\documentclass[matann3,numbook]{svjour1}

\usepackage{amsmath}
\usepackage{amsfonts}
\usepackage{latexsym}
\usepackage{mathrsfs}

\usepackage{times} 
\usepackage{txfonts} 

\usepackage{array}
\usepackage{amscd}
\usepackage{a4}
\usepackage[mathcal]{euscript}
\usepackage{amssymb}
\usepackage{epsf, epic, eepic} 
\usepackage{pstricks}
\usepackage[alphabetic]{amsrefs} 
\usepackage{theorem}

\spnewtheorem*{remark*}{Remark}{\bf}{\rm}
\spnewtheorem*{remarks*}{Remarks}{\it}{\rm}
\spnewtheorem*{aside}{Aside}{\it}{\rm}
\spnewtheorem{vexercise}{Exercise}[section]{\it}{\rm}
\spnewtheorem*{sketchproof}{Sketch proof}{\it}{\rm}
\spnewtheorem*{proofs}{Proofs}{\it}{\rm}
\spnewtheorem*{proofext}{Proof of the Extension Theorem}{\it}{\rm}
\spnewtheorem*{prooftower}{Proof of the Tower Law}{\it}{\rm}
\spnewtheorem*{theoremA}{Theorem A (The division algorithm)}{\bf}{\it}
\spnewtheorem*{theoremB}{Theorem B (Constructing Fields)}{\bf}{\it}
\spnewtheorem*{theoremC}{Theorem C (Constructible Numbers)}{\bf}{\it}
\spnewtheorem*{theoremD}{Theorem D (Simple Extensions)}{\bf}{\it}
\spnewtheorem*{theoremE}{Theorem E}{\bf}{\it}
\spnewtheorem*{sylow}{Sylow's First Theorem}{\bf}{\it}
\spnewtheorem*{theoremF}{Theorem F (The Extension Theorem)}{\bf}{\it}
\spnewtheorem*{corollaryG}{Corollary G}{\bf}{\it}
\spnewtheorem*{isomthm}{First Isomorphism Theorem}{\bf}{\it}
\spnewtheorem*{factthm}{The Factor Theorem}{\bf}{\it}
\spnewtheorem*{fundthmalg}{Fundamental Theorem of Algebra}{\bf}{\it}
\spnewtheorem*{eisenstein}{Eisenstein Irreducibility
  Theorem}{\bf}{\it}
\spnewtheorem*{reduction}{The Reduction Test}{\bf}{\it}
\spnewtheorem*{ufd}{Unique factorisation in $F[x]$}{\bf}{\it}
\spnewtheorem*{kronecker}{Kronecker's Theorem}{\bf}{\it}
\spnewtheorem*{towerlaw}{The Tower Law}{\bf}{\it}
\spnewtheorem*{bogusdefn}{``Definition''}{\bf}{\it}
\makeatletter
\renewcommand{\paragraph}{\@startsection
{paragraph}{3}{0mm}{\baselineskip}
{-\fontdimen2\font plus -\fontdimen3\font minus -\fontdimen4\font}
{}}%
\makeatother

\numberwithin{paragraph}{section}

\renewcommand{\theparagraph}{{\normalfont{\bfseries 
(\arabic{section}.\arabic{paragraph})}}}
\input xy 
\xyoption{all} 
\xyoption{2cell} 
\input{boxedeps} 
\SetEPSFDirectory{} 
\HideDisplacementBoxes
\SetRokickiEPSFSpecial  
\newtheorem{theoremH}{Theorem H (Galois).}

\newtheorem{galoiscorrespondenceA}{The Galois Correspondence (part 1).}
\newtheorem{galoiscorrespondenceB}{The Galois Correspondence (part
  2).}
\DeclareMathAlphabet{\ams}{U}{msb}{m}{n}
\DeclareMathAlphabet{\goth}{U}{euf}{m}{n}
\def\sss{\scriptstyle}

\def\sl{\text{SL}}
\def\psl{\text{PSL}}

\def\psp{\text{PSP}}
\def\psu{\text{PSU}}

\def\gal{\text{Gal}}

\def\id{\text{id}}

\def\im{\text{im}\,}
\def\ker{\text{ker}\,}
\def\ov{\overline}

\def\imag{\text{i}}

\def\CC{\mathcal C}

\def\LL{\mathcal L}
\def\CC{\mathcal C}

\def\aa{\alpha}
\def\ww{\omega}
\def\bb{\beta}
\def\ss{\sigma}

\def\ll{\lambda}
\def\ev{\varepsilon}
\def\ve{\varepsilon}

\def\re{\text{Re}\,}
\def\im{\text{Im}\,}
\def\Z{\ams{Z}}
\def\R{\ams{R}}
\def\C{\ams{C}}\def\Q{\ams{Q}}
\def\F{\ams{F}}

\def\M{\ams{M}}
\def\0{\mathbf{0}}
\def\quo{/\kern -.45em\sim}
\newpsobject{showgrid}{psgrid}{subgriddiv=1,griddots=10,gridlabels=6pt,gridcolor=red}
\def\ds{\displaystyle}

\def\dv{\,|\,}
\def\lg{\langle}
\def\rg{\rangle}
\def\Langle{\langle\kern -2pt\langle}
\def\Rangle{\rangle\kern -1.9pt\rangle}
\newgray{lightergray}{.85}
\title{Galois Theory -- a first course}

\author{Brent Everitt\thanks{version \today.}}

\institute{{\sc Brent Everitt:}
Department of Mathematics, University of York, York
YO10 5DD, United Kingdom. \email{brent.everitt@york.ac.uk}. 
}

\titlerunning{Galois Theory -- a first course}
\authorrunning{Brent Everitt}

\begin{document}

\maketitle

\begin{pspicture}(0,0)(14,6)
\rput(6.5,3){\BoxedEPSF{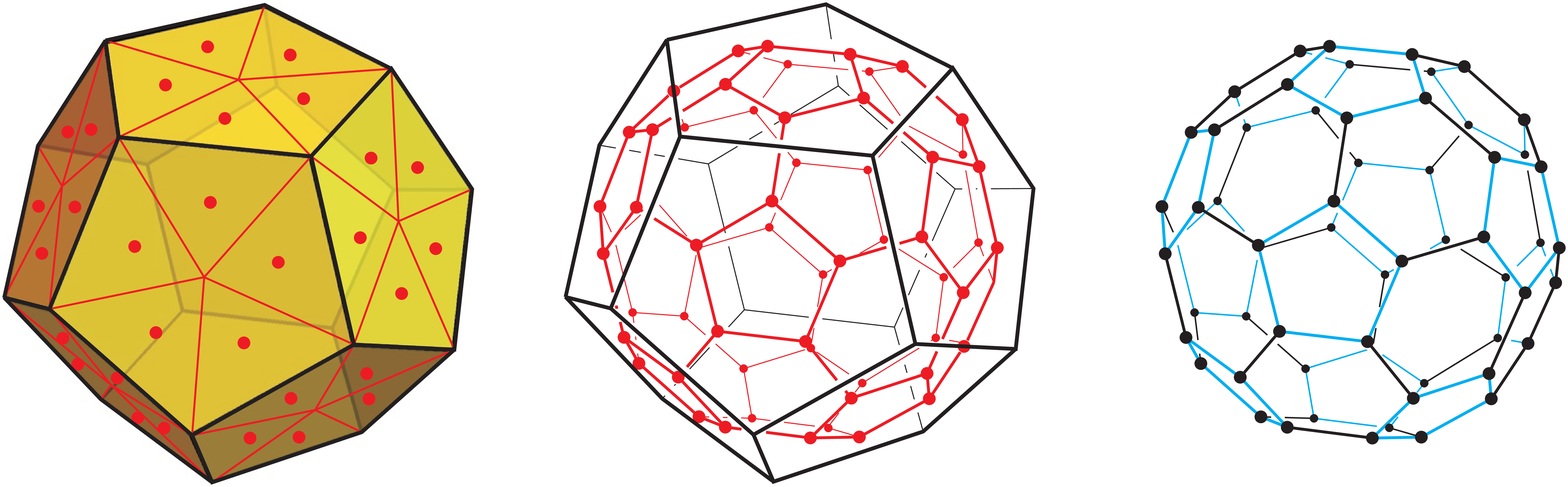 scaled 250}}
\end{pspicture}

\tableofcontents


\section*{Introductory Note}\label{preface}
\label{`preface.page'}

These notes are a self-contained introduction to Galois theory,
designed for the student who has done a first course in abstract algebra. 

To not clutter up the theorems too much, I have made some
restrictions in generality. 
For example, all rings are with 1; all ideals are 
principal; all fields are perfect -- in fact, 
extensions of $\Q$ or of finite fields; consequently all
field extensions are separable; and so on.
This won't be to everyone's taste.

The following prerequisites are assumed, although there are reminders:
the basics of linear algebra, particularly the span and independence
of a set of vectors; the idea of a basis and hence the 
dimension of a vector space. In group theory the fundamentals upto
Lagrange's theorem and the first isomorphism theorem. In ring 
and field theory the definitions and some examples, but probably not
much else.

There are many books on linear algebra and group theory for beginners. My personal favourite is:

Most of the results and proofs are standard and can be found in any book on
Galois theory, but I am particularly indebted to the book of Joseph
Rotman:

In particular the proofs I give of 
Theorems C and E, the Fundamental Theorem of Algebra and the Theorem
of Abels-Ruffini are Rotman's
proofs with some elaboration added. 
The statements (although not the proofs) of Theorems F and G are also his.

The figure depicting the $(a,b)$-plane at the end of Section \ref{solving.equations} is
redrawn from the Mathematica poster \emph{Solving the Quintic}. 


\subsection*{The Cover}

The cover shows a Cayley graph for the smallest non-Abelian simple
group -- 
the alternating group $A_5$. 
We will see that the simplicity of this group means there
is no formula for the roots of the polynomial $x^5-4x+2$, using only
the ingredients
$$
\frac{a}{b}\in\Q, +,-,\times,\div,\sqrt[2]{},\sqrt[3]{},\sqrt[4]{},\sqrt[5]{},\ldots
$$
Therefore, there can be no formula for the solutions of a quintic equation
$$ax^5+bx^4+cx^3+dx^2+ex+f=0$$
that works for all possible $a,b,c,d,e,f\in\C$.

\begin{figure}
  \centering
\begin{pspicture}(0,0)(14,12)
\rput(0,0){
\rput(7,6){\BoxedEPSF{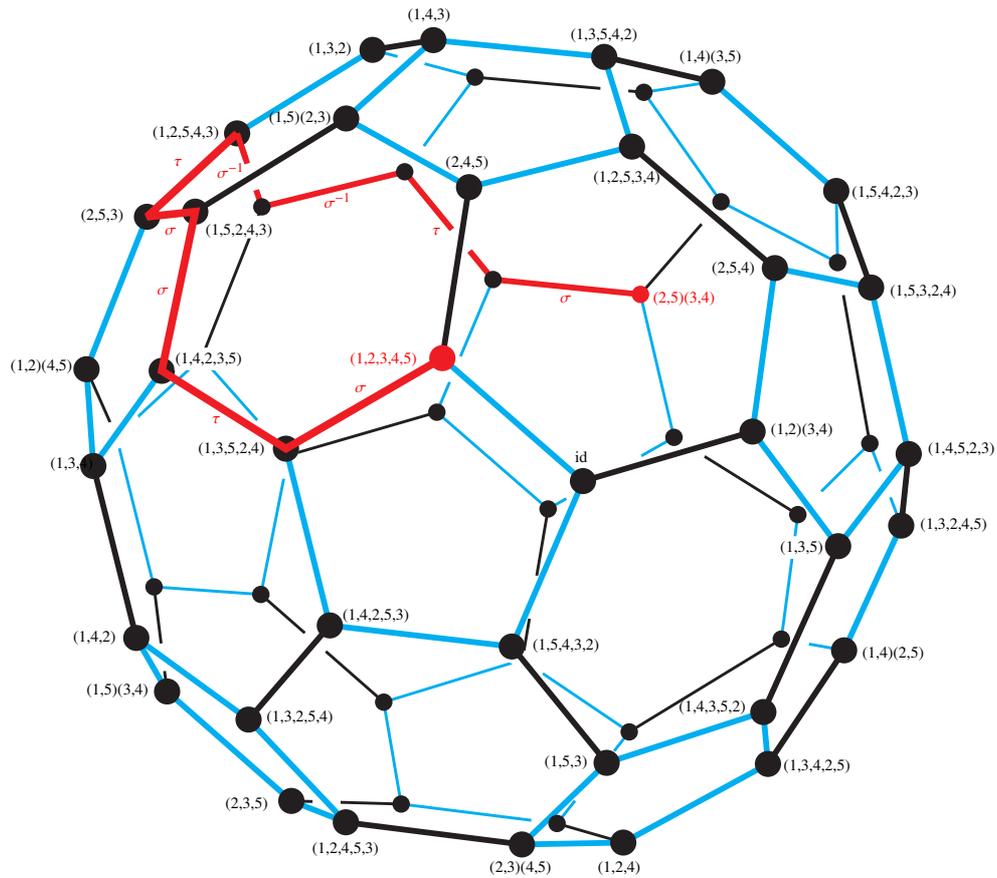 scaled 725}}
\rput(5.2,6.7){${\sss{\red\ss}}$}
\rput(3.3,6.3){${\sss{\red\tau}}$}
\rput(2.6,8){${\sss{\red\ss}}$}
\rput(2.7,8.8){${\sss{\red\ss}}$}
\rput(2.8,9.7){${\sss{\red\tau}}$}
\rput(3.5,9.6){${\sss{\red\ss^{-1}}}$}
\rput(4.9,9.2){${\sss{\red\ss^{-1}}}$}
\rput(6.2,8.8){${\sss{\red\tau}}$}
\rput(7.9,7.9){${\sss{\red\ss}}$}
\rput(8.1,5.8){${\sss\text{id}}$}
\rput(5.5,7.1){${\sss{\red (1,2,3,4,5)}}$}
\rput(3.5,5.9){${\sss (1,3,5,2,4)}$}
\rput(5.4,3.7){${\sss (1,4,2,5,3)}$}
\rput(7.9,3.3){${\sss (1,5,4,3,2)}$}
\rput(11,6.15){${\sss (1,2)(3,4)}$}
\rput(6.6,9.7){${\sss (2,4,5)}$}
\rput(4.8,11.2){${\sss (1,3,2)}$}
\rput(4.4,10.3){${\sss (1,5)(2,3)}$}
\rput(10.1,8.3){${\sss (2,5,4)}$}
\rput(10.7,4.45){\psframe*[linecolor=white](0,0)(0.6,.3)}
\rput(11,4.6){${\sss (1,3,5)}$}
\rput(9.825,2.475){${\sss (1,4,3,5,2)}$}
\rput(8.4,11.4){${\sss (1,3,5,4,2)}$}
\rput(6.1,11.65){${\sss (1,4,3)}$}
\rput(8.7,9.5){${\sss (1,2,5,3,4)}$}
\rput(7.9,1.75){${\sss (1,5,3)}$}
\rput(9.8,11.1){${\sss (1,4)(3,5)}$}
\rput(12.1,9.3){${\sss (1,5,4,2,3)}$}
\rput(2.9,10.1){${\sss (1,2,5,4,3)}$}
\rput(12.6,8){${\sss (1,5,3,2,4)}$}
\rput(13.1,5.9){${\sss (1,4,5,2,3)}$}
\rput(9.45,7.9){${\sss{\red (2,5)(3,4)}}$}
\rput(3.2,8.7){\psframe*[linecolor=white](0,0)(1.2,.3)}
\rput(3.6,8.8){${\sss (1,5,2,4,3)}$}
\rput(1.8,9){${\sss (2,5,3)}$}
\rput(2.75,6.95){\psframe*[linecolor=white](0,0)(1.2,.3)}
\rput(3.2,7.1){${\sss (1,4,2,3,5)}$}
\rput(1.4,5.7){${\sss (1,3,4)}$}
\rput(1,7){${\sss (1,2)(4,5)}$}
\rput(1.7,3.4){${\sss (1,4,2)}$}
\rput(4.4,2.35){${\sss (1,3,2,5,4)}$}
\rput(2,2.7){${\sss (1,5)(3,4)}$}
\rput(13,4.9){${\sss (1,3,2,4,5)}$}
\rput(12.2,3.2){${\sss (1,4)(2,5)}$}
\rput(11.2,1.7){${\sss (1,3,4,2,5)}$}
\rput(5,0.6){${\sss (1,2,4,5,3)}$}
\rput(3.7,1.2){${\sss (2,3,5)}$}
\rput(7.3,0.35){${\sss (2,3)(4,5)}$}
\rput(8.6,0.35){${\sss (1,2,4)}$}
}
\end{pspicture}
\caption{The Cayley graph for the smallest non-Abelian simple group,
the alternating group $A_5$, with respect to $\ss=(1,2,3,4,5)$ -- the
blue edges -- and $\tau=(1,2)(3,4)$ -- the black edges.}
  \label{fig:the_cover:Cayley_graphA5}
\end{figure}

A Cayley graph is a picture of the multiplication in the
group. Let $\ss=(1,2,3,4,$ $5)$. Each blue pentagonal
face can be oriented anti-clockwise when you look at it from the
outside of the ball. Crossing a blue edge anti-clockwise corresponds
to $\ss$ and crossing in the reverse direction (clockwise)
corresponds to $\ss^{-1}$. Crossing a black edge in either direction
corresponds to the element $\tau=(1,2)(3,4)$.

The vertices correspond to the 60 elements of $A_5$ -- the front ones
are marked, with the identity element in the center. If a path
$\gamma$ starts at the vertex corresponding to $\mu_1\in A_5$ and
finishes at $\mu_2\in A_5$, then
reading the $\ss$ and $\tau$ labels off $\gamma$ as you travel along
it gives $\mu_1\gamma=\mu_2$. For example, the red path gives 
$(1,2,3,4,5)\cdot\ss\tau\ss^{2}\tau\ss^{-2}\tau\ss=(2,5)(3,4)$. 

It is a curious coincidence that the smallest non-Abelian simple group
has Cayley graph the
the simplest known pure form of Carbon -- Buckminsterfullerine $C_{60}$.


\section{What is Galois Theory?}
\label{lect1}

A quadratic equation $ax^2+bx+c=0$ has exactly two -- possibly
repeated -- solutions in the complex numbers.
There is a formula for them,
that appears in the ninth century Persian book 
{\em Hisab al-jabr w'al-muqabala}\footnote{\emph{al-jabr\/}, hence ``algebra''.},
by Abu Abd-Allah ibn Musa al'Khwarizmi.
In modern notation it says:
$$
x=\frac{-b\pm\kern-2pt\sqrt{b^2-4ac}}{2a}.
$$
Less familiar maybe,  
$ax^3+bx^2+cx+d=0$ has three $\C$-solutions, and they too can be
expressed algebraically using 
Cardano's formula.
One solution turns out to be
\begin{equation*}
\begin{split}
-\frac{b}{3a}
&+\sqrt[3]{-\frac{1}{2}\biggl(\frac{2b^3}{27a^3}-\frac{bc}{a^2}+\frac{d}{a}\biggr)
+\sqrt{\frac{1}{4}
\biggl(\frac{2b^3}{27a^3}-\frac{bc}{a^2}+\frac{d}{a}\biggr)^2
+\frac{1}{27}\biggl(\frac{c}{a}-\frac{b^2}{3a^2}\biggr)^3}}\\
&+\sqrt[3]{-\frac{1}{2}\biggl(\frac{2b^3}{27a^3}-\frac{bc}{a^2}+\frac{d}{a}\biggr)
-\sqrt{\frac{1}{4}
\biggl(\frac{2b^3}{27a^3}-\frac{bc}{a^2}+\frac{d}{a}\biggr)^2
+\frac{1}{27}\biggl(\frac{c}{a}-\frac{b^2}{3a^2}\biggr)^3}},
\end{split}
\end{equation*}
and the other two have similar expressions.
There is an even more complicated formula, attributed to Descartes, for
the roots of a quartic polynomial equation. 

What is kind of miraculous is not that the solutions exist, but they can always
be expressed in terms of the coefficients and the basic algebraic
operations,
$$
+,-,\times,\div,\sqrt{},\sqrt[3]{},\sqrt[4]{},\sqrt[5]{},\ldots
$$
By the turn of the 19th century, no equivalent formula for the solutions to a
quintic (degree five) polynomial equation had materialised, and it was Abels who
had the crucial realisation: {\em no such formula exists}.

Such a statement can be interpreted in a number of ways. Does it mean that there are
always algebraic expressions for the roots of quintic polynomials, but their form is
too complex for one {\em single\/} formula to describe all the possibilities? It would therefore
be necessary to have a number, maybe even infinitely many, formulas. The reality turns out to
be far worse: there are specific polynomials, such as $x^5-4x+2$, whose solutions cannot be
expressed algebraically in any way whatsoever. 

A few decades later, Evarist\'{e} Galois started thinking about the
deeper problem: {\em why\/} don't these formulae exist?
Thus, Galois theory was originally motivated by the desire to understand,
in a much more precise way, the solutions to
polynomial equations.

Galois' idea was this: study the solutions
by studying their ``symmetries''. 
Nowadays, when we hear the word symmetry, we normally think of
group theory. 
To reach his conclusions, Galois kind of invented group theory
along the way.
In studying the
symmetries of the solutions to a polynomial, Galois theory establishes
a link between these two areas of mathematics.
We illustrate the idea, in a somewhat loose manner, with an example.

\subsection{The symmetries of the solutions to $x^3-2=0$.}

\paragraph{\hspace*{-0.3cm}}
\parshape=2 0pt\hsize 0pt.75\hsize 
We work in $\C$. Let $\aa$ be the real cube
root of $2$, ie: $\aa=\sqrt[3]{2}\in\R$ and,
$\ww=-\frac{1}{2}+\frac{\sqrt{3}}{2}\imag.$ 
Note that $\ww$ is
a cube root of $1$, and so $\ww^3=1$.
\vadjust{\hfill\smash{\lower 72pt
\llap{
\begin{pspicture}(0,0)(3,2)
\rput(1,0.25){
\uput[0]{270}(-.1,2){\pstriangle[fillstyle=solid,fillcolor=lightgray](1,0)(2,1.73)}
\rput(2,1){$\aa$}
\rput(-0.2,-0.2){$\aa\ww^2$}
\rput(-0.20,2.2){$\aa\ww$}
\rput(-0.7,1){${\red s}$}\rput(1.3,2.2){${\red t}$}
\psline[linecolor=red](-0.5,1)(1.75,1)
\psline[linecolor=red](0.08,0)(1.26,2)
}
\end{pspicture}
}}}\ignorespaces

\parshape=7 0pt.75\hsize 0pt.75\hsize 0pt.75\hsize 0pt.75\hsize
0pt.75\hsize 0pt.75\hsize 0pt\hsize  
The three solutions to $x^3-2=0$ (or {\em roots\/} of $x^3-2$) are the
complex numbers $\aa,\aa\ww$ and $\aa\ww^2$, forming the vertices of
the equilateral triangle shown. The triangle has what we might
call 
``geometric symmetries'': three reflections, a
counter-clockwise rotation through $\frac{1}{3}$ of a turn,
a counter-clockwise rotation through $\frac{2}{3}$ of a turn and
a counter-clockwise rotation through $\frac{3}{3}$ of a turn $=$ the
identity symmetry. Notice for now that if $s$ and $t$ are the
reflections in the lines shown, the geometrical symmetries are
$s$, $t$, $tst$, $ts$, $(ts)^2$ and $(ts)^3=\id$ (read these
expressions  from right to left).

The symmetries referred to in the preamble are not so much geometric
as ``number theoretic''.
It will take a little explaining before we see what this means. 

\begin{definition}[field -- version ${\mathbf 1}$]
\label{def_field1}
A field is a set $F$ with two operations, called, purely
  for convenience, $+$ and $\times$, 
such that for any $a,b,c\in F$,
\begin{enumerate}
\item $a+b$ and $a\times b$ ($=ab$ from now on) are uniquely defined elements of $F$,
\item $a+(b+c)=(a+b)+c$,
\item $a+b=b+a$, 
\item there is an element $0\in F$ such that $0+a=a$,
\item for any $a\in F$ there is an element $-a\in F$ with $(-a)+a=0$,
\item $a(bc)=(ab)c$,
\item $ab=ba$,
\item there is an element $1\in F\setminus\{0\}$ with $1\times a=a$,
\item for any $a\not= 0\in F$ there is an $a^{-1}\in F$ with $aa^{-1}=1$,
\item $a(b+c)=ab+ac$.
\end{enumerate}
\end{definition}

A field is just a set of things that you can add, subtract, multiply and
divide so that the ``usual'' rules of algebra are satisfied.
Familiar examples of fields are $\Q$, $\R$ and $\C$; familiar non-examples of
fields are $\Z$, polynomials and matrices (you cannot in general
divide integers, polynomials and matrices to get integers, polynomials
or matrices). 

\paragraph{\hspace*{-0.3cm}}
A {\em subfield\/}
of a field $F$ is a subset that also forms a field under the
same $+$ and $\times$. Thus, $\Q$ is a subfield of $\R$ which
is in turn a subfield of $\C$, and so on. On the other hand,
$\Q\cup\{\kern-2pt\sqrt{2}\}$ is not a subfield of $\R$: it is a subset
but axiom 1 fails, as both $1$ and $\kern-2pt\sqrt{2}$ are elements but $1+\kern-2pt\sqrt{2}$
is not.

\begin{definition}
\label{def:section0_defn10}
If $F$ is a subfield of the complex numbers $\C$ and $\bb\in\C$, then $F(\bb)$
is the
smallest subfield of $\C$ that contains both $F$ and the
number $\bb$.
\end{definition}

What do we mean by smallest? That there is no other 
field $F'$ having the same properties as $F(\bb)$ which is smaller, ie: no $F'$ with
$F\subset F'\text{ and }\bb\in F'\text{ too,}$
but $F'$ properly $\subset F(\bb)$. It is usually more useful to say it the
other way around:
\begin{equation}\label{eq1}\text{If $F'$ is a subfield }
\text{ that also contains }F\text{ and }\bb,
\text{ then $F'$ contains }F(\bb)\text{ too}\tag{*}.\end{equation}
Loosely speaking, $F(\bb)$ is all the complex
numbers we get by adding, subtracting, multiplying and dividing the
elements of $F$ and $\bb$ together in all possible ways.

The construction of Definition \ref{def:section0_defn10} can be
continued: write $F(\bb,\gamma)$ for the
smallest subfield of $\C$ containing $F$ and the
numbers $\bb$ and $\gamma$, and so on.

\paragraph{\hspace*{-0.3cm}}
To illustrate with some trivial examples, $\R(\imag)$ can be shown to be all
of $\C$: it must contain all expressions of the form $b\imag$
for $b\in\R$, and hence all expressions of the form $a+b\imag$ with $a,b\in\R$,
and this accounts for all the complex numbers; $\Q(2)$ is
equally clearly just $\Q$ back again. 

Slightly less trivially, $\Q(\kern-2pt\sqrt{2})$, the smallest subfield of $\C$
containing all the rational numbers and $\kern-2pt\sqrt{2}$, is a field
that is strictly bigger than $\Q$ (eg: it contains $\kern-2pt\sqrt{2}$) but is
much, much smaller than all of $\R$.

\begin{vexercise}
Show that $\kern-2pt\sqrt{3}\not\in\Q(\kern-2pt\sqrt{2})$.
\end{vexercise}

\paragraph{\hspace*{-0.3cm}}
Returning to the symmetries of the solutions to 
$x^3-2=0$, we look at the field $\Q(\aa,\ww)$, where 
$\aa=\sqrt[3]{2}\in\R\text{ and }
\ww=-\frac{1}{2}+\frac{\sqrt{3}}{2}\imag,$
as before. Since $\Q(\aa,\ww)$ is by definition a field, and fields are closed under
$+$ and $\times$, we have
$$
\aa\in\Q(\aa,\ww)\text{ and }\ww\in\Q(\aa,\ww)\Rightarrow
\aa\times\ww=\aa\ww,\aa\times\ww\times\ww=\aa\ww^2\in
\Q(\aa,\ww)\text{ too.}
$$
So, 
$\Q(\aa,\ww)$ contains all the solutions to the equation $x^3-2=0$. 
On the other hand:

\begin{vexercise}
Show that $\Q(\aa,\ww)$ has ``just enough'' numbers to solve the
equation $x^3-2=0$. More precisely, $\Q(\aa,\ww)$ is the {\em
smallest\/} subfield of $\C$ that contains all the solutions to this
equation.
({\em hint\/}: you may find it useful to do Exercise \ref{ex_lect1.0} first).
\end{vexercise}

\paragraph{\hspace*{-0.3cm}}
A very loose definition of a symmetry of the solutions of $x^3-2=0$ is
that it is a
``rearrangement" of $\Q(\aa,\ww)$ that does not disturb (or is compatible with) the $+$ and
$\times$.

To see an example, consider the two fields $\Q(\aa,\ww)$
and $\Q(\aa,\ww^2)$. Despite first appearances they are
actually the same: certainly 
$$
\aa,\ww\in\Q(\aa,\ww)\Rightarrow \aa,\ww^2\in\Q(\aa,\ww).
$$
But $\Q(\aa,\ww^2)$ is the smallest field containing $\Q,\aa$ and $\ww^2$, so by (*),
$$
\Q(\aa,\ww^2)\subseteq\Q(\aa,\ww).
$$
Conversely, 
$$\aa,\ww^2\times\ww^2=\ww^4=\ww\in\Q(\aa,\ww^2)\Rightarrow
\Q(\aa,\ww)\subseteq\Q(\aa,\ww^2).
$$
Remember that $\ww^3=1$ so $\ww^4=\ww$. Thus $\Q(\aa,\ww)$ and
$\Q(\aa,\ww^2)$ are indeed the same. In fact, we 
should think of $\Q(\aa,\ww)$ and $\Q(\aa,\ww^2)$ as two different ways
of looking at the same field, or more suggestively, the same field
viewed from two different angles. 

When we hear the phrase, ``the same field
viewed from two different angles'', it suggests that there is a
symmetry that moves the field from one point of
view to the other. In the case above, there should be a symmetry of the
field $\Q(\aa,\ww)$ that puts it into the form $\Q(\aa,\ww^2)$. Surely this symmetry should 
send
$$
\aa\mapsto\aa,\text{ and }\ww\mapsto\ww^2.
$$
We haven't yet defined what we mean by, ``is compatible with the $+$ and $\times$''.
It will turn out to mean that if $\aa$ and $\ww$ are sent to $\aa$ and $\ww^2$ respectively, then
$\aa\times\ww$ should go to $\aa\times\ww^2$; similarly $\aa\times\ww\times\ww$ should go to 
$\aa\times\ww^2\times\ww^2=\aa\ww^4=\aa\ww$, and so on.
The symmetry thus moves the vertices of the equilateral triangle
determined by the roots in the same way that the reflection $s$ of the
triangle does (see Figure \ref{fig:figure2}).

\begin{figure}
  \centering
\begin{pspicture}(0,0)(14,4)
\rput(3,1){
\uput[0]{270}(-.1,2){\pstriangle[fillstyle=solid,fillcolor=lightgray](1,0)(2,1.73)}
\pscurve[linecolor=red]{<->}(-0.5,0)(-1,1)(-0.5,2)
\psbezier[linecolor=red]{->}(2.1,0.8)(3.1,0)(3.1,2)(2.1,1.2)
\rput(2,1){$\aa$}
\rput(-0.2,-0.2){$\aa\ww^2$}
\rput(-0.20,2.2){$\aa\ww$}
\rput(-0.7,1){${\red s}$}
\psline[linecolor=red](-0.5,1)(1.75,1)
}
\rput(9,-0.25){
\rput(0,.5){\rput(0,1){
\rput{270}(-.1,2){\pstriangle[fillstyle=solid,fillcolor=lightgray](1,0)(2,1.73)}
\rput(-0.05,0.05){\pscurve[linecolor=red]{<->}(0.2,2.2)(1.5,2.5)(2,1.2)}
\rput{-120}(-.3,1.95){\psbezier[linecolor=red]{->}(2.1,0.8)(3.1,0)(3.1,2)(2.1,1.2)}
\rput(2,1){$\aa$}
\rput(-0.2,-0.2){$\aa\ww^2$}
\rput(-0.20,2.2){$\aa\ww$}
\rput(1.2,2.1){${\red t}$}
}
\psline[linecolor=red](-.1,1)(1,3)
}}
\end{pspicture}
\caption{The symmetry $\Q(\aa,\ww)=\Q(\aa,\ww^2)$ (\emph{left}) and
  the symmetry $\Q(\aa\ww,\ww^2)=\Q(\aa,\ww)$ (\emph{right}) of the
  equation $x^3-2=0$.}
  \label{fig:figure2}
\end{figure}

(This compatibility also means that it would have made no sense to have the symmetry
send $\aa\mapsto\ww^2$ and $\ww\mapsto\aa$. A symmetry should not
fundamentally change the algebra 
of the field, so that if an element like $\ww$ cubes to give $1$, then its image under the 
symmetry should too: but $\aa$ {\em doesn't\/} cube to give $1$.)

\paragraph{\hspace*{-0.3cm}}
In exactly the same way, we can consider the fields $\Q(\aa\ww,\ww^2)$ and $\Q(\aa,\ww)$. We have
$$
\aa,\ww\in\Q(\aa,\ww)\Rightarrow \ww^2,\aa\ww\in\Q(\aa,\ww)\Rightarrow
\Q(\aa\ww,\ww^2)\subseteq\Q(\aa,\ww);
$$
and conversely,
$\aa\ww,\ww^2\in\Q(\aa\ww,\ww^2)\Rightarrow
\aa\ww\ww^2=\aa\ww^3=\aa\in\Q(\aa\ww,\ww^2)$, and hence also
$$\aa^{-1}\aa\ww=\ww\in\Q(\aa\ww,\ww^2)
\Rightarrow\Q(\aa,\ww)\subseteq\Q(\aa\ww,\ww^2).
$$

Thus, $\Q(\aa,\ww)$ and $\Q(\aa\ww,\ww^2)$ are the same field, and
we can define another symmetry that sends
$$
\aa\mapsto\aa\ww,\text{ and }\ww\mapsto\ww^2.
$$
To be compatible with the $+$ and $\times$,
$$
\aa\times\ww\mapsto\aa\ww\times\ww^2=\aa\ww^3=\aa,\text{ and
}\aa\times\ww\times\ww\mapsto 
\aa\ww\times\ww^2\times\ww^2=\aa\ww^5=\aa\ww^2.
$$
So the symmetry is like the reflection $t$ of the triangle (see
Figure \ref{fig:figure2}).

Finally, if we have two symmetries of the solutions to some equation, we
would like their composition to be a symmetry too. So if the
symmetries $s$ and $t$ of the original triangle are to
be considered, so should $tst,st,(st)^2$ and $(st)^3=1$. 

\paragraph{\hspace*{-0.3cm}}
The symmetries of the solutions to $x^3-2=0$ include all
the geometrical symmetries of the equilateral triangle. We will see later that any symmetry
of the solutions is uniquely determined as a permutation of the solutions. Since there
are $3!=6$ of these, we have accounted for all of them. So the solutions to 
$x^3-2=0$ have symmetry precisely the geometrical symmetries of the equilateral triangle.

\paragraph{\hspace*{-0.3cm}}
If this was always the case, things would be a little disappointing: Galois theory would just be the study
of the ``shapes" formed by the roots of polynomials, and the symmetries of those shapes.
It would 
be a branch of planar geometry.


Fortunately, if we look at the solutions to $x^5-2=0$, given in Figure
\ref{fig:figure2a}, then
something quite different 
happens. Exercise \ref{ex_lect1.1a} shows you how to find these expressions for the
roots.

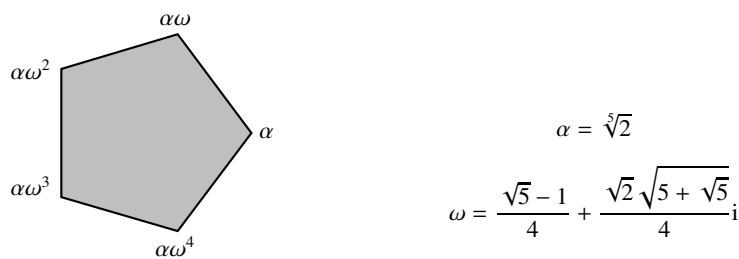
\begin{figure}
  \centering
\begin{pspicture}(14,3.5)
\rput(3,.6){
\rput(2.7,1.3){$\alpha$}\rput(1.5,2.8){$\alpha\ww$}
\rput(-.4,.6){$\alpha\ww^3$}
\rput(-.4,2.15){$\alpha\ww^2$}
\rput(1.5,-.2){$\alpha\ww^4$}
\pspolygon[fillstyle=solid,fillcolor=lightgray](0,.45)(1.53,0)(2.5,1.3)(1.53,2.61)(0,2.15)
}
\rput(10,2){$\alpha=\sqrt[5]{2}$}
\rput(10,1){${\displaystyle \ww=\frac{\sqrt{5}-1}{4}+\frac{\sqrt{2}\sqrt{5+\sqrt{5}}}{4}\imag}$}
\end{pspicture}
\caption{The solutions in $\C$ to the equation $x^5-2=0$.}
  \label{fig:figure2a}
\end{figure}

A pentagon has 10 geometric symmetries, and you can check that all arise as symmetries of the
roots of $x^5-2$ using the same reasoning as in the previous example.
But this reasoning also gives a symmetry that 
moves the vertices
of the pentagon according to:
$$
\begin{pspicture}(0,-.5)(3,3)
\pspolygon[fillstyle=solid,fillcolor=lightgray](0,.45)(1.53,0)(2.5,1.3)(1.53,2.61)(0,2.15)
\psbezier[linecolor=red]{->}(2.85,1.1)(3.85,0.3)(3.85,2.3)(2.85,1.5)
\psline[linecolor=red]{->}(1.45,2.5)(0.05,.5)
\psline[linecolor=red]{->}(1.45,.1)(.05,2.05)
\pscurve[linecolor=red]{->}(.15,.5)(1,.9)(1.5,.1)
\pscurve[linecolor=red]{->}(.15,2.1)(.95,1.73)(1.55,2.5)
\rput(2.7,1.3){$\alpha$}\rput(1.5,2.8){$\alpha\ww$}
\rput(-.4,.6){$\alpha\ww^3$}
\rput(-.4,2.15){$\alpha\ww^2$}
\rput(1.5,-.2){$\alpha\ww^4$}
\end{pspicture}
$$
This is not a geometrical symmetry -- if it was, it would be pretty
disastrous for the poor pentagon. Later we will see that for $p>2$ a prime number, the solutions
to $x^p-2=0$ have $p(p-1)$ symmetries. While agreeing with the six obtained for $x^3-2=0$, it gives 
twenty for $x^5-2=0$. In fact, it was a bit of a fluke that all the number theoretic symmetries
were also geometric ones for $x^3-2=0$. A $p$-gon has $2p$ geometrical
symmetries and $2p\leq p(p-1)$ 
with equality only when $p=3$. 

\subsection*{Further Exercises for Section \thesection}

\begin{vexercise}\label{ex1.-1}
Show that the picture on the left of Figure \ref{fig:figure3} depicts
a symmetry of the solutions to $x^3-1=0$, but the one on the right
does not.

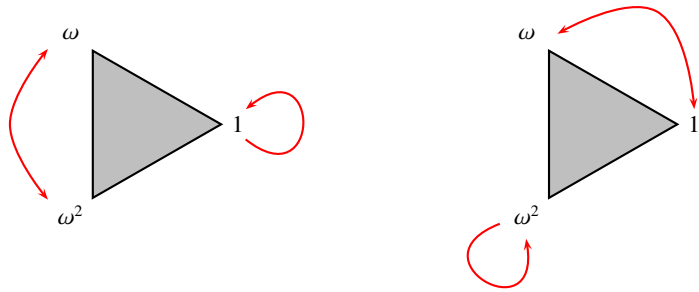
\begin{figure}
  \centering
\begin{pspicture}(14,4)
\rput(3,1.25){
\uput[0]{270}(-.1,2){\pstriangle[fillstyle=solid,fillcolor=lightgray](1,0)(2,1.73)}
\pscurve[linecolor=red]{<->}(-0.5,0)(-1,1)(-0.5,2)
\psbezier[linecolor=red]{->}(2.1,0.8)(3.1,0)(3.1,2)(2.1,1.2)
\rput(2,1){$1$}
\rput(-0.2,-0.2){$\ww^2$}
\rput(-0.20,2.2){$\ww$}
}
\rput(9,1.25){
\uput[0]{270}(-.1,2){\pstriangle[fillstyle=solid,fillcolor=lightgray](1,0)(2,1.73)}
\pscurve[linecolor=red]{<->}(0.2,2.2)(1.5,2.5)(2,1.2)
\rput{-120}(-.2,1.9){\psbezier[linecolor=red]{->}(2.1,0.8)(3.1,0)(3.1,2)(2.1,1.2)}
\rput(2,1){$1$}
\rput(-0.2,-0.2){$\ww^2$}
\rput(-0.20,2.2){$\ww$}
}
\end{pspicture}
\caption{A symmetry (\emph{left}) and
  non-symmetry (\emph{right}) of the equation $x^3-1=0$ from Exercise \ref{ex1.-1}.}
  \label{fig:figure3}
\end{figure}
\end{vexercise}

\begin{vexercise}\label{ex_lect1.1a}
You already know that the $3$-rd roots of 1 are 
$1$ and ${\displaystyle -\frac{1}{2}\pm\frac{\sqrt{3}}{2}\imag}$. What about the $p$-th roots
for higher primes?
\begin{enumerate}
\item If $\ww\not= 1$ is a $5$-th root it satisfies
$\ww^4+\ww^3+\ww^2+\ww+1=0$. Let
$u=\ww+\ww^{-1}$. Find a quadratic polynomial satisfied by $u$, and solve it to obtain $u$.
\item Find another quadratic satisfied this time by $\ww$, with {\em coefficients involving\/} $u$, and 
solve it to find explicit expressions for the four primitive $5$-th roots of 1.
\item Repeat the process with the $7$-th roots of $1$.
\end{enumerate}

\noindent\emph{factoid\/}: the $n$-th roots of 1 can be expressed 
in terms of field operations and extraction of pure roots of rationals for any $n$. The details 
-- which are a little complicated -- were completed by the work of Gauss and Galois.
\end{vexercise}

\begin{vexercise}\label{ex_lect1.0}
Let $F$ be a field such that the element
$$
\underbrace{1+1+\cdots +1}_{n\text{ times}}\not= 0,
$$
for any $n>0$. 
Arguing intuitively, show that $F$
contains a copy of the rational numbers $\Q$
(see also Section \ref{lect4}).
\end{vexercise}

\begin{vexercise}\label{ex_lect1.1b}
Let $\alpha=\sqrt[6]{5}\in\R$ and ${\ds
\ww=\frac{1}{2}+\frac{\sqrt{3}}{2}\imag}$.
Show that $\Q(\aa,\ww)$, $\Q(\aa\ww^2,\ww^5)$ and $\Q(\aa\ww^4,\ww^5)$
are all the same field.
\end{vexercise}

\begin{vexercise}\label{ex_lect1.1}
\hspace{1em}
\begin{enumerate}
\item
Show that there is a symmetry of the solutions to $x^5-2=0$ that moves
the vertices of the pentagon according to:
$$
\begin{pspicture}(0,-.5)(3,3)
\pspolygon[fillstyle=solid,fillcolor=lightgray](0,.45)(1.53,0)(2.5,1.3)(1.53,2.61)(0,2.15)
\psbezier[linecolor=red]{->}(2.85,1.1)(3.85,0.3)(3.85,2.3)(2.85,1.5)
\psline[linecolor=red]{->}(1.45,2.5)(0.05,.5)
\psline[linecolor=red]{->}(1.45,.1)(.05,2.05)
\pscurve[linecolor=red]{->}(.15,.5)(1,.9)(1.5,.1)
\pscurve[linecolor=red]{->}(.15,2.1)(.95,1.73)(1.55,2.5)
\rput(2.7,1.3){$\alpha$}\rput(1.5,2.8){$\alpha\ww$}
\rput(-.4,.6){$\alpha\ww^3$}
\rput(-.4,2.15){$\alpha\ww^2$}
\rput(1.5,-.2){$\alpha\ww^4$}
\end{pspicture}
$$
$\text{where }\alpha=\sqrt[5]{2}, \text{ and }\ww^5=1, \ww\in\C.$
\item
Show that the solutions in $\C$ to the equation $x^6-5=0$ have $12$
symmetries.
\end{enumerate}
\end{vexercise}



\section{Rings I: Polynomials}\label{lect2}

\paragraph{\hspace*{-0.3cm}}
There are a number of 
basic facts about polynomials 
that we will need. 
Suppose $F$ is a field ($\Q,\R$ or $\C$ will do for now).
A \emph{polynomial over $F$\/} is an expression of the form      
$$
f=a_0+a_1x+\cdots a_nx^n,
$$
where the $a_i\in F$ and $x$ is a ``formal symbol" (sometimes called
an indeterminate). We don't tend to think of $x$ as a variable -- it is
purely an object on which to perform algebraic manipulations.
Denote the set of all polynomials over $F$ by $F[x]$. If $a_n\not= 0$,
then $n$ is called the {\em degree\/} of $f$, written
$\deg(f)$. 
If the
leading coefficient $a_n=1$, then $f$ is {\em monic\/}.

(The degree of a non-zero constant polynomial is thus $0$, but to
streamline some statements
define $\deg(0)=-\infty$, where $-\infty<n$ for all $n\in\Z$.
The arithmetic of degrees is just the arithmetic of
non-negative integers, except we decree that $-\infty+n=-\infty$. A
polynomial $f$ is \emph{constant\/} if $\deg f\leq 0$, and
\emph{non-constant\/} otherwise). 

\paragraph{\hspace*{-0.3cm}}
We can add and multiply elements of $F[x]$ in the usual
way:   
$$
\text{if }f=\sum_{i=0}^n a_ix^i\text{ and }g=\sum_{i=0}^m b_ix^i,
$$
then,
\begin{equation}\label{polyops}
f+g=\sum_{i=0}^{\text{max}(m,n)}(a_i+b_i)x^i\text{ and }fg=
\sum_{k=0}^{m+n} c_k x^k\text{ where }c_k=\sum_{i+j=k}a_ib_j.
\end{equation}
that is, $c_k=a_0b_k+a_1b_{k-1}+\cdots +a_kb_0$. The arithmetic of the
coefficients (ie: how to work out $a_i+b_i, a_ib_j$ and so on) is just that of
the field $F$.

\begin{vexercise}
\hspace{1em}
Convince yourself that this multiplication is really just the ``expanding brackets" 
multiplication of polynomials that you know so well.
\end{vexercise}

\paragraph{\hspace*{-0.3cm}}
The polynomials $F[x]$ together with this addition form an example of
an Abelian group:

\begin{definition}[Abelian group]
An Abelian group is a set $G$ endowed with an operation $(f,g)\mapsto f+g$ 
such that for all $f,g,h\in G$:
\begin{enumerate}
\item $f+g$ is a uniquely defined element of $G$ (closure);
\item $f+(g+h)=(f+g)+h$ (associativity);
\item there is an $0\in G$ such that $0+f=f=f+0$ (identity),;
\item for any $f\in G$ there is an element $-f\in G$ with
  $f+(-f)=0=(-f)+f$ (inverses).
\item $f+g=g+f$ (commutativity).
\end{enumerate}
\end{definition}

We will see more general kinds of groups in Section
\ref{groups.stuff}, where we
will write the operation as juxtaposition.
In an Abelian group however, it is customary to write the operation as
addition, as we have done above. 
In $F[x]$ the identity $0$ is the zero polynomial, and the inverse of $f$
is
$$
-\biggl(\sum_{i=0}^n a_ix^i\biggr)=\sum_{i=0}^n (-a_i)x^i.
$$
(To see that $F[x]$ forms an abelian group, we have $f+g=g+f$ exactly when 
$a_i+b_i=b_i+a_i$ 
for all $i$. But the coefficients of our polynomials come from the field $F$, and 
addition is 
always commutative in a field.)

\paragraph{\hspace*{-0.3cm}}
If we want to include the multiplication, we need the formal
concept of a ring:

\begin{definition}[ring]
A ring is a set $R$ endowed with two operations $(a,b)\mapsto
a+b$ and $a\times b$ such that for 
all $a,b\in R$,
\begin{enumerate}
\item $R$ is an Abelian group under $+$;
\item for any $a,b\in R$, $a\times b$ is a uniquely determined element of
$R$ (closure of $\times$);
\item $a\times(b\times c)=(a\times b)\times c$ (associativity of $\times$);
\item there is an $1\in R$ such that $1\times a=a=a\times 1$ (identity
of $\times$);
\item $a\times(b+ c)=(a\times b)+ (a\times c)$ and $(b+c)\times
  a=(b\times a) +(c\times a)$ (the
distributive law).
\end{enumerate}
\end{definition}

Loosely, a ring is a set on which you can {\em add\/} ($+$), {\em subtract\/} 
(the inverse of $+$ in the 
Abelian group) and {\em multiply\/} ($\times$), but {\em not\/} necessarily divide 
(there is no inverse axiom 
for $\times$).

Here are some well known examples of rings:
$$
\Z,F[x]\text{ for $F$ a field}, \Z_n\text{ and }M_n(F),
$$
where $\Z_n$ is addition and multiplication of integers modulo $n$ and
$M_n(F)$ are the $n\times n$ matrices, with
entries from $F$, together with the usual addition and multiplication of matrices.

A ring is {\em commutative\/} if the second operation $\times$ is commutative: 
$a\times b=b\times a$ for all $a,b$.

\begin{vexercise}
\hspace{1em}
\begin{enumerate}
\item
Show that $fg=gf$ for polynomials $f,g\in F[x]$, hence $F[x]$ is a
commutative ring.
\item
Show that $\Z$ and $\Z_n$ are commutative rings, but $M_n(F)$ is not for
{\em any\/} field $F$ if $n>2$.
\end{enumerate}
\end{vexercise}

\paragraph{\hspace*{-0.3cm}}
The observation that $\Z$ and $F[x]$ are both commutative rings is not just some vacuous 
formalism. A concrete way of putting it is this: at a very fundamental level, 
integers and polynomials share the same algebraic properties. 

When we work with polynomials, we need to be able to add and multiply the coefficients of the polynomials
in a way that doesn't produce any nasty surprises--in other words, the coefficients have to satisfy
the basic rules of algebra that we all know and love. But these basic rules of algebra can be 
found among the
axioms of a ring. Thus, to work with polynomials successfully, all we need is that the 
coefficients come 
from a ring. 

This observation means that for a ring $R$, we can form the set of all polynomials with 
coefficients from $R$ and
add and multiply them together as we did above. In fact, we are just repeating what we did above, but
are replacing the field $F$ with a ring $R$. In practice, rather than
allowing our coefficients to some from an arbitrary ring, we take $R$ to
be commutative. This leads to,

\begin{definition}
Let $R[x]$ be the set of all polynomials with coefficients from some
commutative ring $R$,
together with the $+$ and $\times$ defined at (\ref{polyops}).
\end{definition}

\begin{vexercise}
\hspace{1em}
\begin{enumerate}
\item
Show that $R[x]$ forms a ring.
\item
Since $R[x]$ forms a ring, we can consider polynomials with coefficients
from $R[x]$: take a new variable, say $y$, and consider $R[x][y]$. Show
that this is just the set of polynomials in two variables $x$ and $y$
together with the `obvious' $+$ and $\times$.
\end{enumerate}
\end{vexercise}

\paragraph{\hspace*{-0.3cm}}
A commutative ring $R$ is called an {\em integral domain\/} iff for any $a,b\in R$
with $a\times b=0$, we have $a=0$, or $b=0$ or both. Clearly $\Z$ is an integral domain.

\begin{vexercise}
\hspace{1em}
\begin{enumerate}
\item
Show that any field $F$ is an integral domain.
\item
For what values of $n$ is $\Z_n$ an integral domain?
\end{enumerate}
\end{vexercise}

\begin{lemma}\label{sect2lemma1}
Let $f,g\in R[x]$ for $R$ an integral domain. Then
\begin{enumerate}
\item $\deg(fg)=\deg(f)+\deg(g)$. 
\item $R[x]$ is an integral domain.
\end{enumerate}
\end{lemma}

The second part means that given polynomials $f$ and $g$ (with coefficients from an integral 
domain), we have
$fg=0\Rightarrow f=0$ or $g=0$. You have been implicitly using this fact
when you solve polynomial equations by factorising
them. 

\begin{proof}
We have
$$
fg= \sum_{k=0}^{m+n} c_k x^k\text{ where }c_k=\sum_{i+j=k}a_ib_j,
$$
so in particular $c_{m+n}=a_nb_m\not= 0$ as $R$ is an integral domain.
Thus $\deg(fg)\geq m+n$ and since the reverse
inequality is obvious, we have part (1) of the Lemma. Part (2) now
follows immediately since $fg=0\Rightarrow\deg(fg)=-\infty\Rightarrow
\deg f+\deg g=-\infty$, which can
only happen if at least one of $f$ or $g$ has degree $=-\infty$ (see the
footnote at the bottom of the first page). 
\qed
\end{proof}

All your life you have been happily adding the degrees of polynomials when you multiply them. But
as Lemma \ref{sect2lemma1} shows, {\em this is only possible when the
  coefficients of the polynomial come from 
an integral domain\/}. For example, $\Z_6$, the integers under addition and multiplication
modulo $6$, is a ring that is not an integral domain (as $2\times 3=0$ for example), and sure enough,
$$
(3x+1)(2x+1)=5x+1,
$$ 
where all of this is happening in $\Z_6[x]$.

\paragraph{\hspace*{-0.3cm}}
Although we cannot necessarily divide two polynomials and get another polynomial, we 
{\em can\/} divide upto a possible ``error term", or, as it is more
commonly called, a remainder.

\begin{theoremA}
Suppose  $f$ and $g$ are elements of $R[x]$ where the leading coefficient of 
$g$ has a multiplicative inverse in the ring $R$. Then there exist $q$ and $r$ in
$R[x]$ (quotient and remainder) such that
$$
f=qg+r,
$$
where 
the degree of $r$ is $<$ the degree of $g$.
\end{theoremA}

When $R$ is a field (where you may be more used to doing long division)
all the non-zero coefficients of a polynomial have multiplicative inverses (as
they lie in a field) so the condition on $g$ becomes $g\not= 0$.

\begin{proof}
For all $q\in R[x]$, consider those polynomials of the form 
$f-gq$ and 
choose one, say $r$, of smallest degree. 
Let $d=\deg r$ and $m=\deg g$. We claim that $d<m$.
This will give the result, as the $r$ chosen has he form $r=f-gq$ for some $q$,
giving $f=gq+r$.
Suppose that $d\geq m$ and consider
$$
\bar{r}=(r_d)(g_m^{-1})x^{(d-m)}g,
$$
a polynomial since $d-m\geq 0$.
Notice also that we have used
the fact that the leading coefficient of $g$ has a multiplicative inverse.
The leading term 
of $\bar{r}$ is $r_dx^d$, which is also the 
leading term of $r$. Thus, $r-\bar{r}$ has degree $<d$. But 
$r-\bar{r}=f-gq-r_{d}g_{m}^{-1}x^{d-m}g$ by definition, which equals
$f-g(q-r_{d}g_{m}^{-1}x^{d-m})=f-g\bar{q}$, say. Thus $r-\bar{r}$ has the
form $f-g\bar{q}$ too, but with smaller degree than $r$, which was of minimal degree amongst
all polynomials of this form--this is our desired contradiction.
\qed
\end{proof}

\begin{vexercise}
\hspace{1em}
\begin{enumerate}
\item If $R$ is an integral domain, show that the quotient and remainder are unique.
\item Show that the quotient and remainder are not unique when you divide polynomials
in $\Z_6[x]$. 
\end{enumerate}
\end{vexercise}

\paragraph{\hspace*{-0.3cm}}
Other familiar concepts from $\Z$ are those of divisors, common divisors and greatest 
common divisors. Since we need no more algebra to define these notions
than given by
the axioms for a ring, these concepts carry
pretty much straight over to polynomial rings. We will state these in the setting of polynomials
from $F[x]$ for $F$ a field.

\begin{definition}
For $f,g\in F[x]$, we say that $f$ divides $g$ iff $g=fh$ for some $h\in F[x]$. Write
$f\dv g$.
\end{definition}

\begin{definition}
Let $f,g\in F[x]$. Suppose that $d$ is a polynomial satisfying 
\begin{enumerate}
\item $d$ is a common divisor of $f$ and $g$, ie: $d\dv f$ and $d\dv g$;
\item 
if $c$ is a polynomial with $c\dv f$ and $c\dv g$ then $c\dv d$;
\item $d$ is monic.
\end{enumerate}
Then $d$ is called (the) greatest common divisor of $f$ and $g$.
\end{definition}


As with the division algorithm, we have tweaked the definition from
$\Z$ to make it work in $F[x]$. The reason is that we want {\em the\/}
gcd to be unique. In $\Z$ you ensure this by insisting that all gcd's
are positive; in $F[x]$ we insist they are monic. 

\paragraph{\hspace*{-0.3cm}}
$x^2-1$ and $2x^3-2x^2-4x\in\Q[x]$ have greatest common divisor $x+1$:
it is certainly a common divisor as 
$x^2-1=(x+1)(x-1)$ and $2x^3-2x^2-4x=2x(x+1)(x-2)$. From the two
factorisations, any other common divisor must have the form
$\lambda(x+1)$ for some $\lambda\in\Q$, and so divides $x+1$.

\paragraph{\hspace*{-0.3cm}} They key result on gcd's is:

\begin{theorem}\label{gcd}
Any two $f,g\in F[x]$ have a greatest common divisor $d$. Moreover,
there are $a_0,b_0\in F[x]$ such that
$$
d=a_0f+b_0g.
$$
\end{theorem}

Compare this with $\Z$! You can replace $F[x]$ by $\Z$ in the
following proof to get the corresponding fact for the
integers.

\begin{proof}
Consider the set $I=\{af+bg\,|\,a,b\in F[x]\}$. Let $d\in I$ be a monic
polynomial with minimal degree. Then $d\in I$ gives that $d=a_0f+b_0g$
for some $a_0,b_0\in F[x]$. We claim that $d$ is the gcd of $f$ and $g$.
The following two basic facts are easy to verify:
\begin{enumerate}
\item The set $I$ is a subgroup of the Abelian group $F[x]$--exercise.
\item If $u\in I$ and $w\in F[x]$ then $uw\in I$, since
$wu=w(af+bg)=(wa)f+(wb)g\in I$. 
\end{enumerate}
Consider now the set
$P=\{hd\,|\,h\in F[x]\}$. Since $d\in I$ and by the second observation
above, $hd\in I$, and we have $P\subseteq I$. Conversely, if $u\in I$ then by
the division algorithm, $u=qd+r$ where $r=0$ or $\deg(r)<\deg(d)$. Now,
$r=u-qd$ and $d\in I$, so $qd\in I$ by (2). But $u\in I$ and $qd\in I$ so
$u-dq=r\in I$ by (1) above. Thus, if $\deg(r)<\deg(d)$ we would have a
contradiction to the degree of $d$ being minimal, and so we must have
$r=0$, giving $u=qd$. This means that any element of $I$ is a multiple
of $d$, so $I\subseteq P$.

Now that we know that $I$ is just the set of all multiples of $d$, and
since letting $a=1,b=0$ or $a=0,b=1$ gives that $f,g\in I$, we have that
$d$ is a common divisor of $f$ and $g$. Finally, if $d'$ is another
common divisor, then $f=u_1d'$ and $g=u_2d'$, and since $d=a_0f+b_0g$, we
have $d=a_0u_1d'+b_0u_2d'=d'(a_0u_1+b_0u_2)$ giving $d'\dv d$. Thus $d$
is indeed the greatest common divisor.
\qed
\end{proof}

\paragraph{\hspace*{-0.3cm}}
Here is another fundamental concept:

\begin{definition}[Ring homomorphism]
\label{def_hom}
Let $R$ and $S$ be rings. A mapping $\varphi:R\rightarrow S$ is called a
ring homomorphism if and only if for all $a,b\in R$,
\begin{enumerate}
\item $\varphi(a+b)=\varphi(a)+\varphi(b)$;
\item $\varphi(ab)=\varphi(a)\varphi(b)$;
\item $\varphi(1_R)=1_S$ (where $1_R$ is the multiplicative identity
  in $R$ and $1_S$ the multiplicative identity in $S$).
\end{enumerate}
\end{definition}

The reason we need the last item but not $\varphi(0)=0$ is because
$\varphi(0)=\varphi(0+0)=\varphi(0)+\varphi(0)$, and since $S$ is an
group under addition, we can cancel (using the existence of
inverses under addition!) to get $\varphi(0)=0$. 
We can't do this to get $\varphi(1)=1$ as
we don't
have inverses under multiplication.

You should think of a homomorphism as being like an ``algebraic analogy", or a way of
transferring algebraic
properties; the algebra in the image of $\varphi$ is analogous
to the algebra of $R$. 

\paragraph{\hspace*{-0.3cm}}
We will have more to say about general homomorphisms later; for
now we satisfy ourselves with an example: let
$R[x]$ be a ring of
polynomials over a commutative ring $R$, and let $c\in R$. Define a mapping
$\ev_c:R[x]\rightarrow R$ by 
$$
\ev_c(f)=f(c)\stackrel{\text{def}}{=} a_0+a_1c+\cdots+a_nc^n.
$$
ie: substitute $c$ into $f$. This is a ring homomorphism from $R[x]$
to $R$, called the {\em evaluation at $c$ homomorphism\/}:
to see this, certainly
$\ev_c(1)=1$, and I'll leave $\ev_c(f+g)=\ev_c(f)+\ev_c(g)$ to you.
Now,
$$
\ev_c(fg)=
\ev_c\biggl(\sum_{k=0}^{m+n} d_k x^k\biggr)
=\sum_{k=0}^{m+n} d_k c^k\text{ where }d_k=\sum_{i+j=k}a_ib_j.
$$
But $\sum_{k=0}^{m+n} d_k c^k=\biggl(\sum_{i=0}^n
a_ic^i\biggr)\biggl(
\sum_{j=0}^m b_jc^j\biggr)=\ev_c(f)\ev_c(g)$ and we are done.

One consequence of $\ev_c$ being a
homomorphism is that given a factorisation of a polynomial, say $f=gh$, we have
$\ev_c(f)=\ev_c(g)\ev_c(h)$, ie: if we substitute $c$ into $f$
we
get the same answer as when we substitute into $g$ and $h$ and multiply
the answers. 

\subsection*{Further Exercises for Section \thesection}

\begin{vexercise}
Let $f,g$ be polynomials over the field $F$ and $f=gh$. Show that $h$ is also a 
polynomial over $F$.
\end{vexercise}

\begin{vexercise}\label{ex2.2}
Let $\ss:R\rightarrow S$ be a homomorphism of (commutative)
rings. Define $\ss^*:R[x]\rightarrow S[x]$ by 
$$
\ss^*:\sum_i a_ix^i\mapsto \sum_i \ss(a_i)x^i.
$$
Show that $\ss^*$ is a homomorphism.
\end{vexercise}

\begin{vexercise}\label{ex2.3}
Let $R$ be a commutative ring and define $\partial:R[x]\rightarrow R[x]$ by
$$
\partial:\sum_{k=0}^n a_kx^k\mapsto \sum_{k=1}^{n} (ka_k)x^{k-1}\text{
and } \partial(a)=0,
$$
for any constant $a$.
(Ring a bell?) Show that $\partial(f+g)=\partial(f)+\partial(g)$ and 
$\partial(fg)=\partial(f)g+f\partial(g)$. The map $\partial$ is called
the {\em formal derivative\/}.
\end{vexercise}

\begin{vexercise}\label{ex2.4}
Let $p$ be a fixed polynomial in the ring $F[x]$ and consider the map
$\ev_p:F[x]\rightarrow F[x]$ given by $f(x)\mapsto f(p(x))$. Show
that $\ve _p$ is a homomorphism. (The homomorphism $\ve_p$
is a generalisation of the evaluation at $\lambda$ homomorphism $\ve_\lambda$.)
\end{vexercise}


\section{Roots and Irreducibility}\label{lect3}

\paragraph{\hspace*{-0.3cm}}
The early material in this section is familiar for polynomials with
real coefficients. The point is that these results are
still true for polynomials with coefficients coming from an arbitrary
field $F$, and quite often, for polynomials with coefficients from a
ring $R$.

Let 
$$
f=a_0+a_1x+\cdots +a_nx^n
$$
be a polynomial in $R[x]$ for $R$ a ring. We say that $c\in R$ is a {\em root\/}
of $f$ if
$$
f(c)=a_0+a_1c+\cdots+a_nc^n=0\text{ in }R.
$$
As a trivial example, the polynomial $x^2+1$ is in all three
rings $\Q[x],\R[x]$ and $\C[x]$. It has no roots in either $\Q$ or $\R$,
but two in $\C$. 

\paragraph{\hspace*{-0.3cm}} We start with a familiar result:

\begin{factthm}
An element $c\in R$ is a root of $f$ if and only if $f=(x-c)g$ for
some $g\in R[x]$. 
\end{factthm}

In English, $c$ is a root precisely when $x-c$ is a factor.

\begin{proof}
This is an illustration of the power of the division algorithm, Theorem A. Suppose
that $f$ has the form $(x-c)g$ for some $g\in R[x]$. Then 
$$
f(c)=(c-c)g(c)=0.g(c)=0,
$$
so that $c$ is indeed a root (notice we used that $\ev_c$ is a
homomorphism, ie: that $\ev_c(f)=\ev_c(x-c)\ev_c(g)$). 
On the other hand, by the division
algorithm, we can divide $f$ by the polynomial $x-c$ to get,
$$
f=(x-c)g+a,
$$
where $a\in R$ (we can use the division algorithm, as the leading
coefficient of $x-c$, being $1$, has 
an inverse in $R$). 
Since $f(c)=0$, we must also have
$(c-c)g+a=0$, hence $a=0$. Thus $f=(x-c)g$ as required.
\qed
\end{proof}

\paragraph{\hspace*{-0.3cm}}
Here is another familiar result that is
reassuringly true for polynomials over (almost) any ring.

\begin{theorem}\label{degree.number.of.roots}
Let $f\in R[x]$ be a non-zero polynomial with coefficients from the
integral domain $R$. Then $f$ has at
most $\deg(f)$ roots in $R$.
\end{theorem}

\begin{proof}
We use induction on the degree, which is $\geq 0$ since $f$ is
non-zero. If $\deg(f)=0$ then $f=\mu$ a nonzero constant in $R$, which clearly has no
roots, so the result holds. Assume $\deg(f)\geq 1$ and that
the result is true for any polynomial of degree $<\deg(f)$. If $f$ has no roots
in $R$ then we are done. Otherwise, $f$ has a root $c\in R$ and
$$
f=(x-c)g,
$$
for some $g\in R[x]$ by the Factor Theorem. Moreover, as $R$ is an
integral domain, 
$f(a)=0$ iff either $a-c=0$ or $g(a)=0$, so
the roots of $f$ are
$c$, together with the roots of $g$. Since the degree of $g$ must be
$\deg(f)-1$ (by Lemma \ref{sect2lemma1}, again using the fact that $R$ is an integral domain), 
it has at most $\deg(f)-1$ roots by the inductive hypothesis, and these combined with
$c$ give at most $\deg(f)$ roots for $f$.
\qed
\end{proof}

\paragraph{\hspace*{-0.3cm}}
A cherished fact such as Theorem \ref{degree.number.of.roots} will not hold
if the coefficients
do not come from an integral domain. For instance, if $R=\Z_6$, then the
quadratic polynomial $(x-1)(x-2)=x^2+3x+2$ has  roots $1,2,4$ and $5$ in $\Z_6$.

\begin{vexercise}\label{ex3.20}
A polynomial like $x^2+2x+1=(x+1)^2$ has $1$ as a {\em repeated root\/}.
It's derivative, in the sense of calculus, is $2(x+1)$, which also
has $1$ as a root. In general, and in light of the Factor Theorem, call 
$c\in F$ a repeated root of $f$ iff $f=(x-c)^kg$ for some $k>1$.
\begin{enumerate}
\item Using the formal derivative $\partial$ (see Exercise \ref{ex2.3}), show that
$c$ is a repeated root of $f$ if and only if $c$ is a root of $\partial(f)$.
\item Show that the roots of $f$ are distinct if and only
if $\gcd(f,\partial(f))=1$.
\end{enumerate}
\end{vexercise}

\paragraph{\hspace*{-0.3cm}}
For reasons that will become clearer later, a very important role 
is played by polynomials that cannot be ``factorised".

\begin{definition}[irreducible polynomial over ${\mathbf F}$]
Let $F$ be a field and $f\in F[x]$ a non-constant polynomial. A 
non-trivial factorisation
of $f$ is an expression of the form $f=gh$, where $g,h\in F[x]$ and
$\deg g,\deg h\geq 1$ (equivalently, $\deg g,\deg h<\deg f$).
Call $f$ reducible over $F$ iff it has a non-trivial factorisation,
and irreducible over $F$ otherwise. 
\end{definition}

Thus, a polynomial over a field $F$ is irreducible
precisely when it cannot be written as a product of non-constant
polynomials. Put another way,
$f\in F[x]$ is irreducible precisely when it is divisible only
by a constant $c\in F$, or $c f$. 

\begin{aside}
For polynomials over a ring the definition is slightly more complicated:
let $f\in R[x]$ a non-constant polynomial with coefficients from the
ring $R$. A {\em non-trivial factorisation\/}
of $f$ is an expression of the form $f=gh$, where $g,h\in R[x]$ and either,
\begin{enumerate}
\item $\deg g,\deg h\geq 1$, or
\item if either $g$ or $h$ is a constant $\lambda\in R$, then
  $\lambda$ has {\em no\/} multiplicative inverse in $R$. 
\end{enumerate}
Say $f$ is {\em reducible\/} over $R$ iff it has a non-trivial 
factorisation, and {\em irreducible\/} 
over $R$ otherwise. 
If $R=F$ a field, then the second possibility never arises, 
as every non-zero
element of $F$ has a multiplicative inverse. 
As an example, $3x+3=3(x+1)$ is a non-trivial
factorisation in $\Z[x]$ but a trivial one in $\Q[x]$.
\end{aside}

\paragraph{\hspace*{-0.3cm}}
The ``over $F$" that follows reducible or irreducible is crucial;
polynomials are never absolutely reducible or irreducible. For
example $x^2+1$ is irreducible over $\R$ but reducible over 
$\C$. 

There is one exception to the previous sentence:
a linear polynomial $f=ax+b\in F[x]$ is irreducible over any field $F$.
If $f=gh$ then since $\deg f=1$, we cannot have both $\deg(g), \deg(h)\geq
1$, for then $\deg(gh)=\deg(g)+\deg(h)\geq 1+1=2$, a contradiction. 
Thus, one of $g$ or $h$ must be a constant with $f$ thus irreducible
over $F$.

\begin{vexercise}\label{ex3.1}
\begin{enumerate}
\item Let $F$ be a field and $a\in F$. Show that $f$ is an
irreducible polynomial over $F$ if and only if $a f$ is
irreducible over $F$ for any $a\not= 0$.
\item Show that if $f(x+a)$ is irreducible over $F$ then $f(x)$ is too.
\end{enumerate}
\end{vexercise}

\paragraph{\hspace*{-0.3cm}} There is the famous:
\begin{fundthmalg}
Any non-constant $f\in\C[x]$ has a root in $\C$.
\end{fundthmalg}
So if $f\in\C[x]$ has $\deg f\geq 2$, then $f$ has a root in $\C$, hence a linear factor 
over $\C$, hence is reducible over $\C$. Thus, the only irreducible polynomials over $\C$ are
the linear ones.


\begin{vexercise}\label{ex3.21}
Show that if $f$ is irreducible over $\R$ then $f$ is either linear or quadratic.
\end{vexercise}

\paragraph{\hspace*{-0.3cm}}
A common mistake is to equate having no roots in $F$ with being
irreducible over $F$. 
But:
\begin{description}
\item[--] {\em A polynomial can be irreducible over $F$ and still have
roots in $F$\/}: we saw above that a linear polynomial
$ax+b$ is always irreducible,
and yet has a root in $F$, namely $-b/a$. It is true though that if a polynomial
$f$ has degree $\geq 2$ and had a root in $F$, then by the factor theorem it would
have a linear factor so would be reducible. Thus, if $\deg(f)\geq 2$ and $f$
is irreducible over $F$, then $f$ has no roots in $F$.
\item[--]{\em A polynomial can have no roots in $F$ but not be irreducible over 
$F$\/}: the polynomial $x^4+2x^2+1=(x^2+1)^2$ is reducible over $\Q$, but
with roots $\pm\imag\not\in\Q$.
\end{description}

\paragraph{\hspace*{-0.3cm}}
There is no general method for deciding if a polynomial over an arbitrary field $F$ 
is irreducible.
The best we can hope for is an ever expanding
list of techniques, of which the first is:

\begin{proposition}\label{irr1}
Let $F$ be a field and $f\in F[x]$ be a polynomial of degree $\leq 3$. 
If $f$ has no roots in $F$ then it is irreducible over $F$.
\end{proposition}

\begin{proof}
Arguing by the contrapositive, 
if $f$ is reducible then $f=gh$ with $\deg g,\deg h\geq 1$. Since $\deg g+\deg h=\deg f\leq 3$,
we must have for $g$ say, that $\deg g=1$. Thus $f=(ax+b)h$ and $f$ has the root
$-b/a$. 
\qed
\end{proof}

\paragraph{\hspace*{-0.3cm}}
For another, possibly familiar, example of a field: let $p$ be a prime and $\F_p$ the set
$\{0,1\ldots,p-1\}$. Define addition and multiplication on this set to
be addition and multiplication of integers modulo $p$. 
You can verify that $\F_p$ is a field by directly checking the axioms. The only
tricky one is the existence of inverses under multiplication: to show
this use the
gcd theorem from Section \ref{lect2}, but for $\Z$ rather than polynomials.

\begin{vexercise}
Show that a field $F$ is an integral domain. Hence show that if $n$ is {\em not} prime,
then the addition and multiplication of integers modulo $n$ is {\em not\/} a field.
\end{vexercise}

Arithmetic modulo $n$, for the various $n$, thus gives the sequence
$$
\F_2,\F_3,\Z_4,\F_5,\Z_6,\F_7,\Z_8,\Z_9,\Z_{10},\F_{11},\ldots
$$
of fields $\F_p$ for $p$ a prime, and rings $\Z_n$ for $n$
composite. In Section \ref{fields2} we will see that there are fields
$\F_4,\F_8$ and $\F_9$ of orders $4,8$ and $9$, but these fields are
not $\Z_4,\Z_8$ or $\Z_9$. They are something quite different.

\paragraph{\hspace*{-0.3cm}}
\label{irreducible_polynomials:paragraph50}
Consider
polynomials with coefficients from $\F_2$ ie: the ring $\F_2[x]$, and
in particular, the polynomial
$$
f=x^4+x+1\in\F_2[x].
$$
Now $0^4+0+1\not= 0\not=1^4+1+1$, so $f$ has no roots in $\F_2$.
This doesn't mean that $f$ is irreducible over $\F_2$, but certainly
any factorisation of $f$ over $\F_2$, if there is one, must be
as a product of two quadratics. Moreover, these quadratics must
themselves be irreducible over $\F_2$, for if not, they would factor
into linear factors and the
factor theorem would then give roots of $f$. 

There are only four quadratics over $\F_2$:
$$
x^2,x^2+1,x^2+x\text{ and }x^2+x+1
$$
with $x^2=xx, x^2+1=(x+1)^2$ and $x^2+x=x(x+1)$. You might have to
stare at the second of these factorisations for a second.
By 
Proposition \ref{irr1} $x^2+x+1$ is irreducible. Thus, any factorisation of $f$
into irreducible quadratics must be of the form,
$$
(x^2+x+1)(x^2+x+1).
$$
But, $f$ {\em doesn't\/} factorise this way -- just expand the
brackets. Thus $f$ is irreducible over
$\F_2$.

\paragraph{\hspace*{-0.3cm}}
The most important field for the Galois theory of these notes is the
rationals $\Q$. 
Consequently,
determining the irreducibility of polynomials over $\Q$ will be of
great importance to us. 
The first useful test for irreducibility over $\Q$ has the following
main ingredient: to see if a polynomial can be factorised over $\Q$ it 
suffices to see whether it can be factorised over $\Z$. 

First we recall Exercise \ref{ex2.2}, which is used a number of times in these notes so
is worth placing in a,

\begin{lemma}\label{lemma3.20}
Let $\ss:R\rightarrow S$ be a homomorphism of rings. Define $\ss^*:R[x]\rightarrow S[x]$ by
$$
\ss^*:\sum_i a_ix^i\mapsto \sum_i \ss(a_i)x^i.
$$
Then $\ss^*$ is a homomorphism.
\end{lemma}

\begin{lemma}[Gauss]
Let $f$ be a polynomial with integer coefficients. Then $f$ can be factorised
non-trivially as a product
of polynomials with integer coefficients if and only if it can be 
factorised non-trivially as a product
of polynomials with rational coefficients.
\end{lemma}

\begin{proof}
If the polynomial can be written as a product of $\Z$-polynomials then it clearly can 
as a product of $\Q$-polynomials as integers are rational. Suppose on
the other hand that
$f=gh$ in $\Q[x]$ is a non-trivial factorisation. By multiplying through by a multiple
of the denominators of the coefficients of $g$ we get a polynomial $g_1=m g$ with 
$\Z$-coefficients. Similarly we have $h_1=nh\in\Z[x]$ and so 
\begin{equation}\label{gauss.lemmaeq}
mnf=g_1h_1\in\Z[x].
\end{equation}
Now let $p$ be a prime dividing $mn$, and consider the homomorphism $\ss:\Z\rightarrow \F_p$
given by $\ss(k)=k\mod p$. Then by the lemma above, the map $\ss^*:\Z[x]\rightarrow
\F_p[x]$ given by
$$
\ss^*:\sum_i a_ix^i\mapsto \sum_i \ss(a_i)x^i,
$$
is a homomorphism. Applying the homomorphism to (\ref{gauss.lemmaeq}) gives
$0=\ss^*(g_1)\ss^*(h_1)$ in $\F_p[x]$, as $mn\equiv 0\mod p$. 
As the ring $\F_p[x]$ is an integral domain the
only way that this can happen is if one of the polynomials is equal to the zero polynomial
in $\F_p[x]$, ie: one of the original polynomials, say $g_1$, has all of its coefficients
divisible by $p$. Thus we have $g_1=pg_2$ with $g_2\in\Z[x]$, and (\ref{gauss.lemmaeq})
becomes 
$$
\frac{mn}{p}f=g_2h_1.
$$
Working our way through all the prime factors of $mn$ in this way, we can remove the
factor of $mn$ from (\ref{gauss.lemmaeq}) and obtain a factorisation of $f$ into 
polynomials with $\Z$-coefficients.
\qed
\end{proof}

So to determine whether a polynomial with $\Z$-coefficients is
irreducible over $\Q$, you need only check that it has no non-trivial factorisations
with all the coefficients integers. 

\begin{eisenstein}
Let 
$$
f=c_nx^n+\cdots +c_1x+c_0,
$$
be a polynomial with integer coefficients. If there is a
prime $p$ that divides all the $c_i$ for $i<n$, does not divide $c_n$,
and such that
$p^2$ does not divide $c_0$, then $f$ is irreducible over $\Q$.
\end{eisenstein}

\begin{proof}
By virtue of the previous discussion, we need only show that under the conditions stated, 
there is no factorisation of $f$ using integer
coefficients. Suppose otherwise, ie: $f=gh$ with
$$
g=a_rx^r+\cdots +a_0\text{ and }h=b_sx^s+\cdots +b_0,
$$
and the $a_i,b_i\in\Z$. Expanding $gh$ and equating coefficients,
$$
\begin{tabular}{l}
$c_0=a_0b_0$\\
$c_1=a_0b_1+a_1b_0$\\
\hspace*{2em}$\vdots$\\
$c_i=a_0b_i+a_1b_{i-1}+\cdots+a_ib_0$\\
\hspace*{2em}$\vdots$\\
$c_n=a_rb_s$.\\
\end{tabular}
$$
By hypothesis, $p\dv c_0$. 
Write both $a_0$ and $b_0$
as a product of primes, so if $p\dv c_0$, ie: $p\dv a_0b_0$, then $p$ must be one of the 
primes in this factorisation, hence
divides one of $a_0$ or $b_0$.
Thus, either $p\dv a_0$ or $p\dv
b_0$,
{\em but not both\/} (for then $p^2$ would divide $c_0$). Assume that it is $p\dv
a_0$ that we have. Next, $p\dv c_1$, and this coupled with $p\dv a_0$
gives
$p\dv c_1-a_0b_1=a_1b_0$ (If we had assumed $p\dv b_0$, we would still reach this conclusion). 
Again, $p$ must divide one of the these last
two factors, and since we've already decided that it doesn't divide
$b_0$, it must be $a_1$ that it divides. Continuing in this manner, we
get that $p$ divides all the coefficients of $g$, and in particular,
$a_r$. But then $p$ divides $a_rb_s=c_n$, the contradiction we were
after.
\qed
\end{proof}

The proof above is
a good example of the way mathematics is sometimes created. You
start with as few assumptions as possible (in this case that $p$ divides
some of the coefficients of $f$) and proceed towards some sort of
conclusion, imposing extra conditions as and when you need them. In
this way the statement of the theorem writes itself. 

\paragraph{\hspace*{-0.3cm}}
For example
$$
x^5+5x^4-5x^3+10x^2+25x-35,
$$
is irreducible over $\Q$. Even less obviously
$$
x^n-p,
$$
is irreducible over $\Q$ for any prime $p$. Thus, we can find
polynomials over $\Q$ of arbitrary large degree that are irreducible,
in contrast to the situation for polynomials over $\R$ or $\C$.

\paragraph{\hspace*{-0.3cm}}
Another useful tool arises with polynomials having 
coefficients from a ring $R$ and there is a homomorphism from $R$ to some field $F$. 
If the homomorphism is applied to all the coefficients of the polynomial (turning it from a polynomial
with $R$-coefficients into a polynomial with $F$-coefficients)
then a reducible polynomial cannot turn into 
an irreducible one:

\begin{reduction}\label{reduction}
Let $R$ be an integral domain, $F$ a field and $\ss:R\rightarrow F$ a ring homomorphism. 
Let $\ss^*:R[x]\rightarrow F[x]$ be 
the homomorphism of Lemma \ref{lemma3.20}.
Moreover, let $f\in R[x]$ be such that
\begin{enumerate}
\item $\deg\ss^*(f)=\deg(f)$, and 
\item $\ss^*(f)$ is irreducible over $F$.
\end{enumerate}
Then $f$ cannot be written as a product $f=gh$ with $g,h\in R[x]$ and $\deg g,\deg h<\deg f$.
\end{reduction}

Although it is stated in some generality, the reduction test is very useful for determining 
the irreducibility of polynomials over $\Q$.
As an example, take $R=\Z$; $F=\F_5$ and $f=8x^3-6x-1\in\Z[x]$. For $\ss$, take reduction 
modulo $5$, ie: 
$\ss(n)=n\text{ mod }5$. It is not hard to show that $\ss$ is a homomorphism. Since 
$\ss(8)\equiv 3\text{ mod }5$, and so
on, we get 
$$
\ss^*(f)=3x^3+4x+4\in\F_5[x].
$$
The degree has not changed, and by substituting the five elements of $\F_5$ into $\ss^*(f)$, one
can see that it has no roots in $\F_5$. Since the polynomial is a cubic, it must therefore be irreducible
over $\F_5$. Thus, by the reduction test, $8x^3-6x-1$ cannot be written as a product of 
smaller degree polynomials with
$\Z$-coefficients. But by Gauss' lemma, this gives that this 
polynomial is
irreducible over $\Q$.

$\F_5$ was chosen because with $\F_2$ condition (i) fails; with $\F_3$
condition (ii) fails.

\begin{proof}
Suppose on the contrary that $f=gh$ with $\deg g,\deg h<\deg f$. 
Then $\ss^*(f)=\ss^*(gh)=\ss^*(g)\ss^*(h)$, the
last part because $\ss^*$ is a homomorphism. Now $\ss^*(f)$ is irreducible, so the 
only way it can factorise
like this is if one of the factors, $\ss^*(g)$ say, is a constant, hence $\deg \ss^*(g)=0$. Then
$$
\deg
f=\deg\ss^*(f)=\deg\ss^*(g)\ss^*(h)=\deg\ss^*(g)+\deg\ss^*(h)=\deg\ss^*(h)\leq
\deg h<\deg f, 
$$
a contradiction. (That $\deg\ss^*(h)\leq \deg h$ rather than equality necessarily, is because the 
homomorphism $\ss$ may
send some of the coefficients of $h$ -- including the
leading one -- to $0\in F$.) 
\qed
\end{proof}

\paragraph{\hspace*{-0.3cm}}
We've already observed the similarity between polynomials and
integers. 
One thing we know about integers is
that they can be written uniquely as products of primes. 
We might hope that something similar is true for
polynomials, and it is in certain situations.
For the next few results, we deal only with polynomials $f\in F[x]$ for $F$ a field 
(although they are true in more generality).

\begin{lemma}\label{lem_ufd}
\begin{enumerate}
\item If $\gcd(f,g)=1$ and $f\dv gh$ then $f\dv h$.
\item If $f$ is irreducible and monic, then for any $g$ monic with
$g\dv f$ we have either $g=1$ or $g=f$.
\item If $g$ is irreducible and monic and $g$ does not divide $f$, then $\gcd(g,f)=1$.
\item If $g$ is irreducible and monic and $g\dv f_1f_2\ldots f_n$ then
$g|f_i$ for some $i$.
\end{enumerate}
\end{lemma}

\begin{proof}
\begin{enumerate}
\item Since $\gcd(f,g)=1$ there are $a,b\in F[x]$ such that $1=af+bg$,
hence $h=afh+bgh$. We have that $f\dv bgh$ by assumption, and it
clearly divides $afh$, hence it divides $afh+bgh=h$ also. 
\item If $g$ divides $f$ and $f$ is irreducible, then by definition $g$ must be
either a constant or a constant multiple of $f$. But $f$ is monic, so
$g=1$ or $g=f$ are the only possibilities.
\item The $\gcd$ of $f$ and $g$ is certainly a divisor of $g$, and hence
by irreducibility must be either a constant, or a constant times $g$. As
$g$ is also monic, the gcd must in fact be either $1$ or $g$ itself, and
since $g$ does not divide $f$ it cannot be $g$, so must be $1$.
\item Proceed by induction, with the first step for $n=1$ being
immediate. Since $g\dv f_1f_2\ldots f_n=(f_1f_2\ldots f_{n-1})f_n$, we
either have $g\dv f_n$, in which case we are finished, or not, in which
case $\gcd(g,f_n)=1$ by part (3). But then part (1) gives that $g\dv
f_1f_2\ldots f_{n-1}$, and the inductive hypothesis kicks in.\qed
\end{enumerate}
\end{proof}

The best way of summarising the lemma is this: monic irreducible
polynomials are like the ``prime numbers'' of $F[x]$.

\paragraph{\hspace*{-0.3cm}}
Just as any integer can be decomposed uniquely as a product of primes,
so too can any polynomial as a product of irreducible polynomials:

\begin{ufd}\label{ufd}
Every polynomial in $F[x]$ can be written in the form
$$
c p_1p_2\ldots p_r,
$$
where $c$ is a constant and the $p_i$ are monic and irreducible $\in
F[x]$. Moreover, if $a q_1q_2\ldots q_s$ is another factorisation with
the $q_j$ monic and irreducible, then $r=s$, $c=a$ and the $q_j$ are
just a rearrangement of the $p_i$.
\end{ufd}

The last part says that the 
factorisation is unique, except for the order you
write down the factors. 

\begin{proof}
To get the factorisation just keep
factorising reducible polynomials until they become irreducible. At the
end, pull out the coefficient of the leading term in each factor, and 
place them all at the front. 

For uniqueness, suppose that 
$$
c p_1p_2\ldots p_r=a q_1q_2\ldots q_s.
$$
Then $p_r$ divides $a q_1q_2\ldots q_s$ which by Lemma
\ref{lem_ufd} part (4) means that $p_r\dv q_i$ for some $i$. Reorder the
$q$'s so that it is $p_r\dv q_s$ that in fact we have. Since both $p_r$
and $q_s$ are monic, irreducible, and hence non-constant, $p_r=q_s$,
which leaves us with
$$
c p_1p_2\ldots p_{r-1}=a q_1q_2\ldots q_{s-1}.
$$
This gives $r=s$ straight away: if say $s>r$, then repetition of the
above leads to $c=a q_1q_2\ldots q_{s-r}$, which is absurd, as
consideration of degrees gives different answers for each side. 
Similarly if $r>s$.
But then
we also have that the $p$'s are just a rearrangement of the $q$'s, and
canceling down to $c p_1=a q_1$, that $c=a$.
\qed
\end{proof}

\paragraph{\hspace*{-0.3cm}}
It is worth repeating that everything depends on the ambient field $F$,
even the uniqueness of the decomposition. For example, $x^4-4$
decomposes as,
$$\begin{tabular}{l}
$(x^2+2)(x^2-2)\text{ in }\Q[x]$,\\
$(x^2+2)(x-\kern-2pt\sqrt{2})(x+\kern-2pt\sqrt{2})\text{ in }\R[x]\text{ and }$\\
$(x-\kern-2pt\sqrt{2}\imag)(x+\kern-2pt\sqrt{2}\imag)(x-\kern-2pt\sqrt{2})(x+\kern-2pt\sqrt{2})\text{ 
    in }\C[x]$.\\ 
\end{tabular}$$

To illustrate how unique factorisation can be used to determine irreducibility,
we have in $\C[x]$ that,
$$
x^2+2=(x-\kern-2pt\sqrt{2}\imag)(x+\kern-2pt\sqrt{2}\imag).
$$
Since the factors on the right are not in $\R[x]$ 
this polynomial ought to be irreducible over $\R$. To make this more
precise, any factorisation in $\R[x]$ would be of the form
$$
x^2+2=(x-c_1)(x-c_2)
$$
with the $c_i\in\R$. But this would be a factorisation in $\C[x]$ too,
and there is only one such by unique factorisation. This forces the $c_i$ to be
$\kern-2pt\sqrt{2}\imag$ and $-\kern-2pt\sqrt{2}\imag$, contradicting $c_i\in\R$. Hence $x^2+2$ is
indeed irreducible over $\R$. Similarly, $x^2-2$ is
irreducible over $\Q$.

\begin{vexercise}
Formulate the  example above into a general Theorem.
\end{vexercise}

\subsection*{Further Exercises for Section \thesection}

\begin{vexercise}
Prove that if a polynomial equation has all its coefficients in $\C$ then it must have all its roots
in $\C$.
\end{vexercise}

\begin{vexercise}\label{ex_lect2.1}
\hspace{1em}\begin{enumerate}
\item Let $f=a_nx^n+a_{n-1}x^{n-1}+\cdots +a_1x+a_0$ be a polynomial in $\R[x]$, that is, all
the $a_i\in\R$. Show that complex roots of $f$ occur in conjugate pairs, ie:
$\zeta\in\C$ is a root of $f$ if and only if $\bar{\zeta}$ is.
\item Find an example of a polynomial in $\C[x]$ for which part (a)
is not true.
\end{enumerate}

\end{vexercise}

\begin{vexercise}\label{ex_lect2.2}
\hspace{1em}\begin{enumerate}
\item Let $m,n$ and $k$ be integers with $m$ and $n$ relatively prime (ie: $\gcd(m,n)=1$). 
Show that if 
$m$ divides $nk$ then $m$ must divide $k$ ({\em hint}: there are two
methods here. One is to use Lemma \ref{lem_ufd} but in $\Z$. The other is to use the
fact that any integer
can be written uniquely as a product of primes. Do this for $m$ and $n$, and ask yourself what 
it means for this factorisation that $m$ and $n$ are relatively prime).
\item Show that if  $m/n$ is a root of $a_0+a_1x+ ...+a_rx^r$,
$a_i\in\Z$, where
$m$ and $n$ are relatively prime integers, then $m|a_0$ and $n|a_r$.
\item Deduce that if 
$a_r=1$ then $m/n$ is in fact an integer.
\end{enumerate}

\noindent\emph{moral}: 
If a monic polynomial with integer coefficients has a rational root $m/n$, then
this rational number is in fact an integer.
\end{vexercise}

\begin{vexercise}
If $m\in\Z$ is not a perfect square, show that $x^2-m$ is irreducible over $\Q$ (note:
it is {\em not\/} enough to merely assume that under the conditions stated $\kern-2pt\sqrt{m}$
is not a rational number).
\end{vexercise}

\begin{vexercise}
Find the greatest common divisor of $f(x)=x^3-6x^2+x+4$
and $g(x)=x^5-6x+1$ ({\em hint}: look at linear factors of $f(x)$).
\end{vexercise}

\begin{vexercise}\label{ex_lect3.1}
Determine which of the following polynomials are irreducible over the stated field: 
\begin{enumerate}
\item $1+x^8$ over $\R$;
\item $1+x^2+x^4+x^6+x^8+x^{10}$ over $\Q$ (\emph{hint}: Let $y=x^2$
and factorise $y^n-1$);
\item $x^4+15x^3+7$ over $\R$ (\emph{hint}: use the intermediate value
theorem from analysis);
\item $x^{n+1}+(n+2)!\,x^n+\cdots+(i+2)!\,x^i+\cdots+3!\,x+2!$ over $\Q$.
\item $x^2+1$ over $\F_7$.
\item Let $\F$ be the field of order 8 from Section \ref{lect4}, and let $\F[X]$ be
polynomials with coefficients from $\F$ and indeterminate $X$. 
Is $X^3+(\aa^2+\aa)X+(\aa^2+\aa+1)$
irreducible over $\F$?
\item $a_4x^4+a_3x^3+a_2x^2+a_1x+a_0$ over $\Q$ where the $a_i\in\Z$;
$a_3,a_2$ are even and $a_4,a_1,a_0$ are odd.
\end{enumerate}
\end{vexercise}

\begin{vexercise}\label{ex3.3}
If $p$ is prime, show that $p$ divides 
${\displaystyle \binom{p}{i}}$
for $0<i<p$. Show that $p$ divides ${\displaystyle \binom{p^n}{i}}$
for $n\geq 1$ and $0<i<p$.
\end{vexercise}

\begin{vexercise}
Show that 
$$
x^{p-1}+px^{p-2}+\cdots+
\binom{p}{i} x^{p-i-1}+\cdots+p,
$$
is irreducible over $\Q$.
\end{vexercise}

\begin{vexercise}\label{ex_lect3.2}
A complex number $\ww$ is an $n$-th {\em root of unity\/} if $\ww^n=1$.  
It is a {\em primitive\/} $n$-th root of unity if $\ww^n=1$, but $\ww^r\not=1$ for any 
$0<r<n$.  So for example, $\pm 1,\pm \imag$  are the 4-th roots of 1, but only $\pm i$ are 
primitive 4-th roots.

Convince yourself that for any $n$, 
$$\ww=\cos\frac{2\pi}{n}+\imag\sin\frac{2\pi}{n}$$ 
is an $n$-th root of $1$.
In fact, the other $n$-th roots are $\ww^2,\ldots,\ww^n=1$.
\begin{enumerate}
\item Show that if $\ww$ is a {\em primitive\/} $n$-th root of $1$
  then $\ww$ is a root of the polynomial 
\begin{equation}\label{eq1}
x^{n-1}+x^{n-2}+\cdots+x+1.
\end{equation}
\item Show that for (\ref{eq1}) to be irreducible over $\Q$, $n$ cannot be even.
\item Show that a polynomial $f(x)$ is irreducible over a field $F$ if $f(x+1)$ is
irreducible over $F$.
\item Finally, if 
$$\Phi_p(x)=x^{p-1}+x^{p-2}+\cdots+x+1$$ 
for $p$ a prime number, show
that $\Phi_p(x+1)$
is irreducible over $\Q$, and hence $\Phi_p(x)$
is too ({\em hint\/}: consider $x^p-1$ and use
the binomial theorem, Exercise \ref{ex3.3} and Eisenstein).
\end{enumerate}
The polynomial $\Phi_p(x)$ is called the {\em $p$-th cyclotomic polynomial\/}.
\end{vexercise}


\section{Fields I: Basics, Extensions and Concrete Examples}
\label{lect4}

This course studies the solutions to polynomial equations.
Questions about these solutions can be restated as
questions about fields. It is to these that we now turn. 

\paragraph{\hspace*{-0.3cm}}
We remembered the definition of a field in Section \ref{lect1}; we can restate it
as:

\begin{definition}[field -- version ${\mathbf 2}$]
\label{def_field2}
A field is a set $F$ with two operations, $+$ and $\times$,
such that for any $a,b,c\in F$,
\begin{enumerate}
\item $F$ is an Abelian group under $+$;
\item $F\setminus\{0\}$ is an Abelian group under $\times$;
\item the two operations are linked by the distributive law.
\end{enumerate}
\end{definition}

The two groups are called the {\em additive\/} and {\em multiplicative\/} groups of
the field. In particular, we will write $F^*$ to denote the multiplicative group (ie:
$F^*$ is the group with elements $F\setminus\{0\}$ and operation the multiplication 
from the field).
Even more succinctly,

\begin{definition}[field -- version ${\mathbf 3}$]
\label{def_field3}
A field is a set $F$ with two operations, $+$ and $\times$,
such that for any $a,b,c\in F$,
\begin{enumerate}
\item $F$ is a commutative ring under $+$ and $\times$;
\item for any $a\in F\setminus\{0\}$ there is an $a^{-1}\in F$ with
  $a\times a^{-1}=1=a^{-1}\times a$, 
\end{enumerate}
\end{definition}

In particular a field is a special kind of ring. 

\paragraph{\hspace*{-0.3cm}}
More concepts from the first lecture that can now be properly defined are:

\begin{definition}[extensions of fields]
Let $F$ and $E$ be fields with $F$ a subfield of $E$. We call $E$ an
extension of $F$. 
If $\bb\in E$, we write $F(\bb)$, as in
Section \ref{lect1}, for the smallest subfield of $E$ containing both
$F$ and $\bb$ (so in particular $F(\bb)$ is an extension of $F$). 
In general, if $\bb_1,\ldots,\bb_k\in E$, define $F(\bb_1,\ldots,\bb_k)=
F(\bb_1,\ldots,\bb_{k-1})(\bb_k)$. 
\end{definition}

The standard notation for an extension is to write
$E/F$, but in these notes we will use the more concrete $F\subseteq E$, being mindful
that this means $F$ is a subfield of $E$, and not just a
subset.

We say that  $\bb$ is {\em adjoined\/} to $F$ to obtain $F(\bb)$.
The last bit of the definition says that to adjoin several elements to
a field you adjoin them one at a time. The notation seems to
adjoin them in a particular order, but the order doesn't matter.
If we have an extension $F\subseteq E$ and there is a $\bb\in E$
such that $E=F(\bb)$, then we call $E$ a {\em simple extension\/}
of $F$.

\paragraph{\hspace*{-0.3cm}}
$\R$ is an extension of $\Q$; $\C$ is an extension of
$\R$, and so on.
Any field is an extension of itself!

\paragraph{\hspace*{-0.3cm}}\label{para4.4}
Let $\F_2$ be the field of integers modulo
$2$ arithmetic. 
Let $\aa$ be an ``abstract symbol" that can be multiplied so that it has the following property:
$\aa\times\aa\times\aa=\aa^3=\aa+1$ (a bit like decreeing that the 
imaginary $i$ squares to give $-1$).
Let  
$$\F=\{a+b\aa+c\aa^2\,|\,a,b,c\in\F_2\},$$ 
Define addition on $\F$ by:
$(a_1+b_1\aa+c_1\aa^2)+(a_2+b_2\aa+c_2\aa^2)=
(a_1+a_2)+(b_1+b_2)\aa+(c_1+c_2)\aa^2$, where the addition of
coefficients happens in $\F_2$. For multiplication, ``expand"
the expression $(a_1+b_1\aa+c_1\aa^2)(a_2+b_2\aa+c_2\aa^2)$
like you would a polynomial with
$\aa$ the indeterminate, so that $\aa\aa\aa=\aa^3$, the
coefficients are dealt with using the arithmetic from $\F_2$, and so on. Replace any
$\aa^3$ that result using the rule $\aa^3=\aa+1$. 

For example,
$$
(1+\aa+\aa^2)+(\aa+\aa^2)=1\text{ and }
(1+\aa+\aa^2)(\aa+\aa^2)=\aa+\aa^4=\aa+\aa(\aa+1)=\aa^2.
$$
It turns out that $\F$ forms a field with this
addition and multiplication -- see Exercise \ref{ex_tables}. Taking
those elements of $\F$ with $b=c=0$ we 
obtain (an isomorphic) copy of $\F_2$ inside of $\F$, and so
we have an extension of $\F_2$ that contains $8$ elements.

\paragraph{\hspace*{-0.3cm}}
$\Q(\kern-2pt\sqrt{2})$ is a simple extension of $\Q$
while $\Q(\kern-2pt\sqrt{2},\kern-2pt\sqrt{3})$ would appear not to be.
But consider $\Q(\kern-2pt\sqrt{2}+\kern-2pt\sqrt{3})$: certainly
$\kern-2pt\sqrt{2}+\kern-2pt\sqrt{3}\in\Q(\kern-2pt\sqrt{2},\kern-2pt\sqrt{3})$, and so
$\Q(\kern-2pt\sqrt{2}+\kern-2pt\sqrt{3})\subset
\Q(\kern-2pt\sqrt{2},\kern-2pt\sqrt{3})$. On the other hand, 
$$
(\kern-2pt\sqrt{2}+\kern-2pt\sqrt{3})^3=11\kern-2pt\sqrt{2}+9\kern-2pt\sqrt{3},
$$
as is readily checked using the Binomial Theorem. Since
$(\kern-2pt\sqrt{2}+\kern-2pt\sqrt{3})^3\in
\Q(\kern-2pt\sqrt{2}+\kern-2pt\sqrt{3})$, we get
$$
(11\kern-2pt\sqrt{2}+9\kern-2pt\sqrt{3})-9(\kern-2pt\sqrt{2}+\kern-2pt\sqrt{3})\in\Q(\kern-2pt\sqrt{2}+\kern-2pt\sqrt{3})\Rightarrow 
2\kern-2pt\sqrt{2}\in\Q(\kern-2pt\sqrt{2}+\kern-2pt\sqrt{3}).
$$
And so $\kern-2pt\sqrt{2}\in\Q(\kern-2pt\sqrt{2}+\kern-2pt\sqrt{3})$ as ${\displaystyle
\frac{1}{2}}$ is there too. Similarly it can be shown that
$\kern-2pt\sqrt{3}\in\Q(\kern-2pt\sqrt{2}+\kern-2pt\sqrt{3})$ and hence $\Q(\kern-2pt\sqrt{2},\kern-2pt\sqrt{3})\subset\Q(\kern-2pt\sqrt{2}+\kern-2pt\sqrt{3})$.
So 
$$\Q(\kern-2pt\sqrt{2},\kern-2pt\sqrt{3})=\Q(\kern-2pt\sqrt{2}+\kern-2pt\sqrt{3})$$
{\em is\/} a simple extension! 

\paragraph{\hspace*{-0.3cm}}
What do the elements of $\Q(\kern-2pt\sqrt{2})$ actually
look like? Later we will be answer this question in general, but for
now we give an ad-hoc answer.

Firstly $\kern-2pt\sqrt{2}$ and any $b\in\Q$ are in $\Q(\kern-2pt\sqrt{2})$
by definition. Since fields are closed under $\times$, any number of the
form $b\kern-2pt\sqrt{2}\in \Q(\kern-2pt\sqrt{2})$. Similarly, fields are closed under
$+$, so any $a+b\kern-2pt\sqrt{2}\in\Q(\kern-2pt\sqrt{2})$ for $a\in\Q$. Thus, the set
$$
\F=\{a+b\kern-2pt\sqrt{2}\,|\,a,b\in\Q\}\subseteq\Q(\kern-2pt\sqrt{2}).
$$
But $\F$ is a field in its own right using the usual addition and
multiplication of complex numbers. This is easily checked from the
axioms; for instance, the inverse of $a+b\kern-2pt\sqrt{2}$ can be calculated:
$$
\frac{1}{a+b\kern-2pt\sqrt{2}}\times\frac{a-b\kern-2pt\sqrt{2}}{a-b\kern-2pt\sqrt{2}}=\frac{a-b\kern-2pt\sqrt{2}}{a^2-2b^2}=
\frac{a}{a^2-2b^2}-\frac{b}{a^2-2b^2}\kern-2pt\sqrt{2}\in\F,
$$ 
and you can check the other axioms for yourself. We also have $\Q\subset\F$ (letting
$b=0$) and $\kern-2pt\sqrt{2}\in\F$ (letting $a=0,b=1$). Since $\Q(\kern-2pt\sqrt{2})$ is
the smallest field having these two properties, we have
$\Q(\kern-2pt\sqrt{2})\subseteq\F$. Thus,
$$
\Q(\kern-2pt\sqrt{2})=\F=\{a+b\kern-2pt\sqrt{2}\,|\,a,b\in\Q\}.
$$

\begin{vexercise}
Let $\aa$ be a complex number such that $\aa^3=1$ and consider the set
$$
\F=\{a_0+a_1\aa+a_2\aa^2\,|\,a_i\in\Q\}
$$
\begin{enumerate}
\item By row reducing the matrix,
$$
\left(\begin{array}{cccc}
a_0&2a_2&2a_1&1\\
a_1&a_0&2a_2&0\\
a_2&a_1&a_0&0\\
\end{array}\right)
$$
find an element of $\F$ that is the inverse under multiplication of 
$a_0+a_1\aa+a_2\aa^2$.
\item Show that $\F$ is a field, hence $\Q(\aa)=\F$.
\end{enumerate}
\end{vexercise}

\paragraph{\hspace*{-0.3cm}}\label{para4.20}
The previous exercise shows that the following two fields
have the form,
$$
\Q(\sqrt[3]{2})=\{a+b\sqrt[3]{2}+c\sqrt[3]{2}^2\,|\,a,b,c\in\Q\}\text{ and }
\Q(\bb)=\{a+b\bb+c\bb^2\,|\,a,b,c\in\Q\},
$$
where
$$
\bb=\sqrt[3]{2}\biggl(-\frac{1}{2}+\frac{\sqrt{3}}{2}\imag\biggr)\in\C.
$$
These two fields are different: the first is 
completely contained in $\R$, but the second contains $\bb$, which is obviously 
complex but not real. Hold that thought.

\begin{definition}[ring isomorphism]
A bijective homomorphism of rings $\varphi:R\rightarrow S$ is called
an isomorphism. 
\end{definition}

\paragraph{\hspace*{-0.3cm}} 
A silly but instructive example is given by the Roman ring, whose elements are
$$
\{\ldots,-V,-IV,-III,-II,-I,0,I,II,III,IV,V,\cdots\},
$$
and with addition and multiplication $IX+IV=XIII$ and $IX\times VI
=LIV,\text{ etc}\ldots$ Obviously the ring is isomorphic to $\Z$, 
and it is this idea of a trivial relabeling that is captured by 
an isomorphism -- two rings are isomorphic if they are really the same,
just written in different languages.

But we place a huge emphasis on
the way things are labelled.
The two fields of the previous paragraph are a good example, for,
$$
\Q(\sqrt[3]{2})\text{{ and }}
\Q\biggl(\sqrt[3]{2}\biggl(-\frac{1}{2}+\frac{\sqrt{3}}{2}\imag\biggr)\biggr)\text{{ are
isomorphic}}
$$
(we will see why in Section \ref{fields2}). 
To illustrate how we might now come unstuck, suppose we were to
formulate the following,

\begin{bogusdefn}
A subfield of $\C$ is called real if and only if it is contained
in $\R$.
\end{bogusdefn}

So $\Q(\sqrt[3]{2})$ is a real field, but ${\displaystyle \Q\biggl(
\sqrt[3]{2}\biggl(-\frac{1}{2}+\frac{\sqrt{3}}{2}\imag\biggr)\biggr)}$ is not. But they
are the same field! A definition should not depend on the way the
elements are labelled. We will resolve this problem in Section \ref{fields2} by thinking about
fields in a more abstract way.

\paragraph{\hspace*{-0.3cm}}
In the remainder
of this section we introduce a few more concepts associated with fields.

It is well known that $\kern-2pt\sqrt{2}$ and $\pi$ are both irrational real numbers. Nevertheless,
from an algebraic point of view, $\kern-2pt\sqrt{2}$ is slightly more tractable than $\pi$,
as it is a root of a very simple  equation $x^2-2$, whereas there is no polynomial
with integer coefficients having $\pi$ as a root (this is not obvious). 

\begin{definition}[algebraic element]
Let $F\subseteq E$ be an extension of fields and $\aa\in E$. Call $\aa$
algebraic over $F$ if and only if 
$$
a_0+a_1\aa+a_2\aa^2+\cdots+a_n\aa^n=0,
$$
for some $a_0,a_1,\ldots,a_n\in F$.  
\end{definition}

In otherwords, $\aa$ is a root of 
the polynomial $f=a_0+a_1x+a_2x^2+\cdots+a_nx^n$ in $F[x]$. If $\aa$ is not 
algebraic, ie: not the root of any polynomial with $F$-coefficients, then
we say that it is {\em transcendental\/} over $F$.

\paragraph{\hspace*{-0.3cm}}
Some simple examples: 
$$
\kern-2pt\sqrt{2}, \frac{1+\sqrt{5}}{2}\text{ and }\sqrt[5]{\sqrt{2}+5\sqrt[3]{3}},
$$
are algebraic over $\Q$, whereas $\pi$ and $e$ are transcendental over $\Q$;
$\pi$ is algebraic over $\Q(\pi)$. 

\paragraph{\hspace*{-0.3cm}}
A field can contain many subfields: $\C$ contains 
$\Q(\kern-2pt\sqrt{2}),\R,\ldots$. It also contains $\Q$, but no subfields
that are smaller than this. Indeed, any subfield of $\C$ contains
$\Q$, so the rationals are the smallest subfield of the complex
numbers.

\begin{definition}[prime subfield]
\label{definition:prime_subfield}
The prime subfield of a field $F$ is 
the intersection of all the
subfields of $F$. 
\end{definition}

In particular the prime subfield is contained in
every subfield of $F$.


\begin{vexercise}\label{ex4.2}
Consider the field of rational numbers $\Q$ or the finite field $\F_p$ having
$p$ elements. Show that neither of these fields contain a proper subfield
(hint: for $\F_p$, consider the additive group and use Lagrange's Theorem
from Section \ref{groups.stuff}. For $\Q$, any subfield must contain $1$, and show that
it must then be all of $\Q$).
\end{vexercise}

The prime subfield must contain $1$, hence any expression of the
form $1+1+\cdots+1$ for any number of summands. If no such expression equals $0$
then we have infinitely many distinct such elements, and their inverses
under addition, hence a copy of $\Z$ in $F$. Otherwise, if $n$
is the smallest number of summands for which such an expression equals $0$, then
the elements
$$
1,1+1,1+1+1,\ldots,\underbrace{1+1+\cdots+1}_{n\text{ times}}=0,
$$
forms a copy of $\Z_n$ inside of $F$.
These comments can be made precise as in the following exercise. It looks ahead a little, 
requiring the first isomorphism theorem for
rings in Section \ref{lect5}.

\begin{vexercise}
Let $F$ be a field and define a map $\Z\rightarrow F$ by
$$
n\mapsto
\left\{\begin{array}{l}0,\text{ if }n=0,\\
1+\cdots+1, (n\text{ times), if }n>0\\
-1-\cdots-1, (n\text{ times), if }n<0.
\end{array}\right.
$$
Show that the map is a ring homomorphism. If the kernel consists of just $\{0\}$, then
show that $F$ contains $\Z$ as a subring. Otherwise, let $n$ be the smallest positive integer 
contained in the kernel, and show that $F$ contains $\Z_n$ as a subring. As $F$ is a field,
hence an integral domain, show that we must have $n=p$ a prime in this situation.
\end{vexercise}

Thus any field contains a subring isomorphic to $\Z$ or to $\Z_p$ for some 
prime $p$. But the ring $\Z_p$ is the field $\F_p$, and we saw in Exercise \ref{ex4.2}
that $\F_p$ contains no subfields. The conclusion is that in the second case
the prime subfield is $\F_p$. In the first case, $\Z$ is not 
a field, but each $m$ in this copy of $\Z$ has an inverse $1/m$ in $F$, and the
product of this with any other $n$ gives an element $m/n\in F$. 
The set of all such elements obtained is a copy of $\Q$ inside $F$.

\begin{vexercise}\label{ex4.3}
Make these loose statements precise: let $F$ be a field and $R$ a subring of
$F$ with $\varphi:\Z\rightarrow R$ an isomorphism of rings (this is what we mean
when we say that $F$ contains a copy of $\Z$). 
Show that this can be extended to an isomorphism $\widehat{\varphi}:\Q\rightarrow F'
\subseteq F$
with $\widehat{\varphi}|_{\Z}=\varphi$.
\end{vexercise}

\paragraph{\hspace*{-0.3cm}}
Putting it together: the prime subfield of a field is isomorphic
either to the rationals $\Q$ or to the finite field $\F_p$ for some prime $p$.
Define the {\em characteristic\/} of a field to be $0$ if the prime subfield is $\Q$,
or $p$ if the prime subfield is $\F_p$. 
Thus fields like $\Q,\R$ and $\C$ have characteristic zero, and indeed, any field of
characteristic zero must be infinite. Fields like 
$\F_2,\F_3\ldots$ and the field $\F$ of order $8$ given above have characteristic
$2,3$ and $2$ respectively. 

\begin{vexercise}
Show that a field $F$ has characteristic $p>0$ if and only if $p$ is the smallest 
number of summands such that the expression $1+1+\cdots+1$ is equal to $0$. Show
that $F$ has characteristic $0$ if and only if no such expression is equal to $0$.
\end{vexercise}

Thus, all fields of characteristic $0$ are infinite, and the only examples we know
of fields of characteristic $p>0$ are finite. It is not true though that a field
of characteristic $p>0$ must be finite. We give some examples of infinite
fields of characteristic $p>0$ below.

\begin{vexercise}\label{ex4.30}
Suppose that $f$ is an irreducible polynomial over a field $F$ of characteristic $0$.
Recalling Exercise \ref{ex3.20}, show that the roots of $f$ in any extension $E$ of $F$ 
are distinct.
\end{vexercise}

\paragraph{\hspace*{-0.3cm}}
It turns out that we can construct $\Q$ abstractly from 
$\Z$, without having to first position it inside another field:
consider the set
$$
\F=\{(a,b)\,|\,a,b\in\Z,b\not= 0,\mbox{ where }(a,b)=(c,d)\mbox{ iff }ad=bc\}
$$
i.e. ordered pairs of integers with two ordered
pairs $(a,b)$ and $(c,d)$ being the same if $ad=bc$.

\begin{aside}
These loose statements are made precise by defining an equivalence relation on
the set of ordered pairs $\Z\times\Z$ by $(a,b)\sim(c,d)$ if and only if $ad=bc$. The elements
of $\F$ are then the equivalence classes under this relation.
\end{aside}

Define addition and multiplication on $\F$ by:
$$
(a,b)+(c,d)=(ad+bc,bd)\mbox{ and }(a,b)(c,d)=(ac,bd). 
$$

\begin{vexercise}
\hspace{1em}
\begin{enumerate}
\item
Show that these definitions are well-defined, ie: if $(a,b)=(a',b')$ and $(c,d)=
(c',d')$, then $(a,b)+(c,d)=(a',b')+(c',d')$ and $(a,b)(c,d)=(a',b')(c',d')$.
\item
Show that $\F$ is a field. 
\item
Define a map $\varphi:\F\rightarrow \Q$ by $\varphi(a,b)=a/b$. Show that the map is 
well defined (ie: if $(a,b)=(a',b')$ then $\varphi(a,b)=\varphi(a',b')$)
and that $\varphi$ is an isomorphism.
\end{enumerate}
\end{vexercise}

This construction can be generalised as the following Exercise shows:

\begin{vexercise}\label{ex4.20}
Repeat the construction above with $\Z$ replaced by an arbitrary integral 
domain $R$. 
\end{vexercise}

The resulting field is called the {\em field of fractions of $R$}.
The field of fractions construction provides some interesting examples of fields,
possibly new in the reader's experience. Let $F[x]$ be the ring of polynomials 
with $F$-coefficients where $F$ is any field. 
The field of fractions
of this integral domain has elements of the form $f(x)/g(x)$ for $f$ and $g$ polynomials,
in other words, rational functions with $F$-coefficients.
The field is denoted $F(x)$ and is called the {\em field of rational 
functions over $F$\/}. 

\begin{description}
\item[--] {\em An infinite field of characteristic $p$:\/}
if $\F_p$ is a finite field of order $p$, then the field of rational
functions $\F_p(x)$ is infinite as it contains all the polynomials
over $\F_p$. But the rational function
$1$ still adds to itself only $p$ times to give $0$, hence the field
has characteristic $p$. 
\item[--] {\em A field properly containing the complex numbers:\/} $\C$ is
properly contained in the field of rational functions $\C(x)$.
\end{description}


\subsection*{Further Exercises for Section \thesection}

\begin{vexercise}
Let $\F$ be the set of all matrices of the form 
$\left[\begin{array}{cc}a&b\\2b&a\\\end{array}\right]$ where $a,b$ are in
the field $\F_5$. Define addition and multiplication to be the usual addition and multiplication 
of matrices (and also the addition and multiplication in $\F_5$). Show that
$\F$ is a field. How many elements does it have?
\end{vexercise}

\begin{vexercise}\label{ex_tables}
Let $\F_2$ be the field of integers modulo $2$, and $\aa$ be an
``abstract symbol" that can be multiplied so that it has the following
property: 
$\aa\times\aa\times\aa=\aa^3=\aa+1$ (a bit like decreeing that the
imaginary $i$ squares to give $-1$). 
Let  
$$\F=\{a+b\aa+c\aa^2\,|\,a,b,c\in\F_2\},$$ 
Define addition on $\F$ by:
$(a_1+b_1\aa+c_1\aa^2)+(a_2+b_2\aa+c_2\aa^2)=
(a_1+a_2)+(b_1+b_2)\aa+(c_1+c_2)\aa^2$, where the addition of
coefficients happens in $\F_2$. For multiplication, ``expand"
the expression $(a_1+b_1\aa+c_1\aa^2)(a_2+b_2\aa+c_2\aa^2)$
like you would a polynomial with
$\aa$ the indeterminate, 
the
coefficients are dealt with using the arithmetic from $\F_2$, and so on. Replace any
$\aa^3$ that result using the rule above. 
\begin{enumerate}
\item Write down all the elements of $\F$.
\item Write out the addition and multiplication tables for $\F$ (ie: the tables with rows
and columns indexed by the elements of $\F$, with the entry in the $i$-th row and $j$-th
column the sum/product of the $i$-th and $j$-th elements of the field). 
Hence show that $\F$ is a field (you can assume that the addition and
multiplication are associative as well as the distributive law, as these
are a bit tedious to verify!)
Using your tables, find the inverses (under multiplication) of the
elements $1+\aa$ and $1+\aa+\aa^2$, 
ie: find
$$
\frac{1}{1+\aa}\mbox{ and }\frac{1}{1+\aa+\aa^2}\mbox{ in }\F.
$$
\item Is the extension $\F_2\subset \F$ a simple one?
\end{enumerate}
\end{vexercise}

\begin{vexercise}
Take the set $\F$ of the previous exercise, and define addition/multiplication in 
the same way except that the rule for simplification is now $\aa^3=\aa^2+\aa+1$. Show that in
this case you {\em don't\/} get a field.
\end{vexercise}

\begin{vexercise}
Verify the claim in lectures that the set $\F=\{a+b\kern-2pt\sqrt{2}\,|\,a,b\in\Q\}$ is a subfield of 
$\C$.
\end{vexercise}

\begin{vexercise}
Verify the claim in lectures that $\Q(\sqrt[3]{2})=\{a+b(\sqrt[3]{2})+
c(\sqrt[3]{2})^2\,|\,a,b,c\in\Q\}$.
\end{vexercise}

\begin{vexercise}
Find a complex number $\alpha$ such that $\Q(\kern-2pt\sqrt{2},i)=\Q(\alpha)$.

\end{vexercise}

\begin{vexercise}
Is ${\Q}(\kern-2pt\sqrt{2}, \kern-2pt\sqrt{3}, \kern-2pt\sqrt{7})$ a
simple extension of ${\Q}(\kern-2pt\sqrt{2}, \kern-2pt\sqrt{3})$,  
${\Q}(\kern-2pt\sqrt{2})$ or even of $\Q$? 
\end{vexercise}

\begin{vexercise}
Let $\nabla$ be an ``abstract symbol" that has the following property:
$\nabla^2=-\nabla-1$ (a bit like $i$ squaring to give $-1$). Let
$$
\F=\{a+b\nabla\,|\,a,b\in\R\},
$$
and define an addition on $\F$ by:
$(a_1+b_1\nabla)+(a_2+b_2\nabla)=(a_1+a_2)+(b_1+b_2)\nabla$. For multiplication, expand
the expression $(a_1+b_1\nabla)(a_2+b_2\nabla)$ normally (treating
$\nabla$ like an indeterminate, so that $\nabla\nabla=\nabla^2$, and so on), and replace the
resulting $\nabla^2$ using the rule above. 
Show that $\F$ is a field, and is just the complex
numbers $\C$.
Do exactly the same thing, but with symbol $\triangle$ 
satisfying $\triangle^2=\kern-2pt\sqrt{2}\triangle-\sqrt[3]{5}$. Show that you
{\em still\/} get the complex numbers.
\end{vexercise}

\section{Rings II: Quotients}
\label{lect5}

In the last section we saw the need to think about fields more
abstractly. This section introduces the machinery we need to do this.

\paragraph{\hspace*{-0.3cm}}
A subset $I$ of a ring $R$ is an \emph{ideal\/} if and only if $I$ is a subgroup
of the abelian group $(R,+)$ and for any $s\in R$ we have
$sI=\{sa\,|\,a\in I\}=Is\subseteq I$.

In the rings that most interest us, ideals turn out to have a
very simple form:

\begin{proposition}
\label{fields2:all_ideals_are_principle}
Let $I$ be an ideal in $F[x]$ for $F$ a field. Then there is a
polynomial $f\in F[x]$ such that $I=\{fg\,|\,g\in F[x]\}$.
\end{proposition}

An ideal in a ring of polynomials over a field thus consists of all the multiples of
some fixed polynomial. For $f$ the polynomial given in the
Proposition, write $\langle f\rangle$ for the ideal that it gives,
i.e. $\langle f\rangle=\{fg\,|\,g\in F[x]\}$, and call $f$ a
\emph{generator\/} of the ideal.

\begin{proof}
If $I=\{0\}$ (which is an ideal!) then we have $I=\langle
0\rangle$, and so the result holds. Otherwise, $I$ contains non-zero
polynomials. Choose $f$ to be one of minimal degree $\geq 0$. Then
$Ig\subseteq I$ for all $g$ gives $\langle f\rangle\subseteq
I$. Conversely, if $h\in I$ then dividing $h$ by $f$ gives $h=qf+r$.
As $qI\subseteq I$ we have $qf\in I$, hence
$h-qf\in I$, as $I$ is a subgroup under $+$. Thus $r\in I$, and as
$\deg\,r<\deg\,f$ we are only saved from a contradiction if
$\deg\,r<0$; that is, if $r=0$. Thus $h=qf\in\langle f\rangle$ and so
$I\subseteq \langle f\rangle$.
\qed
\end{proof}

To emphasise that from now on, all our ideals will have this special form, we 
restate the definition:

\begin{definition}[ideals of polynomial rings over a field]
An ideal in $F[x]$ is a set of the form 
$$\langle f\rangle=\{fg\,|\,g\in F[x]\}$$
for $f$ some fixed polynomial.
\end{definition}

\begin{vexercise}
\label{fields2:Exercise20}
\begin{enumerate}
\item Show that $\lg f\rg=\lg h\rg$ if and only if $h=c f$ for some constant 
$c\in F$. Similarly, $\lg h\rg=F[x]$ if and only if $h=c$ some
constant. \emph{Moral}: generators are not unique.
\item Let $I\subset\Z[x]$ consist of those polynomials having even
  constant term. Show that $I$ is an ideal but $I\not=\langle
  f\rangle$ for any $f\in\Z[x]$. \emph{Moral}: ideals in $R[x]$ for
  $R$ a commutative ring need not have the special form of
  Proposition \ref{fields2:all_ideals_are_principle}.
\end{enumerate}
\end{vexercise}


\paragraph{\hspace*{-0.3cm}}
In any ring there are the trivial ideals $\langle 0\rangle=\{0\}$
(which we have met already in the proof of Proposition
\ref{fields2:all_ideals_are_principle}) and $\langle 1\rangle=R$. 

\begin{vexercise}\label{ex5.1}
\hspace{1em}
\begin{enumerate}
\item Show that the only ideals in a field $F$ are the two trivial ones (hint: use the property
of ideals mentioned at the end of the last paragraph).
\item If $R$ is a commutative ring whose only ideals are $\{0\}$ and $R$, then show that
$R$ is a field.
\item Show that in the non-commutative ring $M_n(F)$ of $n\times n$ matrices with entries from
the field $F$ there are only the two trivial ideals, but that $M_n(F)$ is not a field. 
\end{enumerate}
\end{vexercise}

\paragraph{\hspace*{-0.3cm}}
For another example of an ideal, consider the ring $\Q[x]$, the number $\kern-2pt\sqrt{2}\in\R$,
and the evaluation homomorphism
$\ev_{\sqrt{2}}:\Q[x]\rightarrow\R$ given by
$$
\ev_{\sqrt{2}}(a_nx^n+\cdots+a_0)=a_n(\kern-2pt\sqrt{2})^n+\cdots+a_0.
$$
(see Section \ref{lect2}). Let $I$ be the set of all polynomials in $\Q[x]$ that
are sent to $0\in\R$ by this map. Certainly $x^2-2\in I$ (as
$\kern-2pt\sqrt{2}^2-2=0$). If $f=(x^2-2)g\in\Q[x]$, then 
$$
\ev_{\sqrt{2}}(f)=\ev_{\sqrt{2}}(x^2-2)\ev_{\sqrt{2}}(g)=0\times\ev_{\sqrt{2}}(g)=0,
$$
using the fact that $\ev_{\sqrt{2}}$ is a homomorphism.
Thus, $f\in I$, and so the ideal $\langle
x^2-2\rangle$ is $\subseteq I$. 

Conversely, if $h$ is sent to $0$ by $\ev_{\sqrt{2}}$, ie: $h\in I$, we
can divide it by $x^2-2$ using the division algorithm, 
$$
h=(x^2-2)q+r,
$$
where $\deg r<2$, so that $r=ax+b$ for some $a,b\in \Q$. But since
$\ev_{\sqrt{2}}(h)=0$ we have
$$
(\kern-2pt\sqrt{2}^2-2)q(\kern-2pt\sqrt{2})+r(\kern-2pt\sqrt{2})=0\Rightarrow
r(\kern-2pt\sqrt{2})=0\Rightarrow
a\kern-2pt\sqrt{2}+b=0.
$$
If $a\not= 0$, then $\kern-2pt\sqrt{2}\in\Q$ as $a,b\in\Q$, which is plainly
nonsense. Thus $a=0$, hence $b=0$ 
too, so that $r=0$, and
hence $h=(x^2-2)q\in\langle x^2-2\rangle$, and we get that
$I\subseteq \langle x^2-2\rangle$.

The conclusion is that the set of polynomials in $\Q[x]$ sent to zero by the
evaluation homomorphism $\ev_{\sqrt{2}}$ is an ideal. 

\paragraph{\hspace*{-0.3cm}}
This always happens: if
$R,S$ are rings and $\varphi:R\rightarrow S$ a ring homomorphism, then
the {\em kernel\/} of $\varphi$, denoted $\ker\varphi$, is the set of all
elements of $R$ sent to $0\in S$ by $\varphi$, ie:
$$
\ker\varphi=\{r\in R\,|\,\varphi(r)=0\in S\}.
$$

\begin{proposition}
If $F$ is a field and $S$ a ring then the kernel of a homomorphism 
$\varphi:F[x]\rightarrow S$ is an ideal.
\end{proposition}

\begin{proof}
Is very similar to the previous example. To get a polynomial that
plays the role of $x^2-2$, 
choose $g\in\ker\varphi$, non-zero, of smallest degree. We
claim that $\ker\varphi=\langle g\rangle$, for which we need to show that
these two sets are mutually contained within each other. On the one
hand, if $pg\in\langle g\rangle$ then 
$$
\varphi(pg)=\varphi(p)\varphi(g)=\varphi(p)\times 0=0,
$$
since $g\in\ker\varphi$. Thus, $\langle g\rangle\subseteq\ker\varphi$.
On the other hand, let $f\in\ker\varphi$ and use the division algorithm to
divide it by $g$,
$$
f=qg+r,
$$
where $\deg r<\deg g$. Now, $r=f-qg\Rightarrow
\varphi(r)=\varphi(f-qg)=\varphi(f)-\varphi(q)\varphi(g)=0-\varphi(q).0=0$, since both
$f,g\in\ker\varphi$. Thus, $r$ is also in the kernel of $\varphi$. If $r$ was a
non-zero polynomial, then we would have a contradiction because $\deg
r<\deg g$, but $g$ was chosen from $\ker\varphi$ to have smallest degree.
Thus we must have that $r=0$, hence $f=qg\in\langle g\rangle$, ie:
$\ker\varphi\subseteq\langle g\rangle$. 
\qed
\end{proof}

\paragraph{\hspace*{-0.3cm}}
Let $\langle f\rangle\subset F[x]$ be an ideal and $g\in F[x]$ any
polynomial. The set
$$
g+\langle f\rangle=\{g+h\,|\,h\in\langle f\rangle\},
$$
is called the {\em coset of $\langle f\rangle$ with representative
$g$\/}
(or the coset of $\langle f\rangle$ {\em determined\/} by $g$).

\paragraph{\hspace*{-0.3cm}}
As an example, consider the ideal $\langle x\rangle$ in $\F_2[x]$.
Thus $\langle x\rangle$ is the set of all multiples of $x$, which is the same
as the polynomials in $\F_2[x]$ that have
no constant term.
What are the cosets of $\langle x\rangle$? Let $g$ be any polynomial and consider the
coset $g+\langle x\rangle$. The only possibilities are that $g$
has no constant term, or it does, in which case this term is $1$
(we are in $\F_2[x]$). 

If $g$ has no constant term, then
$$
g+\langle x\rangle=\langle x\rangle.
$$
For, $g$ added to a polynomial with no constant term is another polynomial with
no constant term, ie: $g+\langle x\rangle\subseteq \langle x\rangle$. On
the other hand, if $f\in\langle x\rangle$ is any polynomial with no constant term, then
$f-g\in\langle x\rangle$ so $f=g+(f-g)\in g+\langle x\rangle$, ie:
$\langle x\rangle\subseteq g+\langle x\rangle$. 

If $g$ does have a constant term, you can convince yourself in exactly
the same way that,
$$
g+\langle x\rangle=1+\langle x\rangle.
$$
Thus, there are only two cosets of $\langle x\rangle$ in $\F_2[x]$,
namely
the ideal $\langle x\rangle$ itself and $1+\langle x\rangle$. 

Notice that these two cosets are completely disjoint, but every
polynomial is in one of them. 

\paragraph{\hspace*{-0.3cm}}
Here are some basic properties of cosets:

\begin{figure}
  \centering
\begin{pspicture}(14,4)
\rput(3,1){
\rput(1.2,1){\BoxedEPSF{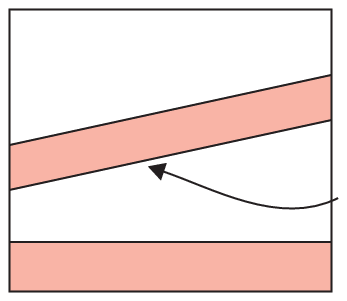 scaled 1000}}
\rput(1,-0.2){$\langle f\rangle$}
\rput(.5,1){$g_1$}
\rput(1.5,1.25){$g_2$}
\rput(0,-.6){\rput(3.5,.75){$g_2+\langle f\rangle$}
\rput(3.5,1.1){$=$}
\rput(3.55,1.5){$g_1+\langle f\rangle$}}
}
\rput(8.75,1){
\rput(1.2,1){\BoxedEPSF{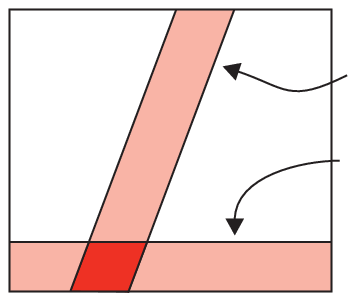 scaled 1000}}
\rput(1,1){$g_1$}
\rput(1.6,-0.2){$g_2$}
\rput(3.5,.9){$g_2+\langle f\rangle$}
\rput(3.55,1.8){$g_1+\langle f\rangle$}
\rput(.55,-0.2){$h$}
}
\end{pspicture}
\caption{Two different names for the same coset (\emph{left}) and a
  prohibited situation (\emph{right}).}
  \label{fig:figure4}
\end{figure}

\begin{description}
\item[--] \emph{Every polynomial $g$ is in some coset of $\langle f\rangle$}: for 
$g=g+0\times f\in g+\langle f\rangle$.
\item[--] \emph{For any $q$, we have $qf+\langle f\rangle=\langle f\rangle$}: so multiples of 
$f$ get ``absorbed'' into
the ideal $\langle f\rangle$.
\item[--]
\emph{The following three things are equivalent: (i). $g_1$ and $g_2$ lie in the same coset of
$\langle f\rangle$; (ii). $g_1+\langle f\rangle=g_2+\langle f\rangle$;
(iii). $g_1$ and $g_2$ differ by a multiple of $f$}:
to see this: (iii) $\Rightarrow$ (ii) If $g_1-g_2=pf$ then $g_1=g_2+pf$ so that
$g_1+\langle f\rangle=g_2+pf+\langle f\rangle=g_2+\langle f\rangle$; (ii) $\Rightarrow$ (i)  
Since $g_1\in g_1+\langle f\rangle$ and $g_2\in g_2+\langle f\rangle$, and these cosets are equal 
we have that $g_1,g_2$ lie in the same coset; (i) $\Rightarrow$ (iii) If $g_1$ and $g_2$ 
lie in the same coset, ie:
$g_1,g_2\in h+\langle f\rangle$, then each $g_i=h+p_if\Rightarrow g_1-g_2=(p_1-p_2)f$.

It can be summarised by saying that $g_1$ and $g_2$ lie in the same coset if and only if 
this coset has the two different names, $g_1+\langle f\rangle$ and $g_2+\langle f\rangle$, 
as in the left of Figure \ref{fig:figure4}.

\item[--] \emph{The situation on the right of Figure \ref{fig:figure4} opposite
  never happens, where distinct cosets have non-empty intersection}: if the two cosets pictured are  
called respectively
$g_1+\langle f\rangle$ and $g_2+\langle f\rangle$, then $h$ is in both and so differs 
from $g_1$ and $g_2$ by 
multiples of $f$, ie: $g_1-h=p_1f$ and $h-g_2=p_2f$, so that $g_1-g_2=(p_1+p_2)f$. Since $g_1$ 
and $g_2$ differ by a multiple of $f$, we have $g_1+\langle f\rangle=
g_2+\langle f\rangle$. 
\end{description}

Thus, the cosets of an ideal partition the ring.

\paragraph{\hspace*{-0.3cm}}
As an example let $x^2-2\in\Q[x]$ and consider the ideal 
$$
\langle x^2-2\rangle=\{p(x^2-2)\,|\,p\in\Q[x]\}.
$$
$(x^3-2x+15)+\langle x^2-2\rangle$ is then a coset, but it is not written in the nicest 
possible form. If 
we divide $x^3-2x+15$ by $x^2-2$:
$$
x^3-2x+15=x(x^2-2)+15,
$$
we have $x^3-2x+15 $ and $15$ differ by a multiple of $x^2-2$, so that
$$
(x^3-2x+15)+\langle x^2-2\rangle=15+\langle x^2-2\rangle.
$$

\paragraph{\hspace*{-0.3cm}}
If we look again at the ideal $\langle x\rangle$ in
$\F_2[x]$, there were only two cosets, 
$$ 
\langle x\rangle=0+\langle x\rangle \text{ and } 1+\langle x\rangle,
$$
that corresponded to the polynomials with constant term $0$ and the
polynomials with constant  
term $1$. We
could try ``adding'' and ``multiplying'' these two 
cosets together according to,
$$
(0+\langle x\rangle)+(0+\langle x\rangle)= 0+\langle
x\rangle,(1+\langle x\rangle)+(0+\langle x\rangle) 
=1+\langle x\rangle,(1+\langle x\rangle)+(1+\langle x\rangle)=0+\langle x\rangle,
$$
and so on, where all we have done is to add the representatives of the cosets together 
using the addition from $
\F_2$. Similarly for multiplying the cosets. 
This looks like $\F_2$, but with $0+\langle x\rangle$ and $1+\langle x\rangle$ 
replacing $0$ and $1$.

\paragraph{\hspace*{-0.3cm}}
Again this always happens. Let $\langle f\rangle$ be an ideal in $F[x]$, and define an 
addition and multiplication
of cosets of $\langle f\rangle$ by,
$$
(g_1+\langle f\rangle)+(g_2+\langle f\rangle)=(g_1+g_2)+\langle f\rangle\text{ and }
(g_1+\langle f\rangle)(g_2+\langle f\rangle)=(g_1g_2)+\langle f\rangle,
$$
where the addition and multiplication of the $g_i$'s is happening in $F[x]$.

\begin{theorem}
The set of cosets $F[x]/\langle f\rangle$ together with the $+$ and $\times$ above is a ring.
\end{theorem}

Call this the {\em quotient ring\/} of $F[x]$ by the ideal $\langle f\rangle$. 
All our rings have a ``zero'', 
a ``one'', and so on, and for the quotient ring these are,
$$
\begin{tabular}{cc}\hline
element of a ring & corresponding element in $F[x]/\langle f\rangle$\\\hline
$a$&$g+\langle f\rangle$\\
$-a$&$(-g)+\langle f\rangle$\\
$0$&$0+\langle f\rangle=\langle f\rangle$\\
$1$&$1+\langle f\rangle$\\\hline
\end{tabular}
$$

\begin{vexercise}
To prove this theorem:
\begin{enumerate}
\item Show that the addition of cosets is {\em well defined\/}, ie: if $g_i'+\langle f\rangle
=g_i+\langle f\rangle$, then
$$
(g_1'+g_2')+\langle f\rangle=(g_1+g_2)+\langle f\rangle.
$$
\item Similarly, show that the multiplication is well defined. 
\item Now verify the axioms for a ring.
\end{enumerate}
\end{vexercise}

\paragraph{\hspace*{-0.3cm}}
Let $x^2+1\in\R[x]$,
and consider the ideal $\langle x^2+1\rangle$. We want to see what the quotient 
$\R[x]/\langle x^2+1\rangle$
looks like. First, any coset can be put into a nice form: for example,
$$
x^4+x^2+x+1+\langle x^2+1\rangle=x^2(x^2+1)+(x+1)+\langle x^2+1\rangle,
$$
where we have divided $x^4+x^2+x+1$ by $x^2+1$ using the division algorithm. But
$$
x^2(x^2+1)+(x+1)+\langle x^2+1\rangle=x+1+\langle x^2+1\rangle,
$$
as the multiple of $x^2+1$ gets absorbed into the ideal. In fact, for any $g\in\R[x]$ 
we can make this argument:
$$
g+\langle x^2+1\rangle=q(x^2+1)+(ax+b)+\langle x^2+1\rangle=ax+b+\langle x^2+1\rangle,
$$
for some $a,b\in\R$,
so the set of cosets can be written as
$$
\R[x]/\langle x^2+1\rangle=\{ax+b+\langle x^2+1\rangle\,|\,a,b\in\R\}.
$$
Now take two elements of the quotient, say $(x+1)+\langle x^2+1\rangle$ and 
$(2x-3)+\langle x^2+1\rangle$, and add/multiply them together:
$$
\biggl\{(x+1)+\langle x^2+1\rangle\biggr\}+\biggl\{(2x-3)+\langle x^2+1\rangle\biggr\}
=3x-2+\langle x^2+1\rangle,
$$
and 
\begin{equation*}
\begin{split}
\biggl\{(x+1)+\langle x^2+1\rangle\biggr\}\times
\biggl\{(2x-3)+\langle x^2+1\rangle\biggr\}&=(2x^2-x-3)+ 
\langle x^2+1\rangle\\
&=2(x^2+1)+(-x-5)+\langle x^2+1\rangle\\
&=-x-5+\langle x^2+1\rangle.
\end{split}
\end{equation*}

Now ``squint your eyes'', so that 
$ax+b+\langle x^2+1\rangle$ becomes the complex number $a\imag+b\in\C$. Then
$$
(\imag+1)+(2\imag-3)=3\imag-2\text{ and }(\imag+1)(2\imag-3)=-\imag-5.
$$
The addition and multiplication of cosets in $\R[x]/\langle x^2+1\rangle$ looks exactly 
like the addition and
multiplication of complex numbers!

\paragraph{\hspace*{-0.3cm}}
To see what quotient rings look like we use:

\begin{isomthm}
Let $\varphi:F[x]\rightarrow S$ be a ring homomorphism with kernel
$\langle f\rangle$. Then the map
$g+\langle f \rangle\mapsto\varphi(g)$ is an isomorphism 
$$
F[x]/\langle f\rangle\rightarrow\im\,\varphi\subset S.
$$
\end{isomthm}

\paragraph{\hspace*{-0.3cm}}
In the example above let $R=\R[x]$ and $S=\C$. 
Let the homomorphism $\varphi$ be the evaluation at $i$ homomorphism,
$$
\ev_i:\biggl(\sum a_kx^k\biggr)\mapsto \sum a_k(i)^k.
$$
In exactly the same way as an earlier example, you can
show that 
$$
\ker\ev_i=\langle x^2+1\rangle.
$$
On the other hand, if $ai+b\in\C$, then $ai+b=\ev_i(ax+b)$, so the
image of the homomorphism  
$\ev_i$ is all of $\C$. Feeding this into the first homomorphism theorem gives,
$$
\R[x]/\langle x^2+1\rangle\cong \C.
$$

\subsection*{Further Exercises for Section \thesection}

\begin{vexercise}
Let $\phi=(1+\sqrt5)/2$ (in fact the {\it Golden Number\/}).
\begin{enumerate}
\item Show that the kernel of the evaluation map
$\epsilon_{\phi}:\Q[x]\rightarrow\C$ (given by $\epsilon_{\phi}(f)=f(\phi)$) is the
ideal $\langle x^2-x-1\rangle$. 
\item Show that $\Q(\phi)=\{a+b\phi\mid a,b\in\Q\}$. 
\item Show that $\Q(\phi)$ is the image in $\C$ of the map $\epsilon_{\phi}$.
\end{enumerate}
\end{vexercise}

\begin{vexercise}\label{ex5.3}
Going back to the general case of an ideal $I$ in a ring $R$, consider the map 
$\eta: R\rightarrow R/I$ given by,
$$
\eta(r)=r+I,
$$
sending an element of $R$ to the coset of $I$ determined by it.
\begin{enumerate}
\item Show that $\eta$ is a homomorphism.
\item Show that if $J$ is an ideal in $R$ containing $I$ then $\eta(J)$ is an ideal of $R/I$.
\item Show that if $J'$ is an ideal of $R/I$ then there is an ideal $J$ of $R$, containing $I$,
such that $\eta(J)=J'$.
\item Show that in this way, $\eta$ is a bijection between the ideals of $R$ containing $I$ and
the ideals of $R/I$.
\end{enumerate}
\end{vexercise}

\section{Fields II: Constructions and More Examples}\label{fields2}

\paragraph{\hspace*{-0.3cm}}
A proper ideal $\lg f\rg$ of $F[x]$ is {\em maximal\/} if and only if 
the only ideals of 
$F[x]$ containing $\lg f\rg$ are $\lg f\rg$ itself and the whole ring $F[x]$, ie:
$$
\lg f\rg\subseteq I\subseteq F[x],
$$
with $I$ an ideal implies that $I=\lg f\rg$ or $I=F[x]$.

\paragraph{\hspace*{-0.3cm}}
The main result of this section is,

\begin{theoremB}
The quotient ring $F[x]/\lg f\rg$ is a field if and only if $\lg f\rg$ 
is a maximal ideal.
\end{theoremB}

\begin{proof}
By Exercise \ref{ex5.1}, a commutative ring $R$ is a field if and only if the
only ideals of $R$ are the trivial one $\{0\}$ and the whole ring $R$.  
Thus the quotient $F[x]/\lg f\rg$ is a field if and only if its only ideals
are the trivial one $\lg f\rg$ and the whole ring $F[x]/\lg f\rg$.
By Exercise \ref{ex5.3}, there is a one to one correspondence between the
ideals of the quotient $F[x]/\lg f\rg$ and the ideals of $F[x]$ that 
contain $\lg f\rg$. Thus $F[x]/\lg f\rg$ has only the two trivial ideals
precisely when there are only two ideals of $F[x]$ containing $\lg
f\rg$, namely $\lg f\rg$ and $F[x]=\lg 1\rg$, 
which is the same as saying that $\lg f\rg$ is maximal.
\qed
\end{proof}

\paragraph{\hspace*{-0.3cm}}
Suppose now that $f$ is an irreducible polynomial over $F$, and let $\langle f\rangle
\subseteq I\subseteq F[x]$ with $I$ an ideal. Then $I=\langle h\rangle$ hence 
$\lg f \rg\subseteq \lg h\rg$, and so $h$ divides $f$. Since $f$ is irreducible this
means that $h$ must be either a constant $c\in F$ or $c f$, so that
the ideal $I$ is either $\lg c\rg$ or $\lg c f\rg$. But $\lg c f\rg$ is just the 
same as the ideal $\lg f\rg$. On the other hand, any polynomial $g$ can be written
as a multiple of $c$, just by setting $g=c(c^{-1}g)$, and so $\lg c\rg=F[x]$.
Thus if $f$ is an irreducible polynomial then the ideal $\lg f\rg$
is maximal.

Conversely, if $\lg f\rg$ is maximal and $h$ divides $f$, then $\lg f\rg\subseteq 
\lg h\rg$, so that by maximality $\lg h\rg=\lg f\rg$ or $\lg
h\rg=F[x]$. By Exercise \ref{fields2:Exercise20} we have $h=c$ a
constant, or $h=c f$, and so $f$ is irreducible over $F$.

Thus, the ideal $\lg f\rg$ is maximal precisely when $f$ is irreducible.

\begin{corollary} 
$F[x]/\lg f\rg$ is a field if and only if $f$ is an irreducible polynomial over $F$.
\end{corollary}

\paragraph{\hspace*{-0.3cm}}
The polynomial $x^2+1$ is irreducible over the reals $\R$, so 
the quotient ring $\R[x]/\lg x^2+1\rg$ is a field. 

\paragraph{\hspace*{-0.3cm}}
The polynomial $x^2-2x+2$ has roots $1\pm \imag$, hence is irreducible over $\R$, giving
the field,
$$
\R[x]/\lg x^2-2x+2\rg.
$$
Consider the evaluation map $\ve_{1+\imag}:\R[x]\rightarrow \C$ given as usual by 
$\ve_{1+\imag}(f)=f(1+\imag)$. In exactly the same way as the example for $\ve_{\sqrt{2}}$
in Section \ref{lect5}, one can show that $\ker\ve_{1+\imag}=\lg x^2-2x+2\rg$. Moreover,
$a+b\imag=\ve_{1+\imag}(a-b+bx)$ so that the evaluation map is onto $\C$. Thus, by the
first isomorphism theorem we get that,
$$
\R[x]/\lg x^2-2x+2\rg\cong\C.
$$
What this means is that we can construct the complex numbers in the following (slightly
non-standard) way: start with the reals $\R$, and define a new symbol, $\nabla$ say, which
satisfies the algebraic property,
$$
\nabla^2=2\nabla-2.
$$
Now consider all expressions of the form $c+d\nabla$ for $c,d\in\R$. Add and multiply
two such expressions together as follows:
\begin{equation*}
\begin{split}
(c_1+d_1\nabla)+(c_2+d_2\nabla)&=(c_1+c_2)+(d_1+d_2)\nabla\\
(c_1+d_1\nabla)(c_2+d_2\nabla)&=c_1c_2+(c_1d_2+d_1c_2)\nabla +d_1d_2\nabla^2\\
&=c_1c_2+(c_1d_2+d_1c_2)\nabla +d_1d_2(2\nabla-2)\\
&=(c_1c_2-2d_1d_2)+(c_1d_2+d_1c_2+2d_1d_2)\nabla.\\
\end{split}
\end{equation*}

\begin{vexercise}
By solving the equations $cx-2dy=1$ and $cy+dx+2dy=0$ for $x$ and $y$
in terms of $c$ and $d$, find the inverse of the element 
$c+d\nabla$. 
\end{vexercise}

\begin{vexercise}
According to Exercise \ref{ex3.21}, if $f$ is irreducible over $\R$ then $f$ must be either
quadratic or linear. Suppose that $f=ax^2+bx+c$ is an irreducible quadratic over $\R$.
Show that the field $\R[x]/\lg ax^2+bx+c\rg\cong\C$. 
\end{vexercise}

\paragraph{\hspace*{-0.3cm}}
The next few paragraphs illustrate the construction for finite fields,
using a field of order four as a running example.

In the process of doing the example in
\ref{irreducible_polynomials:paragraph50}
we saw that the only irreducible quadratic over the field $\F_2$ is $x^2+x+1$.
Thus the quotient
$$
\F_2[x]/\lg x^2+x+1\rg,
$$
is a field. Each of its elements is a coset of the form $g+\lg x^2+x+1\rg$. Use the 
division algorithm, dividing $g$ by $x^2+x+1$, to get 
$$
g+\lg x^2+x+1\rg=q(x^2+x+1)+r+\lg x^2+x+1\rg=r+\lg x^2+x+1\rg,
$$
where the remainder $r$ is of the form $ax+b$, for
$a,b\in\F_2$. Thus every element of the field has the form
$ax+b+\lg x^2+x+1\rg$, of which there are at most $4$ possibilities 
($2$ choices for
$a$ and $2$ choices for $b$). 

Indeed these $4$ are distinct, for if
$$
a_1x+b_1+\lg x^2+x+1\rg=a_2x+b_2+\lg x^2+x+1\rg
$$
then,
\begin{equation*}
\begin{split}
(a_1-a_2)x&+(b_1-b_2)+\lg x^2+x+1\rg\\
&=\lg x^2+x+1\rg
\Leftrightarrow (a_1-a_2)x+(b_1-b_2)\in\lg x^2+x+1\rg.
\end{split}
\end{equation*}
Since the non-zero elements of the ideal are multiples of a degree two polynomial, they 
have degrees that are  at least two. Thus the only way the linear polynomial can be an element
is if it is the zero polynomial. In particular, $a_1-a_2=b_1-b_2
=0$, so the two cosets are the same.
The quotient ring is thus a field having the four elements:
$$
\F_4=\{ax+b+\lg x^2+x+1\rg\,|\,a,b\in\F_2\}
$$

\paragraph{\hspace*{-0.3cm}}
Generalising the example of the field of order $4$ above, 
if $\F_p$ is the finite field with $p$
elements and $f\in\F_p[x]$ is an irreducible polynomial of degree $d$, then
the quotient $\F_p[x]/\lg f \rg$ is a field containing elements
of the 
form,
$$
a_{d-1} x^{d-1}+\cdots a_0+\lg f\rg,
$$
where 
the $a_i\in\F_p$. 
Any two such are distinct by exactly the
same argument as above, so we have a field $\F_q$ with exactly $q=p^d$
elements. 

\paragraph{\hspace*{-0.3cm}}
Returning to the general situation of a quotient 
$F[x]/\lg f\rg$ by an irreducible polynomial $f$, the resulting field contains 
a copy of the original field $F$, obtained by considering the cosets
$a+\lg f\rg$ for $a\in F$.

\begin{vexercise}
Show that the map $a\mapsto a+\lg f\rg$ is an injective homomorphism
$F\rightarrow F[x]/\lg f\rg$, and so $F$ is isomorphic to its image in $F[x]/\lg f\rg$.
\end{vexercise}

Blurring the distinction between the original $F$ and this copy inside 
$F[x]/\lg f\rg$, we get that $F\subset F[x]/\lg f\rg$ is an extension of fields.

\paragraph{\hspace*{-0.3cm}}
Back to the field $\F_4$ of order $4$ and a more convenient notation. Let 
$$
\aa=x+\lg x^2+x+1\rg
$$
and write $a\in\F_2$ for the coset $a+\lg x^2+x+1\rg$ as in the previous
paragraph. Addition and multiplication of cosets gives:
$$
ax+b+\lg x^2+x+1\rg
=(a+\lg x^2+x+1\rg)(x+\lg x^2+x+1\rg)+(b+\lg x^2+x+1\rg)
=a\aa+b.
$$
So we now have that
$\F_4=\{a\aa+b\,|\,a,b\in\F_2\}=\{0,1,\aa,\aa+1\}$. But we also have
the coset property $f+\lg f\rg=\lg f\rg$, which for $f=x^2+x+1$
translates into 
$$
(x+\lg x^2+x+1\rg)^2+(x+\lg x^2+x+1\rg) + (1+\lg x^2+x+1\rg)=\lg x^2+x+1\rg,
$$
or, $\aa^2+\aa+1=0$. Our field is now $\F_4=\{0,1,\aa,\aa+1\}$,
together with the ``rule'' $\aa^2=\aa+1$.

At the risk of labouring the point, here are the multiplication tables
for the field $\F_4$ and the ring $\Z_4$:
$$
\begin{tabular}{c|cccc}
  $\F_4$&$0$&$1$&$\aa$&$\aa+1$\\
   \hline
  $0$&$0$&$0$&$0$&$0$\\
  $1$&$0$&$1$&$\aa$&$\aa+1$\\
  $\aa$&$0$&$\aa$&$\aa+1$&$1$\\
  $\aa+1$&$0$&$\aa+1$&$1$&$\aa$\\
\end{tabular}
\quad \quad
\begin{tabular}{c|cccc}
  $\Z_4$&$0$&$1$&$2$&$3$\\
   \hline
  $0$&$0$&$0$&$0$&$0$\\
  $1$&$0$&$1$&$2$&$3$\\
  $2$&$0$&$2$&$0$&$2$\\
  $3$&$0$&$3$&$2$&$1$\\
\end{tabular}
$$
$1$ appears in every non-zero row of the $\F_4$ table -- so every
non-zero element has an inverse -- but does not appear in every
non-zero row of $\Z_4$.

\paragraph{\hspace*{-0.3cm}}
In general, when $f\in\F_p[x]$ is irreducible of degree $d$, we 
let $\aa=x+\lg f\rg$ and replace $\F_p$ by its copy in 
$\F_p[x]/\lg f\rg$ (ie: identify $a\in\F_p$ with $a+\lg f\rg\in\F_p[x]/\lg f\rg$).
This gives,
$$
\F_p[x]/\lg f \rg=\{a_{d-1} \aa^{d+1}+\cdots a_0\,|\,a_i\in\F_p\},
$$
where two such expressions are added and multiplied like ``polynomials'' in $\aa$. 
If $f=b_{d}x^{d}+\cdots+b_1x+b_0$, and 
since $f+\lg f\rg=\lg f \rg$, we have the 
``rule''\/ $b_{d}\aa^{d}+\cdots+b_1\aa+b_0=0$, which allows us to remove
any powers of $\aa$ bigger than $d$ that occur in such expressions.
The element $\aa$ is called a {\em generator\/} for the field.

\paragraph{\hspace*{-0.3cm}}
The polynomial $x^3+x+1$ is irreducible over the
field $\F_2$ (it is a cubic and has no roots) so that
$$
\F_2[x]/\lg x^3+x+1\rg,
$$
is a field with $2^3=8$ elements of the form $\F=\{a+b\aa+c\aa^2\,|\,a,b,c\in\F_2\}$
subject to the rule $\aa^3+\aa+1=0$, ie: $\aa^3=\aa+1$. This is the field $\F$ of 
order $8$ from Section \ref{lect4}.

\begin{vexercise}
Explicitly construct fields with exactly:
$$
1. \,\, 125\text{ elements}\qquad
2. \,\, 49\text{ elements}\qquad
3. \,\, 81\text{ elements}\qquad
4. \,\, 243\text{ elements}\qquad
$$
(By explicity I mean give a general description of the elements and
any algebraic rules that are needed for adding and multiplying them together.)
\end{vexercise}

\paragraph{\hspace*{-0.3cm}}
To explicitly construct a field of order $p^d$ with $d>3$ is
harder -- finding irreducible polynomials of degree bigger than a
cubic is not straightforward, as the example in
\ref{irreducible_polynomials:paragraph50} shows. One solution is to
create the field in a series of steps (or extensions), each of which only involves
quadratics or cubics. 

We do this for a field of order $729=3^6$. As $3^6=(3^2)^3$, we first
create a field of order $3^2$, and then extend this using a cubic.

Consider the polynomial $f=x^2+x+2\in\F_3[x]$. Substituting the three elements of
$\F_3$ into $f$ gives
$$
0^2+0+2=2, 1^2+1+2=1\text{ and }2^2+2+2=2,
$$
so that $f$ has no roots in $\F_3$. As $f$ is quadratic it is 
irreducible over the field $\F_3$, and so $\F_9=\F_3[x]/\lg x^2+x+2\rg$ is a field of
order $3^2$. 
Let $\aa=x+\langle x^2+x+2\rangle$ in $\F_9$ be a generator
so that the elements have the form $a+b\aa$ with $a,b\in\F_3$
and multiplication satisfying the rule $\aa^2+\aa+2=0$, or 
equivalently 
$\aa^2=2\aa+1$ ($-1=2$ and $-2=1$ in $\F_3$).

Now let $X$ be a new variable, and consider the polynomials $\F_9[X]$ over $\F_9$
in this new variable. In particular the polynomial:
\begin{equation}
  \label{eq:8}
g=X^3+(2\aa+1)X+1.  
\end{equation}
As $g$ is a cubic, it will be irreducible over $\F_9$ precisely when it has no roots 
in this field, which can be verified as usual by straight substitution
of the nine elements of $\F_9$. For example:
\begin{equation*}
\begin{split}
g(2\aa+1)&=(2\aa+1)^3+(2\aa+1)(2\aa+1)+1=2\aa^3+1+\aa^2+\aa+1+1\\
&=2\aa(2\aa+1)+\aa^2+\aa\\
&=\aa^2+2\aa+\aa^2+\aa=\aa+2
\end{split}
\end{equation*}
and the others are similar. We have a used an energy saving device in these computations:

\begin{vexercise}
\label{ex9.1}
If $a,b\in F$, a field of characteristic $p>0$, then $(a+b)^p=a^p+b^p$ (hint:
Exercise \ref{ex3.3}).
\end{vexercise}

Thus the polynomial $g$ in (\ref{eq:8}) is irreducible over $\F_9$,
and we have a field:
$$
\F_9[X]/\lg X^3+(2\aa+1)X+1\rg
$$
of order $9^3=3^6=729$, called $\F_{729}$. 
The elements have the form,
$$
A_0+A_1\bb+A_2\bb^2,
$$
where the $A_i\in\F_9$ and $\bb=X+\lg g\rg$ is a generator. Multiplication is given by the rule 
$\bb^3=(\aa+2)\bb+2$. 
Replacing the $A_i$ by the earlier description of $\F_9$ in terms of
the generator $\aa$ gives elements:
$$
a_0+a_1\bb+a_2\bb^2+a_3\aa+a_4\aa\bb+a_5\aa\bb^2,
$$
with the $a_i\in\F_3$, subject to the rules $\aa^2=2\aa+1$ and $\bb^3=(\aa+2)\bb+\aa$.

\begin{vexercise}
\hspace{1em}\begin{enumerate}
\item Construct a field $\F_8$ with 8 elements by showing that $x^3+x+1$ is
irreducible over $\F_2$.
\item Find a cubic polynomial that is irreducible in $\F_8[x]$ (hint:
refer to Exercise \ref{ex_lect3.1}).
\item Hence, or otherwise, construct a field with $2^9=512$ elements.
\end{enumerate}
\end{vexercise}

\begin{vexercise}
Explicitly construct fields with exactly:
$$
1. \,\, 64\text{ elements}\qquad
2. \,\, \text{\emph{challenge:}}\,\,4096\text{ elements}\qquad
$$
\end{vexercise}

\paragraph{\hspace*{-0.3cm}}
Theorem B and its Corollary solves the problem that we encountered in
Section \ref{lect4} where the fields
$$
\Q(\sqrt[3]{2})\text{ and }
\Q\biggl(\frac{-\sqrt[3]{2}+\sqrt[3]{2}\kern-2pt\sqrt{3}\imag}{2}\biggr)=\Q(\beta)
$$
were different but isomorphic.
The polynomial $x^3-2$ is irreducible over $\Q$, either by Eisenstein, or by observing
that its roots do not lie in $\Q$. Thus
$$
\Q[x]/\lg x^3-2\rg,
$$
is an extension field of $\Q$. 
Consider the two evaluation homomorphisms $\ve_{\sqrt[3]{2}}:\Q[x]\rightarrow \C$
and $\ve_{\bb}:\Q[x]\rightarrow \C$. Since, and this is the key bit,
$$
\sqrt[3]{2}\text{ and }
\beta=\frac{-\sqrt[3]{2}+\sqrt[3]{2}\kern-2pt\sqrt{3}\imag}{2}
$$
are both roots of the polynomial $x^3-2$, we can show in a similar manner to 
examples at the end of Section \ref{lect5} that $\ker\ve_{\sqrt[3]{2}}\cong\lg x^3-2\rg
\cong \ker\ve_{\bb}$. Thus,
$$
\begin{pspicture}(0,0)(12,3)
\rput(0,.5){
\rput(6,2){$\Q[x]/\lg x^3-2\rg$}
\rput(2,2){$\Q[x]/\ker\ve_{\sqrt[3]{2}}$}
\rput(10,2){$\Q[x]/\ker\ve_{\bb}$}
\psline[linewidth=.1mm]{->}(4.8,2)(3.2,2)
\psline[linewidth=.1mm]{->}(7.2,2)(8.8,2)
\rput(4,2.2){$=$}
\rput(8,2.2){$=$}
}
\rput(2.2,1.5){$\cong$}
\rput(9.8,1.5){$\cong$}
\rput(0,.5){
\rput(6,1){{\red $1^{\text{st}}$ Isomorphism Theorem}}
\psline[linewidth=.1mm,linecolor=red]{->}(4,1)(2.6,1)
\psline[linewidth=.1mm,linecolor=red]{->}(8,1)(9.4,1)
}
\rput(2,.5){$\im\ve_{\sqrt[3]{2}}$}
\rput(10,.5){$\im\ve_{\bb}$}
\psline[linewidth=.1mm]{->}(2,2.2)(2,.8)
\psline[linewidth=.1mm]{->}(10,2.2)(10,.8)
\rput(13,1.5){(*)}
\end{pspicture}
$$
To find the image of $\ve_{\sqrt[3]{2}}$ write a $g\in\Q[x]$ as
$g=q(x^3-2)+(a+bx+cx^2)$ so that
\begin{equation*}
\begin{split}
\ve_{\sqrt[3]{2}}(g)&=
\ve_{\sqrt[3]{2}}(q(x^3-2)+(a+bx+cx^2))\\
&=
\ve_{\sqrt[3]{2}}(q)\ve_{\sqrt[3]{2}}(x^3-2)+\ve_{\sqrt[3]{2}}(a+bx+cx^2)\\
&=
\ve_{\sqrt[3]{2}}(q).0+\ve_{\sqrt[3]{2}}(a+bx+cx^2)
=a+b\sqrt[3]{2}+c(\sqrt[3]{2})^2.
\end{split}
\end{equation*}
Hence
$\im\ve_{\sqrt[3]{2}}\subseteq\{a+b\sqrt[3]{2}+c(\sqrt[3]{2})^2\in\C\,|\,a,b,c\in\Q\}
=\Q(\sqrt[3]{2})$.

On the other hand $a+b\sqrt[3]{2}+c(\sqrt[3]{2})^2$
is the image of $a+bx+cx^2$ and so $\im\varepsilon_{\sqrt[3]{2}}=\Q(\sqrt[3]{2})$. 
Similarly 
$\im\ve_\bb=\Q(\bb)$. 
Filling this information into the diagram (*) above gives the claimed
isomorphism between
$\Q(\sqrt[3]{2})$ and $\Q(\bb)$:
$$
\begin{pspicture}(0,0)(12,3)
\rput(0,.5){
\rput(6,2){$\Q[x]/\lg x^3-2\rg$}
\rput(2,2){$\Q[x]/\ker\ve_{\sqrt[3]{2}}$}
\rput(10,2){$\Q[x]/\ker\ve_{\bb}$}
\psline[linewidth=.1mm]{->}(4.8,2)(3.2,2)
\psline[linewidth=.1mm]{->}(7.2,2)(8.8,2)
\rput(4,2.2){$=$}
\rput(8,2.2){$=$}
}
\rput(2.2,1.5){$\cong$}
\rput(9.8,1.5){$\cong$}
\rput(0,.5){
}
\psline[linewidth=.1mm,linecolor=red]{->}(6,1.75)(6,2.1)
\rput(6,1.5){{\red abstract field}}
\rput(2,.5){$\Q(\sqrt[3]{2})$}
\rput(10,.5){$\Q(\bb)$}
\psline[linewidth=.1mm]{->}(2,2.2)(2,.8)
\psline[linewidth=.1mm]{->}(10,2.2)(10,.8)
\rput(6,.5){{\red concrete versions in $\C$}}
\psline[linewidth=.1mm,linecolor=red]{<-}(2.75,.5)(4.2,.5)
\psline[linewidth=.1mm,linecolor=red]{->}(7.8,.5)(9.5,.5)
\end{pspicture}
$$

\paragraph{\hspace*{-0.3cm}}
In algebraic number theory 
a field $\Q[x]/\lg f\rg$, for $f$ an irreducible
polynomial over $\Q$, is called a \emph{number field}.
If $\{\bb_1,\ldots,\bb_n\}$ are the roots of $f$, then we have $n$ mutually
isomorphic fields $\Q(\bb_1),\ldots, \Q(\bb_n)$ inside $\C$. The isomorphisms
from $\Q[x]/\lg f\rg$ to each of these are called the {\em Galois monomorphisms\/}
of the number field. 

\paragraph{\hspace*{-0.3cm}}
Returning to a general field:

\begin{kronecker}
\label{kronecker}
Let $f$ be a polynomial in $F[x]$. Then there is an extension field of $F$ 
containing a root of $f$.
\end{kronecker}

\begin{proof}
If $f$ is not irreducible over $F$, then factorise as $f=gh$ with $g$ irreducible
over $F$ and proceed as below but with $g$ instead of $f$. The result will be an 
extension field containing a root of $g$, and hence of $f$.
Thus we may suppose that $f$ is irreducible over $F$ and
$f=a_nx^n+a_{n-1}x^{n-1}+\cdots a_1x+a_0$ with the $a_i\in F$. Replace 
$F$ by its isomorphic copy in the quotient $F[x]/\lg f\rg$, so that instead of
$a_i$, we write $a_i+\lg f\rg$, ie,
$$
f=(a_n+\lg f\rg)x^n+(a_{n-1}+\lg f\rg)x^{n-1}+\cdots +(a_1+\lg f\rg)x+(a_0+\lg f\rg).
$$
Consider the field $E=F[x]/\lg f\rg$ which is an extension of 
$F$ and the element $\mu=x+\lg f\rg\in E$. If we substitute $\mu$ into the
polynomial then we perform all our arithmetic in $E$, ie: we perform the arithmetic
of cosets, and the zero of this field is the coset $\lg f\rg$:
\begin{equation*}
\begin{split}
f(\mu) &=f(x+\lg f\rg)\\
&=
(a_n+\lg f\rg)(x+\lg f\rg)^n+(a_{n-1}+\lg f\rg)(x+\lg f\rg)^{n-1}
+\cdots +(a_1+\lg f\rg)(x+\lg f\rg)+(a_0+\lg f\rg)\\
&=
(a_nx^n+\lg f\rg)+(a_{n-1}x^{n-1}+\lg f\rg)+\cdots +(a_1x+\lg f\rg)+(a_0+\lg f\rg)\\
&=
(a_nx^n+a_{n-1}x^{n-1}+\cdots +a_1x+a_0)+\lg f\rg=f+\lg f\rg=\lg f\rg=0.
\end{split}
\end{equation*}
i.e. for $\mu=x+\lg f\rg\in E$ we have $f(\mu)=0$. 
\qed
\end{proof}

\begin{corollary}
\label{kronecker:corollary}
Let $f$ be a polynomial in $F[x]$. Then there is an extension field of $F$ 
that contains all the roots of $f$.
\end{corollary}

\begin{proof}
Repeat the process described in the proof of Kronecker's Theorem at most
$\deg f$ number of times, until the desired field is obtained. 
\qed
\end{proof}

\subsection*{Further Exercises for Section \thesection}

\begin{vexercise}
Show that $x^4+x^3+x^2+x+1$ is irreducible over $\F_3$.
How many elements does the resulting extension of $\F_3$ have?
\end{vexercise}

\begin{vexercise}
As linear polynomials are always irreducible, show that the field
$F[x]/\lg ax+b\rg$ is isomorphic to $F$.
\end{vexercise}

\begin{vexercise}\label{ex6.1}
\hspace{1em}\begin{enumerate}
\item Show that $1+2x+x^3\in\F_3[x]$ is irreducible and hence that 
$\F=\F_3[x]/\langle
1+2x+x^3\rangle$ is a field.
\item Show that every coset can be written uniquely in 
the form $(a+bx+cx^2)+\langle
1+2x+x^3\rangle$ with $a,b,c\in\F_3$.
\item Deduce that the field $\F$ has exactly 27 elements.
\end{enumerate}
\end{vexercise}

\begin{vexercise}
Find an irreducible polynomial $f(x)$ in $\F_5[x]$ of degree $2$. Show
that $\F_5[x]/\langle f(x)\rangle$ is a field with $25$ elements.
\end{vexercise}

\begin{vexercise}\label{ex8.3}
\hspace{1em}\begin{enumerate}
\item Show that the polynomial $x^3-3x+6$ is irreducible over $\Q$.
\item Hence, or otherwise, if
$$
\alpha=\sqrt[3]{2\kern-2pt\sqrt{2}-3},\beta=-\sqrt[3]{2\kern-2pt\sqrt{2}+3}\mbox{ and
}\ww=-\frac{1}{2}+
\frac{\sqrt{3}}{2}\imag,
$$
prove that 
\begin{enumerate}
\item the fields $\Q(\alpha+\beta)$ and
$\Q(\omega\alpha+\overline{\omega}\beta)$ are 
distinct (that is, their elements are different), but,
\item $\Q(\alpha+\beta)$ and
$\Q(\omega\alpha+\overline{\omega}\beta)$ are isomorphic (You can
assume that $\omega\alpha+\overline{\omega}\beta$ is not a 
real number.)
\end{enumerate}
\end{enumerate}
\end{vexercise}


\section{Ruler and Compass Constructions I}
\label{ruler.compass}

If you are a farmer in Babylon around 2500 BC, how do you subdivide your
land into plots? You survey it of course.
The most basic surveying instruments are wooden pegs and rope, with which
you can do two very basic things: two pegs can be set a distance apart and the
rope stretched taut between them; also, one of the pegs can be kept stationary
and you can take the path traced by the other as you walk around keeping the rope
stretched tight. In other words, you can draw a line through two
points or you can draw a
circle centered at one point and passing through another.

\paragraph{\hspace*{-0.3cm}}
Instead of the Euphrates river valley, we work in the
complex plane $\C$. We are thus able, given two numbers $z,w\in\C$, to draw
a line through them using a straight edge, or to place one end of a
compass at $z$, and draw 
the circle passing through $w$:
$$
\begin{pspicture}(0,0)(10,3.25)
\rput(5,1.5){\BoxedEPSF{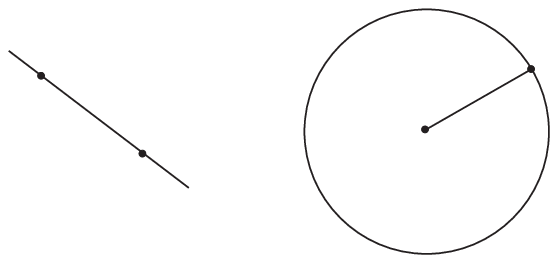 scaled 1500}}
\rput(1.35,2.15){$z$}
\rput(2.95,1.45){$w$}
\rput(7,1.5){$z$}
\rput(9.1,2.5){$w$}
\end{pspicture}
$$
Neither of these operations involves any ``measuring''. There are no
units on the ruler and we don't know the radius of the circle. 

\paragraph{\hspace*{-0.3cm}}
With these two constructions we call a complex number $z$ 
{\em constructible\/}
iff there is a sequence of numbers 
$$
0,1,\imag=z_1,z_2,\ldots,z_n=z,
$$
with $z_j$ obtained from earlier numbers in the sequence in one of the 
following three ways:
$$
\begin{pspicture}(0,0)(14,3)
\rput(1,0){
\rput(2,1.34){\BoxedEPSF{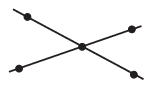 scaled 1500}}
\rput(2.1,1.6){$z_j$}
\rput(1.1,.7){$z_p$}\rput(2.8,1.9){$z_q$}
\rput(1.2,2){$z_r$}\rput(2.9,.6){$z_s$}
\rput(2,.2){(i)}
}
\rput(1,0){
\rput(6,1.5){\BoxedEPSF{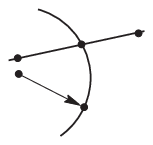 scaled 1500}}
\rput(6.1,2.3){$z_j$}
\rput(5,2){$z_p$}\rput(6.9,2.4){$z_q$}
\rput(6.4,.8){$z_r$}\rput(5.1,1.2){$z_s$}
\rput(6.2,.2){(ii)}
}
\rput(1,0){
\rput(10,1.58){\BoxedEPSF{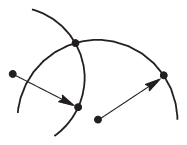 scaled 1500}}
\rput(9.8,2.3){$z_j$}
\rput(10.4,.8){$z_p$}\rput(11.3,1.7){$z_q$}
\rput(9.5,1){$z_r$}\rput(8.7,1.8){$z_s$}
\rput(10,.2){(iii)}
}
\end{pspicture}
$$
In these pictures, $p,q,r$ and $s$ are all $<j$. We are given 
$0,1,\imag$ ``for free'', so they are indisputably constructible.
The reasoning is this: if you stand in a plane (without
coordinates) then your position can be taken as $0$; declare a direction to be the
real axis and a distance along it to be length $1$; construct the perpendicular 
bisector of the segment from $-1$ to $1$ (as in the next paragraph) and 
measure a unit distance along this new axis (in either direction) to get $\imag$.

\paragraph{\hspace*{-0.3cm}}
In addition to the two basic moves there are others
that follow immediately from them.
For example, we can construct the perpendicular bisector of a segment $AB$
as in Figure \ref{fig:constructions:figure50}.

\begin{figure}
  \centering
  \begin{pspicture}(0,0)(14,4.5)
\rput(0,0.25){
\rput(0,0.5){
\rput(2,1.5){\BoxedEPSF{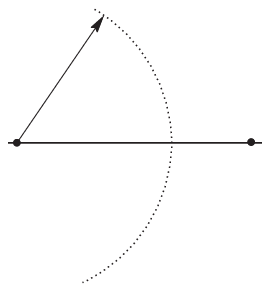 scaled 1500}}
\rput(.7,1.3){$A$}\rput(3.1,1.3){$B$}
\rput(1.1,2.2){$r$}
\pscircle[linecolor=white,fillstyle=solid,fillcolor=red](.5,2.5){.25}
\rput(.5,2.5){{\white{\bf 1}}}
}
\rput(1,0.5){
\rput(6,1.5){\BoxedEPSF{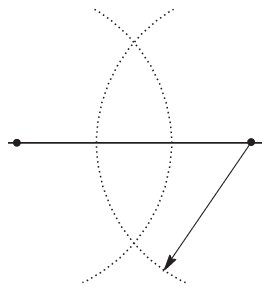 scaled 1500}}
\rput(6.9,.8){$r$}
\pscircle[linecolor=white,fillstyle=solid,fillcolor=red](5,2.5){.25}
\rput(5,2.5){{\white{\bf 2}}}
}
\rput(2,0.5){
\rput(10,1.5){\BoxedEPSF{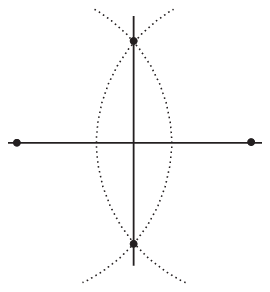 scaled 1500}}
\pscircle[linecolor=white,fillstyle=solid,fillcolor=red](9,2.5){.25}
\rput(9,2.5){{\white{\bf 3}}}
}}
\end{pspicture}
  \caption{Constructing the perpendicular bisector of a segment.}
  \label{fig:constructions:figure50}
\end{figure}

\begin{figure}
  \centering
\begin{pspicture}(0,0)(14,3.5)
\rput(0,0.25){
\rput(0,0){
\rput(2,1.34){\BoxedEPSF{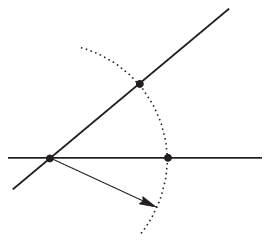 scaled 1500}}
\pscircle[linecolor=white,fillstyle=solid,fillcolor=red](.5,2.5){.25}
\rput(.5,2.5){{\white{\bf 1}}}
}
\rput(1,0){
\rput(6,1.5){\BoxedEPSF{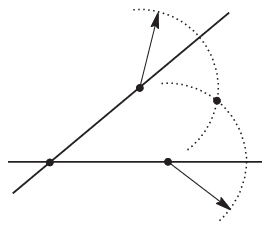 scaled 1500}}
\pscircle[linecolor=white,fillstyle=solid,fillcolor=red](5,2.5){.25}
\rput(5,2.5){{\white{\bf 2}}}
\rput(6,2.5){$r$}\rput(6.7,.4){$r$}
}
\rput(2,0){
\rput(10,1.58){\BoxedEPSF{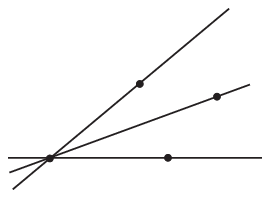 scaled 1500}}
\pscircle[linecolor=white,fillstyle=solid,fillcolor=red](9,2.5){.25}
\rput(9,2.5){{\white{\bf 3}}}
}}
\end{pspicture}  
  \caption{Bisecting an angle.}
  \label{fig:constructions:figure60}
\end{figure}

To explain these pictures (and the rest): a
ray, centered at some point and tracing out a dotted circle is the
compass. If the ray is marked $r$ -- as in the first two  
pictures above -- this means that in passing from the
first picture to the second, the setting on the compass is kept the
same. It does not mean that we know the setting. 

The
construction works for the following reason: let $S$ be the set of
points in $\C$ that 
are an equal distance from both $A$ and $B$. 
After a moments thought, you see that this must be
the perpendicular bisector of the line segment $\overline{AB}$ that we
are constructing. Lines are determined by any two of their points, so
if we can find two points equidistant from $A$ and $B$, and we draw
a line through them, this must be the set $S$ that we want (and hence
the perpendicular bisector). But the intersections of the two circular
arcs are clearly equidistant from $A$ and $B$, so we are done.

\paragraph{\hspace*{-0.3cm}}
As well as bisecting segments, we can bisect angles, ie: if two lines meet in
some angle we can construct a third line meeting these in angles that are each
half the original one -- see Figure \ref{fig:constructions:figure60}.
Remember: none of the angles in this picture can be measured. 
Nevertheless, the two new angles are half the old one.

\paragraph{\hspace*{-0.3cm}}
Given a line and a point $P$ not on it, we can construct a new line passing through
$P$ and perpendicular to the line, as in Figure
\ref{fig:constructions:figure70}. This is called ``dropping a
perpendicular from 
a point to a line''.

\begin{figure}
  \centering
\begin{pspicture}(0,0)(14,3.5)
\rput(0,0.25){
\rput(1,-0.2){
\rput(2,1.5){\BoxedEPSF{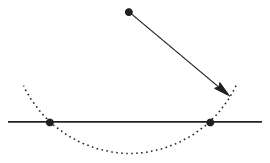 scaled 1500}}
\rput(.6,.65){$A$}\rput(3.2,.65){$B$}
\rput(1.6,2.5){$P$}
\pscircle[linecolor=white,fillstyle=solid,fillcolor=red](.5,2.5){.25}
\rput(.5,2.5){{\white{\bf 1}}}
}
\rput(3,-0.2){
\rput(6,1.6){\BoxedEPSF{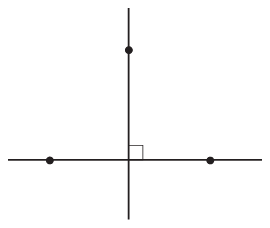 scaled 1500}}
\psline[linecolor=red,linewidth=.1mm]{->}(6.4,1.6)(5.95,1.6)
\rput(8.5,1.6){{\red $\text{ perpendicular bisector of }AB$}}
\rput(4.6,.65){$A$}\rput(7.1,.65){$B$}
\pscircle[linecolor=white,fillstyle=solid,fillcolor=red](5,2.5){.25}
\rput(5,2.5){{\white{\bf 2}}}
}}
\end{pspicture}  
  \caption{Dropping a perpendicular from a point to a line.}
  \label{fig:constructions:figure70}
\end{figure}

\paragraph{\hspace*{-0.3cm}}
Given a line $\ell$ and a point $P$ not on it we can construct a new line through $P$
parallel to $\ell$ -- see Figure \ref{fig:figure5}. Some explanation
for this one: the first step is to drop a perpendicular 
from $P$ to the line $\ell$, meeting it at the new point $Q$. Next, set your compass to
the distance from $P$ to $Q$, and transfer this circular distance along the line to 
some point, drawing a semicircle that meets $\ell$ at the points $A$ and $B$. 
Construct the perpendicular bisector of the segment from $A$ to $B$, which meets
the semicircle at the new point $R$. Finally, draw a line through the points $P$
and $R$. 

\begin{figure}
  \centering
\begin{pspicture}(0,0)(11,6)
\rput(.5,0){
\rput(1,4.55){\BoxedEPSF{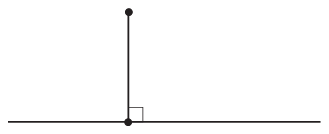 scaled 1500}}
\rput(.7,5.3){$P$}\rput(.7,3.5){$Q$}
\rput(2.8,3.5){{\red $\text{line }\ell$}}
\psline[linecolor=red,linewidth=.1mm]{<-}(.5,4.5)(.9,4.5)
\rput(2.8,4.5){{\red $\text{perpendicular from }P\text{ to }\ell$}}
}
\rput(-.5,0){
\rput(9.5,4.5){\BoxedEPSF{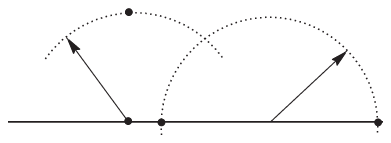 scaled 1500}}
\rput(8.5,3.4){$Q$}\rput(8.5,5.1){$P$}
\rput(9,3.4){$A$}\rput(12.25,3.4){$B$}
}
\rput(7.65,4.4){$r$}\rput(10.4,4.2){$r$}
\rput(.5,0){
\rput(1,1.68){\BoxedEPSF{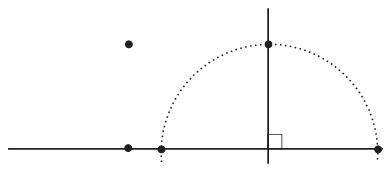 scaled 1500}}
\rput(0,.4){$Q$}\rput(.2,2.3){$P$}
\rput*(2.4,2.3){$R$}
\rput(.5,.4){$A$}\rput(3.7,.4){$B$}
\psline[linecolor=red,linewidth=.1mm]{->}(1.65,1.5)(2.05,1.5)
\rput*(-.45,1.5){{\red $\text{perpendicular bisector of }AB$}}
}
\rput(-.5,0){
\rput(9.5,1.5){\BoxedEPSF{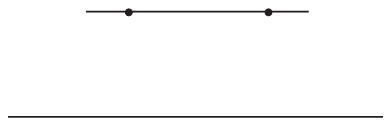 scaled 1500}}
\rput(8.5,2){$P$}
\rput(10.65,2){$R$}
}
\pscircle[linecolor=white,fillstyle=solid,fillcolor=red](-.5,5.5){.25}
\rput(-.5,5.5){{\white{\bf 1}}}
\pscircle[linecolor=white,fillstyle=solid,fillcolor=red](6,5.5){.25}
\rput(6,5.5){{\white{\bf 2}}}
\pscircle[linecolor=white,fillstyle=solid,fillcolor=red](-.5,2.5){.25}
\rput(-.5,2.5){{\white{\bf 3}}}
\pscircle[linecolor=white,fillstyle=solid,fillcolor=red](6,2.5){.25}
\rput(6,2.5){{\white{\bf 4}}}
\end{pspicture}
\caption{Constructing a line through a point $P$ and parallel to
  another line $\ell$.}
  \label{fig:figure5}
\end{figure}

\begin{figure}
  \centering
\begin{pspicture}(0,0)(14,3.5)
\rput(0,0.25){
\rput(2,1.34){\BoxedEPSF{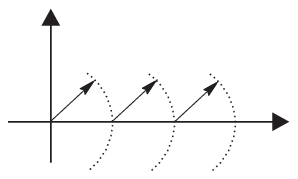 scaled 1500}}
\rput(7,1.3){\BoxedEPSF{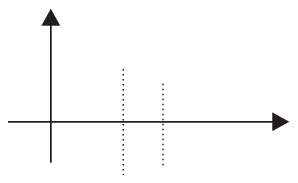 scaled 1500}}
\rput(12,1.58){\BoxedEPSF{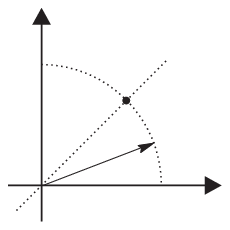 scaled 1500}}
\rput(1.4,.6){$1$}\rput(2.4,.6){$2$}\rput(3.4,.6){${\red 3}$}
\rput(12.7,.2){$1$}
\rput(13.5,1.8){${\displaystyle{\red \frac{1}{\sqrt{2}}+\frac{1}{\sqrt{2}}\imag}}$}
\psline[linewidth=.1mm,linecolor=red]{<-}(12.25,1.8)(12.6,1.8)
\rput(6.4,.4){${\displaystyle \frac{1}{2}}$}
\rput(7,.4){${\displaystyle{\red \frac{3}{4}}}$}
\rput(7.7,.6){$1$}
\rput*(12.5,2.7){{\red bisector of the right angle}}
\rput*(7,2.5){{\red bisect $[0,1]\ldots$}}
\rput*(7.5,2){{\red $\ldots$then bisect $[\frac{1}{2},1]$}}
\psline[linecolor=red,linewidth=.1mm]{->}(12,2.6)(12.5,2.2)
}
\end{pspicture}  
  \caption{Constructing $3,\frac{3}{4}$ and $\frac{1}{\sqrt{2}}+\frac{1}{\sqrt{2}}\imag$.}
  \label{fig:constructions:figure80}
\end{figure}

\paragraph{\hspace*{-0.3cm}}
Figure \ref{fig:constructions:figure80} shows some basic examples of constructible numbers.
It is less clear how to construct $\frac{27}{129}$, or the golden ratio:
$$
\phi=\frac{1+\sqrt{5}}{2}.
$$
But these numbers {\em are\/} constructible, and the reason 
is the first non-trivial fact about constructible numbers: they can be added,
subtracted, multiplied and 
divided.
Defining $\CC$ to be the set of constructible numbers in $\C$, we have,

\begin{theoremC}
$\CC$ is a subfield\footnote{In principle you can now throw away your calculator and instead use
ruler and compass! To compute $\cos x$ of a constructible number $x$
for example,construct as many terms of the Taylor series,
$$
\cos x=1-\frac{x^2}{2!}+\frac{x^4}{4!}-\cdots
$$
as you need (your calculator only ever gives you approximations anyway).} 
of 
$\C$.
\end{theoremC}

\begin{proof}
We show first that the \emph{real\/} constructible numbers form a subfield of
the reals, i.e. that $\CC\cap\R$ is a subfield of $\R$, 
for which we need to show that if $a,b\in \CC\cap \R$ then so too are
$a+b, -a, ab$
and $1/a$. 
\begin{enumerate}
\item {\em $\CC\cap \R$ is closed under $+$ and $-$}:
The picture on the left of Figure \ref{fig:figure20} shows that if
$a\in\CC\cap \R$  then so is 
$-a$.
Similarly, the two on the right of Figure \ref{fig:figure20} give $a,b\in
\CC\cap\R\Rightarrow a+b\in 
\CC\cap\R$. (In these pictures $a$ and $b$ are $>0$. You can draw the
other cases yourself).

\begin{figure}
  \centering
\begin{pspicture}(0,0)(14,4)
\rput(0,0.5){
\rput(2,1.5){\BoxedEPSF{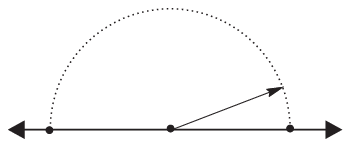 scaled 1500}}
\rput(3.4,.5){$a$}\rput(.4,.5){${\red -a}$}
}
\rput(0.25,0.5){
\rput(7,1.5){\BoxedEPSF{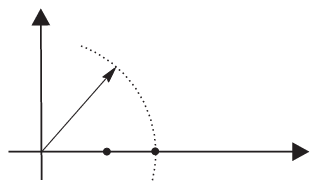 scaled 1500}}
\rput(6.3,1){$a$}\rput(7.1,1){$b$}
\rput(0,-0.5){
\pscircle[linecolor=white,fillstyle=solid,fillcolor=red](7,3){.25}
\rput(7,3){{\white{\bf 1}}}
}
\rput(0.75,0){
\rput(11,1.5){\BoxedEPSF{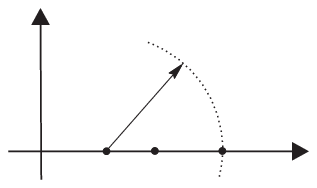 scaled 1500}}
\rput(10.3,1){$a$}\rput(11,1){$b$}\rput(12.4,1){${\red a+b}$}
\rput(0,-0.5){
\pscircle[linecolor=white,fillstyle=solid,fillcolor=red](10.3,3){.25}
\rput(10.3,3){{\white{\bf 2}}}
}
}
}
\end{pspicture}  
  \caption{$\CC\cap \R$ is closed under $+$ \emph{(right)} and $-$ \emph{(left)}.}
  \label{fig:figure20}
\end{figure}

\item {\em $\CC\cap \R$ is closed under $\times$}: as can be seen by following
through the steps in Figure \ref{fig:figure6}. Seeing that the
construction works involves studying the pair of similar  
triangles shown in red.

\begin{figure}
  \centering
\begin{pspicture}(0,0)(14,9.5)
\rput(-1.5,0){
\rput(3,4.5){
\rput(1.5,2.5){\BoxedEPSF{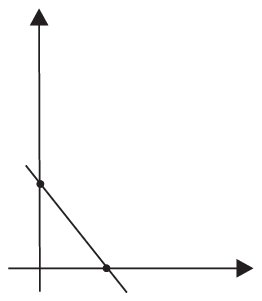 scaled 1500}}
\pscircle[linecolor=white,fillstyle=solid,fillcolor=red](2,4){.25}
\rput(2,4){{\white{\bf 1}}}
\rput(1.1,.5){$1$}\rput(-.1,2){$a\imag$}
}
\rput(5.25,5){
\rput(5,1.975){\BoxedEPSF{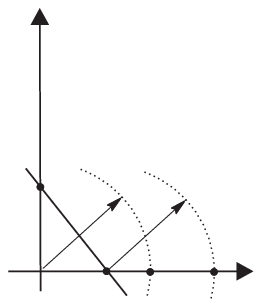 scaled 1500}}
\rput(0.9,-0.5){
\pscircle[linecolor=white,fillstyle=solid,fillcolor=red](5,4){.25}
\rput(5,4){{\white{\bf 2}}}
}
\rput(4.6,.5){$1$}\rput(3.4,2){$a\imag$}\rput(5.4,.5){$b$}
\rput(4.35,1.5){$r$}\rput(5.4,1.5){$r$}	
}
\rput(-4,0){
\rput(8,2.445){\BoxedEPSF{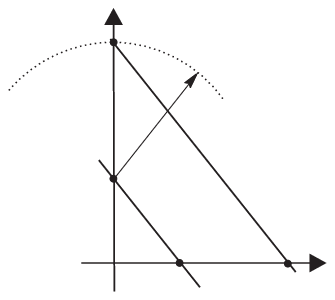 scaled 1500}}
\pscircle[linecolor=white,fillstyle=solid,fillcolor=red](9,4){.25}
\rput(9,4){{\white{\bf 3}}}
\rput(8.1,.5){$1$}\rput(6.9,2){$a\imag$}
\rput(7.65,2.75){$s$}
\rput(8.4,1.5){{\red parallel}}
\psline[linecolor=red,linewidth=.1mm]{->}(8.2,1.3)(7.95,1.05)
\psline[linecolor=red,linewidth=.1mm]{->}(8.4,1.65)(8.8,2.05)
}
\rput(-1.9,0){
\rput(12.1,2.15){\BoxedEPSF{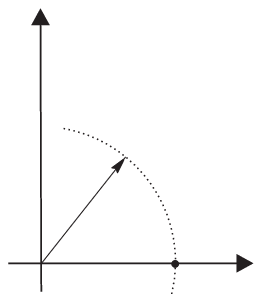 scaled 1500}}
\pscircle[linecolor=white,fillstyle=solid,fillcolor=red](13,4){.25}
\rput(13,4){{\white{\bf 4}}}
\rput(11.3,1.6){$s$}
\rput(13,.65){${\red ab}$}
}
}
\rput(0.5,0){
\rput(12.5,5){\BoxedEPSF{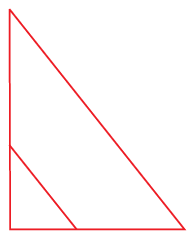 scaled 1500}}
\rput(11.7,3){${\red 1}$}\rput(10.9,4){${\red a}$}
\rput(10.9,5.6){${\red x}$}\rput(13,3){${\red b}$}
\rput(12.5,2){${\red {\displaystyle\frac{x+a}{a}=b+1}}$}
\rput(12.5,1.4){${\red \Rightarrow x=ab}$}
}
\end{pspicture}
\caption{$\CC\cap \R$ is closed under $\times$.}
  \label{fig:figure6}
\end{figure}

\item {\em $\CC\cap \R$ is closed under $\div$}: is just the
previous construction backwards -- see Figure \ref{fig:figure30}.

\begin{figure}
  \centering
\begin{pspicture}(0,0)(14,5)
\rput(0,0){
\rput(2,2.4){\BoxedEPSF{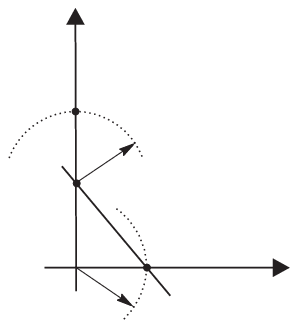 scaled 1500}}
\pscircle[linecolor=white,fillstyle=solid,fillcolor=red](2.5,4){.25}
\rput(2.5,4){{\white{\bf 1}}}
\rput(.6,2.1){$a\imag$}\rput(2.1,1){$1$}\rput(1.3,.4){$r$}\rput(1.3,2.5){$r$}
}
\rput(1,0){
\rput(6,2.4){\BoxedEPSF{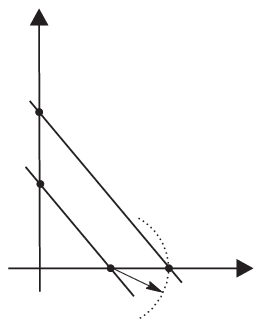 scaled 1500}}
\pscircle[linecolor=white,fillstyle=solid,fillcolor=red](6.5,4){.25}
\rput(6.5,4){{\white{\bf 2}}}
\rput(6.2,.725){$s$}
}
\rput(2,0){
\rput(9.6,1.525){\BoxedEPSF{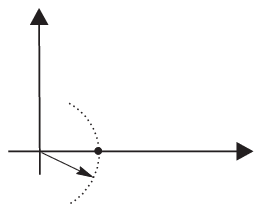 scaled 1500}}
\pscircle[linecolor=white,fillstyle=solid,fillcolor=red](10,4){.25}
\rput(10,4){{\white{\bf 3}}}
\rput(8.6,.5){$s$}\rput(9.25,1.3){${\displaystyle{\red \frac{1}{a}}}$}
}
\end{pspicture}  
  \caption{$\CC\cap \R$ is closed under $\div$}
  \label{fig:figure30}
\end{figure}

\end{enumerate}

Now to the complex constructible numbers. Observe that $z\in
\CC$ precisely when $\re z$ and $\im z$ are in $\CC\cap\R$.  
For, if $z\in \CC$ then dropping perpendiculars to the real and imaginary axes
give the numbers $\re z$ and $\im z\cdot\imag$, the second of which can be transferred to
the real axis by drawing the circle centered at $0$ passing through $\im z \cdot\imag$.
On the other hand, if we have $\re z$ and $\im z$ on the real axis,
then we have $\im z \cdot\imag$ too,
and constructing a line through $\re z$ parallel to the imaginary axis and a line
through $\im z \cdot\imag$ parallel to the real axis gives $z$. 

Suppose then that $z,w\in \CC$ are constructible complex numbers: we show
that $z+w,-z,zw$ and $1/z$ are also constructible. 
We have:
\begin{equation*}
\begin{split}
z+w&=(\re z+\re w)+(\im z+ \im w)\imag\\
-z&=-\re z-\im z\cdot\imag\\
zw&=(\re z\,\re w-\im z\, \im w)+(\re z\,\im w+\im w\,\re z)\imag\\
\frac{1}{z}&=\frac{\re z}{\re z^2+\im z^2}-\frac{\im z}{\re z^2+\im z^2}\imag,
\end{split}
\end{equation*}
so that for example, $z,w\in \CC\Rightarrow \re z, \im z, \re w, \im
w\in \CC\cap\R\Rightarrow \re z+\re w, \im z+\im w\in \CC\cap\R\Rightarrow
\re (z+w), \im (z+w)\in \CC\cap\R\Rightarrow
z+w\in \CC$, and the others are similar. 
\qed
\end{proof}


\begin{corollary}
Any rational number is constructible.
\end{corollary}

\begin{proofs}
  \begin{description}
  \item[\emph{Brute force}:] use the example of the construction of
    $3$ to show that $\Z\subset\CC$; that $\CC\cap\R$ is closed under
    $\times$ and $\div$ then gives $\Q\subset\CC$. 
  \item[\emph{Slightly slicker}:] by Exercise \ref{ex_lect1.0}, any
    subfield of $\C$ contains $\Q$.
  \end{description}
\qed
\end{proofs}

\paragraph{\hspace*{-0.3cm}}
Not only can we perform the four basic arithmetic operations with constructible
numbers, but we can construct square roots too:

\begin{figure}
  \centering
\begin{pspicture}(0,0)(14,8)
\rput(-2,0){
\rput(2.5,4.25){
\rput(2,1.5){\BoxedEPSF{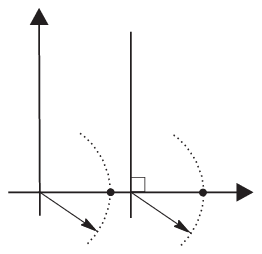 scaled 1500}}
\pscircle[linecolor=white,fillstyle=solid,fillcolor=red](3,2.5){.25}
\rput(3,2.5){{\white{\bf 1}}}
\rput(3.25,.7){$P$}
\rput(1.55,.65){$a$}\rput(1.9,.35){$1$}\rput(1,.1){$r$}\rput(2.3,.1){$r$}
}
\rput(4,4.25){
\rput(5.5,1.625){\BoxedEPSF{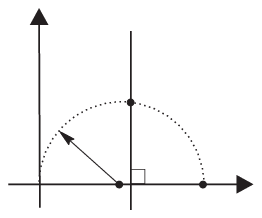 scaled 1500}}
\pscircle[linecolor=white,fillstyle=solid,fillcolor=red](6.5,2.5){.25}
\rput(6.5,2.5){{\white{\bf 2}}}
\psline[linecolor=red,linewidth=.1mm]{->}(5.325,0)(5.325,.45)
\rput(5.4,-.2){{\red midpoint of $0P$}}
}
\rput(-4.4,0.75){
\rput(8.9,1.615){\BoxedEPSF{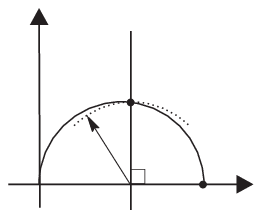 scaled 1500}}
\rput(-0.2,0){
\pscircle[linecolor=white,fillstyle=solid,fillcolor=red](10,2.5){.25}
\rput(10,2.5){{\white{\bf 3}}}
}
\rput(8.5,.8){$s$}
}
\rput(-2.7,0.75){
\rput(12.2,1.15){\BoxedEPSF{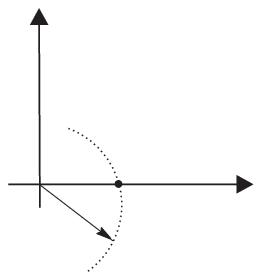 scaled 1500}}
\rput(0.2,0){
\pscircle[linecolor=white,fillstyle=solid,fillcolor=red](13,2.5){.25}
\rput(13,2.5){{\white{\bf 4}}}
}
\rput(11.4,.2){$s$}\rput(12.2,.7){{\red $\kern-2pt\sqrt{a}$}}
}
}
\rput(10.5,2.75){
\rput(1.5,1.2){\BoxedEPSF{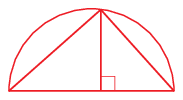 scaled 2000}}
\rput(.8,.1){${\red 1}$}\rput(2.4,.1){${\red a}$}\rput(1.5,1){${\red
    x}$}
\rput(1.5,-0.5){${\red x=\kern-2pt\sqrt{a}}$}
}
\end{pspicture}
\caption{Constructing $\kern-2pt\sqrt{a}$ for $a\in\R$.}
  \label{fig:figure7}
\end{figure}

\begin{theorem}
If $z\in \CC$ then $\kern-2pt\sqrt{z}\in \CC$.
\end{theorem}

\begin{proof}
We can construct the square root of any positive real
number $a\in\R$ as in Figure \ref{fig:figure7}. As an Exercise, show
that in the red picture in Figure \ref{fig:figure7}, 
the length $x=\kern-2pt\sqrt{a}$. Next, the square root of any complex number can be
constructed as in Figure \ref{fig:figure40},
where we have used the construction of real square roots in the second
step. 
\qed
\end{proof}

\begin{figure}
  \centering
\begin{pspicture}(0,0)(14,3)
\rput(0,0){
\rput(2,1.5){\BoxedEPSF{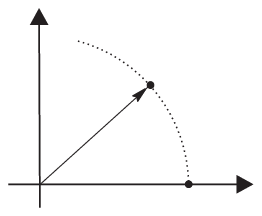 scaled 1500}}
\rput(2.5,2){$z$}\rput(2.9,.1){$a$}
\pscircle[linecolor=white,fillstyle=solid,fillcolor=red](3,2.5){.25}
\rput(3,2.5){{\white{\bf 1}}}
}
\rput(1.5,0){
\rput(5.5,1.5){\BoxedEPSF{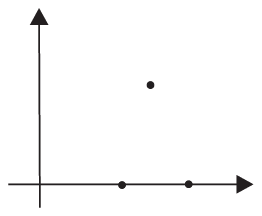 scaled 1500}}
\rput(6,2){$z$}\rput(6.4,.1){$a$}
\rput(5.4,.1){$\kern-2pt\sqrt{a}$}
\pscircle[linecolor=white,fillstyle=solid,fillcolor=red](6.5,2.5){.25}
\rput(6.5,2.5){{\white{\bf 2}}}
}
\rput(3,0){
\rput(9,1.5){\BoxedEPSF{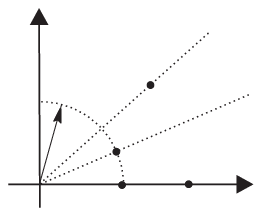 scaled 1500}}
\rput(9.5,2){$z$}\rput(9.9,.1){$a$}\rput(8.9,.1){$\kern-2pt\sqrt{a}$}
\rput(10,1){{\red $\kern-2pt\sqrt{z}$}}
\psline[linewidth=.1mm,linecolor=red]{->}(9.7,1)(8.9,.85)
\rput(12,1.2){{\red bisector}}
\psline[linecolor=red,linewidth=.1mm]{->}(11.4,1.2)(10.5,1.5)
\pscircle[linecolor=white,fillstyle=solid,fillcolor=red](11,2.5){.25}
\rput(11,2.5){{\white{\bf 3}}}
}
\end{pspicture}  
  \caption{Constructing $\kern-2pt\sqrt{z}$ for $z\in\C$.}
  \label{fig:figure40}
\end{figure}


\subsection{Constructing angles and polygons}

\paragraph{\hspace*{-0.3cm}}
We say that an angle can be constructed when we can construct two lines intersecting in 
that angle.

\begin{vexercise}\label{ex7.20}
\hspace{1em}
\begin{enumerate}
\item Show that we can always assume that one of the lines giving an angle is the positive
real axis.
\item Show that an {\em angle\/} $\theta$ can be constructed if and only if the {\em number\/}
$\cos\theta$ can be constructed. Do the same for $\sin\theta$ and $\tan\theta$.
\end{enumerate}
\end{vexercise}

\begin{vexercise}\label{ex7.30}
Show that if $\varphi,\theta$ are constructible angles then so are
$\varphi+\theta$ and $\varphi-\theta$.
\end{vexercise}

\paragraph{\hspace*{-0.3cm}}
A {\em regular $n$-sided polygon\/} or {\em regular $n$-gon\/} is a polygon in $\C$ with 
$n$ sides of equal length and $n$ interior angles of equal size. 

\begin{vexercise}\label{ex7.40}
Show that a regular $n$-gon can be constructed centered at $0\in\C$ if and only if
the angle $\frac{2\pi}{n}$ 
can be constructed. 
Show that a regular $n$-gon can be constructed centered at $0\in\C$ if and only if
the complex number 
$$
z=\cos\frac{2\pi}{n}+\imag\sin\frac{2\pi}{n},
$$
can be constructed.
\end{vexercise}

\begin{vexercise}\label{ex7.50}
Show that if an $n$-gon and an $m$-gon can be constructed for $n$ and $m$
relatively prime, then so can a $mn$-gon (hint: use the $\Z$-version of Theorem 
\ref{gcd}).
\end{vexercise}

\paragraph{\hspace*{-0.3cm}}
For what $n$ can you construct a regular $n$-gon? It makes sense to consider first 
the $p$-gons for $p$ a prime. The complete answer even to this question will
not be revealed until Section \ref{galois.corresapps}. It turns out that the $p$-gons that
can be constructed
are {\em extremely rare\/}. Nevertheless, the first two (odd) primes do work:

\begin{vexercise}\label{ex7.60}
Show that a regular $3$-gon, ie: an equilateral triangle, can be constructed with 
any side length. Using Exercises \ref{ex_lect1.1a} and \ref{ex7.40}, 
show that a regular $5$-gon can also be
constructed.
\end{vexercise}

\paragraph{\hspace*{-0.3cm}}
Here is a proof that a regular $17$-gon is constructible. Gauss proved the
remarkable identity of Figure \ref{fig:gauss}, which is still found in trigonometric tables.
Thus the number $\cos\pi/17$ can be constructed as this expression involves only integers,
the four field operations and square roots, all of which are operations we can perform with
a ruler and compass. Hence, by Exercise \ref{ex7.20}(2) 
the angle $\pi/17$ can be constructed and so 
adding it to itself (Exercise \ref{ex7.30}) gives the angle
$2\pi/17$. Now apply Exercise 
\ref{ex7.40} to get the 
$17$-gon.

\begin{figure}
  \centering
\begin{pspicture}(0,0)(15,2)
\rput(7.25,1){
${\displaystyle
\cos\frac{\pi}{17}=
\frac{1}{8}
\sqrt{
2\biggl(
2\sqrt{
\sqrt{\frac{17(17-\sqrt{17})}{2}}
-
\sqrt{\frac{17-\sqrt{17}}{2}}
-4\sqrt{34+2\sqrt{17}}
+3\sqrt{17}
+17
}
+\sqrt{34+2\sqrt{17}}
+\sqrt{17}
+15
\biggr)
}
}
$
}
\end{pspicture}  
  \caption{A proof that the $17$-gon is constructible.}
  \label{fig:gauss}
\end{figure}

\subsection*{Further Exercises for Section \thesection}

\begin{vexercise}\label{ex7.70}
Using the fact that the constructible numbers include $\Q$, show that
any given line segment can be trisected in length.
\end{vexercise}

\begin{vexercise}\label{ex7.80}
Show that if you can construct a regular $n$-sided polygon, then you can
also construct a regular $2^kn$-sided polygon for any $k\geq1$.
\end{vexercise}

\begin{vexercise}\label{ex7.90}
Show that $\cos\theta$ is constructible if and only if $\sin\theta$ is.
\end{vexercise}

\begin{vexercise}\label{ex7.100}
If $a,b$ and $c$ are constructible numbers (ie: in $\CC$), show that the
roots of the quadratic equation $ax^2+bx+c$ are also constructible.
\end{vexercise}


\section{Vector Spaces I: Dimensions}
\label{lect7}

Having met rings and fields we introduce
our third algebraic object: vector spaces.

\begin{definition}[vector space]
  A vector space over a field $F$ is a set $V$, whose elements
  are called vectors,
together with two operations: addition $u,v\mapsto u+v$ of vectors  and 
scalar multiplication $\ll,v\mapsto \ll v$ of a vector by an element
(or scalar) $\ll$ of the
field $F$, such that:
\begin{enumerate}
\item $(u+v)+w = u+(v+w)$, for all $u,v,w\in V$.
\item There exists a zero vector $0\in V$ such that $v+0=v=0+v$ for all $v\in V$,
\item Every $v\in V$ has a negative $-v$ such that $v+(-v)=0=-v+v$, for all $v\in V$.
\item $u+v=v+u$, for all $u,v\in V$.
\item $\ll(u+v) = \ll u + \ll v$, for all $u,v$ and $\ll\in F$.
\item $(\ll+\mu)v = \ll v+\mu v$, for all $\ll\mu\in F$ and $v\in V$.
\item $\ll(\mu v) = (\ll\mu)v$, for all $\ll\mu\in F$ and $v\in V$.
\item $1 v = v$ for all $v\in V$.
\end{enumerate}
\end{definition}

\begin{aside}
Alternatively, $V$ forms an Abelian group
under $+$ (these are the first four axioms) together with a scalar multiplication
that satisfies the last four axioms.
\end{aside}

\paragraph{\hspace*{-0.3cm}}
A {\em homomorphism\/} of vector spaces is a map 
$\varphi: V_1\rightarrow V_2$ such that 
$\varphi(u+v)=\varphi(u)+\varphi(v)$ and $\varphi(\ll v)=\ll\varphi(v)$
for all $u,v\in V$ and $\ll\in F$. 
(Homomorphisms of vector 
are more commonly called linear maps.) A bijective homomorphism is an
\emph{isomorphism\/}. 

\paragraph{\hspace*{-0.3cm}}
The set $\R^2$ of $2\times 1$ column vectors is the motivating example
of a vector space over $\R$
under the normal addition and scalar multiplication of vectors. Alternatively, 
the complex numbers $\C$ form a vector space over $\R$, and 
these two spaces are isomorphic via the map:
$$
\varphi:\left[\begin{array}{c}
a\\
b\\
\end{array}\right]
\mapsto
a+bi.
$$

\paragraph{\hspace*{-0.3cm}}
\parshape=3 0pt\hsize 0pt.7\hsize 0pt.7\hsize 
The complex numbers are a vector space {\em over themselves\/}: addition of 
complex numbers gives an Abelian group and now we can scalar multiply a complex  number
by another one, using the usual multiplication of complex numbers. 
\vadjust{\hfill\smash{\lower 65pt
\llap{
\begin{pspicture}(0,0)(4,2.5)
\rput(0,.1){
\rput(2,1.25){\BoxedEPSF{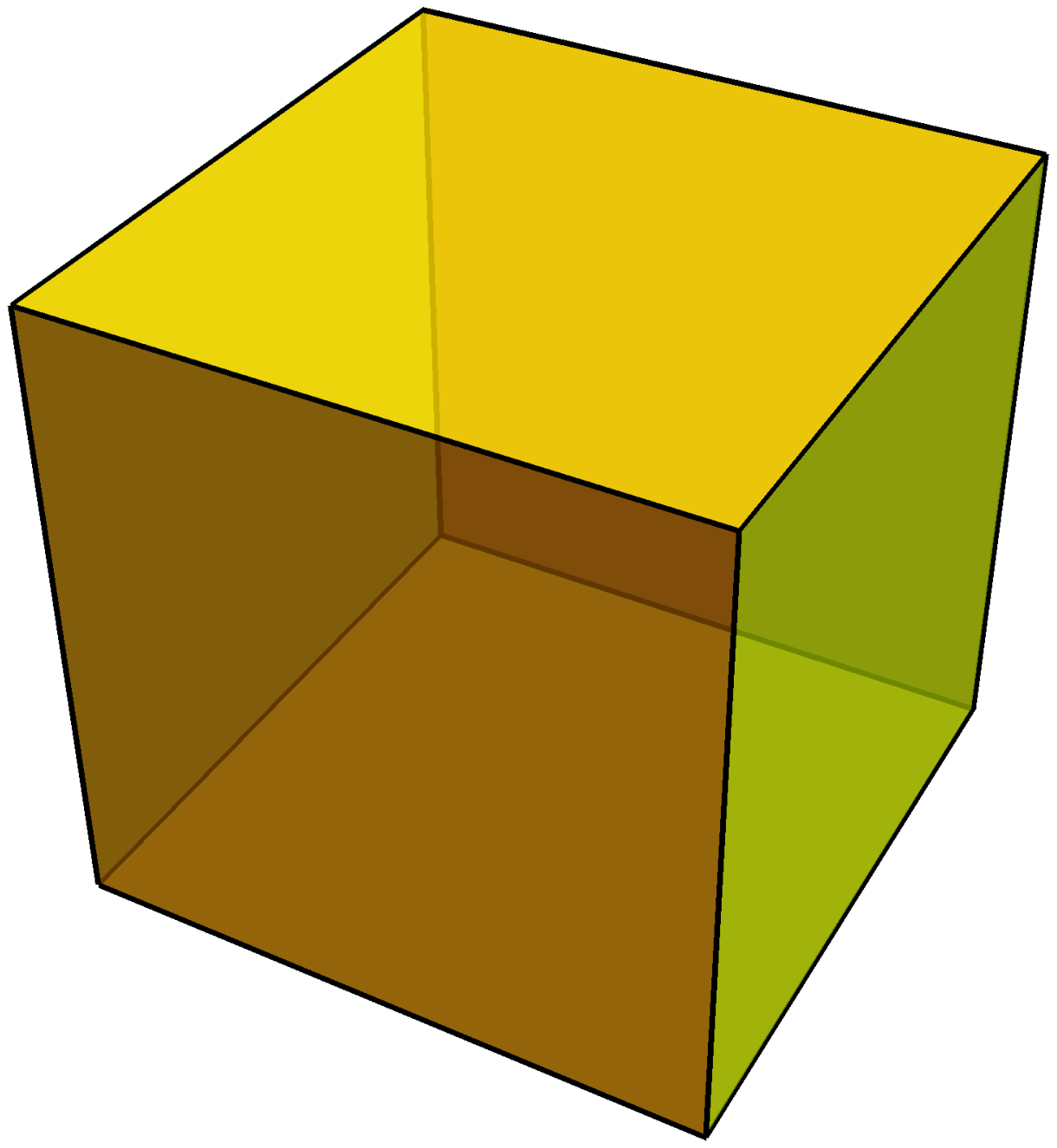 scaled 250}}
\rput(2.5,-0.5){${\sss 000}$}
\rput(0.5,0.4){${\sss 100}$}
\rput(3.65,1){${\sss 001}$}
\rput(2.5,1.7){${\sss 010}$}
\rput(0.25,2.1){${\sss 110}$}
\rput(3.85,2.5){${\sss 011}$}
\rput(1.75,3.1){${\sss 111}$}
}
\end{pspicture}
}}}\ignorespaces

\paragraph{\hspace*{-0.3cm}}
\parshape=6 0pt.7\hsize 0pt.7\hsize 0pt.7\hsize 0pt.7\hsize 0pt.7\hsize
0pt\hsize  
A vector spaces over a finite field:
consider the set of $3$-tuples with coordinates 
from the field $\F_2$ (so are either $0$ or $1$) and add two such
coordinate-wise, using the addition from $\F_2$. Scalar multiply 
a tuple coordinate-wise using the multiplication from $\F_2$. 
As there are only two possibilities for each coordinate and three 
coordinates in total, we get a total of $2^3=8$ vectors in this space.
They can be arranged around the vertices of a cube as shown,
where $abc$ is the vector 
with the three coordinates $a,b,c\in\F_2$. 

\paragraph{\hspace*{-0.3cm}}
We saw in Section \ref{lect4} that the field $\Q(\kern-2pt\sqrt{2})$ has
elements the $a+b\kern-2pt\sqrt{2}$ with $a,b\in\Q$. The identification,
$$
\begin{pspicture}(12,1)
\rput(-1.5,-.5){
\rput(-1,0){
\rput(6.1,1){$\leftrightarrow$}
\rput(5,1){$a+b\kern-2pt\sqrt{2}$}
\rput(7,1){$\left[\begin{array}{c}
a\\
b\\
\end{array}\right]$}
\psline[linewidth=.1mm]{->}(8.5,1.2)(7.5,1.2)
\psline[linewidth=.1mm]{->}(8.5,.8)(7.5,.8)
}
\rput(0,.05){
\rput(9.65,1.2){coordinate in ``$1$ direction''}
\rput(9.8,.8){coordinate in ``$\kern-2pt\sqrt{2}$ direction''}
}
}
\end{pspicture}
$$
is an isomorphism with the vector space $\Q^2$ of $2\times 1$
$\Q$-column vectors with the addition
$(a+b\kern-2pt\sqrt{2})+(c+d\kern-2pt\sqrt{2})=(a+c)+(b+d)\kern-2pt\sqrt{2}$
corresponding to,
$$
\left[\begin{array}{c}
a\\
b\\
\end{array}\right]
+
\left[\begin{array}{c}
c\\
d\\
\end{array}\right]
=
\left[\begin{array}{c}
a+c\\
b+d\\
\end{array}\right],
$$
and scalar multiplication $c(a+b\kern-2pt\sqrt{2})=ac+bc\kern-2pt\sqrt{2}$ corresponding to:
$$
c\left[\begin{array}{c}
a\\
b\\
\end{array}\right]
=
\left[\begin{array}{c}
ac\\
bc\\
\end{array}\right].
$$

\paragraph{\hspace*{-0.3cm}}
The polynomial $x^3-2$ is irreducible over $\Q$ so the quotient ring 
$\Q[x]/\lg x^3-2\rg$ is a field with elements the 
$(a+bx+cx^2)+\lg x^3-2\rg$ for $a,b\in\Q$. It is a $\Q$-vector space,
isomorphic to $\Q^3$ via
$$
\begin{pspicture}(12,1.5)
\rput(2.75,.75){$(a+bx+cx^2)+\lg x^3-2\rg \leftrightarrow
\left[\begin{array}{c}
a\\
b\\
c\\
\end{array}\right]$
}
\rput(-1.3,0.45){
\rput(-.8,-.5){
\psline[linewidth=.1mm]{->}(8.5,1.2)(7.5,1.2)
\psline[linewidth=.1mm]{->}(8.5,.8)(7.5,.8)
\psline[linewidth=.1mm]{->}(8.5,.4)(7.5,.4)
}
\rput(1,-.475){
\rput(9.7,1.2){coordinate in ``$1+\lg x^3-2\rg$ direction''}
\rput(9.8,.8){coordinate in ``$x+\lg x^3-2\rg$ direction''}
\rput(9.8,.4){coordinate in ``$x^2+\lg x^3-2\rg$ direction''}
}}
\end{pspicture}
$$
(Check for yourself that the addition and scalar multiplications match up).

\paragraph{\hspace*{-0.3cm}}
The previous two examples are special cases of the following:
if $F\subseteq E$ is an extension of fields then $E$ 
is a vector space over $F$. The ``vectors'' are the
elements of $E$ and the ``scalars'' are the elements of $F$. Addition of vectors
is just the addition of elements in $E$, and to scalar multiply a $v\in E$ by a $\ll\in F$,
multiply $\ll v$ using the multiplication of the field $E$. The first four
axioms for a vector space hold because of the addition of the field $E$, and the second
four from the multiplication. 

\begin{definition}[span and independence]
If $v_1,\ldots,v_n\in V$ are vectors in a vector space $V$, then a vector of the form 
$$
\alpha_1v_1+\ldots +\alpha_nv_n,
$$ 
for $\alpha_1,\ldots,\alpha_n\in F$,
is called a linear combination of the $v_1,\ldots,v_n$.
The linear span of $\{v_j:j\in J\}$, where $J$ is not
necessarily finite, is the set of all
linear combinations of vectors from the set:
$$
\text{span}
\{v_j:j\in J\} 
= 
\{\aa_1v_{j_1}+\cdots+\aa_kv_{j_k}: \alpha_j\in F\}.
$$
Say $\{v_j:j\in J\}$ {\em span\/} $V$ when $V=\text{span}\{v_j:j\in J\}$.

A set of vectors $v_1,\ldots,v_n\in V$ is linearly dependent
if and only if there exist scalars $\alpha_1,\ldots,\alpha_n$, 
not all zero, such that
$$
\alpha_1v_1+\ldots +\alpha_nv_n = 0,
$$
and linearly independent otherwise, ie:
$\alpha_1v_1+\ldots +\alpha_nv_n = 0$ implies that the $\aa_i$ are all $0$.  
\end{definition}

\paragraph{\hspace*{-0.3cm}}
In the examples above, the complex numbers $\C$ are spanned, as a vector space
over $\R$, by $\{1,\imag\}$, and indeed by any two non-zero complex numbers 
that are not scalar multiples of each other. As a vector space over $\C$, the 
complex numbers are spanned by {\em one\/} element: any
$\zeta\in\C$ can be written as $\zeta\times 1$ for example, so every element
is a \emph{complex\/} scalar multiple of $1$. Indeed, $\C$ is spanned as a complex vector
space by any single one of its non-zero elements. 

\begin{definition}[basis]
A basis for $V$ is a set of vectors $\{v_j:j\in J\}$, with $J$ a not necessarily finite index set,
that span $V$, and such that every finite set of $v_j$'s are linearly
independent. 
\end{definition}
 
It can be proved that there is a 1-1 correspondence between the elements of
any two bases for a vector space $V$. 
When $V$ has a finite basis the \emph{dimension\/} of $V$ is defined to be 
the number of elements in a basis; otherwise $V$ is \emph{infinite dimensional}.

\paragraph{\hspace*{-0.3cm}}
Thus $\C$ is $2$-dimensional as a vector space over $\R$ but 
$1$-dimensional as a vector space over $\C$. We will see later in
this section that $\C$ is {\em infinite\/} dimensional as a vector space over $\Q$.

With the other examples above, $\Q(\kern-2pt\sqrt{2})$ is $2$-dimensional over $\Q$ with basis
$\{1,\kern-2pt\sqrt{2}\}$ and $\Q[x]/\lg x^3-2\rg$ is $3$-dimensional over $\Q$ with basis the
cosets
$$
1+\lg x^3-2\rg,x+\lg x^3-2\rg\text{ and }x^2+\lg x^3-2\rg.
$$
In Exercise \ref{exam00_4} in Section \ref{galois.correspondence}, we will see that if
$\aa=\sqrt[4]{2}$, then $\Q(\aa,\imag)$ is a $2$-dimensional space over $\Q(\aa)$ or $\Q(\aa\imag)$
or even $\Q((1+\imag)\aa)$; a $4$-dimensional space over $\Q(\imag)$ or $\Q(\imag\aa^2)$, and an
$8$-dimensional space over $\Q$ (and these are almost, but not quite, all the possibilities;
see the exercise for the full story).

\begin{definition}[degree of an extension]
Let $F\subseteq E$ be an extension of fields. Consider $E$ as a vector space over $F$,
and define the degree of the extension to be the dimension of this vector
space, denoted $[E:F]$. Call $F\subseteq E$ a finite extension if the
degree is finite. 
\end{definition}

\paragraph{\hspace*{-0.3cm}}
The extensions $\Q\subset\Q(\sqrt{2})$ and $\Q\subset\Q[x]/\lg x^3-2\rg$
have degrees $2$ and $3$.

\paragraph{\hspace*{-0.3cm}}
It is no coincidence that the degree of extensions of the form $F\subseteq F[x]/\lg f\rg$
turn out to be the same as the degree of the polynomial $f$:

\begin{theorem}\label{degree_extension}
Let $f$ be an irreducible polynomial in $F[x]$ of degree $d$. Then the extension,
$$
F\subseteq F[x]/\lg f\rg,
$$
has degree $d$.
\end{theorem}

Hence the name degree!

\begin{proof}
Replace, as usual, the field $F$ by its copy in $F[x]/\lg f\rg$, so that $\ll\in F$ 
becomes $\ll+\lg f\rg\in F[x]/\lg f\rg$. Consider the set of cosets,
$$
B=\{1+\lg f\rg,x+\lg f\rg,x^2+\lg f\rg,\ldots, x^{d-1}+\lg f\rg\}.
$$
Then we claim that $B$ is a basis for $F[x]/\lg f\rg$ over $F$, for which we have to show that it
spans the vector space and is linearly independent. To see that it spans, consider a typical
element, which has the form,
$$
g+\lg f\rg=(qf+r)+\lg f\rg=r+\lg f\rg=(a_0+a_1x+\cdots +a_{d-1}x^{d-1})+\lg f\rg.
$$
using the division algorithm and basic properties of cosets. This is turn gives,
$$
(a_0+a_1x+\cdots +a_{d-1}x^{d-1})+\lg f\rg=
(a_0+\lg f\rg)(1+\lg f\rg)+(a_1+\lg f\rg)(x+\lg f\rg)+\cdots +
(a_{d-1}+\lg f\rg)(x^{d-1}+\lg f\rg),
$$
where the last is an $F$-linear combination of the elements of $B$. Thus this sets spans
the space.

For linear independence, suppose we have an $F$-linear combination of the elements of $B$
giving zero, ie:
$$
(b_0+\lg f\rg)(1+\lg f\rg)+(b_1+\lg f\rg)(x+\lg f\rg)+\cdots +
(b_{d-1}+\lg f\rg)(x^{d-1}+\lg f\rg)=\lg f\rg,
$$
remembering that the zero of the field $F[x]/\lg f\rg$ is the coset $0+\lg f\rg=\lg f\rg$.
Multiplying and adding all the cosets on the left hand side gives,
$$
(b_0+b_1x+\cdots+b_{d-1}x^{d-1})+\lg f\rg=\lg f\rg,
$$
so that $b_0+b_1x+\cdots+b_{d-1}x^{d-1}\in\lg f\rg$ (using another basic property of
cosets). The elements of $\lg f\rg$, being multiples of $f$, must have degree at least
$d$, except for the zero polynomial. On the other hand $b_0+b_1x+\cdots+b_{d-1}x^{d-1}$
has degree $\leq d-1$. Thus it must be the zero polynomial, giving that all the $b_i$
are zero, hence all the $b_i+\lg f\rg$ are $0$, and that the set $B$
is linearly independent over $F$ as claimed. 
\qed
\end{proof}

\paragraph{\hspace*{-0.3cm}}
What is the degree of the extension $\Q\subset\Q(\pi)$? If it was finite,
say $[\Q(\pi):\Q]=d$, then any collection of more than $d$ elements would be linearly
dependent. In particular, the $d+1$ elements,
$$
1,\pi,\pi^2,\ldots,\pi^d,
$$
would be dependent, so that $a_0+a_1\pi+a_2\pi^2+\ldots+a_d\pi^d=0$
for some $a_0,a_1,\ldots,a_d\in\Q$, not all zero,
hence $\pi$ would be a root of the polynomial $a_0+a_1x+a_2x^2+\ldots+a_dx^d$.
But this
contradicts the fact that $\pi$ is transcendental over $\Q$. Thus, the degree of the
extension is infinite.

\paragraph{\hspace*{-0.3cm}}
In fact this is always true:

\begin{proposition}\label{finite.givesalgebraic}
Let $F\subseteq E$ and $\aa\in E$. If the degree of the extension $F\subseteq F(\aa)$ is finite,
then $\aa$ is algebraic over $F$.
\end{proposition}

\begin{proof}
The proof is very similar to the example above. Suppose
that the extension $F\subseteq F(\aa)$ has degree $n$, so that any collection of $n+1$ elements
of $F(\aa)$ must be linearly dependent. In particular the $n+1$ elements
$$
1,\aa,\aa^2,\ldots,\aa^n
$$
are dependent over $F$, so that there are $a_0,a_1,\ldots, a_n$ in $F$  with
$$
a_0+a_1\aa+\cdots+a_n\aa^n=0,
$$
and hence $\aa$ is algebraic over $F$ as claimed.
\qed
\end{proof}

Thus, any field $E$ that contains transcendentals over $F$ will be infinite dimensional
as vector spaces over $F$. In particular, $\R$ and $\C$ are infinite dimensional over $\Q$.

\paragraph{\hspace*{-0.3cm}}
The converse to Proposition \ref{finite.givesalgebraic} is partly true, as we 
summarise now in an important result:

\begin{theoremD}\label{thmD}
Let $F\subseteq E$ and $\aa\in E$ be algebraic over $F$. Then,
\begin{enumerate}
\item There is a unique polynomial $f\in F[x]$ that is monic, irreducible over $F$, and has
$\aa$ as a root.
\item The field $F(\aa)$ is isomorphic to the quotient $F[x]/\lg f\rg$.
\item If $\deg f=d$, then the extension $F\subseteq F(\aa)$ has degree $d$ 
with basis $\{1,\aa,\aa^2,\ldots,\aa^{d-1}\}$,
and so,
$$
F(\aa)=\{a_0+a_1\aa+a_2\aa^2+\cdots+a_{d-1}\aa^{d-1}\,|\,a_0,\ldots,a_{d-1}\in F\}.
$$
\end{enumerate}
\end{theoremD}

\begin{proof}
Hopefully most of the proof will be recognisable from the specific examples we have discussed
already.
As $\aa$ is algebraic over $F$ there is at least one $F$-polynomial having $\aa$
as a root. Choose $f'$ to be a non-zero one having smallest degree. 
This polynomial must then be
irreducible over $F$, for if not, we have $f'=gh$ with $\deg(g),\deg(h)<\deg(f')$, and
$\aa$ must be a root of one of $g$ or $h$, contradicting the original choice of $f'$. 
Divide through by the leading coefficient of $f'$, to get $f$, a monic, 
irreducible (by Exercise \ref{ex3.1})
$F$-polynomial,
having $\aa$ as a root. If $f_1,f_2$ are polynomials with these properties then
$f_1-f_2$ has degree strictly less than either $f_1$ or $f_2$ and still has $\aa$ as
a root, so the only possibility is that $f_1-f_2$ is zero, hence $f$ is unique.

Consider the evaluation homomorphism $\ve_\aa:F[x]\rightarrow E$ defined as
usual by $\ve_\aa(g)=g(\aa)$. To show that the kernel of this homomorphism is the
ideal $\lg f\rg$ is completely analogous to the example at the beginning
of Section \ref{lect5}: clearly $\lg f\rg$ is contained in the kernel, as any multiple of
$f$ must evaluate to zero when $\aa$ is substituted into it. On the
other hand, if
$h$ is in the kernel of $\ve_\aa$, then by division algorithm,
$$
h=qf+r,
$$
with $\deg(r)<\deg(f)$. Taking the $\ve_\aa$ image of both sides gives
$0=\ve_\aa(h)=\ve_\aa(qf)+\ve_\aa(r)=\ve_\aa(r)$, so that $r$ has $\aa$ as a root.
As $f$ is minimal with this property, we must have that $r=0$, so that $h=qf$, ie:
$h$ is in the ideal $\lg f\rg$, and so the kernel is contained in this ideal.
Thus, $\ker\ve_\aa=\lg f\rg$.

In particular we have an isomorphism $\widehat{\ve_\aa}:F[x]/\lg f\rg\rightarrow
\im\ve_\aa\subset E$,  given by,
$$
\widehat{\ve_\aa}(g+\lg f \rg)=\varepsilon_{\aa}(g)=g(\aa),
$$
with $F[x]/\lg f\rg$ a field as $f$ is irreducible over $F$. Thus,
$\im\ve_\aa$ is a subfield of $E$. Clearly, both the element $\aa$ 
($\ve_\aa(x)=\aa$) and the field
$F$ ($\ve_\aa(c)=c$) are contained in $\im\ve_\aa$, hence $F(\aa)$ is too
as $\im\ve_\aa$ is subfield of $E$, and $F(\aa)$ is the smallest one enjoying these
two properties. Conversely, if $g=\sum a_i x^i\in F[x]$ then $\ve_\aa(g)=\sum a_i\aa^i$,
which is an element of $F(\aa)$ as fields are closed under sums and products. Hence
$\im\ve_\aa\subseteq F(\aa)$ and so these two are the same. 
Thus $\widehat{\ve_\aa}$ is an isomorphism between $F[x]/\lg f\rg$ and $F(\aa)$.

The final part follows immediately from Theorem \ref{degree_extension}, where
we showed that the set of cosets 
$$
\{1+\lg f\rg,x+\lg f\rg,x^2+\lg f\rg,\ldots, x^{d-1}+\lg f\rg\},
$$
formed a basis for $F[x]/\lg f\rg$ over $F$. Their images under $\widehat{\ve_\aa}$,
namely $\{1,\aa,\aa^2,\ldots,\aa^{d-1}\}$, must then form a basis for $F(\aa)$ over 
$F$.\qed
\end{proof}

The proof of Theorem D shows that the polynomial $f$ has
the smallest degree of any polynomial having $\aa$ as a root. 

\begin{definition}[minimum polynomial]
 The polynomial $f$ of Theorem D is called the minimum
   polynomial of $\aa$ over $F$. 
\end{definition}

\paragraph{\hspace*{-0.3cm}}
An important property of the minimum polynomial is that it divides {\em any\/} other 
$F$-polynomial that has $\aa$ as a root: for suppose that $g$ is such an $F$-polynomial.
By unique factorisation in $F[x]$, we can decompose $g$ as
$$
g=\ll f_1f_2\ldots f_k,
$$
where the $f_i$ are monic and irreducible over $F$. Being a root of $g$, 
the element $\aa$ must be a root of one of the $f_i$. By uniqueness, 
this $f_i$ must be the minimum polynomial of $\aa$ over $F$. 

\paragraph{\hspace*{-0.3cm}}
The last part of Theorem D tells us that to find the degree of a simple extension 
$F\subseteq F(\aa)$, you find the degree of the minimum polynomial over
$F$ of $\aa$. 

How do you find this polynomial? Its simple: guess! A sensible first
guess is a monic polynomial
with $F$-coefficients that has $\aa$ as root. If your guess is also irreducible,
then you have guessed right (uniqueness).

The only thing that can go wrong is if your guess
is not irreducible. Your next guess should then be a factor of 
your first
guess. In this way, the search for minimum polynomials is ``no harder'' than 
determining irreducibility.

\paragraph{\hspace*{-0.3cm}}
As an example consider the minimum polynomial over $\Q$
of the $p$-th root of $1$,
$$
\cos\frac{2\pi}{p}+\imag\sin\frac{2\pi}{p},
$$
for $p$ a prime. Your first guess is $x^p-1$ which satisfies all the criteria bar
irreducibility as $x-1$ is a factor. Factorising gives:
$$
x^p-1=(x-1)\Phi_p(x),
$$
for $\Phi_p$ the $p$-th cyclotomic polynomial, and this was shown to be irreducible over $\Q$ 
in Exercise \ref{ex_lect3.2}.

\paragraph{\hspace*{-0.3cm}}
How does one find the degree of extensions $F\subseteq F(\aa_1,\ldots,\aa_k)$ that are not 
necessarily simple? Such extensions are a sequence of simple 
extensions. If we can find the degrees of each of these simple 
extensions, all we need is a way to patch the answers together:

\begin{towerlaw}
Let $F\subseteq E\subseteq L$ be a sequence or ``tower'' of extensions. If both of the intermediate
extensions $F\subseteq E$ and $E\subseteq L$ are of finite degree, then $F\subseteq L$ is too, with
$$
[L:F]=[L:E][E:F].
$$
\end{towerlaw}


\paragraph{\hspace*{-0.3cm}}
Before the proof we consider the example 
$\Q\subset \Q(\sqrt[3]{2},\imag)$, which a sequence of two simple extensions:
$$
\Q\subset \Q(\sqrt[3]{2})\subset \Q(\sqrt[3]{2},\imag).
$$
We can use Theorem D to find the degrees of each individual simple
extension. Firstly, the minimum polynomial over $\Q$ of $\sqrt[3]{2}$ is $x^3-2$, for
this polynomial is monic in $\Q[x]$ with $\sqrt[3]{2}$ as a root and irreducible over
$\Q$ by Eisenstein (using $p=2$). Thus the extension $\Q\subset\Q(\sqrt[3]{2})$ has degree 
$\deg(x^3-2)=3$ and $\{1,\sqrt[3]{2},(\sqrt[3]{2})^2\}$ 
is a basis for $\Q(\sqrt[3]{2})$ over $\Q$.

Now let $\F=\Q(\sqrt[3]{2})$ so that the second extension is $\F\subset \F(\imag)$ and where
the minimum polynomial of $\imag$ over  $\F$ is $x^2+1$: it is monic in $\F[x]$ with $\imag$
as a root, and irreducible over $\F$ as its two roots $\pm\imag$ are not in $\F$ (as 
$\F\subset \R$). Thus Theorem D again gives that $\F\subset \F(\imag)$ has degree
$2$ with $\{1,\imag\}$ a basis for $\F(\imag)$ over $\F$.

Now consider the elements,
$$
\{1,\sqrt[3]{2},(\sqrt[3]{2})^2,\imag,\sqrt[3]{2}\imag,(\sqrt[3]{2})^2\imag\},
$$
obtained by multiplying the two bases together.
The claim is that they form a basis for $\Q(\sqrt[3]{2},\imag)=\F(\imag)$ over $\Q$:
we need to show that the $\Q$-span of these six gives every element of
$\Q(\sqrt[3]{2},\imag)$ and that 
they are linearly independent over $\Q$. For the first, let $x$ be an arbitrary element
of $\Q(\sqrt[3]{2},\imag)=\F(\imag)$. As $\{1,\imag\}$ is a basis for $\F(\imag)$ over
$\F$, we can express $x$ as an $\F$-linear combination,
$$
x=a+b\imag, a,b\in\F.
$$
As $\{1,\sqrt[3]{2},(\sqrt[3]{2})^2\}$ is a basis for $\F$ 
over $\Q$, both $a$ and $b$
can be expressed as $\Q$-linear combinations,
$$
a=a_0+a_1\sqrt[3]{2}+a_2(\sqrt[3]{2})^2, b=b_0+b_1\sqrt[3]{2}+b_2(\sqrt[3]{2})^2,
$$
with the $a_i,b_i\in\Q$. This gives,
$$
x=a_0+a_1\sqrt[3]{2}+a_2(\sqrt[3]{2})^2+b_0\imag+b_1\sqrt[3]{2}\imag+b_2(\sqrt[3]{2})^2\imag,
$$
a $\Q$-linear combination for $x$ as required.

Suppose now:
$$
a_0+a_1\sqrt[3]{2}+a_2(\sqrt[3]{2})^2+
b_0\imag+b_1a_3\sqrt[3]{2}\imag+b_2(\sqrt[3]{2})^2\imag=0,
$$
with the $a_i,b_i\in\Q$. Gathering together real and imaginary parts:
$$
(a_0+a_1\sqrt[3]{2}+a_2(\sqrt[3]{2})^2)+(b_0+b_1\sqrt[3]{2}+b_2(\sqrt[3]{2})^2)\imag=
a+b\imag=0,
$$
for $a$ and $b$ now elements of $\F$. As $\{1,\imag\}$ are independent over $\F$ 
the coefficients in this last expression are zero, ie: $a=b=0$. This gives:
$$
a_0+a_1\sqrt[3]{2}+a_2(\sqrt[3]{2})^2=0=b_0+b_1\sqrt[3]{2}+b_2(\sqrt[3]{2})^2,
$$
and as $\{1,\sqrt[3]{2},(\sqrt[3]{2})^2\}$ are independent 
over $\Q$ 
the coefficients in these two expressions are also zero, ie: $a_0=a_1=a_2=b_0=
b_1=b_2=0$.
The six elements are thus independent and form a basis as claimed.

\paragraph{\hspace*{-0.3cm}}
The proof of the tower law is completely analogous to the example above:

\begin{prooftower}
Let $\{\aa_1,\aa_2,\ldots,\aa_n\}$ be a basis for $E$ as an $F$-vector space and
$\{\bb_1,\bb_2,\ldots,\bb_m\}$ a basis for $L$ as an $E$-vector space, both containing a
finite  number of elements as these extensions are finite by assumption. We 
show that the $mn=[L:E][E:F]$ elements
$$
\{\aa_i\,\bb_j\}, 1\leq i\leq n, 1\leq j\leq m,
$$
form a basis for the $F$-vector space $L$, thus giving the
result. Working ``backwards'' as in the example above, 
if $x$ is an element of $L$ we can express it as an $E$-linear combination of
the $\{\bb_1,\ldots,\bb_m\}$:
$$
x=\sum_{i=1}^m a_i\,\bb_i,
$$
where, as they are elements of $E$, each of the $a_i$ can be expressed as $F$-linear
combinations of the $\{\aa_1,\aa_2,\ldots,\aa_n\}$:
$$
a_i=\sum_{j=1}^n b_{ij}\aa_j\Rightarrow x=\sum_{i=1}^m\sum_{j=1}^n b_{ij}\aa_j\,\bb_i.
$$
Thus the elements $\{\aa_i\,\bb_j\}$ span the field $L$. If we have
$$
\sum_{i=1}^m\sum_{j=1}^n b_{ij}\aa_j\,\bb_i=0,
$$
with the $b_{ij}\in F$, we can collect together all the $\bb_1$ terms, all the 
$\bb_2$ terms, and so on (much as we took real and imaginary parts in the example),
to obtain an $E$-linear combination
$$
\biggl(\sum_{j=1}^n b_{1j}\aa_j\biggr)\,\bb_1+
\biggl(\sum_{j=1}^n b_{2j}\aa_j\biggr)\,\bb_2+\cdots+
\biggl(\sum_{j=1}^n b_{mj}\aa_j\biggr)\,\bb_m=0.
$$
The independence of the $\bb_i$ over $E$ forces all the coefficients to be zero:
$$
\biggl(\sum_{j=1}^n b_{1j}\aa_j\biggr)=\cdots=\biggl(\sum_{j=1}^n b_{mj}\aa_j\biggr)=0,
$$
and the independence of the $\aa_j$ over $F$ forces all the coefficients in each of these
to be zero too, ie: $b_{ij}=0$ for all $i,j$. The $\{\aa_i\,\bb_j\}$
are thus independent.
\qed
\end{prooftower}

\paragraph{\hspace*{-0.3cm}}
\label{vector:spacesI:minimum:polynomial}
We find the minimum polynomial over $\Q$ of $\aa+\ww$, where
$\aa=\sqrt[3]{2}$ and $\ww=\frac{1}{2}+\frac{\sqrt{3}}{2}\imag$. Following
the recipe in the proof of Theorem 
\ref{finite.simple} (or just brute force) gives
$\Q(\aa,\ww)=\Q(\aa+\ww)$ with $[\Q(\aa+\ww):\Q]=[\Q(\aa,\ww):\Q]=6$
by the Tower law. So we are after a degree $6$ polynomial. Indeed, it
suffices to find a monic degree $6$ polynomial $g$ over $\Q$ having $\aa+\ww$ as a root,
since the minimum polynomial must then divide $g$, hence be $g$.

Writing $\bb=\aa+\ww$ we thus require $a,b,c,d,e,f\in\Q$ such that
\begin{equation}
  \label{eq:1}
\bb^6+a\bb^5+b\bb^4+c\bb^3+d\bb^2+e\bb+f=0  
\end{equation}
Now compute the powers of $\bb$ and write the answers in terms on the
basis $\{1,\aa,\aa^2,\ww,\aa\ww,\aa^2\ww\}$ for $\Q(\aa,\ww)$ over $\Q$ given
by the tower law. For example,
$$
\bb^3=\aa^3+3\aa^2\ww+3\aa\ww^2+\ww^3=3\aa^2\ww-3\aa\ww-3\aa+3,
$$
and the others are similar using the facts $\aa^3=2,\ww^3=1$ and
$\ww^2=-\ww-1$. Substituting the results into (\ref{eq:1}) and
collecting terms gives a linear combination of the basis vectors
$\{1,\aa,\aa^2,\ww,\aa\ww,\aa^2\ww\}$  equal
to $0$. Independence means the coefficients must be zero, so we get a
linear system of equations in the variables 
$a,\ldots,f$. Solving these gives $a=3,b=6,c=3,d=0,e=f=9$ and hence
the minimum polynomial
$$
x^6+3x^5+6x^4+3x^3+9x+9.
$$

\subsection*{Further Exercises for Section \thesection}


\begin{vexercise}\label{ex8.5}
\hspace{1em}\begin{enumerate}
\item Show that if $F\subseteq L$ are fields with $[L:F]=1$ then $L=F$.
\item Let $F\subseteq L\subseteq E$ be fields with $[E:F]=[L:F]$. Show that 
$E=L$.
\end{enumerate}
\end{vexercise}

\begin{vexercise}\label{ex8.6}
Let $\F={\Q}(a)$, where $a^3=2$. Express $(1+a)^{-1}$ and 
$(a^4+1)(a^2+1)^{-1}$ in the form $ba^2+ca+d$, where
$b,d,c $ are in ${\Q}$.
\end{vexercise}

\begin{vexercise}\label{ex8.7}
Let $\aa = \root{3}\of{5}$. Express the following elements of ${\Q}(\aa)$
as polynomials of degree at most 2 in $\aa$ (with coefficients in
${\Q}$):
$$
1. \,\, 1/\aa \qquad
2. \,\, \aa^5 - \aa^6 \qquad
3. \,\, \aa/(\aa^2+1) \qquad
$$
\end{vexercise}

\begin{vexercise}\label{ex8.8}
Find the minimum polynomial over ${\Q}$ of 
$\alpha = \kern-2pt\sqrt{2}+ \kern-2pt\sqrt{-2}$. 
Show that the following are elements of the field $\Q(\aa)$ and 
express them as 
polynomials in $\alpha$ (with coefficients in ${\Q}$) of degree 
at most 3:
$$
1. \,\, \kern-2pt\sqrt{2}\qquad
2. \,\, \kern-2pt\sqrt{-2}\qquad
3. \,\, i\qquad
4. \,\, \alpha^5 + 4\alpha + 3\qquad
5. \,\, 1/\alpha\qquad
6. \,\, (2\alpha + 3)/(\alpha^2 + 2\alpha + 2)\qquad
$$
\end{vexercise}

\begin{vexercise}\label{ex8.9}
Find the minimum polynomials over ${\Q}$ of the following numbers:
$$
1. \,\, 1+\imag\qquad
2. \,\, \root{3}\of{7}\qquad
3. \,\, \root{4}\of{5}\qquad
4. \,\, \kern-2pt\sqrt{2}+\imag\qquad
5. \,\, \kern-2pt\sqrt{2} + \root{3}\of {3}\qquad
$$
\end{vexercise}

\begin{vexercise}\label{ex8.10}
Find the minimum polynomial over $\Q$ of the following:
$$
1. \,\, \kern-2pt\sqrt{7}\qquad
2. \,\, (\kern-2pt\sqrt{11}+3)/2\qquad
3. \,\, (\imag\kern-2pt\sqrt{3}-1)/2\qquad
$$
\end{vexercise}

\begin{vexercise}\label{ex8.11}
For each of the following fields $L$ and $F$, find $[L:F]$ and compute
a basis for $L$ over $F$.
\begin{enumerate}
\item
$L = {\Q}(\kern-2pt\sqrt{2},\root 3\of{2})$, $F = {\Q}$;
\item
$L = {\Q}(\root 4\of{2},\imag)$, $F = {\Q}(\imag)$;
\item
$L = {\Q}(\xi)$, $F = {\Q}$, where $\xi$ is a primitive complex
7th root of unity;
\item
$L = {\Q}(\imag,\kern-2pt\sqrt{3},\omega)$, $F = {\Q}$, where $\omega$ is a
primitive complex cube root of unity.
\end{enumerate}
\end{vexercise}

\begin{vexercise}\label{ex8.12}
Let $a=e^{\pi i/4}$. Find  $[F(a):F]$ when $F={\R}$ and 
when $F={\Q}$. 
\end{vexercise}


\section{Fields III: Splitting Fields and Finite Fields}
\label{fields3}

\subsection{Splitting Fields}

\paragraph{\hspace*{-0.3cm}}
In Section \ref{lect1} we encountered
fields containing ``just enough'' numbers to solve some polynomial
equation. We now make this more precise. 

Let $f$ be a polynomial with $F$-coefficients. We say that $f$ {\em splits\/}
in an extension $F\subseteq E$ when we can factorise
$$
f=\prod_{i=1}^{\deg f} (x-\aa_i),
$$
in the polynomial ring $E[x]$. Thus $f$ splits in $E$ precisely when $E$ contains all
the roots $\{\aa_1,\aa_2,\ldots,$ $\aa_{\deg f}\}$ of $f$. 

There will in general be many such extension fields -- we are after the
smallest one. By Kronecker's theorem (more accurately, Corollary \ref{kronecker:corollary})
there is an extension $F\subseteq
K$ such that $K$ contains all the roots of $f$. If these roots are
$\aa_1,\aa_2,\ldots,\aa_d\in K$, then let
$E=F(\aa_1,\aa_2,\ldots,\aa_d)$.

\begin{definition}[splitting field of a polynomial]
The field extension $F\subseteq E$ constructed in this way is called 
a splitting field of $f$ over $F$.
\end{definition}

\begin{vexercise}
Show that $E$ is a splitting field of the polynomial $f$ over $F$ if and only 
if $f$ splits in $E$ but not in any subfield of $E$ containing $F$ (so in this sense,
$E$ is the {\em smallest\/} field containing $F$ and all the roots).
\end{vexercise}

\paragraph{\hspace*{-0.3cm}}
The splitting field of $x^2+1$ over $\Q$ is $\Q(\imag)$. The splitting
field of $x^2+1$ over $\R$ is $\C$.

\paragraph{\hspace*{-0.3cm}}
Our example from Section \ref{lect1} again: the polynomial $x^3-2$ has roots
$\aa,\aa\ww,\aa\ww^2$ where $\aa=\sqrt[3]{2}\in\R$ and
$$
\ww=-\frac{1}{2}+\frac{\sqrt{3}}{2}\imag.
$$
Thus a splitting field for $f$ over $\Q$ is given by $\Q(\aa,\aa\ww,\aa\ww^2)$,
which is the same thing as $\Q(\aa,\ww)$. 


\begin{aside}
In Section \ref{galois.groups} we will prove (Theorem G) that 
an isomorphism of a field to itself $\ss:F\rightarrow F$ can always be extended to 
an isomorphism $\widehat{\ss}:E_1\rightarrow E_2$ where $E_1$ is a splitting field 
of some polynomial $f$ over $F$ and $E_2$ is another splitting field of this
polynomial. Thus, any two splitting fields of a polynomial over $F$ are isomorphic. 
\end{aside}

\begin{vexercise}
  \begin{enumerate}
  \item Let $f=ax^2+bx+c\in\Q[x]$ and $\Delta=b^2-4ac$. Show that the
    splitting field of $f$ over $\Q$ is $\Q(\kern-2pt\sqrt{\Delta})$.
  \item Let $f=(x-\aa)(x-\bb)\in\Q[x]$ and $D=(\aa-\bb)^2$. Show that the
    splitting field of $f$ over $\Q$ is $\Q(\kern-2pt\sqrt{D})$. Show that the
    splitting is $F(\aa)=F(\bb)$.  
  \end{enumerate}
\end{vexercise}

\subsection{Finite Fields}

The construction of Section \ref{fields2} produced explicit examples
of fields having order $p^d$ for $p$ a prime. We now show that any
finite field must have order $p^d$ for some prime $p$ and $d>0$, and
there exists a unique such field. 

\paragraph{\hspace*{-0.3cm}}
Recall from Definition \ref{definition:prime_subfield}
that the prime subfield of a field $F$ is the intersection of all the
subfields of $F$. 
It is isomorphic to $\F_p$ for some $p$ or to $\Q$. In particular,
the prime subfield of a finite field $F$ must be isomorphic to $\F_p$.

Using the ideas from Section \ref{lect7}, we have an extension of fields $\F_p\subseteq F$
and hence the finite field $F$ forms a vector space over the field $\F_p$. This
space must be finite dimensional (for $F$ to be finite), so each element of $F$
can be written uniquely as a linear combination,
$$
a_1\aa_1+a_2\aa_2+\cdots+a_d\aa_d,
$$
of some basis vectors $\aa_1,\aa_2,\ldots,\aa_d$ with the $a_i\in\F_p$. In 
particular there are $p$ choices for each $a_i$, and the choices are independent,
giving $p^d$ elements of $F$ in total.

Thus a finite field has $p^d$ elements for some prime $p$.

\paragraph{\hspace*{-0.3cm}}
Here is an extended example that 
shows the converse, ie:
constructs a field 
with $q=p^d$ elements for any prime $p$
and positive integer $d>0$. 

Consider the polynomial $x^q-x$ over the field $\F_p$
of $p$ elements. 
Let $L$ be an extension of the field $\F_p$ containing all the roots of the polynomial,
as guaranteed us by the Corollary to Kronecker's Theorem. 
In Exercise \ref{ex3.20} we used the formal derivative to see whether a polynomial has distinct roots.
We have $\partial(x^q-x)
=qx^{q-1}-1=p^n x^{p^n-1}-1=-1$ as $p^n=0$ in $\F_p$. The constant polynomial $-1$ 
has no roots in $L$,  and so the original polynomial 
$x^q-x$ has no repeated roots
in $L$ by Exercise \ref{ex3.20}.

In fact, the $p^d$ distinct roots of $x^q-x$ 
form a subfield of $L$, and this is the field of order $p^d$ that we seek.
To show this,
let $a,c$ be roots (so that $a^q=a$ and
$c^q=c$). We show that $-a,a+c,
ac$ and $a^{-1}$ are also roots.

Firstly, $(-a)^q-(-a)=(-1)^qa^q+a$. If $p=2$, then $-1=1$ in $\F_2$,
so that $(-1)^qa^q+a=a^q+a=a+a$ 
$=2a=0$. Otherwise $p$ is odd so that $(-1)^q=-1$
and $(-1)^qa^q+a=-a^q+a=-a+a=0$. In either case, $-a$ is a root
of the polynomial $x^q-x$.

Next, 
$$
(a+c)^q=\sum_{i=0}^{q} 
\binom{q}{i}
a^ic^{q-i}=a^q+c^q+p(\text{other terms}),
$$
as $p$ divides the binomial coefficient when $0<i<q$ by Exercise \ref{ex3.3}.
Thus
$(a+c)^q=a^q+a^q$. (Compare this with Exercise \ref{ex9.1}.)
Substituting $a+c$ into $x^q-x$ gives
$$
(a+c)^q-(a+c)=a^q+c^q-a-c=0,
$$
using $a^q=a$ and $c^q=c$. Thus $a+c$ is 
also a root of the polynomial.

The product $(ac)^q-ac=a^qc^q-ac=ac-ac=0$. Finally, 
$(a^{-1})^q-(a^{-1})=(a^q)^{-1}-(a^{-1})=a^{-1}-a^{-1}=0$. In both cases we
have used $a^q=a$.

Thus the $q=p^d$ roots of the polynomial form a subfield of $L$ as claimed, and we have
constructed a field with this many elements.

\paragraph{\hspace*{-0.3cm}}
Looking back at this example, $L$ was an extension of $\F_p$ containing the
roots of the polynomial $x^q-x$. In particular, if these roots are
$\{a_1,\ldots,a_q\}$, then $\F_p(a_1,\ldots,a_q)$ is the
splitting field over $\F_p$ of the polynomial. 
In the example we constructed the subfield $\F$ of $L$ consisting of the roots of
$x^q-x$. As any subfield contains $\F_p$, we have $\F_p(a_1,\ldots,a_q)
\subseteq \F$, whereas $\F=\{a_1,\ldots,a_q\}$ so that 
$\F\subseteq \F_p(a_1,\ldots,a_q)$. Hence the field we constructed in the 
example was the splitting field over $\F_p$ of the polynomial $x^q-q$. 

If $F$ is now an arbitrary field with $q$ elements, then it has prime subfield 
$\F_p$. Moreover, as the multiplicative group of $F$ has order $q-1$,
by Lagrange's Theorem (see Section \ref{groups.stuff}), every element of $F$ satisfies
$x^{q-1}=1$, hence is a root of the $\F_p$-polynomial $x^q-x=0$. Thus,
a finite field of order $q$ is the splitting field over $\F_p$ of 
the polynomial $x^q-x$, and by the uniqueness of such, 
any two fields of order $q$ are isomorphic.

\paragraph{\hspace*{-0.3cm}}
We finish with a fact about finite fields that will prove useful later on.
Remember that a field is, among other things, two groups spliced together in a 
compatible way: the elements form a group under addition (the {\em additive group\/})
and the non-zero 
elements form a group under multiplication (the {\em multiplicative group\/}) .

Looking at the complex numbers as an example, we can find a number of finite
subgroups of the multiplicative group $\C^*$ of $\C$ by considering roots of $1$. 
For any $n$,
the powers of the $n$-th root of $1$,
$$
\ww=\cos\frac{2\pi}{n}+\imag\sin\frac{2\pi}{n},
$$
form a subgroup of $\C^*$ of order $n$. Moreover, this subgroup is
cyclic.

\begin{proposition}
Let $F$ be any field and $G$ a finite subgroup of the multiplicative group $F^*$ of $F$.
Then $G$ is a cyclic group.
\end{proposition}

In particular, if $F$ is a finite field, then the multiplicative group $F^*$ of $F$ is
finite, hence cyclic.

\begin{proof}
By Exercise \ref{ex11.2a} 
there is an element $g\in G$ whose order $m$ is the least common multiple of 
all the orders of elements of $G$.
Thus,
any element $h\in G$ satisfies $h^m=1$. Hence every element of the group is a root
of $x^m-1$, and since this polynomial has at most $m$ roots in $F$, the order of
$G$ must be $\leq m$. As $g\in G$ has order $m$ 
its powers must exhaust the whole group, hence $G$ is cyclic.
\qed
\end{proof}

\subsection{Algebraically closed fields}

\paragraph{\hspace*{-0.3cm}}
In the first part of this section we dealt with fields in which a particular
polynomial of interest split into linear factors. There are
fields like the complex numbers in which {\em any\/} polynomial splits. 

A field $F$ is said to be {\em algebraically closed\/} if and only
if every (non-constant) polynomial over $F$ splits in $F$. 

\paragraph{\hspace*{-0.3cm}}
If $F$ is algebraically closed and $\aa$ is algebraic over $F$ then there is a
polynomial with $F$-coefficients having $\aa$ as a root. As $F$ is algebraically closed,
this polynomial splits in $F$, so that in particular $\aa$ is in $F$. This
explains the terminology: an algebraically closed field is {\em closed\/} with respect to the
taking of algebraic elements. Contrast this with fields like $\Q$, over which there are algebraic
elements like $\kern-2pt\sqrt{2}$ that are not contained in $\Q$. 

\begin{vexercise}
Show that the following are equivalent:
\begin{enumerate}
\item $F$ is algebraically closed;
\item every non-constant polynomial over $F$ has a root in $F$;
\item the irreducible polynomials over $F$ are precisely the linear ones;
\item if $F\subseteq E$ is a finite extension then $E=F$.
\end{enumerate}
\end{vexercise}

\begin{theorem}
Every field $F$ is contained in an algebraically closed one.
\end{theorem}

\begin{sketchproof}
The full proof is beyond the scope of these notes, although the technical
difficulties are not algebraic (or even number theoretical) but set theoretical.
If the field $F$ is finite or countably infinite, the proof sort of goes as follows: there are countably
many polynomials over a countable field, so take
the union of all the splitting fields of these polynomials. Note that for 
a finite field, this is an infinite union, so an algebraically closed field
containing even a
finite field is very large.
\qed
\end{sketchproof}

\subsection{Simple extensions}

\paragraph{\hspace*{-0.3cm}}
We saw in Section \ref{lect4} that the extension $\Q\subset\Q(\kern-2pt\sqrt{2},\kern-2pt\sqrt{3})$
is, despite appearances, simple. The fact that the extension is finite
turns out to be enough to see that it is simple:

\begin{theorem}\label{finite.simple}
Let $F\subset E$ be a finite extension such that the roots of any irreducible
polynomial over $E$ are distinct. Then $E$ is a simple extension of $F$, ie: $E=F(\aa)$
for some $\aa\in E$.
\end{theorem}

The following proof is for the case that $F$ is infinite.

\begin{proof}
Let $\{\aa_1,\aa_2,\ldots,\aa_k\}$ be a basis for $E$ over $F$ and consider the field
$F_1=F(\aa_3,\ldots,\aa_k)$, so that $E=F_1(\aa_1,\aa_2)$.
We will show that $F_1(\aa_1,\aa_2)$ is a simple extension of $F_1$,  ie:
that $F_1(\aa_1,\aa_2)=F_1(\theta)$ for some $\theta\in E$. 
Thus $E=F(\aa_1,\aa_2,\ldots,\aa_k)
=F(\theta,\aa_3\ldots,\aa_k)$, and so by repeatedly applying this procedure,
$E$ is a simple extension of $F$.

Let $f_1,f_2$ be the minimum polynomials over $F_1$ of $\aa_1$ and $\aa_2$, and 
let $L$ be an algebraically closed field containing of the field $F$. 
As the $\aa_i$ are algebraic
over $F$, we have that the fields $F_1$ and $E$ are contained in $L$ too.
In particular the polynomials $f_1$ and $f_2$ split in $L$,
$$
f_1=\prod_{i=1}^{\deg f_1} (x-\bb_i),
f_2=\prod_{i=1}^{\deg f_2} (x-\delta_i),
$$
with $\bb_1=\aa_1$ and $\delta_1=\aa_2$. As the roots of these polynomials are
distinct we have that $\bb_i\not=\bb_j$ and $\delta_i\not=\delta_j$ for all $i\not= j$.
For any $i$ and any $j\not= 1$, the equation.
$\bb_i+x\delta_j=\bb_1+x\delta_1$
has precisely one solution in $F_1$, namely 
$$
x=\frac{\bb_i-\bb_1}{\delta_1-\delta_j}.
$$
As there only finitely many such
equations and infinitely many elements of $F_1$, there must be an $c\in F_1$
which is a solution to {\em none\/} of them, ie: such that,
$$
\bb_i+c\delta_j\not=\bb_1+c\delta_1
$$
for any $i$ and any $j\not= 1$. Let $\theta=\bb_1+c\delta_1=\aa_1+c\aa_2$.
We show that $F_1(\aa_1,\aa_2)=F_1(\theta)=F_1(\aa_1+c\aa_2)$. 

Clearly $\aa_1+c\aa_2\in F_1(\aa_1,\aa_2)$ so that $F_1(\aa_1+c\aa_2)\subseteq
F_1(\aa_1,\aa_2)$. We will show that $\aa_2\in F_1(\aa_1+c\aa_2)=F_1(\theta)$, 
for then
$\aa_1+c\aa_2-c\aa_2=\aa_1\in F_1(\aa_1+c\aa_2)$, and so $F_1(\aa_1,\aa_2)
\subseteq F_1(\aa_1+c\aa_2)$. 

We have $0=f_1(\aa_1)=f_1(\theta-c\aa_2)$, so if we let $r(t)\in F_1(\theta)[t]$ 
be given by $r(t)=f_1(\theta-ct)$, then we have that $\aa_2$ is a root of both
$r(t)$ and $f_2(x)$. If $\gamma$ is another common root of $r$ and $f_2$, then
$\gamma$ is one of the $\delta_j$, and $\theta-c\gamma$ (being a root of $f_1$)
is one of the $\bb_i$, so that,
$$
\gamma=\delta_j\text{ and }\theta-c\gamma=\bb_i
\Rightarrow 
\bb_i+c\delta_j=\bb_1+c\delta_1,
$$
a contradiction. Thus $r$ and $f_2$ have just the single common root $\aa_2$.
Let $h$ be the minimum polynomial of $\aa_2$ over $F_1(\theta)$, so that 
$h$ divides both $r$ and $f_2$ (recall that the minimum polynomial divides any 
other polynomial having $\aa_2$ as a root). This means that $h$ must have degree
one, for a higher degree would give more than one common root for $r$ and $f_2$.
Thus $h=t+b$ for some $b\in F_1(\theta)$. As $h(\aa_2)=0$ we thus get that
$\aa_2=-b$ and so $\aa_2\in F_1(\theta)$ as required.
\qed
\end{proof}

The theorem is true for finite extensions of finite fields -- even without the
condition on the roots of the polynomials -- but we omit the proof here. We saw in
Exercise \ref{ex4.30} that irreducible polynomials over fields of characteristic $0$
have distinct roots. Thus any finite extension of a field of characteristic
zero $0$ is simple. For example, if $\aa_1,\ldots,\aa_k$ are algebraic over $\Q$, then
$\Q(\aa_1,\ldots,\aa_k)=\Q(\theta)$ for some $\theta\in\C$.


\section{Ruler and Compass Constructions II}
\label{ruler.compass2}

\paragraph{\hspace*{-0.3cm}}
We can completely
describe the complex numbers that are constructible:

\begin{theoremE}\label{thmE}
The number $z\in\C$ is constructible if and only if
there exists a sequence of field extensions,
$$
\Q=K_0\subseteq K_1\subseteq K_2\subseteq\cdots\subseteq K_n,
$$
such that $\Q(z)$ is a subfield of $K_n$, and each $K_i$ is 
an extension of $K_{i-1}$ of degree at most $2$.
\end{theoremE}

The idea, which can be a little obscured by the details, is that
points on a line have a linear relationship with the two points
determining the
line, and points on a circle have a quadratic relationship
with the two points determining the circle. 

\begin{proof}
We prove the ``only if'' part first. 
Recall that $z$ is constructible if and only
if there is a sequence of numbers 
$$
0,1,\text{i}=z_1,z_2,\ldots,z_n=z,
$$
with $z_i$ obtained from earlier numbers in the sequence in one of the three
forms shown in Figure \ref{fig:constructions2:fig1}, where $p,q,r,s\in\{1,2,\ldots,i-1\}$.

Let $K_i$ be the field $\Q(z_1,\ldots,z_{i})$, so we have a tower of 
extensions:
$$
\Q\subseteq K_1\subseteq K_2\subseteq\cdots\subseteq K_n.
$$
We will simultaneously show the following two things by induction:
\begin{itemize}
\item Each of the fields $K_i$ is closed
under conjugation, ie: if $z\in K_i$ then $\bar{z}\in K_i$, and 
\item the degree
of each extension $K_{i-1}\subseteq K_i$ is at most two.
\end{itemize}
The first of these is a 
technical convenience, the main point of which is illustrated by
Exercise \ref{ex10.1}
following the proof. 

\begin{figure}
  \centering
\begin{pspicture}(0,0)(12,3)
\rput(2,1.34){\BoxedEPSF{galois7.6a.eps scaled 2000}}
\rput(6,1.5){\BoxedEPSF{galois7.6b.eps scaled 2000}}
\rput(10,1.58){\BoxedEPSF{galois7.6c2.eps scaled 2000}}
\rput(2.1,1.6){$z_i$}
\rput(.8,.55){$z_p$}\rput(3.1,2){$z_q$}
\rput(1,2.2){$z_r$}\rput(4.8,1.2){$z_s$}
\rput(6.1,2.4){$z_i$}
\rput(4.8,2.1){$z_p$}\rput(7.2,2.6){$z_q$}
\rput(6.4,.8){$z_r$}\rput(3.2,.45){$z_s$}
\rput(9.7,2.5){$z_i$}
\rput(10.45,.55){$z_p$}\rput(11.5,1.8){$z_q$}
\rput(9.8,.65){$z_r$}\rput(8.3,1.85){$z_s$}
\rput(2,.2){(i)}\rput(6.2,.2){(ii)}\rput(10,.2){(iii)}
\end{pspicture}
\caption{}
  \label{fig:constructions2:fig1}
\end{figure}

Firstly, $K_1=\Q(\imag)=\{a+b\imag\,:\,a,b\in\Q\}$ is certainly closed
under conjugation and $[K_1:\Q]=[\Q(\imag):\Q]=2$ as the minimum
polynomial of $\imag$ over $\Q$ is $x^2+1$. 
Now fix $i$ and 
suppose that $K_{i-1}$ is closed under conjugation with $K_i=K_{i-1}(z_i)$.

\vspace*{1em}
\noindent (i). Suppose that $z_i$ is obtained as in case
(i) of Figure \ref{fig:constructions2:fig1}. The Cartesian equation for one of the lines 
is $y=m_1x+c_1$, passing through the points $z_p,z_q$, with $z_p,z_q\in K_{i-1}$.
As $K_{i-1}$ is closed under conjugation, Exercise \ref{ex10.1} gives 
the real and imaginary parts of $z_p$ and $z_q$ are in $K_{i-1}$.
Thus,
$$
\begin{pspicture}(0,0)(14,1)
\rput(-.75,0){
\rput(0.75,0){
\rput(3,.2){$\im z_p=m_1\re z_p+c_1$}
\rput(3,.8){$\im z_q=m_1\re z_q+c_1$}
\rput(4.75,.5){$\left.\begin{array}{c}
\vrule width 0 mm height 10 mm depth 0 pt\end{array}\right\}$}
\rput(5.2,.5){$\Rightarrow$}
}
\rput(9.8,.5){$m_1={\ds \frac{\im z_p-\im z_q}{\re z_p-\re z_q}}
\text{ and }
c_1=\im z_p-m_1\re z_p$}
}
\end{pspicture}
$$
so that $m_1,c_1\in K_{i-1}$.
(If the line is vertical with equation $x=c_1$ we get
$c_1=\re z_p\in K_{i-1}$). If the equation of the other line is
$y=m_2x+c_2$, we similarly get $m_2,c_2\in K_{i-1}$. As $z_i$ lies on both these lines we
have
$$
\begin{pspicture}(0,0)(14,1)
\rput(2.25,0){
\rput(-1,0){
\rput(3,.2){$\im z_i=m_1\re z_i+c_1$}
\rput(3,.8){$\im z_i=m_2\re z_i+c_2$}
\rput(4.75,.5){$\left.\begin{array}{c}
\vrule width 0 mm height 10 mm depth 0 pt\end{array}\right\}$}
}
\rput(6.25,.5){with $m_1,m_2,c_1,c_2\in K_{i-1}$}
}
\end{pspicture}
$$
hence
$$
\begin{pspicture}(0,0)(14,1)
\rput(-3,0){
\rput(10,.5){$\re z_i={\ds \frac{c_2-c_1}{m_1-m_2}}
\text{ and }
\im z_i={\ds \frac{m_1(c_2-c_1)}{m_1-m_2}}+c_1$}
}
\end{pspicture}
$$
must lie in $K_{i-1}$ too. 
As $K_{i-1}$ is closed under
conjugation we get $z_i\in K_{i-1}$ too, so in fact 
$K_i=K_{i-1}(z_i)=K_{i-1}$. Thus the degree of the extension $K_{i-1}\subseteq K_i$
(being $1$) is certainly $\leq 2$. Moreover, $K_i=K_{i-1}$ is closed under conjugation as 
$K_{i-1}$ is. 

\vspace*{1em}
\noindent (ii).
Suppose $z_i$ arises as in case (ii) with the line having equation $y=mx+c$ and the circle
having equation $(x-\re z_s)^2+(y-\im z_s)^2=r^2$, where 
$r^2=(\re z_r-\re z_s)^2+(\im z_r-\im z_s)^2$. As before, $m,c\in
K_{i-1}$; moreover, $z_r,z_s\in K_{i-1}$, hence $r^2\in K_{i-1}$. As $z_i$ lies on the line we have 
$\im z_i=m\re z_i+c$, and as it lies on the circle we have
$$
(\re z_i-\re z_s)^2+(m\re z_i+c-\im z_s)^2=r^2.
$$
Thus the polynomial $(x-\re z_s)^2+(mx+c-\im z_s)^2=r^2$ is a quadratic 
with $K_{i-1}$ coefficients and having $\re z_i$ as a root. 
The minimum polynomial of $\re z_i$ over $K_{i-1}$ thus has degree at
most $2$,
giving
$$
[K_{i-1}(\re z_i):K_{i-1}]\leq 2
$$
by Theorem D. 
In fact, $\im z_i\in K_{i-1}(\re z_i)$ as well, since $\im z_i=m\re
z_i+c$. Thus 
$z_i$ itself is in $K_{i-1}(\re z_i)$, as $i$ also is,
and we have the sequence,
$$
K_{i-1}\subseteq K_i=K_{i-1}(z_i)\subseteq K_{i-1}(\re z_i),
$$
giving that the degree of the extension $K_{i-1}\subseteq K_i$ is also
$\leq 2$ by the Tower Law. 

Finally, we show that the field $K_i$ is closed under 
conjugation, for which we can assume that $[K_i:K_{i-1}]=2$ -- it is trivially
the case if the degree is one. Now, $K_i=K_{i-1}(z_i)=K_{i-1}(\re z_i)$, 
so in particular
$z_i$ and $\re z_i$ are in $K_i$, hence 
$$
\im z_i=\frac{z_i-\re z_i}{\imag}
$$
is too. The result is that $\re z_i-\im z_i\cdot\imag=\bar{z_i}$ is in $K_i$
too. A general element of $K_i$ has the form  $a+bz_i$ with $a,b\in
K_{i-1}$, whose conjugate $\bar{a}+\bar{b}\bar{z_i}$ is thus also in $K_i$.

\vspace*{1em}
\noindent (iii). If $z$ arises as in case (iii), then as
$z_i$ lies on both circles we have
$$
(\re z_i-\re z_s)^2+(\im z_i-\im z_s)^2=r^2\text{ and }
(\re z_i-\re z_p)^2+(\im z_i-\im z_p)^2=s^2,
$$
with both $r^2$ and $s^2$ in $K_{i-1}$  for the same reason as in case (ii).
Expanding both expressions gives terms of the form $\re z_i^2+\im z_i^2$,
and equating leads to,
\begin{equation*}
\begin{split}
\im z_i=\frac{\bb_1}{\aa}\re z_i+\frac{\bb_2}{\aa},&\text{ where }
\aa=2(\im z_s-\im z_p), \bb_1=2(\re z_p-\re z_s)\\
&\text{ and }
\bb_2=\re z_s^2+\im z_s^2-(\re z_p^2+\im z_p^2)+s^2-r^2.
\end{split}
\end{equation*}
Combining this $K_{i-1}$-expression for $\im z_i$ with the first of the two circle 
equations above puts us into a similar situation as case (ii), from which the result
follows in the same way.

\vspace*{1em}
Now for the ``if'' part, which is mercifully shorter. Suppose we have a tower of fields
$\Q=K_0\subseteq K_1\subseteq K_2\subseteq\cdots\subseteq K_n,$
with $\Q(z)$ in $K_n$, hence $z\in K_n$. 
We can assume that $z\not\in K_{n-1}$ (otherwise stop one step earlier!) and so 
we have
$$
K_{n-1}\subseteq K_{n-1}(z)\subseteq K_n
$$
where $z\not\in K_{n-1}$ gives $[K_{n-1}(z):K_{n-1}]\geq 2$. On the other
hand $[K_n:K_{n-1}]\leq 2$ so by the tower law we have
$[K_{n-1}(z):K_{n-1}]=[K_n:K_{n-1}]$ and hence $K_n=K_{n-1}(z)$ with $[K_{n-1}(z):K_{n-1}]=2$. 
The minimum polynomial of $z$ over $K_{n-1}$ thus
has the form $x^2+bx+c$, with $b,c\in K_{n-1}$, so that $z$ is one of,
$$
\frac{-1\pm\sqrt{b^2-4c}}{2}
$$
either of which can be constructed from $1,2,4,b,c\in K_{n-1}$, using
the arithmetical and square root 
constructions of Section \ref{ruler.compass}. But in the same way $b,c$ can be
constructed from elements of $K_{n-2}$, and so on, giving that $z$ is
indeed constructible.
\qed
\end{proof}

\begin{vexercise}\label{ex10.1}
Let $K$ be a field such that 
$\Q(\imag)\subseteq K\subseteq \C,$
and suppose that $K$ is closed under conjugation. Show that $z\in K$ if and only if
the real and imaginary parts of $z$ are in $K$. 
\end{vexercise}

\paragraph{\hspace*{-0.3cm}}
It is much easier to use the ``only if'' part of the Theorem, which shows when
numbers {\em cannot\/} be constructed, so we restate this part as a separate,

\begin{corollary}
\label{corollary:construcitble_necessary}
If $z\in\C$ is constructible then the degree of the extension $\Q\subseteq \Q(z)$
must be a power of two.
\end{corollary}

\begin{proof}
If $z$ is constructible then we have the tower of extensions
as given in Theorem E, with $z\in K_n$. Thus we have the sequence of extensions
$\Q\subseteq \Q(z)\subseteq K_n$, which by the tower law gives,
$$
[K_n:\Q]=[K_n:\Q(z)][\Q(z):\Q].
$$
Thus $[\Q(z):\Q]$ divides $[K_n:\Q]$, which is a power of two, so $[\Q(z):\Q]$ must
also be a power of two.
\qed
\end{proof}

To use the ``if'' part to show that numbers {\em can\/}
be constructed by finding a tower of fields as in Theorem E, is a little harder. 
We will need to know more about the fields sandwiched between $\Q$ and
$\Q(z)$ before we can do this. The Galois Correspondence in Section \ref{galois.correspondence}
will give us the control we need.

\paragraph{\hspace*{-0.3cm}}
The Corollary is only stated in one direction. The
converse is
{\em not true\/}. 

\paragraph{\hspace*{-0.3cm}}
\label{constructions2:pgons}
A regular $p$-gon, for $p$ a prime, can be constructed, by Exercise \ref{ex7.40}, 
precisely when the complex
number $z=\cos(2\pi/p)+\imag\sin(2\pi/p)$ can be constructed. By Exercise \ref{ex_lect3.2}, the
minimum polynomial of $z$ over $\Q$ is the $p$-th cyclotomic polynomial,
$$
\Phi_p(x)=x^{p-1}+x^{p-2}+\cdots+x+1.
$$ 
The degree of the extension $\Q\subseteq \Q(z)$ is thus $p-1$, so
$p-1$ must be a power
of two if the $p$-gon is to be constructed, i.e. 
$$
p=2^n+1.
$$

Actually, even more can be said. If $m$ is odd, the polynomial $x^m+1$ has $-1$ as
a root, and so can be factorised as 
$x^m+1=(x+1)(x^{m-1}-x^{m-2}+x^{m-3}-\cdots-x+1).$
Thus if $n=mk$ for $m$ odd, we have
$$
2^n+1=(2^k)^m+1=(2^k+1)((2^k)^{m-1}-(2^k)^{m-2}+(2^k)^{m-3}-\cdots-(2^k)+1),
$$
giving that $2^n+1$ cannot be prime unless $n$ has no \emph{odd\/}
divisors; i.e. $2^n+1$ can only be prime if 
$n$ itself is a power of two. 

Thus for a $p$-gon to be constructible, we must have that $p$ is a prime number
of the form
$$
p=2^{2^t}+1,
$$
a so-called {\em Fermat prime\/}. Such primes are extremely rare: the only ones 
$<10^{900}$ are 
$$
3,5,17,257\text{ and }65537.
$$
We will see in Section \ref{galois.corresapps} that the converse is true: if $p$ is a
Fermat prime, then a regular $p$-gon \emph{can\/} be constructed.

\paragraph{\hspace*{-0.3cm}}
A square plot of land can always be doubled in area using a ruler and compass:
$$
\begin{pspicture}(0,0)(14,2.5)
\rput(0,-0.25){
\rput(7,1.5){\BoxedEPSF{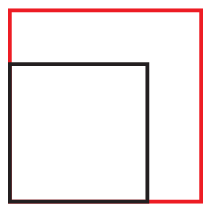 scaled 1000}}
\rput(7.4,0.3){$(t,0)$}\rput(5.6,2){$(0,t)$}
\rput(8.9,2.5){$(\kern-2pt\sqrt{2}t,\kern-2pt\sqrt{2}t)$}
}
\end{pspicture}
$$
Set
the compass to the side length $t$ of the plot. As $\kern-2pt\sqrt{2}$
is a constructible number, we can 
construct the point with coordinates $(\kern-2pt\sqrt{2}t,
\kern-2pt\sqrt{2}t)$, hence doubling the area. 

\paragraph{\hspace*{-0.3cm}}
Is there a similar procedure for a cube? Suppose the original cube
has side length $1$, so that the task is to produce a new cube of {\em volume\/} $2$. 
If this could be accomplished via a ruler and compass construction, then by setting the
compass to the side length of the new cube, we would have constructed $\sqrt[3]{2}$. 
But the minimum polynomial over $\Q$ of $\sqrt[3]{2}$ is clearly $x^3-2$, with the
extension $\Q\subset\Q(\sqrt[3]{2})$ thus having degree three. Such a 
construction cannot therefore be possible.

\paragraph{\hspace*{-0.3cm}}
The subset $\Box^n$ of $\R^n$ given by
$$
\Box^n=\{x\in \R^n\,|\,|x_i|\leq \frac{t}{2}\text{ for all }i\}
$$
is an $n$-dimensional cube of side length $t$ having volume $t^n$. 
In particular, in $4$-dimensions we have the hypercube:
$$
\begin{pspicture}(0,0)(14,6)
\rput(7,3){\BoxedEPSF{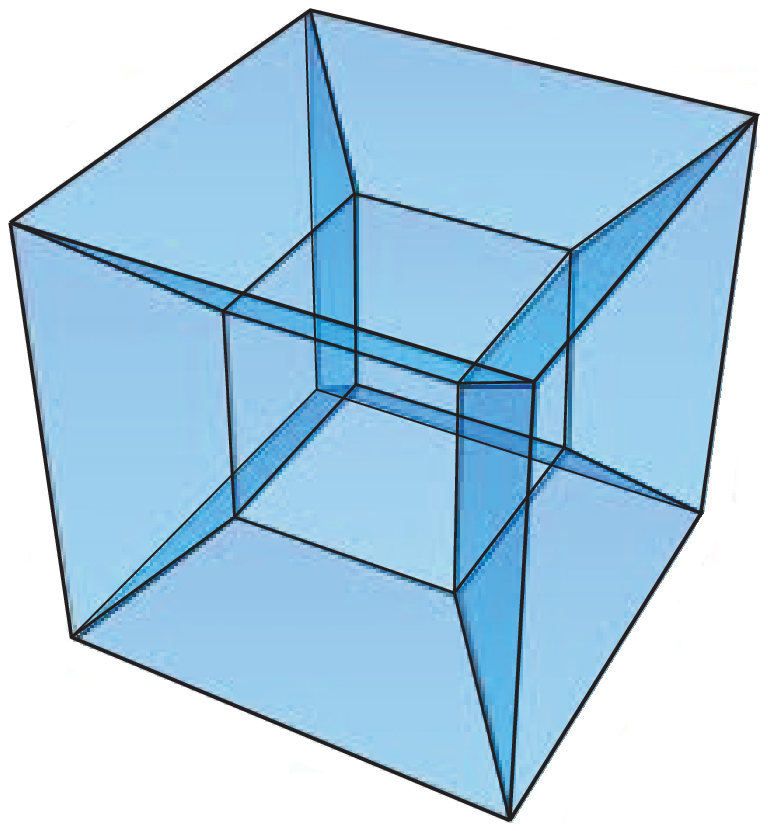 scaled 700}}
\end{pspicture}
$$
The vertices
can be placed on the $3$-sphere $S^3$ in $\R^4$. Stereographically projecting
$S^3$ to $\R^3$ gives the picture above. This object
can be doubled in volume with ruler and compass because the point with coordinates
$(\sqrt[4]{2}t,\sqrt[4]{2}t,\sqrt[4]{2}t,\sqrt[4]{2}t)$ can be
constructed.

\paragraph{\hspace*{-0.3cm}}
One of our fundamental constructions was the bisection of an angle. It is 
natural to ask if there is a construction that {\em trisects\/} an
angle. Certainly there are particular angles that can be trisected:
if the angle $\phi$ is constructible for example, then the angle $3\phi$ can
be trisected.

The angle $\pi/3$ however cannot be trisected. We will see this by
showing that
the angle $\pi/9$ cannot be constructed. 

\begin{vexercise}\label{ex10.2}
Evaluate the complex number $(\cos\phi+\imag\sin\phi)^3$ in two different ways: using
the binomial theorem and De Moivre's theorem. By equating real parts, deduce that
$$
\cos3\phi=4\cos^3\phi-3\cos\phi.
$$ 
Derive similar expressions for $\cos5\phi$ and $\cos7\phi$. 
\end{vexercise}

Exercise \ref{ex7.40} gives that the angle $\pi/9$ is
constructible precisely when the complex  
number $\cos\pi/9$ can be constructed, for which it is necessary in turn that the
degree of the extension $\Q\subseteq \Q(\cos\pi/9)$ be a power of two. Exercise
\ref{ex10.2} with $\phi=\pi/9$ gives
$$
\cos\frac{\pi}{3}=4\cos^3\frac{\pi}{9}-3\cos\frac{\pi}{9},
\text{ hence, }
1=8\cos^3\frac{\pi}{9}-6\cos\frac{\pi}{9}.
$$
Thus, if  $u=2\cos(\pi/9)$, then $u^3-3u-1=0$. This polynomial is irreducible over $\Q$
by the reduction test (with $p=2$) so it is the minimum polynomial over $\Q$ of
$2\cos(\pi/9)$. The extension $\Q\subset
\Q(2\cos(\pi/9))=\Q(\cos(\pi/9))$ thus has 
degree three, and so the angle $\pi/9$ cannot be constructed.

We will be able to say more about which angles of the form $\pi/n$ can be constructed 
in Section \ref{galois.corresapps}.

\begin{vexercise}\label{ex10.50}
\hspace{1em}\begin{enumerate}
\item Can an angle of $40^\circ$ be constructed?
\item Assuming $72^\circ$ is constructible,
what about $24^\circ$ and $8^\circ$?
\item Can $72^\circ$ be constructed? (\emph{hint}: Section \ref{lect1})
\end{enumerate}
\end{vexercise}

\subsection*{Further Exercises for Section \thesection}

\begin{vexercise}\label{platonic_volume}
The octahedron, dodecahedron and icosahedron are
three of the five Platonic solids (the other two are the tetrahedron
and the cube). See Figure \ref{fig:constructions2:fig30}.
The volume of each is given by the formula, where $x$ is the length of
any edge.
Show that in each case, there is no general method, using a ruler and compass, to
construct a new solid
from a given one, and having {\em twice\/} the volume. 
\end{vexercise}

\begin{figure}
  \centering
\begin{pspicture}(0,0)(14,7)
\rput(0,-0.5){
\rput(2,4.5){\BoxedEPSF{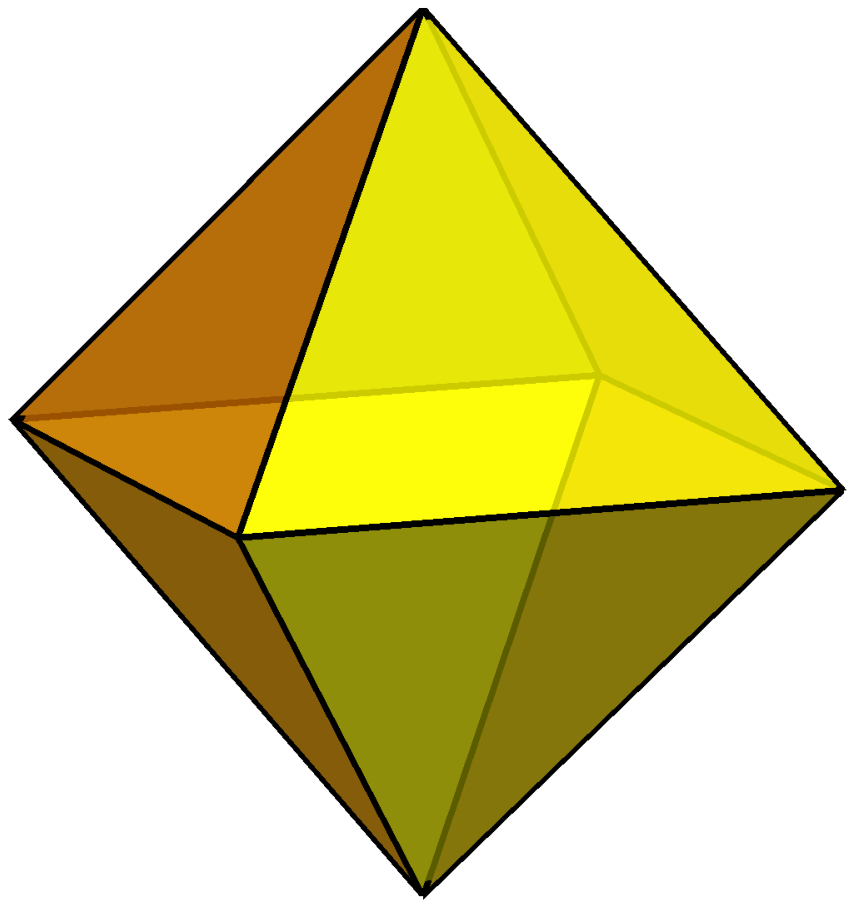 scaled 425}}
\rput(2,1.5){${\ds V_O=\frac{x^3\sqrt{2}}{3}}$}
\rput(7,4.5){\BoxedEPSF{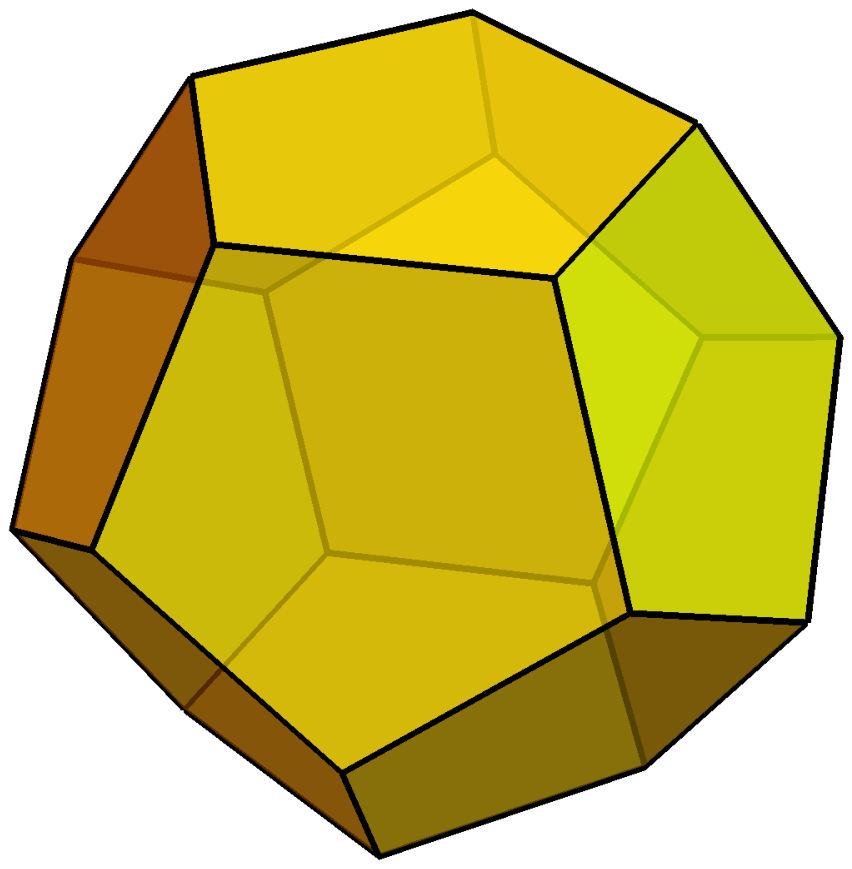 scaled 425}}
\rput(7,1.5){${\ds V_D=\frac{x^3(15+7\sqrt{5})}{4}}$}
\rput(12,4.5){\BoxedEPSF{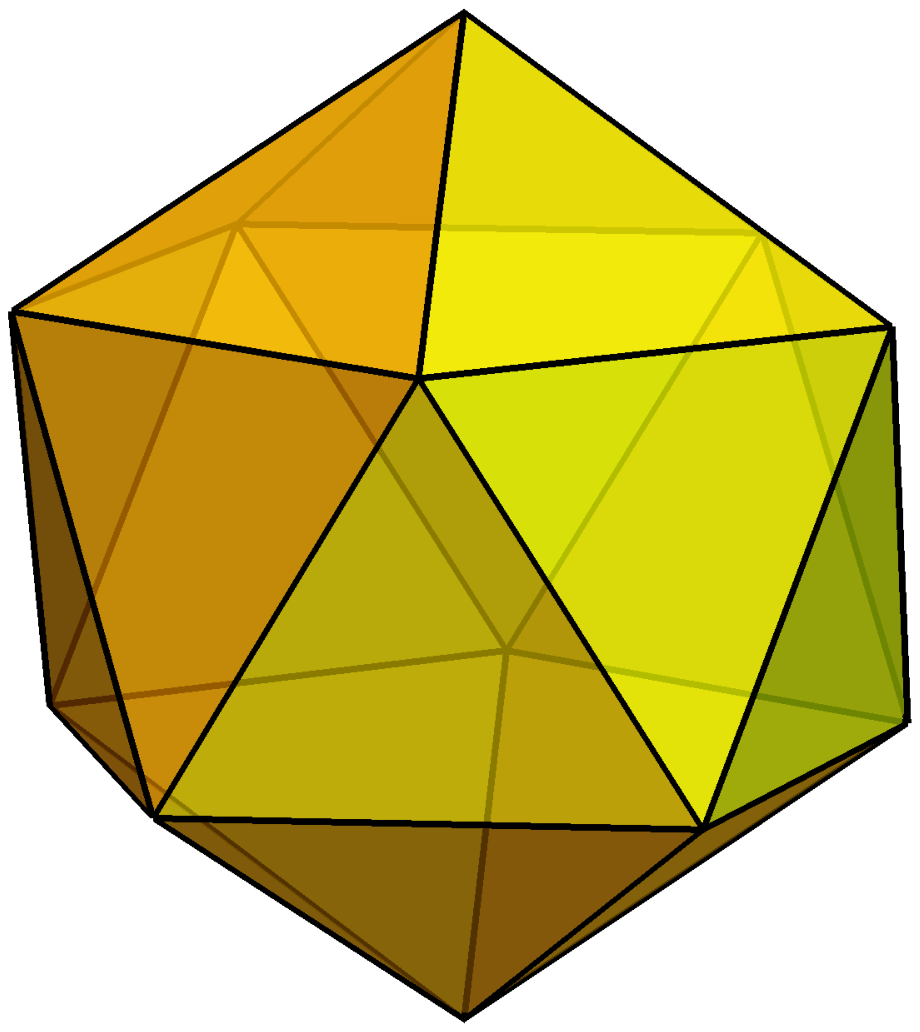 scaled 375}}
\rput(12,1.5){${\ds V_I=\frac{5x^3(3+\sqrt{5})}{12}}$}
}
\end{pspicture}
\caption{The octahedron, dodecahedron and icosahedron, and their volumes.}
  \label{fig:constructions2:fig30}
\end{figure}

\begin{vexercise}
Let $S_O,S_D$ and $S_I$ be the surface areas of the three Platonic
solids of Exercise \ref{platonic_volume}. If,
$$
S_O=2x^2\kern-2pt\sqrt{3},S_D=3x^2\kern-2pt\sqrt{5(5+2\kern-2pt\sqrt{5})}\text{ and }
S_I=5x^2\kern-2pt\sqrt{3},
$$
determine whether or not a solid can be constructed from a given one
with twice the surface area.
\end{vexercise}

\begin{vexercise}\label{angle_quinsect}
  \begin{enumerate}
  \item Using the
identity
$\cos 5\theta=16\cos^5\theta-20\cos^3\theta+5\cos\theta.$
Show that is is impossible, using a ruler and compass, to {\em
quinsect\/} (that is, divide into $5$ equal parts) any angle $\psi$
that satisfies,
$$ 
\cos\psi=\frac{5}{6}
$$
  \item Using the identity,
$\cos7\theta=64\cos^7\theta-112\cos^5\theta+56\cos^3\theta-7\cos\theta$
show that it is impossible, using ruler and compass, to {\em septsect\/}
(that is, divide into {\em seven\/} equal parts) any angle $\varphi$ such
that
$$
\cos\varphi=\frac{7}{8}
$$
  \end{enumerate}
\end{vexercise}


\section{Groups I: Soluble Groups and Simple Groups}
\label{groups.stuff}

This section contains miscellaneous but important reminders from group
theory. Not all our groups will be Abelian, so we return to writing
the group operation as juxtaposition and writing ``$\id$'' for the group
identity. 

\paragraph{\hspace*{-0.3cm}}
A {\em permutation\/} of a set $X$ is a bijection $X\rightarrow X$. Usually
we are interested in the case where $X$ is finite, say $X=\{1,2,\ldots,n\}$,
so a permutation is just a rearrangement of these numbers. Permutations
are most compactly written using cycle notation
$$
(a_{11},a_{12},\ldots,a_{1n_1})(a_{21},a_{22},\ldots,a_{2n_2})\ldots(a_{k1},a_{k2},
\ldots,a_{kn_k})
$$
where the $a_{ij}$ are elements of $\{1,2,\ldots,n\}$. 
Each $(b_1,b_2,\ldots,b_k)$ means that the $b_i$ are permuted in a cycle:
$$
\begin{pspicture}(0,0)(4,4)
\rput(2,2){\BoxedEPSF{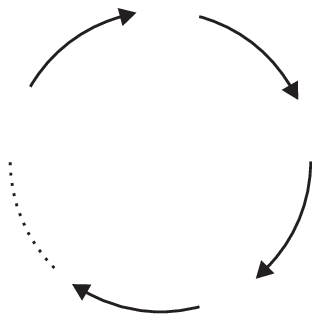 scaled 1000}}
\rput(2.1,3.5){$b_1$}
\rput(3.5,2.35){$b_2$}
\rput(2.75,0.7){$b_3$}
\rput(0.5,2.4){$b_k$}
\end{pspicture}
$$
Cycles are composed from right to left,
eg:
$(1,2)(1,2,4,3)(1,3)(2,4)=(1,2,3)$. In this way a permutation can be 
written as a product of disjoint cycles. 
The set of all permutations of $X$ forms a group under composition of bijections called
the {\em symmetric group\/} $S_{\kern-.3mm X}$, or $S_{\kern-.3mm n}$ if $X=\{1,2,\ldots,n\}$.

\paragraph{\hspace*{-0.3cm}}
A permutation where just two things are 
interchanged, and everything else is left fixed, is called a {\em
  transposition\/} or \emph{swap\/} $(a,b)$. Any permutation can be written as a
composition of transpositions, for example:
$$
(1,2,3)=(1,3)(1,2)=(1,2)(2,3)
\text{ and }
(a_{1},a_{2},\ldots,a_{k})=(a_1,a_k)(a_1,a_{k-1})\ldots (a_1,a_3)(a_1,a_2).
$$	
There will be many such expressions, but they all involve an even
number of transpositions or all involve an odd number of them. 

We can thus call a permutation {\em even\/} if it can be decomposed
into an even number of transpositions, and {\em odd\/} otherwise. The
even permutations in $S_{\kern-.3mm n}$ form a subgroup called the  {\em Alternating
group\/} $A_n$.

\begin{vexercise}\label{ex11.1}
Show that $A_n$ is indeed a group comprising exactly half of the elements of $S_{\kern-.3mm n}$. Show
that the odd elements in $S_{\kern-.3mm n}$ {\em do not\/} form a subgroup.
\end{vexercise}

\begin{vexercise}\label{ex11.2}
Recall that the {\em order\/} of an element $g$ of a group $G$ is the least $n$ such that
$g^n=\id$. 
Show that if $g,h$ are elements such that $gh=hg$ 
then $(gh)^n=g^nh^n$. 
If in addition the order of $g$ is $n$ and the order of $h$ is $m$ with
$\gcd(n,m)=1$, then the order of $gh$ is the lowest common multiple of $n$ and $m$.
\end{vexercise}

\begin{vexercise}\label{ex11.2a}
Let $G$ be a finite Abelian group, and let $1=m_1,m_2,\ldots,m_\ell$ be a list 
of all the possible orders of elements of $G$. Show that there exists an element
whose order is the lowest common multiple of the $m_i$ [\emph{hint}: let $g_i$ be
an element of order $m_i$ and use Exercise \ref{ex11.2} to show that there
are $k_1,\ldots,k_\ell$ with $g_1^{k_1}\cdots g_\ell^{k_\ell}$ the element we seek].
\end{vexercise}

\paragraph{\hspace*{-0.3cm}}
If $G$ is a group and $\{g_1,g_2,\ldots,g_n\}$ are elements of $G$, then we say that the $g_i$
{\em generate\/} $G$ when every element $g\in G$ can be obtained as a product 
$$
g=g_{i_1}^{\pm 1}g_{i_2}^{\pm 1}\ldots g_{i_k}^{\pm 1},
$$
of the $g_i$ and their inverses. Write $G=\lg g_1,g_2,\ldots,g_n\rg$.

\paragraph{\hspace*{-0.3cm}}
We find generators for the symmetric and alternating groups. We have already seen that
the transpositions $(a,b)$ generate $S_{\kern-.3mm n}$, for any permutation can be written as a 
product 
$$
(a_{1},a_{2},\ldots,a_{k})=(a_1,a_k)
(a_1,a_{k-1})\ldots (a_1,a_3)(a_1,a_2).
$$	
The transpositions $(a,b)$ can in turn be expressed in
terms of just some of them: when $a<b$ we have
$$
(a,b)=(a,a+1)(a+1,a+2)\ldots(b-2,b-1)(b-1,b)\ldots(a+1,a+2)(a,a+1)
$$
as can be seen by considering the picture:
$$
\begin{pspicture}(0,0)(14,3)
\rput(7,1.5){\BoxedEPSF{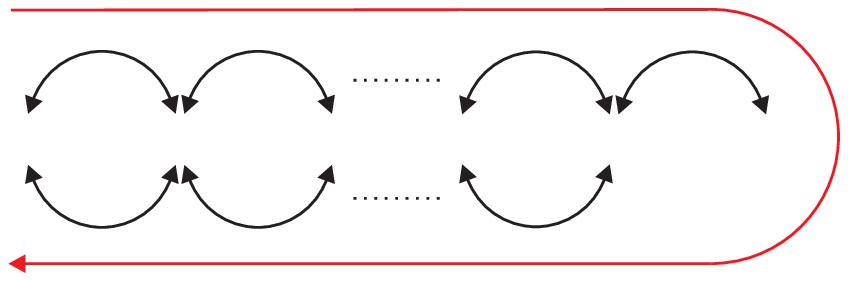 scaled 1000}}
\rput(3,1.5){$a$}
\rput(4.5,1.5){$a+1$}
\rput(6,1.5){$a+2$}
\rput(7.5,1.5){$b-2$}
\rput(9,1.5){$b-1$}
\rput(10.5,1.5){$b$}
\end{pspicture}
$$
and doing the swaps in the order indicated. 
Any number strictly in between $a$ and $b$ moves one place to the right and then one place
to the left, with net effect that it remains stationary. The number $a$ is moved 
to $b$ by the top swaps, but then stays there. Similarly $b$ stays put
for all but the last of the top swaps
and then is moved to $a$ by the bottom swaps.

Any permutation can thus be written as a product of swaps of the form
$(a,a+1)$. 
Even these 
transpositions can be further reduced, by transferring $a$ and $a+1$ to the points
$1$ and $2$, swapping $1$ and $2$ and transferring the answer back to
$a$ and $a+1$. Indeed, if $\tau=(1,2,\ldots,n)$ then doing the
permutations in the order indicated in the picture:
$$
\begin{pspicture}(0,0)(14,3)
\rput(7,1.5){\BoxedEPSF{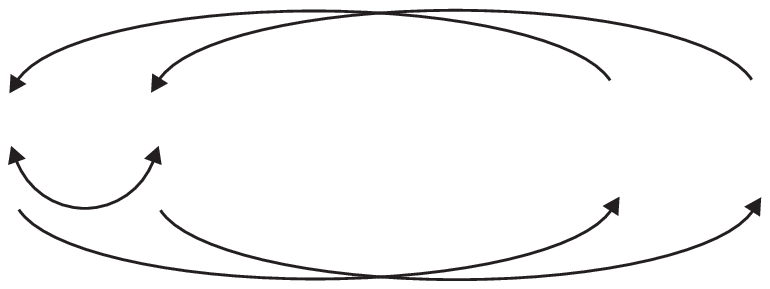 scaled 1000}}
\rput(3.2,1.75){$1$}\rput(4.65,1.75){$2$}
\rput(9.35,1.75){$a$}\rput(10.8,1.75){$a+1$}
\rput(0,-0.5){
\pscircle[linecolor=white,fillstyle=solid,fillcolor=red](7,3){.25}
\rput(7,3){{\white{\bf 1}}}
}
\rput(-3,-1.5){
\pscircle[linecolor=white,fillstyle=solid,fillcolor=red](7,3){.25}
\rput(7,3){{\white{\bf 2}}}
}
\rput(0,-2.5){
\pscircle[linecolor=white,fillstyle=solid,fillcolor=red](7,3){.25}
\rput(7,3){{\white{\bf 3}}}
}
\end{pspicture}
$$
shows that $(a,a+1)=\tau^{a-1}(1,2)\tau^{1-a}$.
The conclusion is that $S_{\kern-.3mm n}$ is generated by just two permutations, namely
$(1,2)$ and $(1,2,\ldots,n)$.

\begin{vexercise}
Show that the Alternating group is generated by the permutations of the form $(a,b,c)$.
Show that just the $3$-cycles of the form $(1,2,a)$ will suffice.
\end{vexercise}

\paragraph{\hspace*{-0.3cm}}
Lagrange's theorem says that if $G$ is a finite group and $H$ a subgroup of $G$, then
the order $|H|$ of $H$ divides the order $|G|$ of $G$. The converse, that if a subset of 
a group has size dividing the order of the  group then it is a subgroup, is false. 

\begin{vexercise}
By considering the Alternating group $A_4$, justify this statement.
\end{vexercise}

\begin{vexercise}\label{cyclicgroup.subgroups}
Show that if $G$ is a cyclic group, then the converse to Lagrange's theorem {\em is\/}
true, ie: if $G$ has order $n$ and $k$ divides $n$ then $G$ has a subgroup of order $k$.
\end{vexercise}

\begin{vexercise}
Use Lagrange's Theorem to show that if a group $G$ has order a prime number $p$, then $G$
is isomorphic to a cyclic group. Thus any two groups of order $p$ are
isomorphic.
\end{vexercise}

There are partial converses to Lagrange's Theorem:

\begin{theorem}[Cauchy]
Let $G$ be a finite group and $p$ a prime dividing the order of $G$.
Then $G$ has a subgroup of order $p$. 
\end{theorem}

Indeed, one can show that $G$ contains an element $g$ of order $p$,
with the subgroup being the elements $\{g,g^2,\ldots,g^p=\id\}$. 

\begin{theorem}[Sylow's 1st]
Let $G$ be a finite group of order $p^k m$, where $p$ does not
divide $m$. Then $G$ has a subgroup of 
order $p^k$. 
\end{theorem}

\paragraph{\hspace*{-0.3cm}}
It will be useful to consider all the subgroups of a group at once, rather than just
one at a time. 

\begin{definition}[lattice of subgroups]
The subgroup lattice is a 
diagram depicting all the subgroups of $G$ and the inclusions between them. If
$H_1,H_2$ are subgroups of $G$ with $H_1\subseteq H_2$ they appear in
the diagram like so:
$$
\begin{pspicture}(0,0)(2,2)
\rput(1,.3){$H_1$}
\rput(1,1.7){$H_2$}
\psline(1,.6)(1,1.4)
\end{pspicture}
$$
At the very base of the diagram is the trivial subgroup $\{\id\}$ and at the apex is the other
trivial subgroup, namely $G$ itself. Denote the lattice by $\LL(G)$.
\end{definition}

\parshape=4 0pt\hsize 0pt\hsize 0pt.75\hsize 0pt.75\hsize  
For example, the group of symmetries of an equilateral triangle has
elements
$$
\{\id,r,r^2,s,rs,r^2s\}
$$
where $r$ is a rotation counter-clockwise through $\frac{1}{3}$ of a
turn (we called it $ts$ in Section \ref{lect1}) and $s$ is the reflection in
the horizontal axis.
\vadjust{\hfill\smash{\lower 30pt
\llap{
\begin{pspicture}(0,0)(3,2)
\rput(1,0.25){
\uput[0]{270}(-.1,2){\pstriangle[fillstyle=solid,fillcolor=lightgray](1,0)(2,1.73)}
\rput(-0.7,1){${\red s}$}\rput(1.3,2.3){${\red r}$}
\psline[linecolor=red](-0.5,1)(2,1)
\rput{-120}(-0.1,1.75){\pscurve[linecolor=red]{<-}(-0.5,0)(-1,1)(-0.5,2)}
}
\end{pspicture}
}}}\ignorespaces

\parshape=6 0pt.75\hsize 0pt.75\hsize  0pt.75\hsize  
0pt\hsize 0pt\hsize 0pt\hsize
The subgroup lattice $\LL(G)$ is on the left in Figure
\ref{fig:groups1:subgroup_lattices}. 
I'll leave you to see that they are all subgroups, so it remains to
see that we have all of them.
Suppose first that $H$ is a subgroup containing $r$. Then it must
contain all the powers
$\{\id,r,r^2\}$
of $r$, and so  $3\leq |H|\leq 6$. By Lagrange's Theorem $|H|$ divides
$6$, so we have $|H|=3$ or $6$, giving that $H$ must be $\{\id,r,r^2\}$ or
all of $G$. This describes all the subgroups that contain 
$r$, and the same argument -- and conclusion -- applies to the subgroups
containing $r^2$. 

This leaves the subgroups containing one
of the reflections $s,rs,r^2s$ but not $r$ or
$r^2$. If $H$ is a subgroup 
containing $s$, then as it also contains $\id$,
and by Lagrange, it must have order $2,3$ or $6$. The first
possibility gives $H=\{\id,s\}$ and the last gives $H=G$. 
On the other hand, 
to have order $3$, the subgroup $H$ must also contain one
of $rs$ or $r^2s$. In the first case it also contains $rss=
r$, a contradiction.  Similarly $H$ cannot contain $r^2s$, so there is
no subgroup $H$ containing $s$ apart from $\{\id,s\}$ and $G$ itself. 
Similarly for subgroups containing $rs$ or $r^2s$. Thus the lattice
$\LL(G)$ is indeed as shown in Figure \ref{fig:groups1:subgroup_lattices}.

The right part of Figure \ref{fig:groups1:subgroup_lattices} gives the
subgroup lattice of the symmetry group of a square. I'll leave the
details to you. 

\begin{figure}
  \centering
\begin{pspicture}(0,0)(14,6)
\rput(3.5,3){\BoxedEPSF{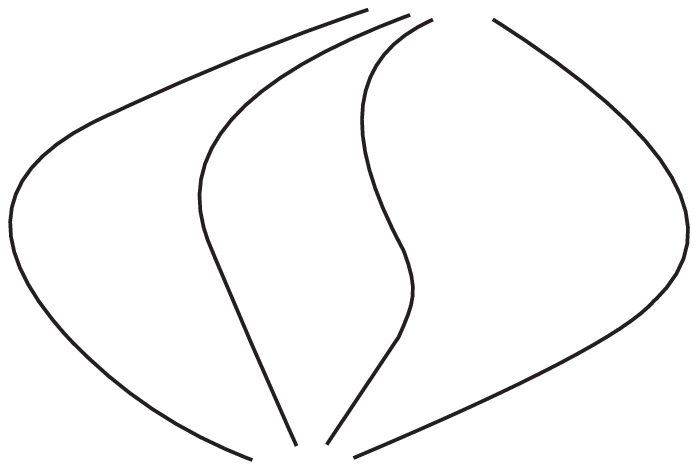 scaled 900}}
\rput(3.1,0.75){$\{\id\}$}
\rput*(0.5,2.5){$\{\id,s\}$}
\rput*(2.25,2.5){$\{\id,rs\}$}
\rput*(4,2.5){$\{\id,r^2s\}$}
\rput*(6,3.7){$\{\id,r,r^2\}$}
\rput(4.3,5.2){$G$}
%
%
\rput(7,0){
\rput(4,3){\BoxedEPSF{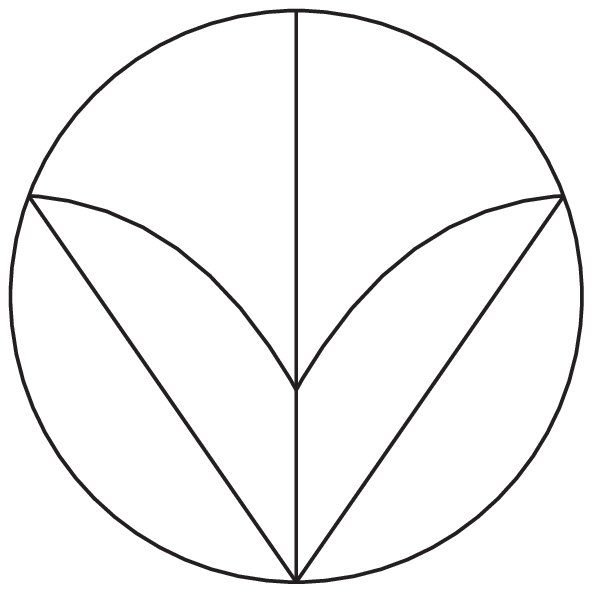 scaled 900}}
\rput*(4,5.6){$G$}
\rput*(1.5,3.8){$\{\id,r^2,s,r^2s\}$}
\rput*(4,3.8){$\{\id,r,r^2,r^3\}$}
\rput*(6.45,3.8){$\{\id,r^2,rs,r^3s\}$}
\rput*(1.5,2.1){$\{\id,r^2s\}$}
\rput*(2.7,2.1){$\{\id,s\}$}
\rput*(4,2.1){$\{\id,r^2\}$}
\rput*(5.2,2.1){$\{\id,r^3s\}$}
\rput*(6.5,2.1){$\{\id,rs\}$}
\rput*(4,0.4){$\{\id\}$}
}
\end{pspicture}
\caption{Subgroup lattices of the group of symmetries of a triangle
  \emph{(left)} and square \emph{(right)}.}
  \label{fig:groups1:subgroup_lattices}
\end{figure}

\paragraph{\hspace*{-0.3cm}}
If $G$ is a finite group and 
$$
\{\id\}=H_0\lhd H_1\lhd \cdots \lhd H_{n-1}\lhd H_n=G,
$$
is a nested sequence of subgroups with each $H_i$ normal in $H_{i+1}$ and
the quotients
$$
H_1/H_0, H_2/H_1,\ldots,H_n/H_{n-1}
$$
Abelian, then $G$ is said to be {\em soluble\/}.

\paragraph{\hspace*{-0.3cm}}
\label{groups1:abelian_are_soluble}
If $G$ is an Abelian group, then we have the sequence
$$
\{\id\}\lhd G,
$$
with the single quotient $G/\{\id\}\cong G$, an Abelian group. Thus Abelian groups
are soluble.

\paragraph{\hspace*{-0.3cm}}
\label{groups1:dihedral_are_soluble}
For another example let $G$ be the symmetries, both rotations and reflections, of a regular 
$n$-gon in the plane. In the sequence:
$$
\{\id\}\lhd \{\text{rotations}\}\lhd G
$$
the normality of the subgroup of rotations in $G$ follows from the fact that the rotations
comprise half of all the symmetries and Exercise \ref{ex11.1}.
Moreover, the rotations are isomorphic to the cyclic group $\Z_n$, and so
the quotients in this sequence are 
$$
\{\text{rotations}\}/\{\id\}\cong
\{\text{rotations}\}\cong\Z_n\text{ and } 
G/\{\text{rotations}\}\cong\Z_2,
$$
both Abelian groups. 

\begin{vexercise}\label{subgroups.solublegroups1}
It turns out, although for slightly technical reasons, that a subgroup
of a soluble group is also soluble. This exercise and the next
demonstrate why. 
Let $G$ be a group, 
$H$ a subgroup and $N$ a normal subgroup. Let
$$
NH=\{nh\,|\,n\in N,h\in H\}.
$$
\begin{enumerate}
\item Define a map $\varphi: H\rightarrow NH/N$ by $\varphi(h)=Nh$. Show that 
$\varphi$ is an onto  homomorphism with kernel $N\cap H$.
\item Use the first isomorphism theorem for groups to deduce that $H/H\cap N$
is isomorphic to $NH/H$.
\end{enumerate}
(This is called the {\em second isomorphism\/} or {\em diamond isomorphism\/} theorem.
Why diamond? Draw a picture of all the subgroups--the theorem says that two ``sides'' of a
diamond are isomorphic).
\end{vexercise}

\begin{vexercise}\label{subgroups.solublegroups2}
Let $G$ be a soluble group via the series,
$$
\{\id\}=H_0\lhd H_1\lhd \cdots \lhd H_{n-1}\lhd H_n=G,
$$
and let $K$ be a subgroup of $G$. 
Show that
$$
\{\id\}=H_0\cap K\lhd H_1\cap K\lhd \cdots \lhd H_{n-1}\cap K\lhd H_n\cap K=K,
$$
is a series with Abelian quotients for $K$, and hence $K$ is also a soluble group.
\end{vexercise}

\paragraph{\hspace*{-0.3cm}}
The antithesis of the soluble groups are the {\em simple\/} ones: 
groups $G$ whose only normal subgroups are the trivial subgroup $\{\id\}$ and the whole group
$G$. 

Whenever we have a normal subgroup we can form a quotient.
A group is thus simple when its only quotients are itself $G/\{\id\}\cong G$
and the trivial group $G/G\cong \{\id\}$. Thus simple groups are analogous to 
prime numbers:
integers whose only 
quotients
are themselves $p/1=p$ and $p/p=1$.

If $G$ is non-Abelian and simple, then $G$ {\em cannot\/} be soluble. For, the only sequence
of normal subgroups that $G$ can have is 
$$
\{\id\}\lhd G,
$$
and as $G$ is non-Abelian the quotients of this sequence are non-Abelian. Thus, non-Abelian
simple groups provide a ready source of non-soluble groups.

\begin{table}
\begin{center}
\begin{tabular}{ll}
\hline
Symbol & Name\\
\hline
&\\
$\Z_p$ & cyclic\\
$A_n$  & alternating\\
&\\
\hline
notes: $p$ is a prime;\\ 
$n\not= 1,2,4$\\
\end{tabular}
\caption{The first two families of simple groups}\label{simple_groups1}
\end{center}
\end{table}

\paragraph{\hspace*{-0.3cm}}
Amazingly, there is a complete list of the finite simple groups, compiled over 
approximately 150 years. The list is contained
in Tables \ref{simple_groups1}-\ref{simple_groups3}.

\begin{vexercise}
Show that if $p$ is a prime number then the cyclic group $\Z_p$ has no non-trivial subgroups
whatsoever, and so is a simple group.
\end{vexercise}

\paragraph{\hspace*{-0.3cm}}
In Table \ref{simple_groups1} we see that the Alternating groups $A_n$ are 
simple for $n\not= 1,2$ or $4$.
In particular these Alternating groups are not soluble, and as any
subgroup of a soluble group is soluble, 
any group containing the Alternating group will also not be soluble. Thus,
the symmetric groups $S_{\kern-.3mm n}$ are not soluble if $n\not=
1,2$ or $4$. 

\begin{table}
\begin{center}
\begin{tabular}{lll}
\hline
Symbol & Name & Discovered \\
\hline
&&\\
$\psl_n\F_q$ & projective  & 1870\\
$\psp_{2n}\F_q$ & simplectic & 1870\\
$\text{P}\Omega^+_{2n}$ & orthogonal & 1870\\
$\text{P}\Omega_{2n+1}$ & orthogonal & 1870\\
$E_6(q)$ & Chevalley  & 1955\\
$E_7(q)$ & Chevalley  & 1955 \\
$E_8(q)$ & Chevalley  & 1955 \\
$F_4(q)$ & Chevalley  & 1955 \\
$G_2(q)$ & Chevalley  & 1955 \\
$^2 A_n(q^2)=\psu_n\F_{q^2}$ & unitary or twisted Chevalley  & 1870 \\
$^2D_n(q^2)=\text{P}\Omega^-_{2n}$ & orthogonal or twisted Chevalley  & 1870 \\
$^2E_6(q^2)$ & twisted Chevalley  & c. 1960 \\
$^3D_4(q^3)$ & twisted Chevalley  & c. 1960 \\
$^2B_2(2^{2e+1})$ & Suzuki  & 1960 \\
$^2G_2(2^{2e+1})$ & Ree  & 1961 \\
$^2F_4(2^{2e+1})$ & Ree  & 1961 \\
&&\\
\hline
notes: $n$ and $e$ are $\in\Z$&There are some restrictions on $n$&\\
$q$ is a prime power;&and $q$, left off here for clarity.&\\
\end{tabular}
\caption{The simple groups of Lie type}\label{simple_groups2}
\end{center}
\end{table}

\paragraph{\hspace*{-0.3cm}}
Tables \ref{simple_groups2} and \ref{simple_groups3} list the really interesting simple groups.
The groups of Lie type are roughly speaking groups of matrices whose entries come from finite fields.
We have already seen that if $q=p^n$ a prime power, then there is a
field $\F_q$ with $q=p^n$ elements. 
The group $\sl_n\F_q$ consists of the $n\times n$ matrices having
determinant $1$ and with entries from this field and the
usual matrix multiplication. This group is not simple as 
$$
N=\{\lambda I_n\,|\,\lambda\in\F_q\},
$$
is a normal subgroup. But it turns out
that the quotient group,
$$
\sl_n\F_q/N,
$$
is a simple group. It is denoted $\psl_n\F_q$, and 
called the $n$-dimensional projective special linear group over $\F_q$.
The remaining groups in Table \ref{simple_groups2} come from more complicated constructions.

Table \ref{simple_groups3} lists groups that don't fall into any of
the other categories. 
For this reason they are called the ``sporadic'' simple groups. They arise from
various -- often quite complicated -- constructions that are  beyond the reach of these
notes. The most interesting of them is the largest one -- the Monster simple group (which
actually contains quite a few of the others as subgroups). 

In any case, the simple groups in Tables
\ref{simple_groups2} and \ref{simple_groups3} 
are all non-Abelian, hence provide more examples of non-soluble
groups.

\begin{table}
\begin{center}
\begin{tabular}{llll}
\hline
Symbol & Name & Discovered & Order\\
\hline
&&&\\
1. \makebox[0pt][l]{{\em First generation of the Happy Family\/}}.\\
$M_{11}$ & Mathieu & 1861 & $2^4\,3^2\,5\,11$\\
$M_{12}$ & Mathieu & 1861 & $2^4\,3^3\,5\,11$\\
$M_{22}$ & Mathieu & 1873 & $2^7\,3^2\,5\,7\,11$\\
$M_{23}$ & Mathieu & 1873 & $2^7\,3^2\,5\,7\,11\,23$\\
$M_{24}$ & Mathieu & 1873 & $2^{10}\,3^3\,5\,7\,11\,23$\\
2. \makebox[0pt][l]{{\em Second generation of the Happy Family\/}}.\\
HJ & Hall-Janko & 1968 & $2^7\,3^3\,5^2\,7$\\
HiS & Higman-Sims & 1968 & $2^9\,3^2\,5^3\,7\,11$\\
McL & McLaughlin & 1969 & $2^7\,3^6\,5^3\,7\,11$\\
Suz & Suzuki & 1969 & $2^{13}3^7\,5^2\,7\,11\,13$\\
$Co_1$ & Conway & 1969 & $2^{21}\,3^9\,5^4\,7^2\,11\,13\,23$\\
$Co_2$ & Conway & 1969? & $2^{18}\,3^6\,5^3\,7\,11\,23$\\
$Co_3$ & Conway & 1969? & $2^{10}\,3^7\,5^3\,7\,11\,23$\\
3. \makebox[0pt][l]{{\em Third generation of the Happy Family\/}}.\\
He & Held & 1968 & $2^{10}\,3^2\,5^2\,7^3\,17$\\
$Fi_{22}$ & Fischer & 1968 & $2^{17}\,3^9\,5^2\,7\,11\,13$\\
$Fi_{23}$ & Fischer & 1968 & $2^{18}\,3^{13}\,5^2\,7\,11\,13\,17\,23$\\
$Fi_{24}$ & Fischer & 1968 & $2^{21}\,3^{16}\,5^2\,7^3\,11\,13\,17\,23\,29$\\
$F_5$  & Harada-Norton & 1973 & $2^{14}\,3^6\,5^6\,7\,11\,19$\\
$F_3$ & Thompson & 1973 & $2^{15}\,3^{10}\,5^3\,7^2\,13\,19\,31$\\
$F_2$ & Fischer or ``Baby Monster'' & 1973 & $2^{41}\,3^{13}\,5^6\,7^2\,11\,13\,17\,19\,23\,47$\\
$\M$ & Fischer-Griess or ``Friendly Giant'' or ``Monster''  & 1973 & $\approx 10^{55}
$\\
4. {\em The Pariahs\/}.\\
$J_1$ & Janko & 1965 & $2^3\,5\,7\,11\,19$\\
$J_3$ & Janko & 1968 & $2^7\,3^5\,5\,17\,19$\\
$J_4$ & Janko & 1975 & $2^{21}\,3^3\,5\,7\,11^3\,23\,29\,31\,37\,43$\\
Ly & Lyons & 1969 & $2^8\,3^7\,5^6\,7\,11\,31\,37\,67$\\
Ru & Rudvalis & 1972 & $2^{14}\,3^3\,5^3\,7\,13\,29$\\
O'N & O'Nan & 1973 & $2^9\,3^4\,5\,7^3\,11\,19\,31$\\
\hline
\end{tabular}
\caption{The sporadic simple groups}\label{simple_groups3}
\end{center}
\end{table}

\subsection*{Further Exercises for Section \thesection}

\begin{vexercise}
Show that any subgroup of an abelian group is normal.
\end{vexercise}

\begin{vexercise}
Let $n$ be a positive integer that is not prime.
Show that the cyclic group $\Z_n$ is not simple.
\end{vexercise}

\begin{vexercise}
Show that $A_2$ and $A_4$ are not simple groups, but $A_3$ is.
\end{vexercise}

\begin{vexercise}\label{ex11.1}
Let $G$ be a group and $H$ a subgroup such that $H$ has exactly two cosets in $G$. 
Let $C_2$ be the group with elements $\{-1,1\}$ and operation the usual multiplication.
Define a map $f:G\rightarrow C_2$ by 
$$
f(g)=\left\{
\begin{array}{ll}
1&g\in H\\
-1&g\not\in H
\end{array}\right.
$$
Show that $f$ is a homomorphism.
Deduce that $H$ is a normal subgroup.
\end{vexercise}

\begin{vexercise}
Consider the group of symmetries (rotations and reflections) of a regular $n$-sided polygon 
for $n\geq 3$. Show that this is
not a simple group.
\end{vexercise}

\begin{vexercise}
Show that $S_{\kern-.3mm 2}$ is simple but $S_{\kern-.3mm n}$ is not
for $n\geq 3$. Show that $A_n$ has no subgroups of 
index $2$ for $n\geq 5$.
\end{vexercise}

\begin{vexercise}
Show that if $G$ is abelian and simple then it is cyclic. Deduce that if $G$ is simple and not
isomorphic to $\Z_p$ then $G$ is non-Abelian.
\end{vexercise}

\begin{vexercise}
For each of the following groups $G$, draw the subgroup lattice
$\LL(G)$:
\begin{enumerate}
\item $G=$ the group of symmetries of a pentagon or hexagon.
\item $G=$ the cyclic group $\{1,g,g^2,\ldots,g^{n-1}\}$ where $g^n=1$.  
\end{enumerate}
\end{vexercise}


\section{Groups II: Symmetries of Fields}
\label{galois.groups}

We are finally able to bring symmetry into the 
solutions of polynomial equations.

\begin{definition}[automorphism or symmetry of a field]
An automorphism of a field $F$ is an isomorphism $\ss:F\rightarrow F$, ie:
a bijective map from $F$ to $F$ such that $\ss(a+b)=\ss(a)+\ss(b)$ and $\ss(ab)=\ss(a)\ss(b)$
for all $a,b\in F$. 
\end{definition}

We remarked in Section \ref{lect4} that an automorphism is a
relabeling of the elements using different symbols
but keeping the algebra the same.
So it is a way of picking the 
field up and
placing it back down without changing the way it essentially looks. 

\begin{vexercise}
Show that if $\ss$ is an automorphism of the field $F$ then $\ss(0)=0$ and $\ss(1)=1$.
\end{vexercise}

\paragraph{\hspace*{-0.3cm}}
A familiar example is complex conjugation: 
$\ss:z\mapsto \ov{z}$ is an automorphism of $\C$, 
since
$$
\ov{z+w}=\ov{z}+\ov{w}\text{ and }\ov{zw}=\ov{z}\,\ov{w},
$$
with conjugation a bijection $\C\rightarrow \C$. This symmetry captures the idea that
from an algebraic point of view, we could have just as easily adjoined $-\imag$ to $\R$,
rather than $\imag$, to obtain the complex numbers -- they look the same upside down as right 
side up!

We will see at the end of this section that if a non-trivial  automorphism of $\C$ fixes pointwise
the real numbers, then it must be complex conjugation. If we drop the requirement that $\R$
be fixed then there may be more possibilities: if we only insist
that $\ss$ fix $\Q$ pointwise then 
there are infinitely many possibilities.

\begin{vexercise}
\label{exercise:groups2:conjugation}
Let $f\in\Q[x]$ with roots $\{\aa_1,\ldots,\aa_d\}\in\C$. Show that
complex conjugation $z\mapsto\ov{z}$ is an automorphism of the
splitting field $\Q(\aa_1,\ldots,\aa_d)$. Is it always non-trivial?
\end{vexercise}

\begin{vexercise}
Show that $a+b\imag
\mapsto -a+b\imag$ is \emph{not\/} an automorphism of $\C$. Show that if
$\ell$ is a line through $0$ in $\C$, then reflecting in $\ell$ is an
automorphism only when $\ell$ is the real axis.
\end{vexercise}

\paragraph{\hspace*{-0.3cm}}
We saw in Section \ref{lect4} that every field $F$ has a prime subfield 
isomorphic to either $\F_p$ or $\Q$. 
The elements have the form:
$$
\frac{\overbrace{1+1+\cdots +1}^{m\text{ times}}}
{\underbrace{1+1+\cdots +1}_{n\text{ times}}}.
$$
If $\ss:F\rightarrow F$ is an automorphism of $F$ then
\begin{equation*}
\begin{split}
\ss\biggl(\frac{\overbrace{1+1+\cdots +1}^{m\text{ times}}}
{\underbrace{1+1+\cdots +1}_{n\text{ times}}}\biggr)
&=
\ss(\overbrace{1+1+\cdots +1}^{m})
\ss\biggl(\frac{1}
{\underbrace{1+1+\cdots +1}_{n}}\biggr)\\
&=
(\overbrace{\ss(1)+\ss(1)+\cdots +\ss(1)}^{m})
\biggl(\frac{1}
{\underbrace{\ss(1)+\ss(1)+\cdots +\ss(1)}_{n}}\biggr)
=
\frac{\overbrace{1+1+\cdots +1}^{m\text{ times}}}
{\underbrace{1+1+\cdots +1}_{n\text{ times}}}.
\end{split}
\end{equation*}
The elements of the prime subfield are thus fixed pointwise by the
automorphism $\ss$.

\paragraph{\hspace*{-0.3cm}}
This example suggests that we should think about symmetries in a
relative way. As symmetries normally arrange themselves into groups we define:

\begin{definition}[Galois group of an extension]
Let $F\subseteq E$ be an extension of fields. The automorphisms of the field $E$ that
fix pointwise the elements of $F$ form a group under composition, called the 
Galois group of $E$ over $F$, and denoted $\gal(E/F)$.
\end{definition}

An element $\ss$ of $\gal(E/F)$ thus has the property that $\ss(a)=a$
for all $a\in F$.

\begin{vexercise}
For $F\subset E$ fields, show that the set of automorphisms $\gal(E/F)$
of $E$ that fix $F$ pointwise do indeed form a group under composition. 
\end{vexercise}

\paragraph{\hspace*{-0.3cm}}
Consider the field $\Q(\kern-2pt\sqrt{2},\imag)$. The tower law gives
basis $\{1,\kern-2pt\sqrt{2},\imag,\kern-2pt\sqrt{2}\imag\}$ over $\Q$, so 
the elements are
$$
\Q(\kern-2pt\sqrt{2},\imag)=\{a+b\kern-2pt\sqrt{2}+c\imag+d\kern-2pt\sqrt{2}\imag\,|\,a,b,c,d\in\Q\}. 
$$
If $\ss\in\gal(\Q(\kern-2pt\sqrt{2},\imag)/\Q)$ then
\begin{align*}
  \label{eq:2}
\ss(a+b\kern-2pt\sqrt{2}+c\imag+d\kern-2pt\sqrt{2}\imag)
&=\ss(a)+\ss(b)\ss(\kern-2pt\sqrt{2})+\ss(c)\ss(\imag) 
+\ss(d)\ss(\kern-2pt\sqrt{2}\imag)\\ 
&=a+b\ss(\kern-2pt\sqrt{2})+c\ss(\imag)+d\ss(\kern-2pt\sqrt{2}\imag)  
\end{align*}
as an element of $\gal(\Q(\kern-2pt\sqrt{2},\imag)/\Q)$ fixes rational numbers by definition.
Thus $\ss$ is completely determined by its effect on the basis
$\{1,\kern-2pt\sqrt{2},\imag,\kern-2pt\sqrt{2}\imag\}$: once their images are known, then 
$\ss$ is known.

(This is no surprise. 
If $F\subseteq E$ is an extension then, among other things,
$E$ is a vector space over $F$ and $\ss\in\gal(E/F)$ is,
among other things, a linear map of vector spaces $E\rightarrow E$, 
hence completely determined by its effect on a basis.)

We can say more: we have $\ss(1)=1$ and
$\ss(\kern-2pt\sqrt{2}\imag)=\ss(\kern-2pt\sqrt{2})\ss(\imag)$. Thus 
$\ss$ is completely determined by its effect on 
$\kern-2pt\sqrt{2}$ and $\imag$, the elements adjoined to obtain $\Q(\kern-2pt\sqrt{2},\imag)$.

\paragraph{\hspace*{-0.3cm}}
This is a general fact: if $F\subseteq F(\aa_1,\aa_2,\ldots,\aa_k)=E$
and $\ss\in\gal(E/F)$, then $\ss$ is completely determined by its effect on
$\aa_1,\ldots,\aa_k$. For, if $\{\bb_1,\ldots,\bb_n\}$ is a basis for
$E$ over $F$, then $\ss$ is completely determined by its effect on the
$\bb_i$. The proof of the tower law gives
$$
\bb_i=\aa_1^{i_1}\aa_2^{i_2}\ldots \aa_k^{i_k},
$$
a product of the $\aa_j$'s, so that $\ss(\bb_i)=\ss(\aa_1)^{i_1}\ss(\aa_2)^{i_2}\ldots \ss(\aa_k)^{i_k}$
is in turn determined by the $\ss(\aa_j)$'s.

\paragraph{\hspace*{-0.3cm}}
The structure of Galois groups can sometimes be determined via ad-hoc arguments,
at least in very simple cases. For example, let $\ww$ be the primitive cube root of $1$,
$$
\ww=-\frac{1}{2}+\frac{\sqrt{3}}{2}\imag,
$$ 
and consider the extension $\Q\subset\Q(\ww)$. 
\vadjust{\hfill\smash{\lower 72pt
\llap{
\begin{pspicture}(0,0)(2.5,2)
\rput(-0.25,1.2){
\uput[0]{270}(-.1,2){\pstriangle[fillstyle=solid,fillcolor=lightgray](1,0)(2,1.73)}
\psbezier[linecolor=red]{->}(2.1,0.8)(3.1,0)(3.1,2)(2.1,1.2)
\rput(-.1,0){\psbezier[linecolor=red]{<->}(-.4,0.1)(-1,.5)(-1,1.5)(-.4,1.9)}
\rput(2,1){$1$}
\rput(-.2,0.1){$\ww^2$}
\rput(-.2,1.9){$\ww$}
\rput(0.75,-0.5){\red $\ss(\ww)=\ww^2=\overline{\ww}$}
}
\end{pspicture}
}}}\ignorespaces

\parshape=6 0pt.7\hsize 0pt.7\hsize 0pt.7\hsize 0pt.7\hsize
0pt.7\hsize 0pt\hsize 
Although $\ww$ is a root of $x^3-1$, this is reducible over $\Q$ ($1$ is
also a root) and the minimum polynomial  of $\ww$ over $\Q$ is in
fact $x^2+x+1$
by Exercise \ref{ex_lect3.2}. By Theorem D, the field 
$\Q(\ww)=\{a+b\ww\,|\,a,b\in\Q\}$, so that $\Q(\ww)$ is $2$-dimensional over $\Q$ with
basis $\{1,\ww\}$. Let $\ss\in\gal(\Q(\ww)/\Q)$, whose effect is
completely determined by where it sends $\ww$. Suppose $\ss(\ww)=a+b\ww$ for some $a,b\in\Q$
to be determined. We have $\ss(\ww^3)=\ss(1)=1$, but also
$$
\ss(\ww^3)=\ss(\ww)^3=(a+b\ww)^3=(a^3+b^3-3ab^2)+(3a^2b-3ab^2)\ww
$$
with the last bit using $\ww^2=-\ww-1$.

As
$\{1,\ww\}$ are independent over $\Q$, the elements of $\Q(\ww)$ have unique
expressions as linear combinations of these two basis elements. We can therefore
``equate the $1$ and $\ww$ 
parts'' in these two expressions for $\ss(\ww^3)$:
$$
1=\ss(\ww^3)=(a^3+b^3-3ab^2)+(3a^2b-3ab^2)\ww,\text{ so that }
a^3+b^3-3ab^2=1\text{ and }3a^2b-3ab^2=0.
$$
Solving these equations (in $\Q$!) gives three solutions $a=0,b=1$ and
$a=1,b=0$ and
$a=-1,b=-1$, corresponding to $\ss(\ww)=\ww$ and $\ss(\ww)=1$ and
$\ss(\ww)=-1-\ww=\ww^2$.
The second one is impossible as $\ss$ is a bijection and we already
have $\ss(1)=1$. The first one is the identity map and the third
$\ss(\ww)=\ww^2=\overline{\ww}$ is complex conjugation (and shown in the
figure above), giving
$\gal(\Q(\ww)/\Q)=\{\id,\ss:z\mapsto\overline{z}\}$ a group of order
two. (Now revisit Exercise \ref{ex1.-1}).

\begin{vexercise}
\label{galois_groups_exercise50}
$\Q(\ww)$ is also spanned, as a vector space, by $\{1,\ww,\ww^2\}$, so
that every element has an expression of the form $a+b\ww+c\ww^2$ for
some $a,b,c\in\Q$. In
particular $\overline{\ww}$ can be written as both $\ww^2$ and as
$-1-\ww$. ``Equating the $1$ and the $\ww$ and the $\ww^2$ parts''
gives $0=-1$ and $1=0$. What has gone wrong?
\end{vexercise}

\paragraph{\hspace*{-0.3cm}}
Our first tool for unpicking the structure of Galois groups is:

\begin{theoremF}
Let $F,K$ be fields, $\tau:F\rightarrow K$ an isomorphism and
$\tau^*:F[x]\rightarrow K[x]$ the ring homomorphism given by
$\tau^*:\sum a_ix^i\mapsto\sum\tau(a_i)x^i$. If $\aa$ is algebraic over $F$, then 
$\tau$ extends to an isomorphism $\ss:F(\aa)\rightarrow K(\bb)$
with $\ss(\aa)=\bb$
if and only 
if $\bb$ is a root of $\tau^*f$, where $f$ is the minimum polynomial of $\aa$ over $F$.
\end{theoremF}

The elements $\aa$ and $\bb$ are assumed to lie in some extensions
$F\subseteq E_1, K\subseteq E_2$;
when we say that $\tau$ extends to $\ss$ we mean that the restriction
of $\ss$ to $F$ is $\tau$. 

The theorem seems technical, but has an intuitive meaning. 
Suppose we have $F=K$ and $\tau$ is the identity isomorphism,
hence $\tau^*$ is also the identity.
Then we have an extension $\ss:F(\aa)\rightarrow F(\bb)$ precisely when $\bb$
is a root of the minimum polynomial $f$ of $\aa$ over $F$.

We can say even more: if $\bb$ is an element of $F(\aa)$, then
$F(\bb)\subseteq F(\aa)$;
as an $F$-vector space
$F(\bb)$ is $(\deg f)$-dimensional over $F$ as $\aa$ and $\bb$ have the
same minimum polynomial over $F$. As $F(\aa)$ has the same dimension
we get $F(\bb)=F(\aa)$.
Thus $\ss$ is an isomorphism of $F(\aa)\rightarrow F(\aa)$ fixing $F$ pointwise, and so an element
of the Galois group $\gal(F(\aa)/F)$.

Here is everything we know about Galois groups so far:

\begin{corollary}
\label{galois_groups:Extension_corollary}
Let $\aa$ be algebraic over $F$ with minimum polynomial $f$ over $F$.
Then $\ss:F(\aa)\rightarrow F(\aa)$ is an
element of the Galois group $\gal(F(\aa)/F)$ if and only if
$\ss(\aa)=\beta$ where $\bb$ is a root of $f$ that is contained in
$F(\aa)$. 
\end{corollary}

The elements of the Galois group thus permute those roots of the minimum
polynomial that are contained in $F(\aa)$.

There are slick proofs of the Extension theorem; ours is not going to
be one of them. But it does make things nice and concrete. The
elements of $F(\aa)$ are polynomials in $\aa$, so the simplest way to
define $\ss$ is
\begin{equation}
  \label{eq:3}
\ss:a_m\aa^m+\cdots+a_1\aa+a_0
\mapsto 
\tau(a_m)\,\bb^m+\cdots+\tau(a_1)\,\bb+\tau(a_0).
\end{equation}
The complication is that the same element will have many such
polynomial expressions; for example $\ov{\ww}\in\Q(\ww)$ can be
written both as $\ww^2$ and $-1-\ww$ (see Exercise
\ref{galois_groups_exercise50} above) making it unclear if
(\ref{eq:3}) is well-defined. The solution is that $\bb$ is a root of $\tau^*f$,
the ``$K[x]$ version'' of $f$.
 
\begin{proofext}
For the ``only if'' part let $f=\sum a_ix^i$ with $f(\aa)=0$. Then $\sum
a_i\aa^i=0\in E_1$ and $\ss(0)=0\in E_2$ gives:
$$
\ss\biggl(\sum a_i\aa^i\biggr)=0\Rightarrow \sum \ss(a_i)
\ss(\aa)^i=0\Rightarrow 
\sum \tau(a_i)\,\bb^i=0
\Rightarrow 
\tau^*f(\bb)=0.
$$
(Compare this argument with the one that shows the roots of a
polynomial with real coefficients occur in complex conjugate pairs).

For the ``if'' part, we need to build an isomorphism $F(\aa)\rightarrow
K(\bb)$ with the desired properties.
Define $\ss$ by the formula (\ref{eq:3}); in particular
$\ss(a)=\tau(a)$ for all $a\in F$ and $\ss(\aa)=\bb$. 

\begin{description}
\item[(i).]\emph{$\ss$ is well-defined and 1-1}: Let
$$
\sum a_i\aa^i=\sum b_i\aa^i,
$$
be two expressions for some element of $F(\aa)$. Then $\sum
(a_i-b_i)\aa^i=0$ and so
$\aa$ is a root of the polynomial $g=\sum (a_i-b_i)x^i\in F[x]$. As $f$ is the minimum
polynomial of $\aa$ over $F$ it is a factor of $g$, so that
$g=fh$, hence $\tau^*(g)=\tau^*(fh)=\tau^*(f)\tau^*(h)$ and
$\tau^*(f)$ is a factor of $\tau^*(g)$. As $\bb$ is a root of $\tau^*(f)$ it
is a root of $\tau^*(g)$:
$$
\tau^*(g)(\bb)=0\Leftrightarrow \sum \tau(a_i-b_i)\,\bb^i=0\Leftrightarrow 
\sum \tau(a_i)\,\bb^i= \sum \tau(b_i)\,\bb^i \Leftrightarrow 
\ss\biggl(\sum a_i\aa^i\biggr)=\ss\biggl(\sum b_i\aa^i\biggr).
$$
The conclusion is that $\sum a_i\aa^i=\sum b_i\aa^i$ in $F(\aa)$ if
and only if $\ss(\sum a_i\aa^i)=\ss(\sum b_i\aa^i)$ in $K(\bb)$, hence
$\ss$ is both well-defined ($\Rightarrow$) and 1-1 ($\Leftarrow$).

\item[(ii).] \emph{$\ss$ is a homomorphism}: Let 
$$
\ll=\sum a_i\aa^i\text{ and }\mu=\sum b_i\aa^i,
$$
be two elements of $F(\aa)$. Then
\begin{equation*}
\begin{split}
\ss(\ll+\mu)=\ss\biggl(\sum(a_i+b_i)
\aa^i\biggr)&=\sum\tau(a_i+b_i)\,\bb^i\\
&=\sum\tau(a_i)\,\bb^i
+\sum\tau(b_i)\,\bb^i=\ss(\ll)
+\ss(\mu).
\end{split}
\end{equation*}
Similarly,
\begin{equation*}
\begin{split}
\ss(\ll\mu)=\ss\biggl(
\sum_{k}\biggl(\sum_{i+j=k}a_ib_j \biggr)\aa^k\biggr)&=
\sum_{k}\tau\biggl(\sum_{i+j=k}a_ib_j\biggr)\,\bb^k
=\sum_{k}\biggl(\sum_{i+j=k}\tau(a_i)\tau(b_j)\biggr)\,\bb^k\\
&=\biggl(\sum\tau(a_i)\,\bb^i\biggr)\biggl(\sum\tau(b_j)\,\bb^j\biggr)
=\ss(\ll)\ss(\mu).
\end{split}
\end{equation*}

\item[(ii).] \emph{$\ss$ is onto}: $\ss(F(\aa))$
is contained in $K(\bb)$ by (\ref{eq:3}).
On the other hand, any $b\in K$
is the image $b=\tau(a)$ of some $a\in F$, as $\tau$ is onto, and
$\bb=\ss(\aa)$ by definition. Thus both $\bb$ and $K$ are in 
$\ss(F(\aa))$,
hence $K(\bb)\subseteq \ss(F(\aa))$.
\qed
\end{description}
\end{proofext}

\paragraph{\hspace*{-0.3cm}}
To compute the Galois group of the extension $\Q\subset\Q(\aa)$, where
$\aa=\sqrt[3]{2}$, any automorphism is completely determined by where
it sends $\aa$. And we are
free to send $\aa$ to those roots of its minimum polynomial over $\Q$ that
are also contained in $\Q(\aa)$. The minimum polynomial is $x^3-2$, which
has roots $\aa,\aa\ww$ and $\aa\ww^2$ where
$$
\ww=-\frac{1}{2}+\frac{\sqrt[3]{2}}{2}\imag.
$$ 
But the roots $\aa\ww$ and $\aa\ww^2$ are not contained in $\Q(\aa)$ as this
field contains only real numbers -- whereas $\aa\ww$ and $\aa\ww^2$
are clearly non-real. Thus the only 
possible image for $\aa$ under an automorphism is $\aa$ itself, and
$\gal(\Q(\aa)/\Q)$ is 
the trivial group $\{\id\}$.

\paragraph{\hspace*{-0.3cm}}
Returning to the example immediately before the 
Extension theorem, any automorphism of $\Q(\ww)$ that fixes $\Q$ pointwise is determined
by where it sends $\ww$, and this must be to a root of the minimum polynomial over $\Q$
of $\ww$. As this polynomial is $1+x+x^2$ with roots $\ww$ and $\ww^2$, we 
have automorphisms that sends $\ww$ to itself or sends $\ww$ to $\ww^2=\ov{\ww}$, ie:
$$
\gal(\Q(\ww)/\Q)=\{\id,\ss:z\mapsto\ov{z}\}.
$$
In particular the figure below left is an automorphism but below right
is not:
$$
\begin{pspicture}(0,0)(13,3.5)
\rput(3,0.95){
\uput[0]{270}(-.1,2){\pstriangle[fillstyle=solid,fillcolor=lightgray](1,0)(2,1.73)}
\rput(-.1,0){\psbezier[linecolor=red]{<->}(-.4,0.1)(-1,.5)(-1,1.5)(-.4,1.9)}
\psbezier[linecolor=red]{->}(2.1,0.8)(3.1,0)(3.1,2)(2.1,1.2)
\rput(2,1){$1$}
\rput(-0.2,-0.2){$\ww^2$}
\rput(-0.20,2.2){$\ww$}
}
\rput(8,0.95){
\uput[0]{270}(-.1,2){\pstriangle[fillstyle=solid,fillcolor=lightgray](1,0)(2,1.73)}
\rput{-120}(0,2){\rput(-.1,0){\psbezier[linecolor=red]{<->}(-.4,0.1)(-1,.5)(-1,1.5)(-.4,1.9)}}
\rput{-120}(-.2,1.9){\psbezier[linecolor=red]{->}(2.1,0.8)(3.1,0)(3.1,2)(2.1,1.2)}
\rput(2,1){$1$}
\rput(-0.2,-0.2){$\ww^2$}
\rput(-0.20,2.2){$\ww$}
}
\end{pspicture}
$$

\paragraph{\hspace*{-0.3cm}}
The ``only if'' part of the Extension Theorem is worth stating separately:

\begin{corollary}
Let $F\subseteq E$ be an extension and 
$g\in F[x]$ having root $a\in
E$. Then for any $\ss\in\gal(E/F)$, 
the image $\ss(a)$ is also a root of $g$.
\end{corollary}

An immediate and important consequence is:

\begin{corollary}
If $F\subseteq E$ is a finite extension then the Galois group $\gal(E/F)$ is finite.
\end{corollary}

\begin{proof}
If $\{\aa_1,\aa_2,\ldots,\aa_k\}$ is a basis for $E$ over $F$, then 
$E=F(\aa_1,\aa_2,\ldots,\aa_k)$, with $\aa_i$ algebraic over $F$
(by Proposition \ref{finite.givesalgebraic})
having minimum polynomial $f_i\in F[x]$. If $\ss\in\gal(E/F)$ 
then $\ss$ is completely determined by the finitely many $\ss(\aa_i)$, which in turn
must be one of the finitely many roots of $f_i$.
\qed
\end{proof}

\paragraph{\hspace*{-0.3cm}}
Let $p$ be a prime and
$$
\ww=\cos\frac{2\pi}{p}+\imag\sin\frac{2\pi}{p},
$$
be a root of $1$.

By Corollary \ref{galois_groups:Extension_corollary},
$\ss\in\gal(\Q(\ww)/\Q)$ 
precisely when it sends $\ww$ to a root, contained in $\Q(\ww)$, of
its minimum polynomial 
over $\Q$. The
minimum polynomial is 
$$
\Phi_p=1+x+x^2+\cdots+x^{p-1},
$$
(Exercise \ref{ex_lect3.2}) with roots 
$\ww,\ww^2,\ldots,\ww^{p-1}$. All these roots are contained in $\Q(\ww)$, and so we
are free to send $\ww$ to any one of them. The Galois group thus has order $p-1$, with 
elements 
$$
\{\ss_1=\id:\ww\mapsto\ww,\ss_2:\ww\mapsto\ww^2,\ldots,\ss_{p-1}:\ww\mapsto\ww^{p-1}\}.
$$
If $\ss(\ww)=\ww^k$ then
$\ss^i(\ww)=\ww^{k^i}$ (keeping $\ww^p=1$ in mind).

We saw in Section \ref{fields3} that the multiplicative group of the finite field $\F_p$ is cyclic:
there is a $k$ with $1<k<p$, such that the powers $k^i$ of $k$ exhaust 
all of the non-zero elements of $\F_p$, ie: the powers $k^i$ run through
$\{1,2,\ldots,p-1\}$ mod $p$ (or $k$ generates $\F_p^*$).

Putting the previous two paragraphs together, let
$\ss\in\gal(\Q(\ww)/\Q)$ be such that 
$\ss(\ww)=\ww^k$ for $k$ a generator of $\F_p^*$. Then the elements
$$
\{\ss(\ww),\ss^2(\ww),\ldots,\ss^{p-1}(\ww)\}=\{\ww,\ww^2,\ldots,\ww^{p-1}\}
$$
and so the powers
$\ss,\ss^2,\ldots,\ss^{p-1}$ exhaust the Galois group.
$\gal(\Q(\ww)/\Q)$ is thus a cyclic group of order $p-1$.

\begin{figure}
  \centering  
\begin{pspicture}(0,0)(4.5,4)
\rput(0,-0.4){
\pspolygon[linecolor=white,fillstyle=solid,fillcolor=lightgray](4.5,2.5)(4,3.5)
(3,4)(1,4)(0,3)(0,2)(1,1)(3,1)(4,1.5)
\psline[linewidth=.2mm](3,1)(4,1.5)(4.5,2.5)(4,3.5)(3,4)
\psline[linewidth=.2mm](1,4)(0,3)(0,2)(1,1)
\rput(4.6,2.5){$1$}
\rput(4.25,3.5){$\ww$}\rput(3.1,4.3){$\ww^2$}
\rput(4.4,1.5){$\ww^{p-1}$}\rput(3.1,.7){$\ww^{p-2}$}
\rput(-.3,3){$\ww^k$}\rput(-.4,2){$\ww^{k+1}$}
\psbezier[linecolor=red,showpoints=false]{->}(3.9,3.45)(3,3)(1.5,3)(.2,3)
\rput(2,3.25){{\red $\ss$}}
}
\end{pspicture}
\caption{The Galois group $\gal(\Q(\ww)/\Q)$ is cyclic for $\ww$ a
  primitive $p$-th root of $1$.}
  \label{fig:galois_groups:cyclotomic}
\end{figure}

\paragraph{\hspace*{-0.3cm}}
The Extension theorem gives the existence of 
automorphisms. We can also say how many there are:

\begin{theorem}\label{galois_groups:number_of_extensions}
Let $\tau:F\rightarrow K$ be an isomorphism and $F\subseteq E_1$ and 
$K\subseteq E_2$ be extensions with $E_1$ a splitting field
of some polynomial $f$ over $F$ and $E_2$ a splitting field of $\tau^*f$ over $K$. 
Assume also that the roots of $\tau^*f$ in $E_2$ are distinct. Then the number of 
extensions of $\tau$
to an isomorphism $\ss:E_1\rightarrow E_2$ is equal to the degree of the extension
$K\subseteq E_2$. 
\end{theorem}

\begin{proof}
\parshape=5 0pt\hsize 0pt\hsize 0pt.8\hsize 0pt.8\hsize 0pt.8\hsize 
Let $\aa$ be a root of $f$ and
$F\subseteq F(\aa)\subseteq E_1$. By the Extension Theorem, $\tau$ extends to an 
isomorphism $\ss:F(\aa)\rightarrow K(\bb)$ if and only if $\bb$ is a root in $E_2$
of $\tau^*(p)$, where $p$ is the minimum polynomial of $\aa$ over
$F$. In this case the minimum polynomial $q$ of $\bb$ over $K$ divides
$\tau^*p$; moreover, $\deg\tau^*p\leq \deg p=[F(\aa):F]=[K(\bb):K]=\deg
q$. Thus $\tau^*p=q$ \emph{is\/} the minimum polynomial of $\bb$ over
$K$. 
\vadjust{\hfill\smash{\lower 72pt
\llap{
\begin{pspicture}(0,0)(2,3)
\rput(.2,.2){$F$}
\rput(1.8,.2){$K$}
\rput(.2,1.8){$F(\aa)$}
\rput(1.8,1.8){$K(\bb)$}
\rput(.2,2.9){$E_1$}
\rput(1.8,2.9){$E_2$}
\psline[linewidth=.1mm]{->}(.4,.2)(1.6,.2)
\psline[linewidth=.1mm]{->}(.7,1.8)(1.3,1.8)
\psline[linewidth=.1mm,linecolor=red]{->}(.4,2.9)(1.6,2.9)
\psline[linewidth=.1mm]{->}(.2,.4)(.2,1.6)
\psline[linewidth=.1mm]{->}(1.8,.4)(1.8,1.6)
\psline[linewidth=.1mm]{->}(.2,2.05)(.2,2.7)
\psline[linewidth=.1mm]{->}(1.8,2.05)(1.8,2.7)
\rput(1,.4){$\tau$}
\rput(1,3.1){\red ?}
\rput(1,2){$\ss$}
\end{pspicture}
}}}\ignorespaces

\parshape=2 0pt.8\hsize 0pt.8\hsize 
As $\aa$ is a root of $f$ we have $f=ph$ in $F[x]$, 
so $\tau^*f=(\tau^*p)(\tau^*h)$ in $K[x]$. As the roots of
$\tau^*f$ are distinct, those of $\tau^*p$ must be too. 

\parshape=2 0pt.8\hsize 0pt.8\hsize 
The number of possible $\ss$ then, which is equal to the number of 
\emph{distinct\/} roots of $\tau^*p$, must in fact be equal to the degree of $\tau^*p$.
This in turn equals the degree $[K(\bb):K]>1$. 

\parshape=3 0pt.8\hsize 0pt.8\hsize 0pt\hsize
We now proceed by induction on the degree $[E_2:K]$. If $[E_2:K]=1$
then $E_2=K$. An isomorphism $\ss:E_1\rightarrow E_2$ extending $\tau$
gives $[E_1:F]=1$, hence $E_1=F$. There can then be only one such
$\ss$, namely $\tau$ itself.  
By the tower law,
$[E_2:K]=[E_2:K(\bb)][K(\bb):K]$ where $[E_2:K(\bb)]<[E_2:K]$ since $[K(\bb):K]>1$.
By induction, any isomorphism $\ss:F(\aa)\rightarrow K(\bb)$ will thus
have 
$$
[E_2:K(\bb)]=\frac{[E_2:K]}{[K(\bb):K]},
$$
extensions to an isomorphism $E_1\rightarrow E_2$. Starting
from the bottom of the diagram, $\tau$ extends to $[K(\bb):K]$ possible $\ss$'s, and extending
each in turn gives,
$$
[K(\bb):K]\frac{[E_2:K]}{[K(\bb):K]}=[E_2:K],
$$
extensions in total.
\qed
\end{proof}

The condition that the roots of $\tau^*f$ are distinct
is not essential to the theory, but makes the
accounting easier:
we can relate the
number of automorphisms to the degrees of extensions by passing through the midway house
of the roots of polynomials. 

\paragraph{\hspace*{-0.3cm}}
Theorem D gives a connection between minimum polynomials and the degrees of 
field extensions, while 
Theorem \ref{galois_groups:number_of_extensions}
connects the degrees of extensions with the number of automorphisms of a field.
Bolting these together:

\begin{corollaryG}
Let $f$ be a polynomial over $F$ having distinct roots and let $E$ be its splitting field
over $F$. Then
\begin{equation}
  \label{eq:4}
|\gal(E/F)|=[E:F].  
\end{equation}
\end{corollaryG}

The polynomial $f$ is over the \emph{field\/} $F$, or is contained in
the {\em ring\/} $F[x]$, 
with $E$ a {\em vector space\/} over $F$ and $\gal(E/F)$ its {\em group\/} of 
automorphisms. The formula (\ref{eq:4}) thus contains the main objects
of undergraduate algebra.

\begin{proof}
By Theorem \ref{galois_groups:number_of_extensions} there are 
$[E:F]$ extensions of the identity automorphism $F\rightarrow F$ to an
automorphism of $E$.
Conversely any automorphism of $E$ fixing $F$ pointwise is an extension 
of the identity automorphism on $F$, so we obtain the whole Galois
group this way.
\qed
\end{proof}

\paragraph{\hspace*{-0.3cm}}
That $E$ be a splitting field is important in 
Corollary G.
Consider the extension $\Q\subseteq \Q(\sqrt[3]{2})$, where
$\Q(\sqrt[3]{2})$ is \emph{not\/} the splitting field over $\Q$ of
$x^3-2$, or indeed any polynomial. 
$\ss$ is an element of the  
Galois group $\gal(\Q(\sqrt[3]{2})/\Q)$ precisely when it sends $\sqrt[3]{2}$ to 
a root, contained in $\Q(\sqrt[3]{2})$, of its minimum polynomial over $\Q$. These
roots are $\sqrt[3]{2}$ itself, with the other two complex, whereas
$\Q(\sqrt[3]{2})$ is completely contained in $\R$. The only
possibility for $\ss$ is
that it sends $\sqrt[3]{2}$ to itself, ie: $\ss=\id$.

The Galois group thus has order $1$, but the degree of the extension is $3$.

\paragraph{\hspace*{-0.3cm}}
The following proposition returns to the kind of examples we saw in
Section \ref{lect1}:

\begin{proposition}
\label{galois_groups:section0_examples}
Let $E$ be the splitting field over $F$ of a polynomial 
with distinct roots. Suppose also that $E=F(\aa_1,\ldots,\aa_m)$ for
some $\aa_1,\ldots,\aa_m\in E$ such that
\begin{equation}
  \label{eq:5}
[E:F]=\prod_i [F(\aa_i):F].  
\end{equation}
Then there is a $\ss\in\gal(E/F)$ with $\ss(\aa_i)=\bb_i$ if and only if $\bb_i$ is a root
of the minimum polynomial of $\aa_i$ over $F$.
\end{proposition}

\begin{proof}
Any $\ss$ in the Galois group must send each $\aa_i$ to a root
of the minimum polynomial $f_i$ of $\aa_i$ over
$F$. 
Conversely, $\ss$ is determined by where it sends the
$\aa_i$'s, and there are at most $\deg(f_i)$ possibilities for these
images, namely the $\deg(f_i)$ roots of $f_i$. As
$$
|\gal(E/F)|=[E:F]=\prod_i [F(\aa_i):F]=\prod_i \deg(f_i),
$$
all these possibilities must arise. For any $\bb_i$ a root
of $f_i$ there must then be a $\ss\in\gal(E/F)$ with $\ss(\aa_i)=\bb_i$. 
\qed
\end{proof}

\paragraph{\hspace*{-0.3cm}}
In Section \ref{lect1} we computed, in an ad-hoc way, the automorphisms of $\Q(\aa,\ww)$ where 
$$
\aa=\sqrt[3]{2}\in\R\text{ and
}\ww=-\frac{1}{2}+\frac{\sqrt{3}}{2}\imag.
$$
The minimum polynomial of $\aa$ over $\Q$ is $x^3-2$ with roots
$\aa,\aa\ww,\aa\ww^2$ and the minimum polynomial of $\ww$ over $\Q$ -- 
and over $\Q(\ww)$ -- is 
$1+x+x^2$ with roots $\ww,\ww^2$. By the Tower law:
$$
[\Q(\aa,\ww):\Q]=[\Q(\aa,\ww):\Q(\aa)][\Q(\aa):\Q]=[\Q(\ww):\Q][\Q(\aa):\Q].
$$
By Proposition \ref{galois_groups:section0_examples} we can send
$\aa$ to any of $\aa,\aa\ww,\aa\ww^2$ 
and $\ww$ to any of $\ww,\ww^2$, and get an automorphism. 
Following this through with the vertices of the triangle gives three
automorphisms with $\ww$ mapped to itself -- the top three in Figure
\ref{fig:galoiscorrespondence1:schematic} -- and another three with
$\ww$ mapped to $\ww^2$ -- as in the bottom three.

\begin{figure}
  \centering
\begin{pspicture}(0,0)(14,7)
\rput(0,4){
\rput(-4.8,0){
\rput(7,2.5){\rput{-90}(0,0){
\pstriangle[fillstyle=solid,fillcolor=lightgray](1,0)(2,1.73)}
}
\rput{0}(6.5,0.5){\psbezier[linecolor=red]{->}(2.1,0.8)(3.1,0)(3.1,2)(2.1,1.2)}
\rput{120}(9.05,0.95){\psbezier[linecolor=red]{->}(2.1,0.8)(3.1,0)(3.1,2)(2.1,1.2)}
\rput{-120}(7.25,3){\psbezier[linecolor=red]{->}(2.1,0.8)(3.1,0)(3.1,2)(2.1,1.2)}
\rput(9.2,.6){$\aa\mapsto\aa$}
\rput(9.2,.3){$\ww\mapsto\ww$}
}
\rput(0,0){
\rput(7,2.5){\rput{-90}(0,0){
\pstriangle[fillstyle=solid,fillcolor=lightgray](1,0)(2,1.73)}
}
\rput{-120}(6.9,2.3){\psbezier[linecolor=red]{<-}(-.4,0.1)(-1,.5)(-1,1.5)(-.4,1.9)}
\rput{0}(7.2,0.5){\psbezier[linecolor=red]{<-}(-.4,0.1)(-1,.5)(-1,1.5)(-.4,1.9)}
\rput{120}(8.6,1.65){\psbezier[linecolor=red]{<-}(-.4,0.1)(-1,.5)(-1,1.5)(-.4,1.9)}
\rput(9.2,.6){$\aa\mapsto\aa\ww$}
\rput(9.2,.3){$\ww\mapsto\ww$}
}
\rput(4.3,0){
\rput(7,2.5){\rput{-90}(0,0){
\pstriangle[fillstyle=solid,fillcolor=lightgray](1,0)(2,1.73)}
}
\rput{-120}(6.9,2.3){\psbezier[linecolor=red]{->}(-.4,0.1)(-1,.5)(-1,1.5)(-.4,1.9)}
\rput{0}(7.2,0.5){\psbezier[linecolor=red]{->}(-.4,0.1)(-1,.5)(-1,1.5)(-.4,1.9)}
\rput{120}(8.6,1.65){\psbezier[linecolor=red]{->}(-.4,0.1)(-1,.5)(-1,1.5)(-.4,1.9)}
\rput(9.2,.6){$\aa\mapsto\aa\ww^2$}
\rput(9.2,.3){$\ww\mapsto\ww$}
}
}
\rput(0,0.5){
\rput(.75,0){
\rput(-5,0){
\rput(6.5,0.5){
\rput{270}(-.1,2){\pstriangle[fillstyle=solid,fillcolor=lightgray](1,0)(2,1.73)}
\rput(0.1,0){\psbezier[linecolor=red]{<->}(-.4,0.1)(-1,.5)(-1,1.5)(-.4,1.9)}
\rput(-.3,0){\psbezier[linecolor=red]{->}(2.1,0.8)(3.1,0)(3.1,2)(2.1,1.2)}
}
\rput(8,.6){$\aa\mapsto\aa$}
\rput(8,.3){$\ww\mapsto\ww^2$}
}
\rput(9,8.3){
\rput{-120}(0,0){\rput(6.5,0.5){
\rput{270}(-.1,2){\pstriangle[fillstyle=solid,fillcolor=lightgray](1,0)(2,1.73)}
\rput(0.1,0){\psbezier[linecolor=red]{<->}(-.4,0.1)(-1,.5)(-1,1.5)(-.4,1.9)}
\rput(-.3,0){\psbezier[linecolor=red]{->}(2.1,0.8)(3.1,0)(3.1,2)(2.1,1.2)}
}}
}
\rput(8,.6){$\aa\mapsto\aa\ww$}
\rput(8,.3){$\ww\mapsto\ww^2$}
\rput(-1,0){
}
\rput(16,-3.8){
\rput{120}(0,0){\rput(6.5,0.5){
\rput{270}(-.1,2){\pstriangle[fillstyle=solid,fillcolor=lightgray](1,0)(2,1.73)}
\rput(0.1,0){\psbezier[linecolor=red]{<->}(-.4,0.1)(-1,.5)(-1,1.5)(-.4,1.9)}
\rput(-.3,0){\psbezier[linecolor=red]{->}(2.1,0.8)(3.1,0)(3.1,2)(2.1,1.2)}
}}
}
\rput(4,0){
\rput(8.95,.6){$\aa\mapsto\aa\ww^2$}
\rput(8.95,.3){$\ww\mapsto\ww^2$}
}
\rput(3,0){
}
}
}
\end{pspicture}    
  \caption{the elements of $\gal(\Q(\aa,\ww)/\Q)$ where
    $\aa=\sqrt[3]{2}$ and  $\bb=-\frac{1}{2}+\frac{\sqrt{3}}{2}\imag$.}
  \label{fig:galoiscorrespondence1:schematic}
\end{figure}

\begin{vexercise}
\label{galois_groups_exercise100}
Let $\aa=\sqrt[5]{2}$ and $\ww=\cos(2\pi/5)+\imag\sin(2\pi/5)$, so that
$\aa^5=2$ and $\ww^5=1$.  
Let $\bb=\aa+\ww$ and eliminate radicals by considering
$(\bb-\ww)^5=2$ to  
find a polynomial of degree $20$ having $\bb$ as a root. Show that
this polynomial is irreducible 
over $\Q$ and hence that 
$$
[\Q(\aa+\ww):\Q]=[\Q(\aa):\Q][\Q(\ww):\Q].
$$
Show that $\Q(\aa+\ww)=\Q(\aa,\ww)$.
\end{vexercise}

\paragraph{\hspace*{-0.3cm}}
For $\aa=\sqrt[5]{2}$ and $\ww$ given by the expression below, 
the
extension $\Q\subset \Q(\aa,\ww)$ 
satisfies (\ref{eq:5}) by Exercise \ref{galois_groups_exercise100}.
An automorphism is thus free to send
$\aa$ to any root of 
$x^5-2$ and $\ww$ to any root of 
$1+x+x^2+x^3+x^4$. This gives twenty elements of the Galois group in
total;
in particular there is an automorphism sending $\aa$ to itself and $\ww$ to $\ww^3$:
$$
\begin{pspicture}(14,3)
\rput(2.5,.2){
\pspolygon[fillstyle=solid,fillcolor=lightgray](0,.45)(1.53,0)(2.5,1.3)(1.53,2.61)(0,2.15)
\psbezier[linecolor=red]{->}(2.85,1.1)(3.85,0.3)(3.85,2.3)(2.85,1.5)
\psline[linecolor=red]{->}(1.45,2.5)(0.05,.5)
\psline[linecolor=red]{->}(1.45,.1)(.05,2.05)
\pscurve[linecolor=red]{->}(.15,.5)(1,.9)(1.5,.1)
\pscurve[linecolor=red]{->}(.15,2.1)(.95,1.73)(1.55,2.5)
\rput(2.7,1.3){$\alpha$}\rput(1.5,2.8){$\alpha\ww$}
\rput(-.4,.6){$\alpha\ww^3$}
\rput(-.4,2.15){$\alpha\ww^2$}
\rput(1.5,-.2){$\alpha\ww^4$}
}
\rput(10,2){$\alpha=\sqrt[5]{2}$}
\rput(10,1){${\displaystyle \ww=\frac{\sqrt{5}-1}{4}+\frac{\sqrt{2}\sqrt{5+\sqrt{5}}}{4}\imag}$}
\end{pspicture}
$$

\paragraph{\hspace*{-0.3cm}}
We can get closer
to the spirit of Section \ref{lect1} by defining:

\begin{definition}[Galois group of a polynomial]
The Galois group over $F$ of the polynomial $f\in F[x]$ is
the group $\gal(E/F)$ where $E$ is the splitting field of $f$ over $F$.  
\end{definition}

\begin{proposition}\label{galoisgroups.subgroupsymmetricgroup}
The Galois group of a polynomial of degree $d$ is isomorphic to a subgroup 
of the symmetric group $S_{\kern-.3mm d}$.
\end{proposition}

\begin{proof}
Let $\{\aa_1,\ldots,\aa_d\}$ be the roots of $f$ and write
$\{\aa_1,\ldots,\aa_d\}=\{\bb_1,\ldots,\bb_k\}$ where the
$\bb$'s are distinct (and $k\leq d$).  
An element $\ss\in\gal(E/F)$, for
$E=F(\aa_1,\ldots,$ $\aa_d)=F(\bb_1,\ldots,\bb_k)$, is determined by where it sends the $\bb_i$'s,
and each $\ss(\bb_i)$ must be a root of (any) polynomial over
$F$ having $\bb_i$ as a root. But $f$ is such a polynomial, hence the
effect of $\ss$ on the $\bb_i$ is to permute them
among themselves ($\ss$ is a bijection). Define a map $\gal(E/F)\rightarrow S_{\kern-.3mm k}$ that
sends $\ss$ to the permutation of the $\bb_i$ that it realizes. As
the group laws in both the Galois group and the symmetric group are
composition, this map is a homomorphism, and is injective as each
$\ss$ is determined by its effect on the roots. Thus the Galois group
is isomorphic to a subgroup of $S_{\kern-.3mm k}$, which in turn is isomorphic to a
subgroup of $S_{\kern-.3mm d}$ by taking those permutations of $\{1,\ldots,d\}$
that permute only the first $k$ numbers.
\qed
\end{proof}

\begin{figure}
  \centering
\begin{pspicture}(0,0)(14,6)
\rput(3.5,3){\BoxedEPSF{galois12.10b.eps scaled 900}}
\rput(3.1,0.75){$\{\id\}$}
\rput*(0.5,2.5){$\{\id,(\aa,\bb)\}$}
\rput*(2.25,2.5){$\{\id,(\aa,\gamma)\}$}
\rput*(4,2.5){$\{\id,(\bb,\gamma)\}$}
\rput*(5.6,3.7){$\{\id,(\aa,\bb,\gamma),(\aa,\gamma,\bb)\}$}
\rput(4.3,5.2){$S_{\kern-.3mm 3}$}
\rput(10.5,3){\BoxedEPSF{galois12.10b.eps scaled 900}}
\rput(10,0.75){$x(x-1)(x-2)$}
\rput*(7.5,2.5){$x(x^2-2)$}
\rput*(9.2,2.5){$x(x^2-2)$}
\rput*(10.9,2.5){$x(x^2-2)$}
\rput*(13,3.7){$x^3-3x+1$}
\rput(11.3,5.3){$x^3-2$}
\end{pspicture}
\caption{The possible Galois groups over $\Q$ of
  $(x-\aa)(x-\bb)(x-\gamma)$: the subgroup lattice of the group of permutations of
  $\{\aa,\bb,\gamma\}$ (\emph{aka\/} the symmetric group $S_{\kern-.3mm 3}$)
  \emph{(left)} and example polynomials having Galois group these
  subgroups \emph{(right)}.}
  \label{fig:groups2:subgroup_lattice_S_3}
\end{figure}
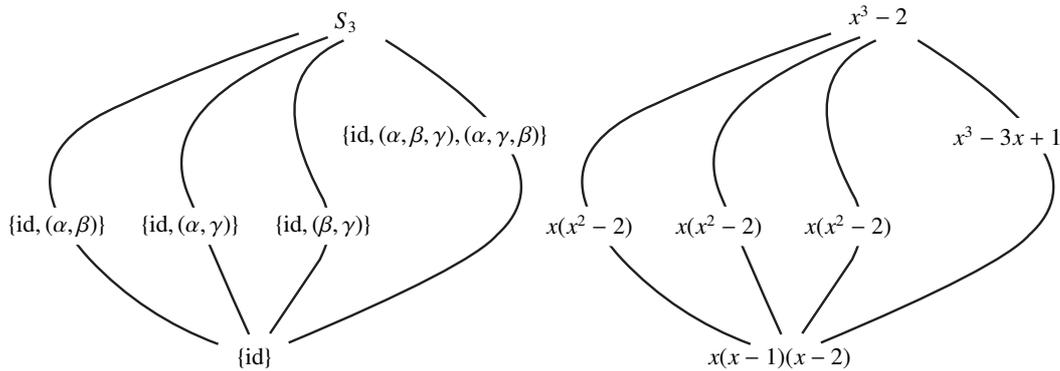

\paragraph{\hspace*{-0.3cm}}
\label{groups2:galois_group_quadratic}
Let $f=(x-\aa)(x-\bb)$ be a quadratic polynomial in $\Q[x]$ with
distinct roots
$\aa\not=\bb\in\C$. Then $f$ has splitting field $\Q(\aa)$ over $\Q$,
since $\aa+\bb$ and $\aa\bb$ are rational numbers. If $\aa\in\Q$
(hence $\bb\in\Q$) then the Galois group of $f$ over $\Q$ is the
trivial group $\{\id\}$. Otherwise both $\aa,\bb\not\in\Q$ and $f$, being irreducible
over $\Q$, is the minimum polynomial of $\aa$ over $\Q$. There is an
element of the Galois group sending $\aa$ to $\bb$, and this must be
the permutation $(\aa,\bb)$, as it is the only element of $S_{\kern-.3mm 2}$ that
does the job. The Galois group is thus
$\{\id,(\aa,\bb)\}$ when $\aa\not\in\Q$.  

\paragraph{\hspace*{-0.3cm}}
Similarly if $f=(x-\aa)(x-\bb)(x-\gamma)$ is a cubic in $\Q[x]$ with
distinct roots $\aa,\bb,\gamma\in\C$. By Proposition
\ref{galoisgroups.subgroupsymmetricgroup}, the Galois group of $f$ is a
subgroup of the symmetric group $S_{\kern-.3mm 3}$, the subgroup lattice of which
is shown in Figure \ref{fig:groups2:subgroup_lattice_S_3}. (You can
come up with this picture either by brute force, or by taking the
symmetry group of the equilateral triangle in Figure
\ref{fig:groups1:subgroup_lattices}, labelling the vertices of the
triangle $\aa,\bb,\gamma$, and taking the permutations of these
effected by the symmetries). 
We can find
polynomials having each of these subgroups as Galois group.

If $\aa,\bb,\gamma\in\Q$ then $f$ has splitting field $\Q$,
and the Galois group is $\{\id\}$. If $\aa,\bb\in\Q$ then, as
$\aa+\bb+\gamma\in\Q$, we get $\gamma\in\Q$ too. The next case then is
$\aa\in\Q$ and $\bb,\gamma\not\in\Q$, so that $(x-\bb)(x-\gamma)$ is a
rational polynomial. As in
\ref{groups2:galois_group_quadratic}, the splitting field of $f$ is
$\Q(\bb)$ and the Galois group is $\{\id,(\bb,\gamma)\}$. The other
two subgroups of order two in Figure
\ref{fig:groups2:subgroup_lattice_S_3} come about in a similar way.

That leaves the case $\aa,\bb,\gamma\not\in\Q$, and where the
key player is the \emph{discriminant\/}:
$$
D=(\aa-\bb)^2(\aa-\gamma)^2(\bb-\gamma)^2
$$
or in fact, its square root. The polynomial $f$ is irreducible over
$\Q$, hence the minimum polynomial over $\Q$ of $\aa$. As the roots
$\aa,\bb,\gamma$ are distinct there are distinct elements
of the Galois group sending $\aa$ to each of $\aa,\bb$ and $\gamma$,
and so the Galois group has order $3$ or $6$. 

Suppose that $\kern-2pt\sqrt{D}\in\Q$. Then $\kern-2pt\sqrt{D}$, like all rational
numbers, is fixed by the elements of the Galois group. The permutation
$(\aa,\bb)$ however sends $\kern-2pt\sqrt{D}\mapsto-\kern-2pt\sqrt{D}$, and so do $(\aa,\gamma)$
and $(\bb,\gamma)$. None of these can therefore be in the Galois
group, which is thus $\{\id,(\aa,\bb,\gamma),(\aa,\gamma,\bb)\}$. 

We illustrate the final case $\kern-2pt\sqrt{D}\not\in\Q$ by example. Suppose that
$\aa\in\R\setminus\Q$ and $\bb,\gamma\in\C\setminus\R$ -- in which case
$\bb,\gamma$ are complex conjugates. Then complex conjugation is a
non-trivial element of the Galois group (see Exercise
\ref{exercise:groups2:conjugation}) having 
effect the permutation $(\bb,\gamma)$. The Galois group must then 
be all of $S_3$. (Incidentally, this and the previous paragraph show that if
$\kern-2pt\sqrt{D}\in\Q$ then $\aa,\bb,\gamma\in\R$.)

\paragraph{\hspace*{-0.3cm}}
Finding a rational polynomial of degree $d$ that has Galois group a
given subgroup of $S_{\kern-.3mm d}$ is possible for small values of $d$ like the
cases $d=2,3$ above. For general $d$ it is an open problem --
called the \emph{Inverse Galois problem\/}.

\subsection*{Further Exercises for Section \thesection}

\begin{vexercise}\label{ex_lect9.1}
Show that the following Galois groups have the given orders:
\begin{enumerate}
\item $|\gal(\Q(\kern-2pt\sqrt{2})/\Q)|=2$.
\item $|\gal(\Q(\sqrt[3]{2})/\Q)|=1$.
\item $|\gal(\Q(-\frac{1}{2}+\frac{\sqrt{3}}{2}\imag)/\Q)|=2$.
\item
$|\gal(\Q(\sqrt[3]{2},-\frac{1}{2}+\frac{\sqrt{3}}{2}\imag)/\Q)|=6$.
\end{enumerate}
\end{vexercise}

\begin{vexercise}\label{ex_lect9.2}
Find the orders of the Galois groups $\gal(L/\Q)$ where $L$ is the
splitting field of the polynomial:
$$
1. \,\, x-2\qquad
2. \,\,x^2-2\qquad
3. \,\,x^5-2\qquad
$$
\end{vexercise}

\begin{vexercise}\label{ex_lect9.21}
Find the orders of the Galois groups $\gal(L/\Q)$ where $L$ is the
splitting field of the polynomial:
$$
1. \,\, 1+x+x^2+x^3+x^4\qquad
2. \,\, 1+x^2+x^4\qquad
$$
(\emph{hint} for the second one: $(x^2-1)(1+x^2+x^4)=x^6-1$).
\end{vexercise}

\begin{vexercise}\label{ex_lect9.3}
Let $p>2$ be a prime number. Show that
\begin{enumerate}
\item ${\ds |\gal(\Q\biggl(\cos\frac{2\pi}{p}+\imag\sin\frac{2\pi}{p}\biggr)/\Q)|=p-1}$.
\item $|\gal(L/\Q)|=p(p-1)$, where $L$ is the splitting field of the
polynomial $x^p-2$. Compare the answer when $p=3$ and $5$ to Section \ref{lect1}.
\end{enumerate}
\end{vexercise}


\section{Vector Spaces II: Solving Equations}
\label{linear.algebra2}

This short section contains some auxiliary technical results on the
solutions of homogeneous linear equations that are needed for the
proof of the Galois correspondence in Section \ref{galois.correspondence}.

\paragraph{\hspace*{-0.3cm}}
Let $V$ be a $n$-dimensional vector space over the field $F$
with fixed basis $\{\aa_1,\aa_2,\ldots,\aa_n\}$.
A {\em homogenous linear equation\/} over $F$ is an equation of the form,
$$
a_1x_1+a_2x_2+\cdots+ a_nx_n=0,
$$
with the $a_i$ in $F$. A vector $u=\sum_{i=1}^{n} t_i\aa_i\in V$ is a solution 
when
$$
a_1t_1+a_2t_2+\cdots+ a_nt_n=0.
$$
A system of homogeneous linear equations,
\begin{equation*}
\begin{split}
a_{11}x_{1}+a_{12}x_2+\cdots+ a_{1n}x_n&=0,\\
a_{21}x_1+a_{22}x_2+\cdots+ a_{2n}x_n&=0,\\
&\vdots\\
a_{k1}x_1+a_{k2}x_2+\cdots+ a_{kn}x_n&=0,
\end{split}
\end{equation*}
is {\em independent\/} over $F$ when the vectors,
$$
v_1=\sum a_{1j}\aa_j,v_2=\sum a_{2j}\aa_j,\ldots,v_k=\sum a_{kj}\aa_j,
$$
are independent. In other words, if $A$ is the matrix of
coefficients of the system of equations, then the rows of $A$
are independent.

Here is the key property of independent systems of equations: 

\begin{proposition}
\label{vectorspaces2:independent_systems}
Let $S$
be an independent system of equations over $F$
and let $S'\subset S$ be a proper
subset of the equations. Then the space of solutions in $V$ to $S$ is a proper
subspace of the space of solutions in $V$ to $S'$. 
\end{proposition}

\begin{vexercise}
Prove Proposition \ref{vectorspaces2:independent_systems}.
\end{vexercise}

\begin{vexercise}
Let $F\subseteq E$ be an extension of fields and $B$ a finite set. Let $V_F$ be the 
$F$-vector space with basis $B$, ie: the elements of $V_F$ are the formal sums
$$
\sum \ll_i b_i,
$$
with the $\ll_i\in F$ and the $b_i\in B$. Formal sums are added together 
and multiplied by scalars in the obvious way. Similarly let $V_E$ be the $E$-vector
space with basis $B$, and identify $V_F$ with a subset (it is not a subspace) of $V_E$
in the obvious way.
Now let $S'\subset S$ be independent systems of equations \emph{over $E$\/}.
Show that 
the space of solutions in $V_F$ to $S$ is a proper subspace of the space
of solutions in $V_F$ to $S'$.
\end{vexercise}

\begin{vexercise}\label{vandermonde}
Let $F$ be a field and $\aa_1,\ldots,\aa_{n+1}\in F$ distinct elements.
Show that the matrix
$$
\left(\begin{array}{cccc}
\aa_1^{n}&\cdots&\aa_1&1\\
\vdots&&\vdots&\vdots\\
\aa_{n+1}^{n}&\cdots&\aa_{n+1}&1\\
\end{array}\right)
$$
has non-zero determinant
(\emph{hint\/}: suppose otherwise, and find a polynomial of degree $n$ with $n+1$
distinct roots in $F$, contradicting Theorem \ref{degree.number.of.roots}).
\end{vexercise}

\begin{lemma}
\label{vectorspaces2:polynomials_same}
Let $F$ be a field and $f,g\in F[x]$ polynomials of degree $n$ over $F$. 
Suppose that
there exist distinct $\aa_1,\ldots,\aa_{n+1}\in F$ such that $f(\aa_i)=g(\aa_i)$ for all
$i$. Then $f=g$.
\end{lemma}

\begin{proof}
Letting
$f(x)=\sum a_i x^i\mbox{ and }g(x)=\sum b_i x^i$
gives $n+1$ expressions $\sum a_i \aa_j^i=\sum b_i \aa_j^i$, hence
the system of equations
\begin{equation}
  \label{eq:6}
  \sum a_j^i y_i=0,
\end{equation}
where $y_i=a_i-b_i$. The matrix of coefficients of these $n+1$
equations is
$$
\left(\begin{array}{cccc}
\aa_1^{n}&\cdots&\aa_1&1\\
\vdots&&\vdots&\vdots\\
\aa_{n+1}^{n}&\cdots&\aa_{n+1}&1\\
\end{array}\right)
$$
with non-zero determinant by Exercise \ref{vandermonde}. 
The system (\ref{eq:6}) thus has
the unique solution $y_i=0$ for all $i$, so that $f=g$.
\qed
\end{proof}

\paragraph{\hspace*{-0.3cm}}
Here is the main result of the section.

\begin{theorem}
\label{theorem:linearalgebra2}
Let $F\subseteq E=F(\aa)$ be a simple extension of fields with the minimum
polynomial of $\aa$ over $F$ having distinct roots. 
Let $\{\ss_1,\ss_2\ldots,\ss_k\}$ be distinct non-identity  elements of the Galois 
group $\gal(E/F)$. Then 
$$
\ss_1(x)=\ss_2(x)=\cdots=\ss_k(x)=x,
$$
is a system of independent linear equations over $E$.
\end{theorem}

\begin{proof}
By Theorem D we have a basis $\{1,\aa,\aa^2,\ldots,\aa^d\}$ for $E$ over $F$ where
the minimum polynomial $f$ of $\aa$ over $f$ has degree $d+1$. Any
$x\in E$ thus has the 
form 
$$
x=x_0+x_1\aa+x_2\aa^2+\cdots+x_d\aa^d,
$$
for some $x_i\in F$. By the Extension Theorem, the elements of the Galois group
send $\aa$ to roots of $f$. Suppose these roots are 
$\{\aa=\aa_0,\aa_1,\ldots,\aa_{d}\}$
where $\ss_i(\aa)=\aa_i$. Then $x$ satisfies $\ss_i(x)=x$ if and only if,
$$
(\aa_0-\aa_i)x_1+(\aa_0^2-\aa_i^2)x_2+\cdots+(\aa_0^d-\aa_i^d)x_d=0.
$$
Thus we have a system of equations $Ax=0$ where the matrix of coefficients $A$ 
is made up of rows from the larger $d\times d$ matrix $\widehat{A}$ given by,
$$
\widehat{A}=
\left(\begin{array}{cccc}
\aa_0-\aa_1&\aa_0^2-\aa_1^2&\cdots&\aa_0^d-\aa_1^d\\
\aa_0-\aa_2 
\vrule width 0mm height 5mm depth 0mm
\vrule width 2mm height 0mm depth 0mm
&\aa_0^2-\aa_2^2
\vrule width 2mm height 0mm depth 0mm
&\cdots
\vrule width 2mm height 0mm depth 0mm
&\aa_0^d-\aa_2^d\\
\vdots&\vdots&&\vdots\\
\aa_0-\aa_d&\aa_0^2-\aa_d^2&\cdots&\aa_0^d-\aa_d^d\\
\end{array}\right)
$$
Let $\widehat{A}b=0$ for some vector $b\in E^n$, so that
$$
b_0\aa_0+b_1\aa_0^2+\cdots+b_d\aa_0^d=b_0\aa_i+b_1\aa_i^2+\cdots+b_d\aa_i^d,
$$
for each $1\leq i\leq d$. Thus if $g=b_0x+b_1x^2+\cdots+b_dx^d$, then we
have $g(\aa_0)=g(\aa_1)=g(\aa_2)=\cdots=g(\aa_d)=a$, say. The degree $d$
polynomial $g-a$ thus agrees with the zero polynomial at $d+1$ distinct values, hence
by Lemma \ref{vectorspaces2:polynomials_same} must be the zero
polynomial, and so all the $b_i$ are 
zero. The columns of $\widehat{A}$ are thus independent, hence so are
the rows, and thus also the rows of $A$.
\qed
\end{proof}


\section{The Fundamental Theorem of Galois Theory}
\label{galois.correspondence}

According to Theorem E, a $z\in\C$ is constructible when there is a
sequence of extensions:
$$
\Q=K_0\subseteq K_1\subseteq K_2\subseteq\cdots\subseteq K_n,
$$
with each $[K_{i+1}:K_i]\leq 2$ and $\Q(z)\subset K_n$.
To show that $z$ can actually {\em be\/} constructed, we need to find
these $K_i$, and so 
we need to understand the fields sandwiched between $\Q$ and $\Q(z)$. 
In this section we prove the theorem that gives us that knowledge.

\paragraph{\hspace*{-0.3cm}}
We will need a picture of the fields sandwiched in an extension, analogous to the 
picture of the subgroups of a group in Section \ref{groups.stuff}. 


\begin{definition}[intermediate fields and their lattice]
Let $F\subseteq E$ be an extension. Then $K$ is an intermediate field
when $K$ is an extension of $F$ and $E$ is an extension of $K$:
ie: $F\subseteq K\subseteq E$. The lattice of intermediate fields is a 
diagram depicting them and the inclusions between them. 
If $F\subseteq K_1\subseteq K_2\subseteq E$ they appear in
the diagram like so:
$$
\begin{pspicture}(0,0)(2,2)
\rput(1,.3){$K_1$}
\rput(1,1.7){$K_2$}
\psline(1,.6)(1,1.4)
\end{pspicture}
$$
At the very base of the diagram is $F$ and at the apex is $E$. Denote the lattice by $\LL(E/F)$.
\end{definition}

\paragraph{\hspace*{-0.3cm}}
\label{galois:correspondence:condition}
From now on we will work in the following situation: 
$F\subseteq E$ is a finite extension such that:
\begin{description}
\item[(\dag)]  Every irreducible polynomial over $F$ that has a root in $E$
  has all its roots in $E$, and these roots are distinct.
\end{description}
We saw in Exercise \ref{ex4.30} that if $F$
has characteristic $0$ then any irreducible polynomial over $F$ has
distinct roots. This is also true if $F$ is a finite field, although we omit the
proof here. 

\begin{galoiscorrespondenceA}
Let $F\subseteq E$ be a finite extension satisfying $(\dag)$
and $G=\gal(E/F)$ its Galois group. Let $\LL(G)$
and $\LL(E/F)$ be the subgroup and intermediate field lattices.
\begin{enumerate}
\item For any subgroup $H$ of $G$, let 
$$
E^H=\{\ll\in E\,|\,\ss(\ll)=\ll\text{ for all }\ss\in H\}.
$$
Then $E^H$ is an intermediate field, called the fixed field of $H$.
\item For any intermediate field $K$, the group $\gal(E/K)$ is a subgroup of $G$.
\item The maps $\Psi:H\mapsto E^H$ 
and $\Phi:K\mapsto \gal(E/K)$ are mutual inverses, hence bijections
$$
\Psi:\LL(G)\rightleftarrows \LL(E/F):\Phi
$$
that reverse order:
$$
H_1 \subset H_2
\stackrel{\Psi}{\longrightarrow} 
E^{H_2}\subset E^{H_1}
\qquad
K_2\subset K_1
\stackrel{\Phi}{\longrightarrow} 
\gal(E/K_1)\subset\gal(E/K_2)
$$
\item The degree of the extension $E^H\subseteq E$ is equal to the order
$|H|$ of the subgroup $H$. Equivalently, the degree of the extension
$F\subseteq E^H$ is equal
to the index $[G:H]$.
\end{enumerate}
\end{galoiscorrespondenceA}

The correspondence in one sentence: turning the lattice of subgroups upside down
gives the lattice of intermediate fields,
and vice-versa. See Figure \ref{fig:galoiscorrespondence1:schematic}.

\begin{figure}
  \centering
\begin{pspicture}(0,0)(14,5)
\rput(0,0){
\rput(3.5,2.5){\BoxedEPSF{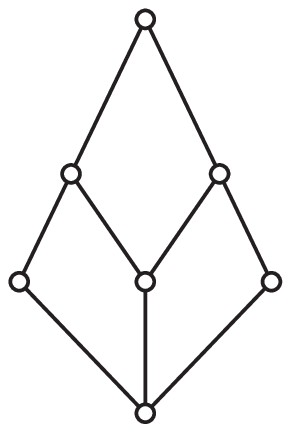 scaled 1000}}
\rput(3.5,4.8){$E$}\rput(3.5,0.15){$F$}
\rput(1.8,2.9){$E^{H_1}=K_1$}\rput(1.3,1.8){$E^{H_2}=K_2$}
\rput(4.7,4){$\LL(E/F)$}\rput(2.3,2.4){$n$}
}
\rput(2,0){
\rput(8.5,2.5){\BoxedEPSF{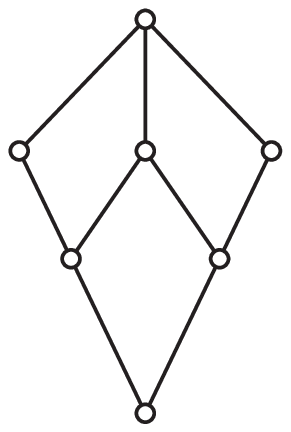 scaled 1000}}
\rput(8.5,4.8){$G=\gal(E/F)$}\rput(8.5,0.15){$\{\id\}$}
\rput(11.15,3.2){$H_2=\gal(E/K_2)$}\rput(10.6,2){$H_1=\gal(E/K_1)$}
\rput(7.2,4){$\LL(G)$}\rput(9.65,2.5){$n$}
}
%
\rput(7,0.5){$[E^{H_1}:E^{H_2}]=n=[H_2:H_1]$}
\rput(0.25,1){
\psline[linewidth=0.9pt]{->}(5,2)(8.5,2)
\psline[linewidth=0.9pt]{<-}(5,1.5)(8.5,1.5)
\rput(6.75,1.75){Galois correspondence}
\rput(6.75,2.25){$\Phi:X\mapsto\gal(E/X)$}
\rput(6.75,1.25){$\Psi:Y\mapsto E^Y$}
}
\end{pspicture}    
  \caption{Schematic of the Galois correspondence.}
  \label{fig:galoiscorrespondence1:schematic}
\end{figure}

The upside down nature of the correspondence may seem puzzling, but it
is just the nature of imposing conditions.
If $H$ is a subgroup,
the fixed field $E^H$ is the set of solutions in $E$ to the system of equations
\begin{equation}
  \label{eq:7}
\ss(x)=x, \text{ for }\ss\in H.  
\end{equation}
The more equations, the greater the number of conditions being imposed on
$x$, hence the smaller the number of solutions. Thus, larger subgroups
$H$ should correspond
to smaller intermediate fields $E^H$ and vice-versa. That the
correspondence is exact -- increasing the size of $H$ decreases the
size of $E^H$ -- 
will follow from Section \ref{linear.algebra2} and the fact that the
equations (\ref{eq:7}) are independent.

\begin{proof}
In the situation described in the Theorem the
extension is of the form 
$F\subseteq F(\aa)$ for some
$\aa\in E$ algebraic over $F$. The minimum polynomial $f$ of $\aa$
over $F$ splits in $E$ by $(\dag)$. On the other hand any
field containing the roots of $f$ contains $F(\aa)=E$. Thus $E$ is the
splitting field of $f$.  

\begin{enumerate}
\item \emph{$E^H$ is an intermediate field:} we
have $E^H\subset E$ by definition, and  
$F\subset E^H$ as every element of
$G$ -- so in particular every element of $H$ -- fixes $F$. If $\ll,\mu\in E^H$
then $\ss(\ll+\mu)=\ss(\ll)+\ss(\mu)=\ll+\mu$, so that $\ll+\mu\in E^H$,
and similarly $\ll\mu,1/\ll\in E^H$. 

\item \emph{$\gal(E/K)$ is a subgroup:} if an
automorphism of $E$ fixes the intermediate field $K$ pointwise, then 
it fixes the field $F$ pointwise, and thus $\gal(E/K)\subset
\gal(E/F)$. If $\ss,\tau$ are automorphisms fixing $K$
then so is $\ss\tau^{-1}$. We thus have a subgroup.

\item \emph{$\Phi$ and $\Psi$ reverse order:} if $\ll$ is fixed by 
  every automorphism in $H_2$, then it is fixed by every  
automorphism in $H_1$, so that $E^{H_2}\subset E^{H_1}$. If $\ss$ fixes every 
element of $K_1$ pointwise then it fixes every element of $K_2$ pointwise,
so that $\gal(E/K_1)\subset\gal(E/K_2)$.

\item \emph{The composition $\Phi\Psi:H\rightarrow E^H\rightarrow \gal(E/E^H)$
    is the identity:} by definition every element of $H$ fixes $E^H$ pointwise, and since 
$\gal(E/E^H)$ consists of {\em all\/} the automorphisms of $E$ that fix
$E^H$ pointwise, we have $H\subset\gal(E/E^H)$. In fact, both 
$H$ and $\gal(E/E^H)$ have the same fixed field, ie: $E^{\gal(E/E^H)}=E^H$.
To see this, any $\ss\in\gal(E/E^H)$ fixes $E^H$ pointwise by definition,
so $E^H\subset E^{\gal(E/E^H)}$. On the other hand
$H\subset \gal(E/E^H)$ and $\Psi$ reverses order, so
$E^{\gal(E/E^H)}\subset E^H$. 

By the results of Section \ref{linear.algebra2}, the elements of the fixed field
$E^{\gal(E/E^H)}$  are obtained by solving the system of linear equations
$\ss(x)=x$ for all $\ss\in\gal(E/E^H)$, and these equations are independent.
In particular, a proper subset of these equations has 
a proper superset of solutions. 
We already have that $H\subset\gal(E/E^H)$. Suppose 
$H$ is a proper subgroup of $\gal(E/E^H)$. 
The fixed field $E^H$ would then properly contain the fixed field $E^{\gal(E/E^H)}$.
As this contradicts the previous paragraph, we have
$H=\gal(E/E^H)$.

\item \emph{The composition $\Psi\Phi:K\rightarrow\gal(E/K)\rightarrow
    E^{\gal(E/K)}$ is the identity:} let $E=K(\beta)$ and suppose
  the minimum polynomial $g$ of $\beta$ over $K$ has degree $d+1$ with
  roots $\{\beta=\beta_0,\ldots,\beta_d\}$. $E$
  thus has basis $\{1,\beta,\ldots,\beta^d\}$ over $K$ and
  $G=\gal(E/K)$ has elements $\{\id=\ss_0,\ldots,\ss_d\}$ by Theorem
  G, labelled so that $\ss_i(\beta)=\beta_i$. An element $x\in E$ has the form
$$
x=x_0+x_1\beta+\cdots+x_d\,\beta^d
$$
with $x\in E^G$ exactly when $\ss_i(x)=x$ for all $i$, i.e. when 
$$
x_1(\bb-\bb_i)+\cdots+x_d(\bb^d-\bb_i^d)=0,
$$
a homogenous system of $d$ equations in $d$ unknowns. The system has
coefficients given by the matrix $\hat{A}$ of Theorem
\ref{theorem:linearalgebra2} (but with $\bb$'s instead of $\aa$'s) and
hence, by the argument given there, has the unique solution
$x_1=\cdots=x_d=0$. Thus $x=x_0\in K$ and so $E^{\gal(E/K)}=K$. 

\item As $E$ is a splitting field we can apply
  Theorem G to get $|\gal(E/E^H)|=[E:E^H]$, 
where $\gal(E/E^H)=H$ gives $|H|=[E:E^H]$.
\qed
\end{enumerate}
\end{proof}

\paragraph{\hspace*{-0.3cm}}
Before an example, a little house-keeping:
the condition $(\dag)$ in \ref{galois:correspondence:condition} can be
replaced by an easier one to verify:

\begin{proposition}
  \label{proposition:galois_extensions}
Let $F\subset E$ be a finite extension such that every irreducible
polynomial over $F$ has distinct roots. Then the following are
equivalent:
\begin{enumerate}
\item Every irreducible polynomial over $F$ that has a root in $E$ has
  all its roots in $E$.
\item $E=F(\aa)$ and the minimum polynomial of $\aa$ over $F$ splits
  in $E$.
\end{enumerate}
\end{proposition}

\begin{proof}
$(1)\Rightarrow (2)$: the minimum polynomial is
irreducible over $F$ with root $\aa\in F(\aa)=E$, hence splits by (1).

$(2)\Rightarrow (1)$: apply the argument of part 5 of the proof of the
Galois correspondence to $K=F$ to get $E^{G}=F$ for $G=\gal(E/F)$. Suppose that
$p\in F[x]$ is irreducible over $F$ and has a root $\aa\in E$ and let
$\{\aa=\aa_1,\ldots,\aa_n\}$ be the distinct elements of the set
$\{\ss(\aa):\ss\in G\}$. The polynomial $g=\prod (x-\aa_i)$ has roots
permuted by the $\ss\in G$, hence its coefficients are fixed by the
$\ss\in G$, i.e. $g$ is a polynomial over $E^G=F$. Both $p$ and $g$
have factor $x-\aa$, hence their gcd is not $1$. As $p$ is irreducible
it must then divide $g$, hence all it roots lie in $E$. 
\qed
\end{proof}

\paragraph{\hspace*{-0.3cm}}
\label{galois:correspondence:example1}
Now to our first example. In Section \ref{galois.groups} we revisited
the example of Section \ref{lect1}, where for $\aa=\sqrt[3]{2}$ and 
$\ww=\frac{1}{2}+\frac{\sqrt{3}}{2}\imag$ we had
$$
G=\gal(\Q(\aa,\ww)/\Q)=\{\id,\ss,\ss^2,\tau,\ss\tau,\ss^2\tau\},
$$
with $\ss(\aa)=\aa\ww,\ss(\ww)=\ww$ and $\tau(\aa)=\aa,\tau(\ww)=\ww^2$. 

In \ref{vector:spacesI:minimum:polynomial} we showed that
$\Q(\aa,\ww)=\Q(\aa+\ww)$ with the minimum 
polynomial of $\aa+\ww$ over $\Q$ having all its roots in
$\Q(\aa,\ww)$. Condition $(\dag)$ thus holds. The subgroup lattice
$\LL(G)$ is shown on the left in Figure
\ref{fig:galois:correspondence:example1} -- adapted from Figure 
\ref{fig:groups2:subgroup_lattice_S_3}. 
Applying the Galois Correspondence then gives the lattice $\LL(E/F)$ of intermediate
fields on the right of Figure \ref{fig:galois:correspondence:example1}
with $F_4$ the fixed field of $\{\id,\ss,\ss^2\}$ and the others the fixed
fields (in no particular order) of the three order two subgroups. By
part (4)
of the Galois correspondence, each of the extensions $F_i\subset\Q(\aa,\ww)$ has 
degree the order of the corresponding subgroup, so that $\Q(\aa,\ww)$ is a degree
three extension of $F_4$, and a degree two extension of the other intermediate
fields. 

Let $F_1$ be the fixed field of the subgroup $\{\id,\tau\}$; we will
explicitly describe its elements. The Tower law gives basis for 
$\Q(\aa,\ww)$ over $\Q$ the set
$$
\{1,\aa,\aa^2,\ww,\aa\ww,\aa^2\ww\},
$$
so that an $x\in\Q(\aa,\ww)$ has the form,
$$
x=a_0+a_1\aa+a_2\aa^2+a_3\ww+a_4\aa\ww+a_5\aa^2\ww,
$$
with the $a_i\in\Q$. The element $x$ is in $F_1$ if and only if $\tau(x)=x$
where,
\begin{equation*}
\begin{split}
\tau(x)&=a_0+a_1\aa+a_2\aa^2+a_3\ww^2+a_4\aa\ww^2+a_5\aa^2\ww^2\\
&=a_0+a_1\aa+a_2\aa^2+a_3(-1-\ww)+a_4\aa(-1-\ww)+a_5\aa^2(-1-\ww)\\
&=(a_0-a_3)+(a_1-a_4)\aa+(a_2-a_5)\aa^2-a_3\ww-a_4\aa\ww^2-a_5\aa^2\ww.
\end{split}
\end{equation*}
Equate coefficients (we are using a basis) to get:
$$
a_0-a_3=a_0, a_1-a_4=a_1,a_2-a_5=a_2,-a_3=a_3,-a_4=a_4\text{ and }-a_5=a_5.
$$
Thus, $a_3=a_4=a_5=0$ and $a_0,a_1,a_2$ are arbitrary. Hence
$$
x=a_0+a_1\aa+a_2\aa^2
$$ 
so is an element of $\Q(\aa)$. 
This gives $F_1\subseteq \Q(\aa)$. On the other hand, $\tau$ fixes $\Q$
pointwise and fixes $\aa$, hence fixes $\Q(\aa)$ pointwise, giving
$\Q(\aa)\subseteq F_1$ and so $F_1=\Q(\aa)$. 

The rest of the picture is described in Exercise
\ref{galois:correspondence:exercise10}.  

\begin{figure}
  \centering
\begin{pspicture}(0,0)(14,6)
\rput(-0.2,0){
\rput(3.5,3){\BoxedEPSF{galois12.10b.eps scaled 900}}
\rput(3.1,0.75){$\{\id\}$}
\rput*(0.5,2.5){$\{\id,\tau\}$}
\rput*(2.25,2.5){$\{\id,\ss\tau\}$}
\rput*(4,2.5){$\{\id,\ss^2\tau\}$}
\rput*(6.2,3.7){$\{\id,\ss,\ss^2\}$}
\rput(4.3,5.3){$\gal(\Q(\aa,\ww)/\Q)$}
}
\rput(0.2,0){
\rput(10.5,3){\BoxedEPSF{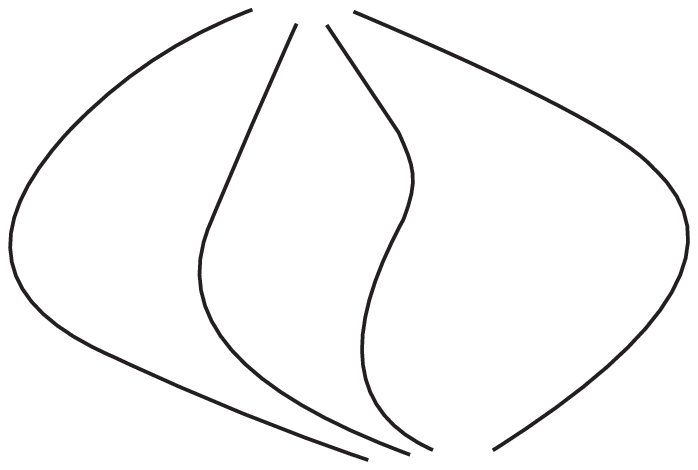 scaled 900}}
\rput(11.3,0.75){$\Q$}
\rput*(7.55,3.5){$F_1=\Q(\aa)$}
\rput*(9.3,3.5){$F_2$}
\rput*(11,3.5){$F_3$}
\rput*(13.3,2.3){$F_4$}
\rput(10,5.3){$\Q(\aa,\ww)$}
}
\end{pspicture}
\caption{The lattice of subgroups of $\gal(\Q(\aa,\ww)/\Q)$ with
  $\aa=\sqrt[3]{2}$ and $\ww=\frac{1}{2}+\frac{\sqrt{3}}{2}\imag$
  \emph{(left)} and the corresponding lattice of intermediate fields
  of the extension $\Q\subseteq\Q(\aa,\ww)$ \emph{(right)}.}
  \label{fig:galois:correspondence:example1}
\end{figure}

\paragraph{\hspace*{-0.3cm}}
Recall that a subgroup $N$ of a group $G$ is normal when $gNg^{-1}=N$
for all $g\in G$. This extra property possessed by normal subgroups
means they correspond to slightly special intermediate fields. 

Let $F\subseteq E$ be an extension with Galois group $\gal(E/F)$. Let
$F\subseteq K\subseteq E$ 
be an intermediate field and $\ss\in\gal(E/F)$. The image of $K$ by 
$\ss$ is another intermediate field, as on the left of Figure
\ref{fig:galois:correspondence:conjugate:subgroups}. Applying the
Galois correspondence gives subgroups $\gal(E/K)$ and $\gal(E/\ss(K))$ 
as on the right. Then:

\begin{proposition}\label{prop14.1}
$\gal(E/\ss(K))=\ss\gal(E/K)\ss^{-1}$
\end{proposition}

\begin{proof}
If $x\in\ss(K)$, then $x=\ss(y)$ for some $y\in K$. If $\tau\in\gal(E/K)$,
then $\ss\tau\ss^{-1}(x)=\ss\tau(y)=\ss(y)=x$,
so that $\ss\tau\ss^{-1}\in\gal(E/\ss(K))$,
giving $\ss\gal(E/K)\ss^{-1}\subseteq\gal(E/\ss(K))$. Replace $\ss$ by
$\ss^{-1}$ to get the reverse inclusion. 
\qed
\end{proof}

\begin{figure}
  \centering
\begin{pspicture}(0,0)(14,3)
\rput(-0.5,0){
\rput(4,1.5){\BoxedEPSF{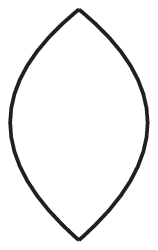 scaled 900}}
\rput*(4,2.5){$E$}
\rput*(3.35,1.5){$K$}
\rput*(4.65,1.5){$\ss(K)$}
\rput*(4,0.5){$F$}
}
\rput(0.5,0){
\rput(10,1.5){\BoxedEPSF{galois12.12.eps scaled 900}}
\rput*(10,2.5){$\gal(E/F)$}
\rput*(9.1,1.5){$\gal(E/K)$}
\rput*(10.9,1.5){$\gal(E/\ss(K))$}
\rput*(10,0.5){$\{\id\}$}
}
\psline[linewidth=0.9pt]{->}(5,1.5)(8.5,1.5)
\rput(6.75,1.75){Galois correspondence}
\rput(6.75,1.25){$X\mapsto\gal(E/X)$}
\end{pspicture}
\caption{}
  \label{fig:galois:correspondence:conjugate:subgroups}
\end{figure}

\begin{galoiscorrespondenceB}
Suppose we have the assumptions of the first part of the Galois
correspondence. If $K$ is an
intermediate field then
$\ss(K)=K$, for all $\ss\in\gal(E/F)$, if and only if $\gal(E/K)$ is a normal
subgroup of $\gal(E/F)$. In this case,
$$
\gal(E/F)/\gal(E/K)\cong \gal(K/F).
$$
\end{galoiscorrespondenceB}

\begin{proof}
If $\ss(K)=K$ for all $\ss$ then by Proposition \ref{prop14.1},
$\ss\gal(E/K)\ss^{-1}=\gal(E/\ss(K))=\gal(E/K)$ for all $\ss$, and so $\gal(E/K)$ is 
normal. Conversely, if $\gal(E/K)$ is normal then Proposition \ref{prop14.1} 
gives $\gal(E/\ss(K))=\gal(E/K)$ for all $\ss$, where $X\mapsto \gal(E/X)$ 
is a 1-1 map by the first part of the Galois correspondence. We thus have $\ss(K)=K$
for all $\ss$.

Define a map $\gal(E/F)\rightarrow\gal(K/F)$ by taking an automorphism $\ss$ of $E$
fixing $F$ pointwise and restricting it to $K$. We get an automorphism 
of $K$ as $\ss(K)=K$. The map is a homomorphism as the operation is composition 
in both groups. A $\ss$ is in the kernel if and only if
it restricts to the identity map on $K$ -- that is, fixes $K$
pointwise -- when restricted, which
happens if and only if $\ss$ is in $\gal(E/K)$. If $\ss$ is an automorphism of $K$
fixing $F$ pointwise then by Theorem F, it can be extended
to an automorphism of $E$ fixing $F$ pointwise. Thus any element of the Galois group
$\gal(K/F)$ can be obtained by restricting an element of $\gal(E/F)$
and the homomorphism is onto. The isomorphism follows by the first
isomorphism theorem. 
\qed
\end{proof}

\paragraph{\hspace*{-0.3cm}}
Here is a simple application:

\begin{proposition}
 Let $F\subseteq E$ be an extension satisfying the conditions of the
 Galois correspondence. If $F\subseteq K\subseteq E$ with $F\subseteq
 K$ an extension of degree two, then any $\ss\in\gal(E/F)$ sends $K$
 to itself.
\end{proposition}

Applying the Galois correspondence (part 1), the subgroup $\gal(E/K)$ has index
two in $\gal(E/F)$, hence is normal by Exercise
\ref{ex11.1}. Now apply the Galois correspondence (part 2).

\subsection*{Further Exercises for Section \thesection}

\emph{In all these exercises, you can assume that the condition (\dag)
  of \ref{galois:correspondence:condition} holds.}

\begin{vexercise}\label{exam00_4}
Let $\aa=\sqrt[4]{2}\in\R$ and $\imag\in\C$, and consider the field
$\Q(\aa,\imag)\subset\C$. 
\hspace{1em}\begin{enumerate}
\item
Show that there are automorphisms $\sigma,\tau$ of
$\Q(\aa,\imag)$ such that
$$
\sigma(\imag)=\imag,\sigma(\aa)=\aa\,\imag,\tau(\imag)=-\imag,\mbox{ and }\tau(\aa)=\aa.
$$
Show that
$$
G=\{\id,\ss,\ss^2,\ss^3,\tau,\ss\tau,\ss^2\tau,\ss^3\tau\},
$$
are then {\em distinct\/} automorphisms of $\Q(\aa,i)$.
Show that $\tau\ss=\ss^3\tau$.
\item
Show that $\gal(\Q(\aa,\imag)/\Q)=G$ and that the lattice $\LL(G)$ is as
on the left of Figure \ref{fig:galois:correspondence:exercise20}. 
\item Find the subgroups $H_1,H_2$ and $H_3$ of $G$.
If the corresponding lattice of subfields is as shown on the right,
then express the fields $F_1$ and $F_2$ in the form
$\Q(\beta_1,\ldots,\beta_n)$ for $\beta_1,\ldots,\beta_n\in\C$. 
\end{enumerate}
\end{vexercise}

\begin{figure}
  \centering
\begin{pspicture}(14,6)
\rput(-0.25,0){
\rput(4,3){\BoxedEPSF{galois12.11a.eps scaled 900}}
\rput*(4,5.6){$G$}
\rput*(1.5,3.8){$\{\id,\ss^2,\tau,\ss^2\tau\}$}
\rput*(4,3.8){$\{\id,\ss,\ss^2,\ss^3\}$}
\rput*(6.45,3.8){$H_1$}
\rput*(1.5,2.1){$H_2$}
\rput*(2.7,2.1){$\{\id,\tau\}$}
\rput*(4,2.1){$\{\id,\ss^2\}$}
\rput*(5.2,2.1){$\{\id,\ss^3\tau\}$}
\rput*(6.5,2.1){$H_3$}
\rput*(4,0.4){$\{\id\}$}
}
\rput(0.25,0){
\rput(10,3){\BoxedEPSF{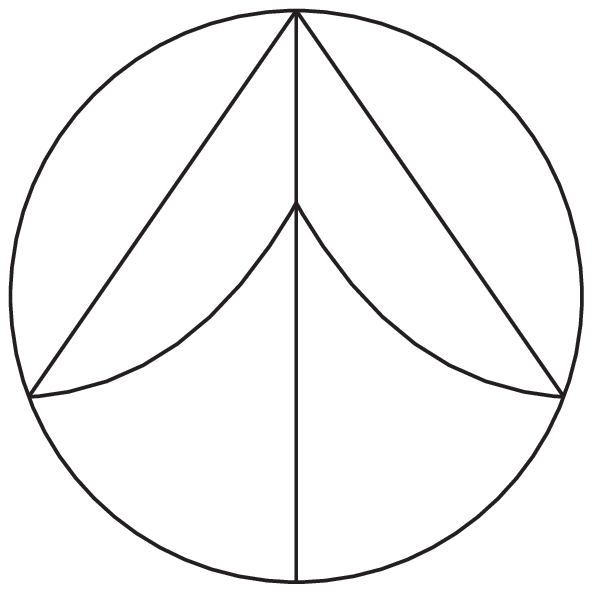 scaled 900}}
\rput*(10,5.6){$\Q(\aa,\imag)$}
\rput*(7.5,2.1){$F_2$}
\rput*(10,2.1){$\Q(\imag)$}
\rput*(12.5,2.1){$\Q(\imag\aa^2)$}
\rput*(7.5,3.8){$\Q(\aa\,\imag)$}
\rput*(8.7,3.8){$\Q(\aa)$}
\rput*(10,3.8){$F_1$}
\rput*(11.2,3.8){$\Q((1-\imag)\aa)$}
\rput*(12.85,3.8){$\Q((1+\imag)\aa)$}
\rput*(10,0.4){$\Q$}
}
%
%
\end{pspicture}
  \caption{Exercise \ref{exam00_4}: the lattice of subgroups of $\gal(\Q(\aa,\imag)/\Q)$ with
  $\aa=\sqrt[4]{2}$
  \emph{(left)} and the corresponding lattice of intermediate fields
  of the extension $\Q\subset\Q(\aa,\imag)$ \emph{(right)}.}
  \label{fig:galois:correspondence:exercise20}
\end{figure}

\begin{vexercise}\label{exam01_4}
\label{galois:correspondence:exercise20}
Let $\ww=\cos\frac{2\pi}{7}+\imag\sin\frac{2\pi}{7}\in\C$.
\begin{description}
\item[\hspace*{5mm}1.]
\parshape=2 0pt.8\hsize 0pt.8\hsize
Show that $\Q(\ww)$ is the splitting field of the polynomial
$$1+x+x^2+x^3+x^4+x^5+x^6.$$ 
and deduce that $|\gal(\Q(\ww)/\Q)|=6$.
\vadjust{\hfill\smash{\lower 30pt
\llap{
\begin{pspicture}(0,0)(2,4)
\rput(1,2){\BoxedEPSF{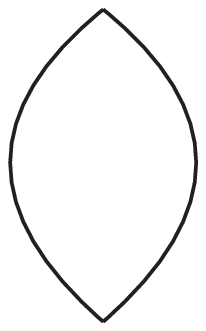 scaled 900}}
\rput*(1,3.3){$\Q(\ww)$}
\rput*(0.35,1.5){$F_1$}
\rput*(1.75,2.35){$F_2$}
\rput*(1,0.7){$\Q$}
\end{pspicture}
}}}\ignorespaces
Let $\ss\in\gal(\Q(\ww)/\Q)$ be such that $\ss(\ww)=\ww^3$.
Show that,
$$
\gal(\Q(\ww)/\Q)=\{\id,\ss,\ss^2,\ss^3,\ss^4,\ss^5\}.
$$
\item[\hspace*{5mm}2.]
\parshape=3 0pt\hsize 0pt\hsize 0pt\hsize
Using the Galois correspondence, show that 
the lattice of intermediate fields is as shown on the right,
where $F_1$ is a degree 2 extension of $\Q$ and $F_2$ a degree 3
extension. Find complex numbers $\beta_1,\ldots,\beta_n$ such that
$F_2=\Q(\beta_1,\ldots,\beta_n)$. 
\end{description}
\end{vexercise}

\begin{vexercise}
\label{galois:correspondence:exercise10}
Complete the lattice of intermediate fields from the example in
\ref{galois:correspondence:example1}: 
\begin{figure}[h]
  \centering
\begin{pspicture}(0,0)(14,5)
\rput(-3.6,-0.5){
\rput(10.5,3){\BoxedEPSF{galois12.10c.eps scaled 900}}
\rput(11.3,0.75){$\Q$}
\rput*(7.45,3.5){$\Q(\aa)$}
\rput*(9.,3.5){$\Q(\aa+\aa\ww)$}
\rput*(11.2,3.5){$\Q(\aa^2+\aa^2\ww)$}
\rput*(13.3,2.3){$\Q(\ww)$}
\rput(10,5.3){$\Q(\aa,\ww)$}
}
\end{pspicture}
\caption{The rest of the lattice of intermediate fields for the
  example in \ref{galois:correspondence:example1}}
  \label{fig:}
\end{figure}
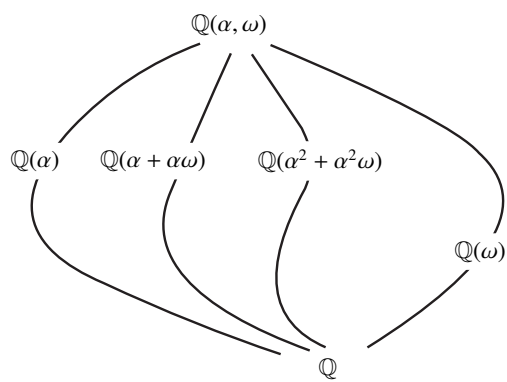
\end{vexercise}

\begin{vexercise}
\label{galois:correspondence:exercise30}
Let $\aa=\sqrt[6]{2}$ and $\ww=\frac{1}{2}+\frac{\sqrt{3}}{2}\imag$
and consider the field extension $\Q\subset \Q(\aa,\ww)$.
\begin{enumerate}
\item Find a basis for $\Q(\aa,\ww)$ over $\Q$ and show that $|\gal(\Q(\aa,\ww)/\Q)|=24$. 
\item Let $\ss,\tau\in\gal(\Q(\aa,\ww)/\Q)$ be such that 
$\tau(\aa)=\aa,\tau(\ww)=\ww^5$ and
$\ss(\aa)=\aa\ww,\ss(\ww)=\ww$.
Show that 
$$
H=\{\id,\ss,\ss^2,\ss^3,\ss^4,\ss^5,\tau,\tau\ss,\tau\ss^2,\tau\ss^3,\tau\ss^4,\tau\ss^5\},
$$
are then distinct elements in $\gal(\Q(\aa,\ww)/\Q)$.
\item Part of the subgroup lattice $\LL(G)$ is shown
on the left of Figure
\ref{fig:galois:correspondence:exercise35}. Fill in the corresponding
part of the lattice of intermediate 
fields on the right.
\end{enumerate}
\end{vexercise}

\begin{vexercise}\label{ex_lect10.2}
Let $\ww=\cos\frac{2\pi}{5}+\imag\sin\frac{2\pi}{5}$.
\begin{description}
\item[\hspace*{7mm}1.]
\parshape=2 0pt.8\hsize 0pt.8\hsize
Show that $\Q(\ww)$ is the splitting field of the polynomial
$1+x+x^2+x^3+x^4$ and deduce that $|\gal(\Q(\ww)/\Q)|=4$.
\vadjust{\hfill\smash{\lower 40pt
\llap{
\begin{pspicture}(0,0)(2,3)
\rput(1,1.5){\BoxedEPSF{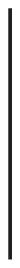 scaled 900}}
\rput*(1,2.75){$\Q(\ww)$}
\rput*(1,1.5){$F$}
\rput*(1,0.25){$\Q$}
\end{pspicture}
}}}\ignorespaces
\item[\hspace*{7mm}2.] Let $\ss\in\gal(\Q(\ww)/\Q)$ be such that $\ss(\ww)=\ww^2$.
Show that
$$
\gal(\Q(\ww)/\Q)=\{\id,\ss,\ss^2,\ss^3\}.
$$
Find the subgroup lattice $\LL(G)$ for
$G=\gal(\Q(\ww)/\Q)$.
\item[\hspace*{7mm}3.]
\parshape=2 0pt\hsize 0pt\hsize
Using the Galois correspondence, deduce that 
the lattice of intermediate fields is as shown on the right.
Find a complex number $\beta$ such that
$F=\Q(\beta)$. 
\end{description}
\end{vexercise}

\begin{vexercise}
\label{galois:correspondence:exercise50}
Consider the polynomial $f(x)=(x^2-2)(x^2-5)\in\Q[x]$. 
\begin{enumerate}
\item Show that $\Q(\kern-2pt\sqrt{2},\kern-2pt\sqrt{5})$ is the splitting field of $f$
over $\Q$ and that the Galois group $\gal(\Q(\kern-2pt\sqrt{2},$ $\kern-2pt\sqrt{5})/\Q)$ has order
four. (You can assume that if $a,b,c\in\Q$ satisfy
$a\kern-2pt\sqrt{2}+b\kern-2pt\sqrt{5}+c=0$ then $a=b=c=0$.)
\item Show that there are automorphisms $\ss,\tau$ of
$\Q(\kern-2pt\sqrt{2},\kern-2pt\sqrt{5})$ defined by
$\ss(\kern-2pt\sqrt{2})=-\kern-2pt\sqrt{2},\ss(\kern-2pt\sqrt{5})=\kern-2pt\sqrt{5}$ and
$\tau(\kern-2pt\sqrt{2})=\kern-2pt\sqrt{2},\tau(\kern-2pt\sqrt{5})=-\kern-2pt\sqrt{5}$. 
List the elements of the
Galois group $\gal(\Q(\kern-2pt\sqrt{2},\kern-2pt\sqrt{5})/\Q)$.
\item Complete the subgroup lattice on the left of Figure
  \ref{fig:galois:correspondence:exercise40} 
by listing the elements of $H$,
and use your answer to write the field $F$ in the form $\Q(\theta)$ for
some $\theta\in\C$.
\end{enumerate}
\end{vexercise}

\begin{figure}
  \centering
\begin{pspicture}(0,0)(14,6.5)
\rput(-0.5,0){
\rput(3.5,3.25){\BoxedEPSF{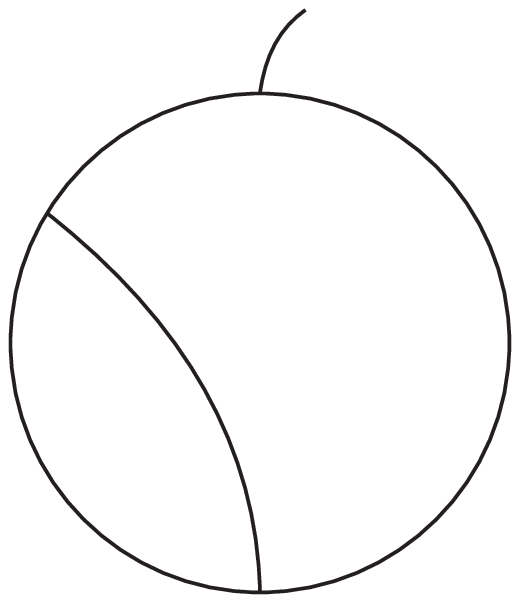 scaled 900}}
\rput*(3.5,5.1){$H$}
\rput*(1.5,4){$\{\id,\tau,\ss^3,\ss^3\tau\}$}
\rput*(5.8,3){$\{\id,\ss^2,\ss^4\}$}
\rput*(3.2,2){$\{\id,\ss^3\}$}
\rput*(1.5,2){$\{\id,\tau\}$}
\rput*(3.5,0.6){$\{\id\}$}
}
\rput(1,0){
\rput(10.5,3.25){\BoxedEPSF{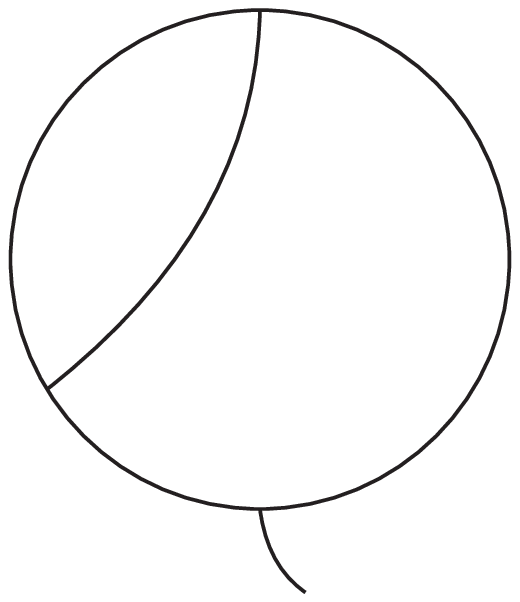 scaled 900}}
\rput*(10.5,5.9){$\Q(\aa,\ww)$}
\rput*(10.5,1.4){$F_1$}
\rput*(8.5,2.5){$F_2$}
\rput*(12.8,3.5){$F_3$}
\rput*(8.4,4.5){$F_4$}
\rput*(10.2,4.5){$F_5$}
}
\rput(0.25,3){
\psline[linewidth=0.9pt]{->}(5,2)(8.5,2)
\psline[linewidth=0.9pt]{<-}(5,1.5)(8.5,1.5)
\rput(6.75,1.75){Galois correspondence}
\rput(6.75,1.25){$X\mapsto\gal(E/X)$}
\rput(6.75,2.25){$Y\mapsto E^Y$}
}
\end{pspicture}
  \caption{Exercise \ref{galois:correspondence:exercise30}:
    $\aa=\sqrt[6]{2}$ and $\ww=\frac{1}{2}+\frac{\sqrt{3}}{2}\imag$} 
  \label{fig:galois:correspondence:exercise35}
\end{figure}


\section{Applications of the Galois Correspondence}
\label{galois.corresapps}

\subsection{Constructing polygons}

If $p$ is a prime number, then a regular $p$-gon can be constructed {\em only if\/}
$p$ is a Fermat prime of the form 
$$
2^{2^t}+1.
$$
This negative result was proved in Section \ref{ruler.compass2}, and required only the
degrees of extensions. We didn't need
any symmetries of fields.

Galois theory proper -- the interplay between fields and their Galois
groups -- allows us to prove positive results:

\begin{figure}
  \centering
\begin{pspicture}(0,0)(14,5)
\rput(-1,0){
\rput(3.5,2.5){\BoxedEPSF{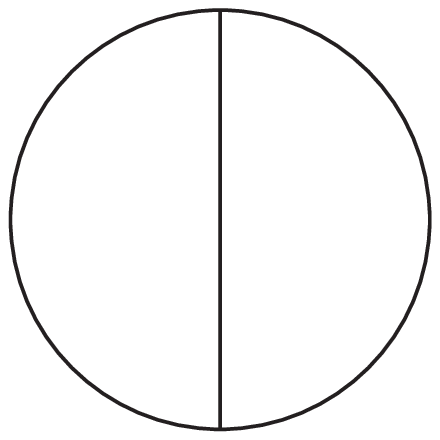 scaled 900}}
\rput*(3.5,4.4){$\gal(\Q(\kern-2pt\sqrt{2},\kern-2pt\sqrt{5})/\Q)$}
\rput*(1.6,2.5){$\{\id,\ss\}$}
\rput*(3.5,2.5){$\{\id,\tau\}$}
\rput*(5.4,2.5){$H$}
\rput*(3.5,0.6){$\{\id\}$}
}
\rput(1,0){
\rput(10.5,2.5){\BoxedEPSF{galois12.15.eps scaled 900}}
\rput*(10.5,4.4){$\Q(\kern-2pt\sqrt{2},\kern-2pt\sqrt{5})$}
\rput*(8.6,2.5){$\Q(\kern-2pt\sqrt{2})$}
\rput*(10.5,2.5){$\Q(\kern-2pt\sqrt{5})$}
\rput*(12.4,2.5){$F$}
\rput*(10.5,0.6){$\Q$}
}
\rput(0.25,0.75){
\psline[linewidth=0.9pt]{->}(5,2)(8.5,2)
\psline[linewidth=0.9pt]{<-}(5,1.5)(8.5,1.5)
\rput(6.75,1.75){Galois correspondence}
\rput(6.75,1.25){$X\mapsto\gal(E/X)$}
\rput(6.75,2.25){$Y\mapsto E^Y$}
}
\end{pspicture}
  \caption{Exercise \ref{galois:correspondence:exercise50}: subgroup
    and intermediate field lattice for the extension $\Q\subset\Q(\kern-2pt\sqrt{2},\kern-2pt\sqrt{5})$.} 
  \label{fig:galois:correspondence:exercise40}
\end{figure}

\begin{theorem}
\label{theorem:fermat_primes}
If $p=2^{2^t}+1$ is a Fermat prime then a regular $p$-gon can be constructed.
\end{theorem}

\begin{proof}
By Theorem E we need a tower of fields,
$$
\Q\subset K_1\subset\cdots\subset K_n=\Q(\zeta),
$$
where $\zeta=\cos(2\pi/p)+\imag\sin(2\pi/p)$ and $[K_i:K_{i-1}]=2$. 
We will get the tower by analysing the Galois group $\gal(\Q(\zeta)/\Q)$ and applying the
Galois correspondence.
As $\Q(\zeta)$ is the splitting field over $\Q$ of the  
$p$-th cyclotomic polynomial 
$$
\Phi_p(x)=x^{p-1}+x^{p-2}+\cdots+x+1,
$$ 
we have by Theorem G:
$$
|\gal(\Q(\zeta)/\Q)|=[\Q(\zeta):\Q]=\deg\Phi_p=p-1=2^{2^t}=2^n.
$$
The roots of $\Phi$ are the powers $\zeta^k$, and these all lie in
$\Q(\zeta)$. We can thus apply the Galois correspondence by
Proposition \ref{proposition:galois_extensions}.
In Section \ref{galois.groups} we
showed that $\gal(\Q(\zeta)/\Q)$ is a cyclic group, and so by Exercise
\ref{cyclicgroup.subgroups}, 
there is a chain of subgroups 
$$
\{\id\}=H_0\subset H_1\subset\cdots\subset H_n=\gal(\Q(\zeta)/\Q),
$$
where the subgroup $H_i$ has order $2^i$. 
Explicitly, if
$\gal(\Q(\zeta)/\Q)=\{g,g^2,\ldots,g^{2^{n-1}},g^{2^n}=\id\}$ then
$$
\{\id\}
\subset
\{h_1,h_1^2=\id\}
\subset
\{h_2,h_2^2,h_2^3,h_2^4=\id\}
\subset
\cdots
\subset
\{h_{n-1},h_{n-1}^2,\ldots,h_{n-1}^{2^{n-1}}=\id\}
\subset
\gal(\Q(\zeta)/\Q)
$$
where $h_i=g^{2^{n-i}}$ and $H_i$ is the subgroup generated by $h_i$
The Galois correspondence thus gives a chain of fields,
$$
\Q=K_0\subset K_1\subset\cdots\subset K_n=\Q(\zeta),
$$
where $K_{n-i}$ is the fixed field $E^{H_i}$ of the subgroup $H_i$. We
have $2^i=[G:H_{n-i}]=[K_i:\Q]$, so by the tower law
$$
2^i=[K_i:\Q]=[K_i:K_{i-1}][K_{i-1}:\Q]=[K_i:K_{i-1}]2^{i-1}
$$
and hence $[K_i:K_{i-1}]=2$ as desired.
\qed
\end{proof}

Theorem \ref{theorem:fermat_primes} and \ref{constructions2:pgons} then give:

\begin{corollary}
  If $p$ is a prime then a $p$-gon can be constructed if and only if
  $p=2^{2^t}+1$ is a Fermat prime. 
\end{corollary}

\begin{corollary}
If $n=2^kp_1p_2\ldots p_m$ with the $p_i$ Fermat primes, then a regular $n$-gon can be 
constructed.
\end{corollary}

\begin{proof}
A $2^k$-gon can be constructed by repeatedly bisecting angles, and thus
an $n$-gon, where $n$ has the form given, by Exercise \ref{ex7.50}.
\hfill$\Box$
\end{proof}

A little more Galois Theory, which we omit, gives the following
complete answer to what $n$-gons can be constructed:

\begin{theorem}
  An $n$-gon can be constructed if and only if $n=2^kp_1p_2\ldots p_m$ with the $p_i$ Fermat primes.
\end{theorem}

\paragraph{\hspace*{-0.3cm}}
The angle $\pi/n$ can be constructed precisely when the angle $2\pi/n$ can be constructed
which in turns happens precisely when the regular $n$-gon can be constructed. Thus, the list
of submultiples of $\pi$ that are constructible runs as,
$$
\frac{\pi}{2},\frac{\pi}{3},\frac{\pi}{4},\frac{\pi}{5},\frac{\pi}{6},\frac{\pi}{8},
\frac{\pi}{10},\frac{\pi}{12},\frac{\pi}{15},\ldots
$$

\begin{vexercise}
Give direct proofs of the non-constructability of the angles,
$$
\frac{\pi}{7},\frac{\pi}{9},\frac{\pi}{11}\text{ and }\frac{\pi}{13}.
$$
\end{vexercise}

\subsection{The Fundamental Theorem of Algebra}

We saw this in Section \ref{lect3}. We now prove it using the Galois
correspondence, starting with two observations:

\begin{description}
\item[(i).] \emph{There are no extensions
of $\R$ of odd degree $>1$}. Any polynomial in $\R[x]$ has roots that are
either real or occur in complex conjugate pairs, hence a real polynomial with
odd degree $>1$ has a real root and is reducible over $\R$. Thus, the minimum polynomial 
over $\R$ of any $\aa\not\in\R$ must have even degree so that the
degree $[\R(\aa):\R]$ is even. If $\R\subset L$ is an extension,
then for $\aa\in L\setminus\R$, we have
$$
[L:\R]=[L:\R(\aa)][\R(\aa):\R],
$$
is also even.
\item[(ii).] \emph{There is no extension of $\C$ of degree two}. For if $\C\subset L$ 
with $[L:\C]=2$ then an $\aa\in L\setminus\C$ gives the intermediate
$\C\subset \C(\aa)\subset L$ with $[\C(\aa):\C]=1$ or $2$ by the Tower
law. 
If this degree equals $1$ then $\aa\in\C$; thus $[\C(\aa):\C]=2$, and hence $L=\C(\aa)$. 
If $f$ is the minimum polynomial of $\aa$ over $\C$ then $f=x^2+bx+c$ for $b,c\in\C$
with $\aa$ one of the two roots
$$
\frac{-b\pm\sqrt{b^2-4c}}{2}.
$$
But these are both in $\C$, contradicting the choice of $\aa$. 
\end{description}

\begin{fundthmalg}
Any non-constant $f\in\C[x]$ has a root in $\C$.
\end{fundthmalg}

\begin{proof}
The proof toggles back and forth between intermediate fields and subgroups of Galois groups
using the Galois correspondence. All the fields and groups appear in
Figure \ref{fig:galois:correspondence:applications10}. 
If $f=pq$ is reducible over $\R$, 
then replace $f$ in what follows by $p$. Thus we may assume that $f$ is
irreducible over $\R$ and let $E$ 
be the splitting field over $\R$, not of $f$,  but of $(x^2+1)f$. 
We have 
$\R$ and $\pm\imag$ are in $E$, hence $\C$ is too, giving the series of
extensions $\R\subset\C\subseteq E$.  

Since $G=\gal(E/\R)$ is a finite group, we can factor from its order all
the powers of $2$, writing $|G|=2^km$, where $m\geq 1$ is odd. 
Sylow's Theorem then gives 
a subgroup $H$ of $G$ of order $2^k$, and the Galois correspondence
gives the intermediate field $F=E^H$ with the extension $F\subset E$ 
of degree $2^k$. As $[E:\R]=[E:F][F:\R]$ with $[E:\R]=|G|=2^km$, 
we have that $F$ is a degree $m$ extension of $\R$. 
As $m$ is odd and no such extensions exist if $m>1$, we must have
$m=1$, so that $|G|=2^k$.

Using the Galois correspondence in the reverse direction,
the subgroup $\gal(E/\C)$ has order dividing $|G|=2^k$, hence order
$2^s$ for some $0\leq s\leq k$. If $s>0$ then 
there is a non-trivial subgroup $K$ of $\gal(E/\C)$ 
of order $2^{s-1}$, 
with $2^{s-1}[E^H:\C]=[E:\C]=|\gal(E/\C)|=2^s$. 
Thus, $E^H$ is a degree $2$ extension of
$\C$, a contradiction to the second observation above. 
We thus have $s=0$,
hence $|\gal(E/\C)|=1$. We now have two
fields, $E$ and $\C$, that map via the 1-1 map $X\mapsto\gal(E/X)$ to
the trivial group.
The conclusion is that
$E=\C$. As $E$ is the splitting field of the polynomial $(x^2+1)f$, we get that $f$ has
a root (indeed {\em all\/} its roots) in $\C$.
\qed
\end{proof}

\begin{figure}
  \centering
\begin{pspicture}(0,0)(14,6)
\rput(-1.5,-0.5){
\rput(3.5,3.25){\BoxedEPSF{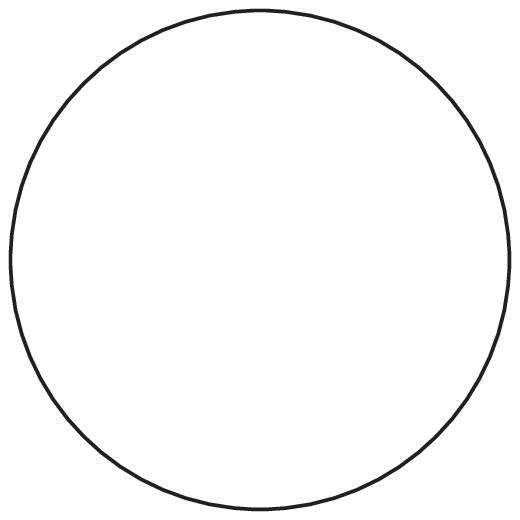 scaled 900}}
\rput*(3.5,5.5){$E$}
\rput*(1.4,4){$E^K$}
\rput*(5.8,3){$F=E^H$}
\rput*(1.5,2){$\C$}
\rput*(3.5,1){$\R$}
\rput(5.15,4.5){$2^k$}
\rput(2.5,5){$2^{s-1}$}
}
\rput(1.5,-0.5){
\rput(10.5,3.25){\BoxedEPSF{galois13.10.eps scaled 900}}
\rput*(10.5,5.5){$G=\gal(E/\R)$}
\rput*(10.5,1){$\{\id\}$}
\rput*(8.5,2){$K$}
\rput*(12.8,3){$H$}
\rput*(8.4,4){$\gal(E/\C)$}
\rput(12.15,2){$2^k$}
\rput(9.55,1.5){$2^{s-1}$}
}
\rput(0.25,1.25){
\psline[linewidth=0.9pt]{->}(5,2)(8.5,2)
\psline[linewidth=0.9pt]{<-}(5,1.5)(8.5,1.5)
\rput(6.75,1.75){Galois correspondence}
\rput(6.75,2.25){$X\mapsto\gal(E/X)$}
\rput(6.75,1.25){$Y\mapsto E^Y$}
}
\end{pspicture}
  \caption{Using the Galois correspondence to prove the Fundamental
    Theorem of Algebra.} 
  \label{fig:galois:correspondence:applications10}
\end{figure}


\section{(Not) Solving Equations}
\label{solving.equations}

We can finally return to the theme of Section \ref{lect1}: finding algebraic
expressions for the roots of polynomials.

\paragraph{\hspace*{-0.3cm}}
The formulae for the roots of quadratics, cubics and quartics express the roots in terms
of the coefficients, the four field operations $+,-,\times,\div$ and
$n$-th roots
$\sqrt{},\sqrt[3]{},\sqrt[4]{}$.
These roots thus lie
in an extension of $\Q$
obtained by adjoining certain $n$-th roots.

\begin{definition}[radical extension of $\Q$]
An extension $\Q\subset E$ is radical when
there is a sequence of simple extensions,
$$
\Q\subset \Q(\aa_1)\subset \Q(\aa_1,\aa_2)\subset \cdots\subset
\Q(\aa_1,\aa_2,\ldots,\aa_k)=E,
$$
with some power $\aa_i^{m_i}$ of $\aa_i$ contained in
$\Q(\aa_1,\aa_2,\ldots,\aa_{i-1})$ for each $i$.   
\end{definition}

Each extension in the sequence is thus obtained by adjoining to the
previous field in the sequence, the 
$m_i$-th root of some element.

A simple example:
$$
\Q\subset\Q(\sqrt{2})\subset\Q(\sqrt{2},\sqrt[3]{5})
\subset\Q\biggl(\sqrt{2}, \sqrt[3]{5},\sqrt{\sqrt{2}-7\sqrt[3]{5}}\biggr).
$$
By repeatedly applying Theorem D, the elements of a radical extension
are seen to
have expressions in terms of rational numbers, $+,-,\times,\div$ 
and $\sqrt[n]{}$ for various $n$.

\begin{definition}[polynomial solvable by radicals]
A polynomial 
$f\in\Q[x]$ is solvable by radicals when its splitting field over $\Q$
is contained in some radical extension. 
\end{definition}

Notice that we are dealing with a fixed specific polynomial, and not an arbitrary one.
The radical extension containing the splitting field will depend on the polynomial. 

\paragraph{\hspace*{-0.3cm}}
Any quadratic polynomial $ax^2+bx+c$ is solvable by radicals, with splitting field
in the radical extension 
$$
\Q\subseteq \Q(\sqrt{b^2-4ac}).
$$
Similarly, the formulae for the roots of cubics and quartics give for any specific such 
polynomial, radical extensions containing their splitting fields. 

\paragraph{\hspace*{-0.3cm}}
Recalling the definition of soluble group given in Section \ref{groups.stuff}:

\begin{theoremH}
A polynomial $f\in\Q[x]$ is solvable by radicals if and only if its Galois group
over $\Q$ is soluble.
\end{theoremH}

The proof, which we omit,  uses the full power of the Galois correspondence, with the sequence of 
extensions in a radical extension corresponding to the sequence of
subgroups 
$$
\{1\}=H_0\lhd H_1\lhd \cdots \lhd H_{n-1}\lhd H_n=G,
$$
in a soluble group.

\paragraph{\hspace*{-0.3cm}}
As a small reality check of Theorem H, we saw in Section \ref{galois.groups} that the
Galois group over $\Q$ of a quadratic polynomial is either the trivial
group $\{\id\}$ or the (Abelian) permutation group $\{\id,(\aa,\bb)\}$
where $\aa,\bb\in\C$ are the roots. 
Abelian groups are soluble -- see \ref{groups1:abelian_are_soluble} --
and this syncs with quadratics 
being solvable by radicals via the quadratic formula.

Similarly, the possible Galois groups of cubic
polynomials are shown in Figure
\ref{fig:groups2:subgroup_lattice_S_3}. Apart from $S_{\kern-.3mm 3}$,
these are also Abelian. But $S_{\kern-.3mm 3}$ is the symmetry group
of an equilateral triangle lying in the plane -- soluble by
\ref{groups1:dihedral_are_soluble}. 

\paragraph{\hspace*{-0.3cm}}
Somewhat out of chronological order, we have:

\begin{theorem}[Abels-Fubini]
The polynomial $f=x^5-4x+2$ is not solvable by radicals.
\end{theorem}

The roots of $x^5-4x+2$ are algebraic numbers, yet there is no
algebraic expression for them. 

\begin{proof}
We show that the Galois group of $f$ over $\Q$ is insoluble. Indeed, we show that
the Galois group is the symmetric group $S_{\kern-.3mm 5}$, which contains the
non-Abelian, finite simple group $A_5$. 
Thus $S_{\kern-.3mm 5}$ contains an insoluble subgroup, hence is insoluble, as any
subgroup of a soluble group is soluble by Exercises \ref{subgroups.solublegroups1}
and \ref{subgroups.solublegroups2}.

If $E$ is the splitting field over $\Q$ of $f$,
then
$$
E=\Q(\aa_1,\aa_2,\aa_3,\aa_4,\aa_5),
$$
where the $\aa_i\in\C$ are the roots of $f$ and the Galois group 
is $\gal(E/\Q)$, itself a subgroup of the group of permutations of
$\{\aa_1,\ldots,\aa_5\}$ -- which is $\cong S_{\kern-.3mm 5}$. 

\begin{figure}
  \centering
\begin{pspicture}(14,10)
\rput(1,0.5){
\rput(5,4.5){\BoxedEPSF{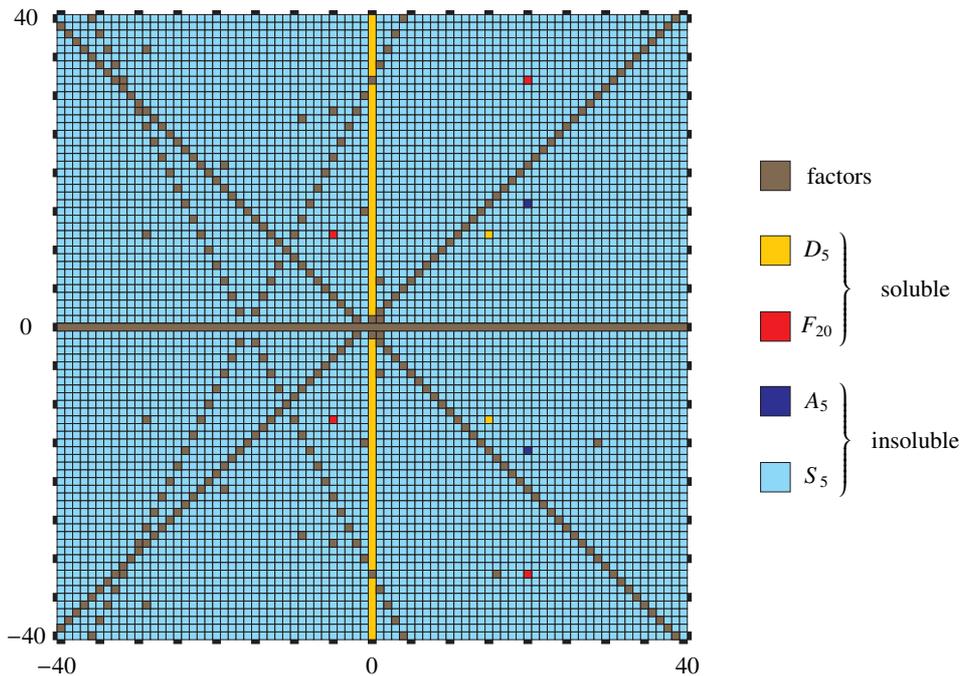 scaled 500}}
\rput(10.2,2.5){$S_5$}
\rput(10.2,3.5){$A_5$}
\rput(10.2,5.5){$D_5$}
\rput(10.2,4.5){$F_{20}$}
\rput(10.5,6.5){factors}
\rput(10.5,3){$\left.\begin{array}{c}
\vrule width 0 mm height 15 mm depth 0 pt\end{array}\right\}$}
\rput(11.5,3){insoluble}
\rput(10.5,5){$\left.\begin{array}{c}
\vrule width 0 mm height 15 mm depth 0 pt\end{array}\right\}$}
\rput(11.5,5){soluble}
\rput(-.2,4.5){$0$}\rput(-.2,.4){$-40$}\rput(-.2,8.6){$40$}
\rput(4.35,0){$0$}\rput(.2,0){$-40$}\rput(8.5,0){$40$}
}
\end{pspicture}
  \caption{The Galois groups of the quintic polynomials $x^5+ax+b$ for
    $-40\leq a,b\leq 40$ (re-drawn from the \emph{Mathematica\/} poster,
``Solving the Quintic'').}
  \label{fig:solving:quintic}
\end{figure}

\parshape=3 0pt\hsize 0pt\hsize 0pt.6\hsize
The polynomial $f$ is irreducible over $\Q$ by Eisenstein, hence is
the minimum polynomial of $\aa_1$ over $\Q$. The extension $\Q\subset \Q(\aa_1)$
thus has degree five, and the Tower law gives 
\vadjust{\hfill\smash{\lower 125pt
\llap{
\begin{pspicture}(0,0)(4.5,3.5)
\rput(0,0.4){
\rput(0,0){
\rput{-90}(2,1.5){\BoxedEPSF{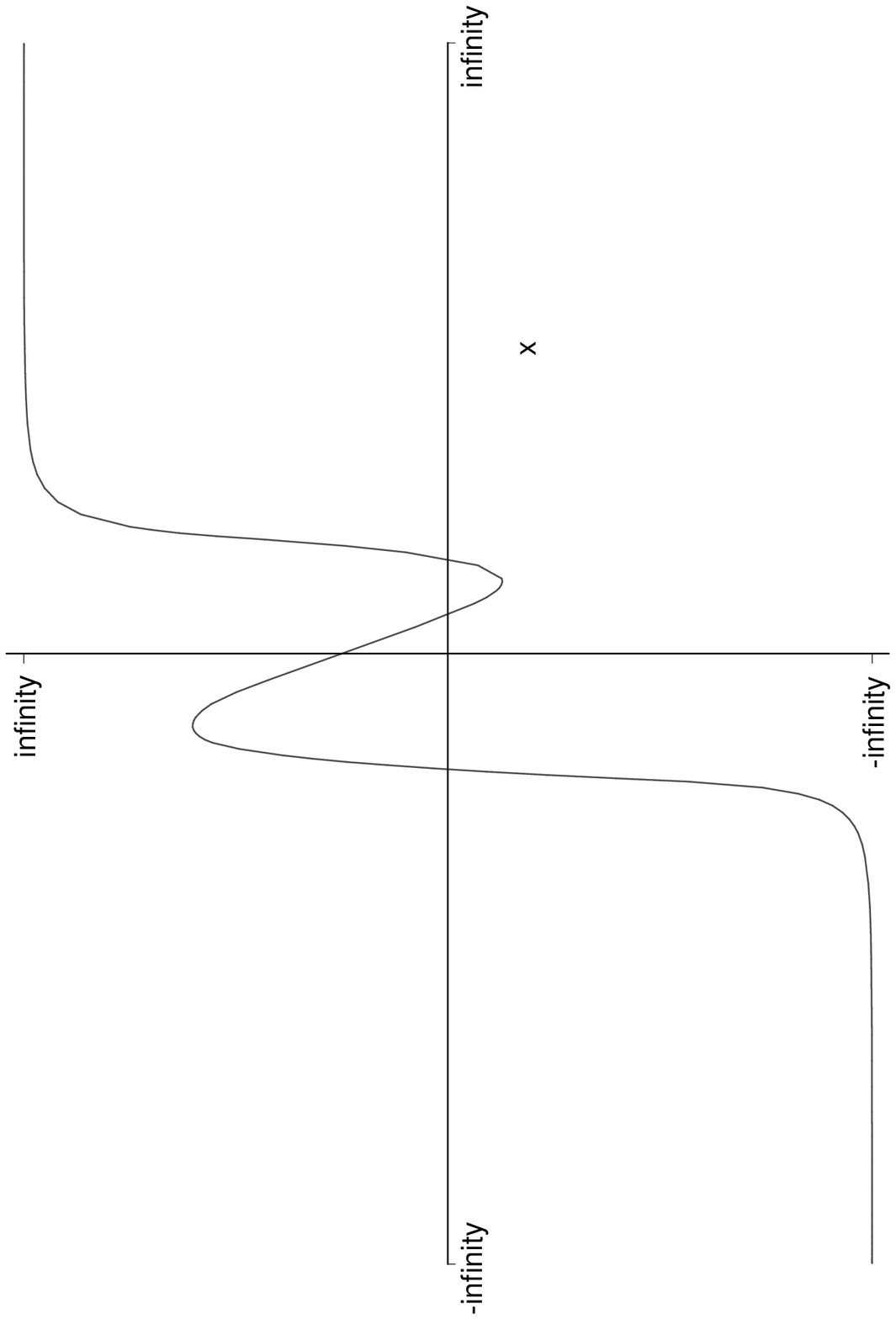 scaled 250}}
}
\psframe*[linecolor=white](1.3,3)(1.9,3.4)
\psframe*[linecolor=white](1.8,-0.3)(2.2,.2)
\psframe*[linecolor=white](3,1)(3.3,1.3)
\rput*(2,3.1){$\infty$}
\rput*(1.9,-0.1){$-\infty$}
\rput*(-.5,1.5){$-\infty$}
\rput*(4.2,1.5){$\infty$}
}
\end{pspicture}
}}}\ignorespaces
$$
[E:\Q]=[E:\Q(\aa_1)][\Q(\aa_1):\Q].
$$
\parshape=6 0pt.6\hsize 0pt.6\hsize 0pt.6\hsize 0pt.6\hsize 0pt.6\hsize 0pt.6\hsize
The degree of the extension $\Q\subset E$ is therefore divisible by the degree
of the extension $\Q\subset\Q(\aa_1)$, ie: divisible by five. 
Moreover, by Theorem G, the group $\gal(E/\Q)$ has order the degree 
$[E:\Q]$, and so the group has order divisible by five. By Cauchy's Theorem, 
the Galois group contains an element $\ss$ of order
$5$, and a subgroup
$$
\{\id,\ss,\ss^2,\ss^3,\ss^4\},
$$
\parshape=2 0pt\hsize 0pt\hsize 
where the permutation $\ss$ is a $5$-cycle 
$\ss=(a,b,c,d,e)$ when
considered as a permutation of the roots. 
The graph
of $f$ on the right shows that three of the roots are real, and the other two
are thus complex conjugates. By Exercise
\ref{exercise:groups2:conjugation}, complex conjugation is 
an element of the Galois group having effect the permutation
$$
\tau=(b_1,b_2),
$$
where $b_1,b_2$ are the two complex roots. But in
Section \ref{groups.stuff} we saw that $S_{\kern-.3mm n}$ is generated by a $n$-cycle and a
transposition, hence the Galois group is $S_{\kern-.3mm 5}$ as claimed.
\qed
\end{proof}

\paragraph{\hspace*{-0.3cm}}
There is nothing particularly special about the polynomial $x^5-4x+2$;
among the polynomials having degree $\geq 5$, those that are not
solvable by radicals are \emph{generic\/}.
We illustrate what we mean with some experimental evidence: consider the quintic polynomials
$$
x^5+ax+b,
$$
for $a,b\in\Z$ with $-40\leq a,b\leq 40$. 

Figure \ref{fig:solving:quintic}
(which is re-drawn from the Mathematica poster,
``Solving the Quintic'')
shows the $(a,b)$ plane for $a$ and $b$ in this range. 
The vertical line through $(0,0)$ corresponds to $f$ with Galois group
the soluble dihedral
group $D_{10}$ of order $10$. The horizontal line through $(0,0)$ and the two sets
of crossing diagonal lines correspond to reducible $f$, as do a few other isolated points.
The (insoluble) alternating group $A_5$ arises in a few sporadic
places, as does another soluble
subgroup of $S_5$. The vast majority of $f$ however, forming the light background, have
Galois group the symmetric group $S_5$, and so have roots that are {\em algebraic\/}, but
cannot be expressed {\em algebraically}.








\begin{biblist}
\bib{MR965514}{book}{
   author={Armstrong, M. A.},
   title={Groups and symmetry},
   series={Undergraduate Texts in Mathematics},
   publisher={Springer-Verlag, New York},
   date={1988},
   pages={xii+186},
   isbn={0-387-96675-7},
   review={\MR{965514}},
   doi={10.1007/978-1-4757-4034-9},
}
\end{biblist}

\begin{biblist}

\bib{Rotman90}{book}{
   author={Rotman, Joseph},
   title={Galois theory},
   series={Universitext},
   publisher={Springer-Verlag, New York},
   date={1990},
   pages={xii+108},
   isbn={0-387-97305-2},
   review={\MR{1064318}},
   doi={10.1007/978-1-4684-0367-1},
}

\end{biblist}
\end{document}